\documentclass[10pt,reqno]{amsart}
\usepackage{url}
\usepackage{amsthm}
\usepackage{fullpage}
\usepackage{amsmath}
\usepackage{amssymb}
\usepackage{amsfonts}
\DeclareMathOperator{\spn}{span}
\usepackage{enumitem}
\usepackage{mathscinet}
\usepackage{lipsum}
\usepackage[english]{babel}
\usepackage[autostyle]{csquotes}
\usepackage{comment}
\usepackage{footmisc}
\usepackage{color}

\usepackage{graphicx}

\usepackage{hyperref}

\hypersetup{
    colorlinks=true,
    linkcolor=red,
    citecolor=blue,
    }

\let\ul\underline

\def\Jac{\ensuremath{\mathrm{Jac}}}
\def\Vol{\ensuremath{\mathrm{Vol}}}

\newtheorem{thm}{Theorem}[section]
\newtheorem{definition}[thm]{Definition}
\newtheorem{theorem}[thm]{Theorem}
\newtheorem*{oseledec*}{Oseledec Theorem}
\newtheorem{cor}[thm]{Corollary}
\newtheorem{claim}[thm]{Claim}
\newtheorem{lemma}[thm]{Lemma}
\newtheorem{prop}[thm]{Proposition}

\makeatletter
\def\moverlay{\mathpalette\mov@rlay}
\def\mov@rlay#1#2{\leavevmode\vtop{%
   \baselineskip\z@skip \lineskiplimit-\maxdimen
   \ialign{\hfil$\m@th#1##$\hfil\cr#2\crcr}}}
\newcommand{\charfusion}[3][\mathord]{
    #1{\ifx#1\mathop\vphantom{#2}\fi
        \mathpalette\mov@rlay{#2\cr#3}
      }
    \ifx#1\mathop\expandafter\displaylimits\fi}
\makeatother

\newcommand{\RST}{\mathrm{RST}}
\newcommand{\ST}{\mathrm{ST}}

\newcommand{\bigcupdot}{\charfusion[\mathop]{\bigcup}{\cdot}}

\DeclareFontFamily{U}{mathx}{\hyphenchar\font45}
\DeclareFontShape{U}{mathx}{m}{n}{
      <5> <6> <7> <8> <9> <10>
      <10.95> <12> <14.4> <17.28> <20.74> <24.88>
      mathx10
      }{}
\DeclareSymbolFont{mathx}{U}{mathx}{m}{n}
\DeclareFontSubstitution{U}{mathx}{m}{n}
\DeclareMathAccent{\widecheck}{0}{mathx}{"71}
\DeclareMathAccent{\wideparen}{0}{mathx}{"75}

\begin{document}

\title{summable orbits}
\author{Snir Ben Ovadia}
\address{Department of Mathematics, Eberly College of Science, Pennsylvania State University} \email{sfb5896@psu.edu}

\date{}

\maketitle

\begin{abstract}
We introduce a class of orbits which may have $0$ Lyapunov exponents, but still demonstrate some sensitivity to initial conditions. We construct a countable Markov partition with a finite-to-one almost everywhere induced coding, and which lifts the geometric potential with summable variations (for a $C^{1+}$ diffeomorphism of a closed manifold of dimension $\geq2$). An important tool we use is a shadowing theory for orbits which may have $0$ Lyapunov exponents. 

We construct (weak) stable and unstable leaves for such orbits using a Graph Transform method, and prove the absolute continuity of these foliations w.r.t holonomies. In particular, we discuss setups where these foliations exist, and are strictly weak- i.e do not demonstrate exponential contraction. One example is a family of non-uniformly hyperbolic diffeomorhpims where we are able to simultaneously code all invariant measures in a finite-to-one almost everyhwere fashion.

   
   
   
\end{abstract}

\tableofcontents
\section{Introduction}
Given a dynamical system, symbolic dynamics are a method of recasting the dynamics in a way where the iterates of the system become easier to study. This approach serves as an important tool in many areas in the field of dynamical systems. In particular, constructions such as Markov partitions (finite or countable) have had many uses in smooth dynamics and in Pesin theory.

The history of Markov partitions of hyperbolic systems goes back to \cite{Aw,AW70}, where Adler and Weiss constructed a finite Markov partition for hyperbolic toral automorphisms. The uniformly hyperbolic case was done in \cite{Si1,Si2} and in \cite{B1}. The non-uniformly hyperbolic case raises many new difficulties which don't exist in the uniformly hyperbolic case, and Sarig was first able to overcome those in \cite{Sarig}, in the context of $C^{1+\beta}$ closed surface diffeomorphisms.

Sarig's methods were later extended in several contexts- maps with singularities \cite{LM16}, flows \cite{SL14}, $C^{1+\beta}$ maps of closed manifolds of dimension $\geq 2$ \cite{SBO}, non-invertible maps \cite{Lima1D, LimaNonInvetHighDim}, and injective semi-conjugacies \cite{InjectiveBuzzi}. In addition, an improved construction in \cite{LifeOfPi} allows to identify the set of points which are covered by the Markov partition. 

The engine which allows Sarig's methods to work is Pesin theory. Namely, given a non-uniformly hyperbolic orbit, one can apply a certain local change of coordinates, where the action of the differential becomes a hyperbolic block-form matrix (plus a smaller order error term). Fixing a $\chi>0$, and a non-uniformly hyperbolic orbit, this approach allows one to study the iterates of the differential on the orbit as iterates of this block-form matrix, whose blocks have bounded norms, but one contracts with rate at least $e^{\chi}$, and one expands in rate at least $e^{\chi}$. 

This approach therefore produces uniform estimates in charts (i.e after the local change of coordinates). These uniform estimates in charts are the main tool for carrying out Pesin's stable manifold theorem, which is then used in Sarig's work to construct a coding. 

The main down-side of these approaches mentioned above and the coding in the non-uniformly hyperbolic contexts which are mentioned, is the fact that one has to restrict to code only $\chi$-hyperbolic orbits, for some fixed $\chi>0$.\footnote{Trying to capture a decreasing sequence of $\chi_n$-hyperbolic sets admits an infinite-to-one coding, which loses many of its advantages as a recasting of the dynamics.} This down-side becomes more evident when one wishes to capture some diffusing phenomena with the coding, in oppose to coding a single measure. For example, let $f\in \mathrm{Diff}^{1+\beta}(M)$ where $M$ is a closed Riemannian manifold, and let $\varphi:=\log\Jac(d_xf^{-1}|_{E^u(x)})$ be the geometric potential which is defined for orbits of hyperbolic points. If one wishes to code a sequence of hyperbolic measures $\mu_n$, where $h_{\mu_n}(f)+\int \varphi d\mu_n\to0$, while the entropy of the measures in this sequence deteriorates to $0$, then their Lyapunov exponents must deteriorate as well, and so they cannot be captured by a single coding for any $\chi>0$. This phenomenon cannot be studied symbolically with those codings.

While this, and natural curiosity, motivate the studying of symbolic codings which capture measures with deteriorating hyperbolicity, it raises several difficulties. The most obvious difficulty is the fact that we could no longer assume uniform estimates in charts, as we do not have a uniform $\chi>0$. Heuristically, one can think of it as ``non-uniform non-uniform hyperbolicity". A natural naive approach in overcoming these difficulties, is assigning each chart $\psi$ a corresponding $\chi_\psi$, which will correspond to the ``local Lyapunov exponent", and carry out such a change of coordinates. However, this approach is conceptually flawed, since Lyapunov exponents are a property of the tail of an orbit (or chain), and not of a point (or symbol); this approach will break the Markovian structure one wishes to construct. Therefore, an effort to code simultaneously (in a finite-to-one fashion) hyperbolic measures with deteriorating Lyapunov exponents must identify some underlying structure which is common to all such measures. 

A corollary of this observation, is that upon a successful construction of a Markovian coding of measures with deteriorating hyperbolicity, one can tailor orbits whose Lyapunov exponents are $0$ (as it is a tail property), but which nonetheless demonstrate sub-exponential sensitivity to initial conditions. By the nature of the coding, one has to carry out Graph Transform constructions which preserve a certain suitable space of admissible manifolds, to carry out the construction of local (weak) stable and unstable manifolds; while the estimates in charts cannot be held uniform over a chain or an orbit  (this is treated in \textsection \ref{GTrans}, see \textsection \ref{SectionOverview} for its use in the construction of the Markov partition). In particular, this implies that even the orbits that we can construct with $0$ Lyapunov exponents admit local stable or unstable leaves which are strictly weak (i.e exhibit strictly sub-exponential contraction). These local manifolds exist, yet were unseen before due to their lack of hyperbolic behavior. A natural question to ask would be, are the foliations by such weak manifolds absolutely continuous w.r.t holonomies? In \textsection \ref{chapter7} we show that the answer is yes.

The tradeoff in losing the uniform estimates in charts is the fact that one has to require a more restrictive ``temperability" property of the orbit. Temperability refers to the asymptotic growth-rate of the norm of the inverse change of coordinates. A classical result in Pesin theory is the Pesin Tempered Kernel Lemma, which implies that for almost every orbit (for every hyperbolic probability measure), this growth-rate is sub-exponential. The uniform exponential contraction in charts therefore overpowers it, and allows for Pesin's constructions to work. When one gives up the uniform exponential contraction in charts, one has to find also a more refined temperability property. In \textsection \ref{codability} we present a condition which is sufficient to show the temperability that this coding requires. We also mention a family of examples where this condition can be tested and allows us to simultaneously code all invariant probabilities of the system in a finite-to-one fashion, and also code Lebesgue almost every point in some cases (when there is no invariant density). 

The search for a suitable refined temperability property (which is sufficient, but not too restrictive), reveals an interesting relationship between the temperability properties of an orbit, and the asymptotic contraction rates on stable/unstable manifolds. This allows for the study of rates of contraction of strictly weak stable manifolds, through the study of the temperability properties of the orbit. We discuss this at \textsection \ref{duality}.

When studying this relationship, one can see that positive recurrence of orbits (which holds almost everywhere w.r.t every invariant probability measure) implies hyperbolicity. This does not contradict the existence of strictly weak foliations, which may carry infinite (conservative, $\sigma$-finite, ergodic, invariant) measures with sensitivity to initial conditions. Through temperability properties, one can study asymptotic contraction on weak stable/unstable leaves, which may be strictly sub-exponential. The question of temperability of a general hyperbolic probability measure, and possibly infinite measures carried by weak leaves, is as interesting as it is difficult. We hope that this work will provide some framework which allows to study these questions.

\subsection{Main results}
Let $\mathcal{G}$ be a directed graph with a countable collection of vertices $\mathcal{V}$ s.t every vertex has at least one ingoing and one outgoing edge. The {\em topological Markov shift} (TMS) associated to $\mathcal{G}$ is the set
$$\Sigma=\Sigma(\mathcal{G}):=\{(v_i)_{i\in\mathbb{Z}} \in\mathcal{V}^\mathbb{Z} :v_i \rightarrow v_{i+1}\text{  }, \forall i\in\mathbb{Z}\},$$
equipped with the left-shift $\sigma:\Sigma\rightarrow\Sigma$, $\sigma((v_i)_{i\in\mathbb{Z}})=(v_{i+1})_{i\in\mathbb{Z}}$, and the metric 
$d(u,v):=\exp(-\min\{n\in\mathbb{N}_0:u_n\neq v_n\text{ or }u_{-n}\neq v_{-n}\})$ for $u=(u_n)_{n\in\mathbb{Z}},$ $v=(v_n)_{n\in\mathbb{Z}}$. With this metric $\Sigma$ is a complete separable metric space. $\Sigma$ is compact iff $\mathcal{G}$ is finite. $\Sigma$ is locally compact iff every vertex of $\mathcal{G}$ has finite ingoing and outgoing degree.

A {\em Markovian subshift} of $\Sigma(\mathcal{G})$ is a subset of the form $\Sigma(\mathcal{G}')$, where $\mathcal{G}'$ is a subgraph of $\mathcal{G}$.

\medskip
The following definition is due to Sarig: given a TMS $\Sigma$,
$$\Sigma^\#:=\{(v_i)_{i\in\mathbb{Z}}\in\Sigma:\exists v',w'\in\mathcal{V}(\Sigma)\text{ , }\exists n_k,m_k\uparrow\infty\text{ s.t }v_{n_k}=v'\text{ , }v_{-m_k}=w',\text{ }\forall k\in\mathbb{Z}\}.$$
Notice that by the Poincar\'e recurrence theorem every $\sigma$-invariant probability measure is carried by $\Sigma^\#$. Furthermore, every periodic point of $\sigma$ is in $\Sigma^\#$. 


Let $f$ be a $C^{1+\beta}$ diffeomorphism on a compact smooth boundaryless manifold $M$ of dimension greater than $1$. Here $\beta\in (0,1)$ is the H\"older exponent of $df$.


A point $x\in M$ is said to be {\em $0$-summable} (see Definition \ref{summability}) if its tangent space decomposes uniquely into $T_xM=H^s(x)\oplus H^u(x)$ where $\forall \xi^s\in H^s(x), \xi^u\in H^u(x)$,
	\begin{align*}
	\sum_{m\geq0}|d_xf^m\xi ^s |^2<\infty,\text{ }	\sum_{m\geq0}|d_xf^{-m}\xi ^u |^2<\infty.
	\end{align*}
	
This splitting allows for the definition of a linear change of coordinates $C_0^{-1}(x):T_xM\to\mathbb{R}^d$ (which depends measurably on $x$) s.t 	$\forall \xi^s\in H^s(x), \xi^u\in H^u(x) $
$$|C_0^{-1}(x)\xi^s|^2=2\sum_{m\geq0}|d_xf^m\xi ^s |^2,\text{ }|C_0^{-1}(x)\xi^u|^2=2\sum_{m\geq0}|d_xf^m\xi ^u |^2.$$
See Theorem \ref{pesinreduction}. This linear change of coordinates is defined up to an orthogonal self-map of $H^s(x)$, $H^s(x)$, but this choice does not affect $\|C_0^{-1}(x)\|$.

For any choice of $\Gamma>0$ and $\gamma>\frac{5}{\beta}$,   let $I(x):=xe^{\Gamma x^\frac{1}{\gamma}}$. Let $\mathcal{I}:=\{I^\frac{-\ell}{4}(1)\}_{\ell\geq0}$, where $I^\frac{-1}{4}$ is a 4th iterative root of the inverse of $I$ (see Definition \ref{ladderFunc} and Lemma \ref{forI}).

A $0$-summable point $x$ is called {\em strongly temperable} (w.r.t $\beta,\Gamma,\gamma$,  and $\epsilon>0$) if  there is a function $q:\{f^n(x)\}_{n\in\mathbb{Z}}\rightarrow \mathcal{I}$ s.t
\begin{enumerate}
	\item $q\circ f=I^{\pm1}(q)$,
	\item $\forall n\in\mathbb{Z}$, $q(f^n(x))\leq \frac{C_{\beta,\epsilon}}{\|C_0^{-1}(f^n(x))\|^{2\gamma}}$,
\end{enumerate}
where $C_{\beta,\epsilon}>0$ is a constant depending only $\epsilon$ and $\beta$ (see Definition \ref{calibrationParams}). If in addition to (1) and (2), $q:\{f^n(x)\}_{n\in\mathbb{Z}}\rightarrow \mathcal{I}$ can be chosen to also satisfy,

\begin{enumerate}
\item[(3)]$\limsup\limits_{n\rightarrow\pm\infty}q(f^n(x))>0$,
\end{enumerate}
we say that $x$ is {\em recurrently strongly temperable}. Denote by $\RST=\RST(\beta,\gamma,\Gamma,\epsilon)$ the set of all recurrently strongly temperable points, which is clearly an invariant set (see Definition \ref{strongtemp}).

\begin{theorem}\label{t4.1.1} For every $\epsilon>0$ sufficiently small (w.r.t $\beta,M,f,\Gamma,$ and $\gamma$) there exist a locally compact TMS $\Sigma$, and a map $\pi:\Sigma\rightarrow M$ s.t: \begin{enumerate}
	\item[{\rm (1)}] $\pi$ is uniformly continuous on $\Sigma$, with summable variations.
   \item[{\rm (2)}] $\pi\circ\sigma=f\circ\pi$.
   \item[{\rm (3)}] $\pi[\Sigma^\#]=\RST$, and every point in $\pi[\Sigma^\#]$ has finitely many pre-images in $\Sigma^\#$.
    
  \item[{\rm (4)}]  For every $x\in \pi[\Sigma]$, $T_xM=H^s(x)\oplus H^u(x)$, where
$$\text{ }\sum_{n\geq0}\|d_xf^n|_{H^s(x)}\|^2<\infty\text{ , }\text{ }\sum_{n\geq0}\|d_xf^{-n}|_{H^u(x)}\|^2<\infty.
$$
The maps $\underline{R}\mapsto E^{s/u}(\pi(\underline{R}))$ have summable variations as maps from $\Sigma$ to $TM$.
    
\end{enumerate}
\end{theorem}

\medskip
\noindent\textbf{Remark:} In \cite[Theorem~1.4]{InjectiveBuzzi}, Buzzi shows that one cannot hope for a H\"older continuous coding map $\pi$ which captures all hyperbolic measures. Our result is ``good news" in light of that, since we indeed get continuity of the factor map which is weaker than H\"older continuity, but we get a uniform modulus of continuity, and further more- summable variations. This is the threshold property needed for Sarig's thermodynamic formalism of countable Markov shifts. Composition of $\pi$ with a Lipschitz function retains its summable variations, and Lipschitz functions are dense in the continuous functions on the manifold. In \cite[Question~3]{InjectiveBuzzi}, Buzzi asks if one can achieve a continuous finite-to-one almost everywhere coding of all hyperbolic measures. Our coding does that on some examples, but the question of wether or not $\RST$ carries all hyperbolic measures in \textit{every} system, is an interesting and open question so far.

\medskip
Notice that every shift-invariant probability measure on $\Sigma$ projects to an invariant probability measure on the manifold. The following theorem implies that the projection preserves the entropy.
\begin{theorem}\label{t4.1.2} For $\Sigma$ given by Theorem \ref{t4.1.1}, we denote the set of states of $\Sigma$ by $\mathcal{V}$. There exists a function $N:\mathcal{V}\rightarrow\mathbb{N}$ s.t for every $x\in M$ which can be written as $x=\pi((v_i)_{i\in\mathbb{Z}})$ with $v_i=u$ for infinitely many $i<0$, and $v_i=w$ for infinitely many $i
>0$, it holds that: $|\pi^{-1}[\{x\}]\cap \Sigma^\#|\leq N(u)\cdot N(w)$.
\end{theorem}

In \textsection \ref{GTrans} we construct local stable and unstable manifolds for chains in $\Sigma$ (and similarly for strongly temperable points). In \textsection \ref{chapter7} we prove the following:
\begin{theorem}
Local stable and unstable leaves of strongly temperable points are absolutely continuous w.r.t holonomies. 
\end{theorem}
In \textsection \ref{chapter7} we give a general condition for the existence of strictly weak leaves (i.e with strictly sub-exponential contraction rates), and in \textsection \ref{codability} we mention some examples.

\subsection{Overview of the construction of the Markov partition}\label{SectionOverview}
Bowen's construction of Markov partitions for uniformly hyperbolic diffeomorphisms \cite{B3,B4} uses Anosov's shadowing theory for pseudo-orbits. This theory fails for general non-uniformly hyperbolic diffeomorphisms. In dimension two, Sarig developed a refined shadowing theory which does work in the non-uniformly hyperbolic setup. It consists of: \begin{enumerate}
    \item $\epsilon$-chains: a definition of pseudo orbits which allows for neighborhoods with varrying size.
    \item Shadowing lemma: every $\epsilon$-chain ``shadows" a unique orbit.
    \item Solution for the inverse problem: control of $\epsilon$-chains which shadow the same orbit.
\end{enumerate}
These properties, once established, allow the construction of Markov partitions using the method of Bowen  (\cite{Sarig}). Sarig worked in dimension two. In \cite{SBO}, we generalize his work to the higher dimensional case. 

In this work, we follow a similar scheme, but with refined adapted notions. 
\begin{enumerate}
    \item $\epsilon$-chains are replaced by {\em $I$-chains}, which require the sizes of neighborhoods on the pseudo orbit to be significantly more restricted. This is the content of \textsection \ref{loclin}.
    \item A shadowing theory for orbits with possible $0$ Lyapunov exponents: charts no longer have uniform estimates, and we need a suitable space of admissible manifolds on which the Graph Transform process converges in $C^1$, with a convergence rate which is summable. This is the content of \textsection \ref{GTrans}.
    \item A refined solution for the inverse problem: control of $I$-chains which shadow the same orbit. The fact that the sizes of neighborhoods in the $I$-chains are significantly more restricted, requires one to get significantly tighter estimates for the inverse problem, while having weaker (non-uniform) estimates in charts to begin with. This is the content of \textsection \ref{chapter333}. 
\end{enumerate}
In \textsection\ref{MarkoPolo} and  \textsection \ref{chapter5} we show how these properties are sufficient for the construction of a Markov partition which induces a finite-to-one almost everywhere coding.

\subsection{Notation}\label{notations}
\begin{enumerate}
    \item For any normed vector space $V$, we denote the unit sphere in $V$ by $V(1)$.
    \item Let $L$ be a linear transformation between two inner-product spaces $V,W$ of finite dimension. The {\em Frobenius norm} of $L$ is defined by $\|L\|_{Fr}=\sqrt{\sum_{i,j}a_{ij}^2}$, where $(a_{ij})$ is the representing matrix for $L$, under the choice of some pair of orthonormal bases for $V$ and $W$. This definition does not depend on the choices of bases.
    \item When a statement about two different cases that are indexed by their subscripts and/or their superscripts holds for both cases, we write in short the two cases together, with the two scripts separated by a ``/". For example, $\pi(u_{x/y})=x/y$ means $\pi(u_x)=x$ and $\pi(u_y)=y$.
    \item Given $A,B,C\in\mathbb{R}^+$ we write $A=e^{\pm B}C$ when $e^{-B}C\leq A\leq e^B C$.
    \item $|\cdot|$ denotes the Riemannian norm on $T_xM$ ($x\in M$) when applied to a tangent vector in $T_xM$, and denotes the Euclidean norm when applied to a vector in $\mathbb{R}^d$.
    \item $|\cdot|_\infty$ denotes the highest absolute value of the coordinates of a vector.
\end{enumerate}

\section*{Acknowledgements}
I would like to thank the Eberly College of Science in the Pennsylvania State University for excellent working conditions. I wish to thank Prof. Omri Sarig for many interesting and fruitful discussions. 
\section{Local linearization- chains of charts as pseudo-orbits}\label{loclin}
Let $f$ be a $C^{1+\beta}$ diffeomorphism of a smooth, compact and boundary-less manifold $M$ of dimension $d\geq 2$.

In this section we introduce $I$-overlapping, and $I$-chains for a refined version of Pesin charts (adapted from \cite{Sarig,SBO}, which in turn were adapted from \cite{BP,KM,K1}). There are several differences between those definitions, which are needed to get the fine Graph Transform estimates of the next section.

\subsection{Summability}\label{summability}

\begin{definition}[Summability]\label{0-summ}
	A point $x\in M$ is called {\em summable} if $T_xM$ decomposes uniquely $T_xM=H^s(x)\oplus H^u(x)$ s.t
	\begin{align*}
	\sup_{\xi^s\in H^s(x),|\xi^s|=1}\sum_{m\geq0}|d_xf^m\xi ^s |^2<\infty,	\sup_{\xi^u\in H^u(x),|\xi ^u |=1}\sum_{m\geq0}|d_xf^{-m}\xi ^u |^2.
	\end{align*}
 The set of summable points is called {\em $0$-$\mathrm{summ}$}. For a summable points $x$, $s(x):=\mathrm{dim}(H^s(x)), u(x):=\mathrm{dim}(H^u(x))$.
\end{definition}

\medskip
\noindent\textbf{Remark:} Aside from the power $\cdot^2$, Definition \ref{0-summ} does not depend on any choice of a constant. Moreover, the parallelogram law tells us that $\cdot^2$ is the only power for which this norm on the tangent space is induced by an inner-product (see Theorem \ref{pesinreduction}).

\begin{theorem}\label{pesinreduction} Oseledec-Pesin $\epsilon$-reduction  theorem (\cite{Pesin},\cite[\textsection~S.2.10]{KM}):

Let $M$  be  a compact Riemannian manifold. Then for each summable point $x\in M$ there exists an invertible linear transformation $C_0(x):\mathbb{R}^d\rightarrow T_xM$ ($C_0(\cdot):0\text{-}\mathrm{summ}\rightarrow GL({\mathbb{R}}^d,T_\cdot M)$) such that the map $D_0(x)=C_0^{-1}(f(x)) \circ d_xf \circ C_0(x)$ has the Lyapunov block form:
$$
\begin{pmatrix}D_s(x)  &   \\  & D_u(x)
\end{pmatrix},
$$
where $D_{s}(x)\in GL(s(x),\mathbb{R}),D_{u}(x)\in GL(u(x),\mathbb{R})$. In addition, for every $x \in  0$-summ we can decompose $T_xM=H^s(x)\oplus H^u(x)$ and $\mathbb{R}^d=\mathbb{R}^{s(x)}\oplus\mathbb{R}^{u(x)}$, and $C_0(x)$ sends each $\mathbb{R}^{s(x)/u(x)}$ to $H^{s/u}(x)$. Furthermore, $\exists \kappa(f)>1$ s.t  $$\kappa^{-1}\leq\frac{|D_sv_s|}{|v_s|}\leq e^{-\frac{1}{S^2(x)}},\kappa^{-1}\leq\frac{|D_u^{-1}v_u|}{|v_u|}\leq e^{-\frac{1}{U^2(x)}},\text{for all  non-zero }v_s\in\mathbb{R}^{s(x)},v_u\in \mathbb{R}^{u(x)}.$$ 
\end{theorem}
\begin{proof} Denote the Riemannian inner product on $T_xM$ by $\langle\cdot,\cdot\rangle$. Set
$$\langle u,v \rangle_{x,s}':=2\sum_{m=0}^\infty\langle d_xf^mu,d_xf^mv\rangle ,\text{ for }u,v\in H^s(x);$$
$$\langle u,v \rangle_{x,u}':=2\sum_{m=0}^\infty\langle d_xf^{-m}u,d_xf^{-m}v\rangle ,\text{ for }u,v\in H^u(x).$$
Since $x\in 0$-summ, the series converge absolutely. Thus, the series can be rearranged, and $\langle \cdot,\cdot \rangle_{x,s}',\langle \cdot,\cdot \rangle_{x,u}'$ are bilinear and are well defined as inner products on $H^s(x),H^u(x)$, respectively.
Define:
\begin{equation}\label{LIP}
\langle u,v \rangle_x':=\langle \pi_s u,\pi_s v \rangle_{x,s}'+\langle \pi_u u,\pi_u v \rangle_{x,u}',
\end{equation}
where $u,v\in T_xM$, and $v=\pi_sv+\pi_uv$ is the unique decomposition s.t $\pi_{s}v\in H^{s}(x),\pi_{u}v\in H^{u}(x)$ (the tangent vector $u$ decomposes similarly).

\medskip
Choose measurably\footnote{Notice that $C_0^{-1}(x)$ are only determined up to an orthogonal self map of $H^{s/u}(x)$- thus a measurable choice is needed. Let $\{D_i\}_{i=1}^N$ be a finite cover of open discs for the compact manifold $M$. Each open disc is orientable, and there exists a continuous choice of orthonormal
(w.r.t to the Riemannian metric) bases of $T_xM$, $(\widetilde{e}_1(x),...,\widetilde{e}_d(x))$. The contiuous choice of bases on the discs induces a meaurable choice of bases on the whole manifold by the partition $D_1, D_2\setminus D_1,D_3\setminus (D_1\cup D_2),...,D_N\setminus\cup_{i=1}^{N-1}D_i$. Using projections to $H^{s/u}(x)$ (which are defined measurably), and the Gram-Schmidt process, we construct orthonormal bases for $H^{s/u}(x)$ in a measurable way, w.r.t to $\langle\cdot,\cdot\rangle_{x,s/u}'$. Denote these bases by $\{b_i(x)\}_{i=1}^{s(x)}$ and $\{b_i(x)\}_{i=s(x)+1}^{d}$, respectively. Define $C_0^{-1}(x):b_i(x)\mapsto e_i$, where $\{e_i\}_{i=1}^d$ is the standard basis for $\mathbb{R}^d$.} $C_0^{-1}(x):T_xM\rightarrow\mathbb{R}^d$ to be a linear transformation satisfying 
\begin{equation}\label{omrisuggestnumber}
    \langle u,v \rangle_x'=\langle C_0^{-1}(x)u,C_0^{-1}(x)v \rangle,
\end{equation}
and $C_0^{-1}[H^s(x)]=\mathbb{R}^{s(x)}$ (and hence $C_0^{-1}[H^u(x)]=(\mathbb{R}^{s(x)})^\perp$). Thus $C_0^{-1}[H^s(x)]\perp C_0^{-1}[H^u(x)]$, and therefore $D_0$ is in this block form. Therefore, when $v_s\neq 0$, we get
$$\langle d_xfv_s,d_xfv_s\rangle_{f(x),s}'=2\sum_{m=0}^\infty |d_xf^{m+1}v_s|^2=(\langle v_s,v_s\rangle_{x,s}'-2|v_s|^2)$$
\begin{equation}\label{lowerboundstoonice}\Rightarrow \frac{\langle d_xfv_s,d_xfv_s\rangle_{f(x),s}'}{\langle v_s,v_s\rangle_{x,s}'}= 1-\frac{2|v_s|^2}{\langle v_s,v_s\rangle_{x,s}'}.\end{equation}
Call $M_f:=\max_{x\in M}\{\|d_xf\|,\|d_xf^{-1}\|\}$ and notice $\langle v_s,v_s\rangle_{x,s}'>2(|v_s|^2+|d_xfv_s|^2)\geq2(|v_s|^2+M_f^{-2}|v_s|^2)$. So:
$$\frac{1}{\kappa^2}\leq\frac{\langle d_xfv_s,d_xfv_s\rangle_{f(x),s}'}{\langle v_s,v_s\rangle_{x,s}'}\leq e^{-\frac{2|v_s|^2}{\langle v_s,v_s\rangle_{x,s}'}}$$
where $\kappa=
(1-\frac{1}{1+M_f^{-2}})^{-1/2}
$.

The same holds for $D_u$:
$$\langle d_xf^{-1}v_u,d_xf^{-1}v_u\rangle_{x,u}'=2\sum_{m=0}^\infty |d_xf^{-m-1}v_u|^2=\langle v_u,v_u\rangle_{f(x),u}'-2|v_u|^2,$$
Then, $$\frac{1}{\kappa^2}\leq\frac{\langle d_xf^{-1}v_u,d_xf^{-1}v_u\rangle_{x,u}'}{\langle v_u,v_u\rangle'_{f(x),u}}\leq e^{-\frac{2|v_u|^2}{\langle v_u,v_u\rangle'_{f(x),u}}}
.$$ Let $w_s\in \mathbb{R}^{s(x)}$, and define $v_s:=C_0(x)w_s$.
$$\frac{|D_s(x)w_s|^2}{|w_s|^2}=\frac{|C_0^{-1}(f(x))d_xfv_s|^2}{|C_0^{-1}(x)v_s|^2}=\frac{|d_xfv_s|_{f(x),s}^{'2}}{|v_s|_{x,s}^{'2}}\in[\kappa^{-2}, e^{-2\frac{1}{S^2(x)}}].$$
Similarly $\frac{|D_u^{-1}(x)w_u|^2}{|w_u|^2}\in [\kappa^{-2}, e^{-2\frac{1}{U^2(x)}}] $.
\end{proof}

We fix three real positive constants which calibrate our construction.
\begin{definition}\label{calibrationParams} The {\em calibration parameters} are
\begin{enumerate}
	\item $\Gamma>0$, the {\em ladder scale},
	\item $\gamma>\frac{5}{\beta}$, the {\em windows' exponent}  (up to \textsection \ref{forGammaBounds} $\gamma>\frac{2}{\beta}$ is sufficient), and
	\item $C_{\beta,\epsilon}:=\frac{1}{3^\frac{6}{\beta}}\epsilon^\frac{90}{\beta}$, given $\epsilon>0$, the {\em windows' scalar}.
\end{enumerate}	
\end{definition}
 The objects we define next depend on the calibration of these three parameters. We abuse notation and not always write when the objects depend on these parameters. 

\begin{definition}[Ladder functions]\label{ladderFunc}
	$I(x):=xe^{\Gamma x^\frac{1}{\gamma}}$, $I:(0,1)\rightarrow(0,\infty)$. $I^\frac{1}{4}:(0,1)\rightarrow(0,\infty)$ is a strictly increasing $C^1$ positive function s.t $(I^\frac{1}{4})^4= I^\frac{1}{4}\circ I^\frac{1}{4}\circ I^\frac{1}{4}\circ I^\frac{1}{4} = I $ on $(0,1)$ and $I^\frac{1}{4}(x)>x$ for all $x\in(0,1)$. $I,I^\frac{1}{4}$ are called {\em the ladder functions}.
\end{definition}
\begin{lemma}\label{forI}\text{ }

\begin{enumerate}
\item	$I^\frac{1}{4}$ is well-defined
. 
\item $I$ and $I^\frac{1}{4}$ are invertible, and $I^{\frac{-n}{4}}(x)\xrightarrow[n\rightarrow\infty]{}0$ $\forall x\in(0,1)$.
\item $I^{-1}(x)\geq xe^{-\Gamma x^\frac{1}{\gamma}}$.
\item $\sum_{n\geq0}(I^{-n} (x))^\frac{1}{\gamma}=\infty$ for all $x>0$.
\end{enumerate}
	
\end{lemma}
\begin{proof}
\begin{enumerate}
\item Notice that $I'>1$ on $(0,1)$, and so there exists an inverse $I^{-1}\in C^1((0,e^\Gamma))$ by The Inverse Function Theorem. Since $I(x)>x$ on $(0,1)$, $I^{-1}(x)<x$ on $(0,e^\Gamma)$; and since $I$ is strictly increasing, $I^{-1}$ is strictly increasing. We now may use \cite [Theorem~1]{IterativeRoot} to conclude the existence of a strictly increasing $C^{1}$ iterative root of $I^{-1}$; i.e $I^\frac{-1}{2}\in C^1((0,e^\Gamma))$ s.t $(I^\frac{-1}{2})^2:= I^\frac{-1}{2} \circ I^\frac{-1}{2} =I^{-1}$. In addition, $I^\frac{-1}{2}(x)<x$ on $(0, e^\Gamma)$, otherwise if $x_0\leq I^\frac{-1}{2}(x_0)$, then by applying $I^\frac{-1}{2}$: $I^\frac{-1}{2}(x_0)\leq(I^\frac{-1}{2})^2(x_0)$; and so $x_0\leq I^\frac{-1}{2}(x_0) \leq(I^\frac{-1}{2})^2(x_0) =I^{-1}(x_0)<x_0$, a contradiction. Therefore, we may use \cite{IterativeRoot} again with $I^\frac{-1}{2}$ to conclude similarly the existence of a strictly increasing $C^{1}$ iterative $4$-th root of $I^{-1}$; i.e $I^\frac{-1}{4}\in C^1((0, e^\Gamma))$ s.t $(I^\frac{-1}{4})^4:= I^\frac{-1}{4} \circ I^\frac{-1}{4}\circ I^\frac{-1}{4}\circ I^\frac{-1}{4} =I^{-1}$; and $I^\frac{-1}{4}(x)<x$ on $(0, e^\Gamma)$. 

Next we claim that $(I^\frac{-1}{4})'>0$ on $(0, e^\Gamma)$. Since $I^\frac{-1}{4}$ is strictly increasing and $C^1$, $(I^\frac{-1}{4})'\geq0$. If $\exists x'\in (0, e^\Gamma)$ s.t $(I^\frac{-1}{4})'(x')=0$, then $0<\frac{1}{I'(I^{-1}(x'))}=(I^{-1})'(x')=(I^\frac{-1}{4})' (I^\frac{-3}{4}(x'))\cdot (I^\frac{-1}{4})'(I^\frac{-2}{4}(-x))\cdot (I^\frac{-1}{4})'(I^\frac{-1}{4} (x'))\cdot (I^\frac{-1}{4})'(x') =0$, which is a contradiction.

Thus, by The Inverse Function Theorem there exists an inverse $I^\frac{1}{4}\in C^1((0,I^\frac{-1}{4}(e^\Gamma)))$. Since $I^\frac{-1}{4}(x)<x$ and is strictly increasing on $(0, e^\Gamma)$, $I^\frac{1}{4}(x)>x$ and is strictly increasing on $(0,I^\frac{-1}{4}(e^\Gamma))$. Notice, $I^\frac{-1}{4}(e^\Gamma) =I^\frac{3}{4}(1)>1$.

\item Invertibility is given by the first item. In addition, we saw that $I^\frac{- 1}{4}(x)<x$ on $(0,e^\Gamma)$. From continuity, the sequence $\{I^\frac{-n}{4}(x)\}_{n\geq0}$ cannot be bounded from below by a positive number and so it converges to $0$. 
\item $I^{-1}(xe^{\Gamma x^\frac{1}{\gamma}})= x=xe^{\Gamma x^\frac{1}{\gamma}} \cdot e^{-\Gamma x^\frac{1}{\gamma}} \geq xe^{\Gamma x^\frac{1}{\gamma}} \cdot e^{-\Gamma \left(x e^{\Gamma x^\frac{1}{\gamma}}\right)^\frac{1}{\gamma}}$, then $I^{-1}(x)\geq xe^{-\Gamma x^\frac{1}{\gamma}}$.
\item Assume $\sum_{n\geq 0}(I^{-n}(x))^\frac{1}{\gamma}<\infty$, then 
$$\infty> \sum_{n\geq 0}(I^{-n}(x))^\frac{1}{\gamma}\geq x^\frac{1}{\gamma}\sum_{n\geq0}e^{-\frac{\Gamma}{\gamma}\sum_{k=0}^n(I^{-k}(x))^\frac{1}{\gamma}}\Rightarrow \sum_{k=0}^n(I^{-k}(x))^\frac{1}{\gamma} \xrightarrow[n\rightarrow\infty]{}\infty.$$
	A contradiction.
\end{enumerate}
	
\end{proof}

\begin{lemma}\label{FaveForNow}
	$\forall \tau>0$, $\forall x>0$, $\sum_{n\geq0}(I^{-n}(x))^\frac{1+\tau}{\gamma}<\infty$.
\end{lemma}
\begin{proof} Let $\tau>0$. Let $S(x):=x^\gamma$, $J(x):=x e^{\frac{\Gamma}{\gamma}x}$, then $I=S\circ J\circ S^{-1}$. $J$ is a strictly increasing function, and so it is invertible. $J(x)=x+\frac{\Gamma}{\gamma}x^2+O(x^3)$. Notice, $\forall n\geq0$, $(I^{-n}(x))^\frac{1+\tau}{\gamma}=(S\circ J^{-n}\circ S^{-1}(x))^{1+\tau}=(J^{-n}(x^\frac{1}{\gamma}))^{1+\tau}$, then it is enough to show $\sum_{n\geq 0}(J^{-n}(x))^{1+\tau}<\infty $ $\forall x>0$. We therefore continue to get estimates for $J^{-n}(x)$. Assume w.l.o.g that $x$ is so small that $J(x)\leq x+ \frac{\Gamma}{\gamma}e^\frac{\tau}{3}x^2
$ (which is possible since $J^{-n}(x)$ is a strictly decreasing diminishing sequence, see Lemma \ref{ladderFunc}). 
$J^{-n}(x)$ does not decrease faster than a linear rate. We give a short induction proof of this which is adapted from \cite{YoungCounterExample}, together with an estimate on the linear rate.

\medskip
\textit{Claim 1:} $\forall n\geq0$, $J^{-n}(x)\geq \frac{1}{A (n+n_x)}$ where $A:=\frac{\Gamma}{\gamma}e^\frac{\tau}{3}$ and $n_x:=\frac{1}{Ax}$.

\textit{Proof:} Proof is by induction, and the basis of the induction is given by the definition of $n_x$. Assume that the statement holds for all $n\leq k$, and assume for contradiction that $J^{-(k+1)}(x)< \frac{1}{A (k+1+n_x)} $, then
\begin{align*}
J^{-k}(x)=&J(J^{-(k+1)}(x))\leq J^{-(k+1)}(x)+A \cdot(J^{-(k+1)}(x))^2<\frac{1}{A (k+1+n_x)}\cdot (1+\frac{1}{k+1+n_x})\\
=& \frac{1}{A (k+n_x)}\cdot(1-\frac{1}{k+1+n_x}) (1+\frac{1}{k+1+n_x})= \frac{1}{A (k+n_x)}\cdot(1-\frac{1}{(k+1+n_x)^2})< \frac{1}{A (k+n_x)}.
\end{align*}
A contradiction.

\medskip
\textit{Claim 2:} $J^{-1}(x)=x-\frac{\Gamma}{\gamma}x^2+O(x^3)$.

\textit{Proof:} Write $J^{-1}(x)=x-\frac{\Gamma}{\gamma}x^2+\delta(x)$, and let $\omega(x)=J(x)-x-\frac{\Gamma}{\gamma} x^2=O(x^3)$. Then,
\begin{align*}x=&J^{-1}(J(x))=x+\frac{\Gamma}{\gamma}x^2+\omega(x)- \frac{\Gamma}{\gamma}(x+\frac{\Gamma}{\gamma}x^2+\omega(x))^2+\delta\circ J(x)=x+O(x^3)+ \delta\circ J(x).
\end{align*}
Then we get $\delta(x)=O((J^{-1}(x))^3)=O(x^3)$. This concludes claim 2.

Therefore we may choose $x$ so small w.l.o.g that $J^{-1}(x)\leq xe^{-Bx}$, $B=\frac{\Gamma}{\gamma}e^\frac{-\tau}{3}$. Then $\forall n\geq0$,

\begin{align*}
	J^{-n}(x)\leq & x e^{-\sum_{k=0}^{n-1}BJ^{-k}(x)}\leq xe^{-\frac{B}{A}\sum_{k=0}^{n-1}\frac{1}{k+n_x}}\leq xe^{-e^\frac{-2\tau}{3}\sum_{k=0}^{n-1}\frac{1}{k+n_x}}\\
	=& xe^{-e^\frac{-2\tau}{3}\sum_{k=0}^{n-1}\frac{1}{k+1}+ e^\frac{-2\tau}{3}(\sum_{k=0}^{n-1}\frac{1}{k+1}-\frac{1}{k+n_x})}= xe^{-e^\frac{-2\tau}{3}\sum_{k=0}^{n-1}\frac{1}{k+1}+ e^\frac{-2\tau}{3}\sum_{k=0}^{n-1}\frac{n_x-1}{(k+1)(k+n_x)}}.
\end{align*}
Let $E_0:=\sup_{n\geq0}|\log (n+1)-\sum_{k=0}^{n-1}\frac{1}{k+1}|<\infty$, then
\begin{align}\label{forSummVar}
	[J^{-n}(x)]^{1+\tau}\leq & \left[e^{E_0+n_x\frac{\pi^2}{6}}xe^{-e^\frac{-2\tau}{3}\log (n+1)}\right]^{1+\tau}= e^{(1+\tau)(E_0+n_x\frac{\pi^2}{6})}x^{1+\tau}\cdot \frac{1}{(n+1)^{(1+\tau)e^\frac{-2\tau}{3}}}.
\end{align}
Therefore, when w.l.o.g $\tau>0$ is so small that $(1+\tau)e^\frac{-2\tau}{3}>1$, $\sum_{n\geq0 } [J^{-n}(x)]^{1+\tau}<\infty$.
\end{proof}

\begin{definition}[Strong temperability]\label{strongtemp}
	$\mathcal{I}:=\{I^\frac{-\ell}{4}(1)\}_{\ell\geq 0}$, $\widetilde{Q}_\epsilon(x):=C_{\beta,\epsilon}\cdot\frac{1}{\|C_0^{-1}(x)\|^{2\gamma}}$,
$$Q_\epsilon(x):=\max\{q\in \mathcal{I} |q\leq \widetilde{Q}_\chi(x)\}$$
\end{definition}
\begin{definition}\label{NUHsharp}
A point $x\in0\text{-}\mathrm{summ}$ is called {\em strongly temperable} if  $\exists q:\{f^n(x)\}_{n\in\mathbb{Z}}\rightarrow \mathcal{I}$ s.t
\begin{enumerate}
	\item $q\circ f=I^{\pm1}(q)$,
	\item $\forall n\in\mathbb{Z}$, $q(f^n(x))\leq Q_\epsilon(f^n(x))$.
\end{enumerate}
If in addition to (1),(2), $q:\{f^n(x)\}_{n\in\mathbb{Z}}\rightarrow \mathcal{I}$ can be chosen to also satisfy,

\begin{enumerate}
\item[(3)]$\limsup\limits_{n\rightarrow\pm\infty}q(f^n(x))>0$,
\end{enumerate}
then we say that $x$ is {\em recurrently strongly temperable}.
Define $\ST:= \{x\in0\text{-}\mathrm{summ}:x\text{ is }\text{strongly temperable}\}$, and  $\RST:= \{x\in0\text{-}\mathrm{summ}:x\text{ is recurrently }\text{strongly temperable}\}$. $\ST$ is the set of {\em strongly temperable points}, and $\RST\subseteq\ST$ is the set of {\em recurrently strongly temperable points}.

A function $q :\{f^n(x)\}_{n\in\mathbb{Z}}\rightarrow \mathcal{I} $ which satisfies conditions (1) and (2) is called {\em a strongly tempered kernel} for the point $x$.
\end{definition}

\noindent\textbf{Remark:} Notice that the supremum over a collection of strongly tempered kernels for a point is also a strongly tempered kernel for the point, and so we may discuss ``the optimal" one. In addition, this definition depends on the choice of the scaling parameters; the scaling parameters are fixed in this paper, and so we abuse notation and don't notate the dependence on them. However, in the general case, ``strong temperability" refers to being strongly temperable w.r.t any (proper) choice of scaling parameters.

\begin{definition}\label{scalingfuncs} {\em Scaling functions ($S/U$ parameters)}: For any $x\in M$, $\xi\in T_xM$ we define:
$$S^2(x,\xi)=2\sum_{m=0}^\infty|d_xf^m\xi|^2, U^2(x,\xi)=2\sum_{m=0}^\infty|d_xf^{-m}\xi|^2\in [2|\xi|^2,\infty].$$
Here, $|\cdot|$ is the Riemannian norm on $T_xM$.
\end{definition}
\noindent\textbf{Remark}: $U(x,\xi),S(x,\xi)$ will be calculated in two cases: \begin{enumerate}
    \item Where $x\in 0$-$\mathrm{summ}$, $\xi\in H^{u/s}(x)$ respectively (in this case the quantities are finite).
    \item Where $x$ belongs to a $u/s$- manifold $V^{u/s}$ (Definition \ref{def135}), and $\xi\in T_x V^{u/s}$ respectively.
\end{enumerate}
\begin{theorem}\label{SandU} There exists an $F_0>0$ only depending on $f$ and $M$  such that for any $x\in 0$-$\mathrm{summ}$: $$\frac{\|C_0^{-1}(f(x))\|}{\|C_0^{-1}(x)\|}\in[F_0^{-1},F_0]$$
\end{theorem}
\begin{proof} Notice that by definition:
$$\|C_0^{-1}(x)\|^2=\sup_{\xi\in H^s(x),\eta\in H^u(x),\xi+\eta\neq0}\frac{S^2(x,\xi)+U^2(x,\eta)
}{|\xi+\eta|^2}.$$

\textit{Claim 1}: For a fixed $x\in 0$-$\mathrm{summ}$ $(\xi,\eta)\mapsto \frac{S^2(x,\xi)+U^2(x,\eta)}{|\xi+\eta|^2}$ is continuous where defined.

\textit{Proof}: The norms of vectors are clearly  continuous. So is
$S^2$,
since $S^2(x,\xi)=\langle\xi,\xi\rangle_{x,s}'$- 
a bilinear form
on a finite-dimensional linear space
. The same argument applies for $U^2$
. This concludes the proof of claim 1.

\medskip
A continuous function on a compact set attains its maximum, and the mapping $(\xi,\eta)\mapsto 
\sqrt{S^2(x,\xi)+U^2(x,\eta)}
$ is homogeneous. So let $\xi\in H^s(f(x))$ and $\eta\in H^u(f(x))$ be any two tangent vectors s.t $\frac{S^2(f(x),\xi)+U^2(f(x),\eta)}{|\xi+\eta|^2}$ is maximal. Whence,
\begin{align}\label{eq1}\frac{\|C_0^{-1}(f(x))\|^2}{\|C_0^{-1}(x)\|^2}=&\frac{\sup_{\xi_1,\eta_1} (S^2(f(x),\xi_1)+U^2(f(x),\eta_1))/|\xi_1+\eta_1|^2}{\sup_{\xi_2,\eta_2} (S^2(x,\xi_2)+U^2(x,\eta_2))/|\xi_2+\eta_2|^2}\nonumber\\
\leq&\frac{\sup_{\xi_1,\eta_1} (S^2(f(x),\xi_1)+U^2(f(x),\eta_1))/|\xi_1+\eta_1|^2}{(S^2(x,d_{f(x)}f^{-1}\xi)+U^2(x,d_{f(x)}f^{-1}\eta))/|d_{f(x)}f^{-1}\xi+d_{f(x)}f^{-1}\eta|^2}\nonumber\\
    =&\frac{S^2(f(x),\xi)+U^2(f(x),\eta)}{S^2(x,d_{f(x)}f^{-1}\xi)+U^2(x,d_{f(x)}f^{-1}\eta)}\cdot\frac{|d_{f(x)}f^{-1}\xi+d_{f(x)}f^{-1}\eta|^2}{|\xi+\eta|^2}.
\end{align}
\textit{Claim 2}: The left fraction of (\ref{eq1}) is bounded away from $0$ and $\infty$ uniformly by a constant (and its inverse) depending only on $f$ and $M$.

\textit{Proof}: Let $v\in H^s(x)$ (w.l.o.g $|v|=1$),
$$S^2(f(x),d_xfv)=2\sum_{m=0}^\infty|d_{f(x)}f^md_xfv|^2=2\sum_{m=0}^\infty|d_{x}f^mv|^2-2=S^2(x,v)-2.$$
This, and the inequality $S^2(f(x),d_xfv)>2M_f^{-2}(M_f^{-2}+1)>2M_f^{-4}$, yields
$$S^2(x,v)(1-\frac{1}{M_f^{-4}})<S^2(f(x),d_xfv)< S^2(x,v).$$
A similar calculation gives $U^2(f(x),d_xfv)=U^2(x,v)+2$, and so since $U^2(x,v)>2$, $\frac{U^2(f(x),d_xfv)}{U^2(x,v)}\in[1,2]$. The claim follows.

\textit{Claim 3}: The right fraction of (\ref{eq1}) is bounded away from $0$ and $\infty$ uniformly by constants
depending only on $f$ and $M$.

\textit{Proof}: Recall that $M_f:=\max_{x\in M}\{\|d_xf\|,\|d_xf^{-1}\|^{-1}\}$. Then
$$\frac{|d_{f(x)}f^{-1}\xi+d_{f(x)}f^{-1}\eta|^2}{|\xi+\eta|^2}=\frac{|d_{f(x)}f^{-1}(\xi+\eta)|^2}{|\xi+\eta|^2}\in [M_f^{-2},M_f^2].$$
The three claims together give the upper bound in the theorem. If we repeat the previous argument with $\xi,\eta$ maximizing the denominator instead of the numerator, we get the lower bound.
\end{proof}
\begin{lemma}\label{contraction}
Given $x\in 0$-$\mathrm{summ}$, $C_0(x)$ is a contraction (w.r.t to the Eucleadan norm on $\mathbb{R}^d$ and the Riemannian norm on the tangent space).
\end{lemma}
\begin{proof} 
$\forall v_{s}\in H^{s}$,
$|C_0^{-1}(x)v_{s}|^2={|v_{s}|'}_{x,s}^2>2|v_{s}|^2$. Similarly $\forall v_{u}\in H^{u}$,
$|C_0^{-1}(x)v_{u}|^2={|v_{u}|'}_{x,u}^2>2|v_{u}|^2$. Hence, $\forall w\in \mathbb{R}^d:w=w_s+w_u$:
$$|C_0(x)w|^2=(|C_0(x)w_s|+|C_0(x)w_u|)^2<(\frac{1}{2}(|w_s|+|w_u|))^2\leq(|w_s+w_u|)^2,$$
where the last inequality is due to the fact that $w_s\bot w_u$. \end{proof}

\subsection{Pesin charts}\label{forrho}

Let $\exp_x:T_xM\rightarrow M$ be the exponential map. Since $M$ is compact, there exist $r=r(M),\rho=\rho(M) >0$ s.t $\forall x\in M$ $\exp_x$ maps $B_{\sqrt{d}\cdot r}^x:=\{v\in T_xM: |v|_x<\sqrt{d\cdot} r\}$ diffeomorphically onto a neighborhood of $B_\rho(x)=\{y\in M: d(x,y)<\rho\}$, where $d(\cdot,\cdot)$ is the distance function on $M\times M$, w.r.t to the Riemannian metric.
\begin{definition}\label{balls}$\text{ }$\\
\begin{enumerate}
\item[(a)] 
    \begin{enumerate}
        \item[(i)] $B_\eta(0), R_\eta(0)$ are open balls located at the origin with radius $\eta$ in $\mathbb{R}^d$ w.r.t Euclidean norm and supremum norm respectively.
        \item[(ii)] For every $x\in M$, we define the following open ball in $T_xM$: $\forall v\in T_xM, r>0$, $B_r^x(x):=\{w\in T_xM: |v-w|<r\}$.
    \end{enumerate}

\item[(b)] We take $\rho$ so small that $(x,y)\mapsto \exp_x^{-1}(y)$ is well defined and 2-Lipschitz on $B_\rho(z)\times B_\rho(z)$ for $z\in M$, and so small that $\|d_y \exp_x^{-1}\|\leq2$ for all $y\in B_\rho(x)$.

\item[(c)] For all $x\in 0$-$\mathrm{summ}$: $$\psi_x:=\exp_x\circ C_0(x):R_r(0)\rightarrow M.$$ Since $C_0(x)$ is a contraction, $\psi_x$ maps $R_r(0)$ diffeomorphically into $M$.

\item[(d)] $$f_x:=\psi_{f(x)}^{-1}\circ f\circ \psi_x :R_r(0)\rightarrow \mathbb{R}^d.$$ Then we get $d_0f_x=D_0(x)$ (c.f Theorem \ref{pesinreduction}).

\end{enumerate}
\end{definition}

\noindent\textbf{Remark}: Depending on context, the notations of $(a)(i)$ will be used to denote balls with the same respective norms in some subspace of $\mathbb{R}^d$ (i.e $\{v\in \mathbb{R}^{s(x)}:|v|<\eta\}$, $x\in 0$-$\mathrm{summ}$).

\begin{theorem}\label{linearization} (See \cite{BP}, theorem 5.6.1) For all $\epsilon$ small enough
, and $x\in 0$-$\mathrm{summ}$:
\begin{enumerate}
    \item $\psi_x(0)=x$ and $\psi_x:R_{Q_\epsilon(x)}(0)\rightarrow M$ is a diffeomorphism onto its image s.t $\|d_u\psi_x\|\leq2$ for every $u\in R_{Q_\epsilon(x)}(0)$.
\item $f_x$ is well defined and injective on $R_{Q_\epsilon(x)}(0)$, and \begin{enumerate}
\item $f_x(0)=0,d_0f_x=D_0(x)$. 
\item $\|f_x-d_0f_x\|_{C^{1+\frac{3\beta}{4}}}\leq \epsilon^2 |v|^\frac{3\beta}{4}$ on $R_{|v|}(0)$ $\forall v\in R_{Q_\epsilon(x)}(0)$ and in particular $\|f_x-d_0f_x\|_{C^{1+\frac{\beta}{2}}}\leq Q_\epsilon(x)^\frac{\beta}{2}$ on $R_{Q_\epsilon(x)}(0)$.    
\end{enumerate}
\item The symmetric statement holds for $f_x^{-1}$.
\end{enumerate}
\end{theorem}
\begin{proof} The proof is similar to \cite[Theorem~5.6.1]{BP} and to \cite[Theorem~2.7]{Sarig}.
\end{proof}
\begin{definition} Suppose $x\in 0$-$\mathrm{summ}$ and $\eta\in (0,Q_\epsilon(x)]$, then the Pesin Chart $\psi_x^\eta$ is the map $\psi_x:R_\eta(0)\rightarrow M$. Recall, $R_\eta(0)=[-\eta,\eta]^d$.
\end{definition}
\begin{lemma}\label{omega0} For every $x\in 0$-$\mathrm{summ}$:

\begin{enumerate}
\item[(1)] $Q_\epsilon(x)<C_{\beta,\epsilon}\leq 2M_f^2$ on $0$-$\mathrm{summ}$.

\item[(2)]$\|C_0^{-1}(f^i(x))\|^{2\gamma}<F_0^{2\gamma}C_{\beta,\epsilon}/Q_\epsilon(x)$ for i=-1,0,1, where $F_0>0$ is a constant given by Theorem \ref{SandU}.

\item[(3)] For all $t>0$ $\#\{Q_\epsilon(x)| Q_\epsilon(x)>t, x\in M\}<\infty$.

\item[(4)] $\exists\omega_0$ depending only on $M,f,\Gamma$ and $\gamma$ s.t $\omega_0^{-1}\leq Q_\epsilon\circ f/Q_\epsilon\leq \omega_0$ on $0$-$\mathrm{summ}$.
\end{enumerate}
\end{lemma}
\begin{proof}

\begin{enumerate}
\item[(1)]  Clear by definition of $\widetilde{Q}_\epsilon(x)$.
\item[(2)] Recall $F_0$ from Theorem \ref{SandU}. Then,
$$\|C_0^{-1}(f^i(x))\|^{2\gamma}\leq F_0^{2\gamma}\|C_0^{-1}(x)\|^{2\gamma}<F_0^{2\gamma} C_{\beta,\epsilon}/Q_\epsilon(x).$$
%
%
%

\item[(3)]  For all $t>0$: $\{Q_\epsilon(x)| Q_\epsilon(x)>t, x\in M\}\subset\{I^{- \frac{\ell}{4}}(1)\}_{\ell\in\mathbb{N}}\cap (t,C_{\beta,\epsilon}]$.

\item[(4)] By Theorem \ref{SandU}:  $F_0^{-1}\leq\frac{\|C_0^{-1}(f(x))\|}{\|C_0^{-1}(x)\|}\leq F_0$ on $0$-$\mathrm{summ}$. Now, $\widetilde{Q}_\epsilon(x)=(\|C_0^{-1}(x)\|)^{-2\gamma}C_{\beta,\epsilon}$ hence
$$\frac{\widetilde{Q}_\epsilon(f(x))}{\widetilde{Q}_\epsilon(x)}\in[F_0^{-2\gamma},F_0^{2\gamma}].$$
By the definition of $Q_\epsilon(x)$: $Q_\epsilon\in(I^{-1/4}(\widetilde{Q}_\epsilon),\widetilde{Q}_\epsilon]\subseteq (e^{-\Gamma}\widetilde{Q}_\epsilon,\widetilde{Q}_\epsilon]$. Hence,
 $$F_0^{-2\gamma}e^{-\Gamma }\leq Q_\epsilon\circ f/Q_\epsilon\leq F_0^{2\gamma}e^{\Gamma }.$$
 Set $\omega_0:= F_0^{2\gamma}e^{\Gamma }$.

\end{enumerate}
\end{proof}

\subsection{Overlapping charts}
The following definition is motivated by \cite[\textsection~3.1]{Sarig}.
\begin{definition}\label{isometries}$\text{ }$\\
\begin{enumerate}
\item[(a)] For every $x\in M$ there is an open neighborhood $D$ of diameter less than $\rho$ and a smooth map $\Theta_D:TD\rightarrow\mathbb{R}^d$ s.t:
\begin{enumerate}
    \item[(1)] $\Theta_D:T_xM\rightarrow\mathbb{R}^d$ is a linear isometry for every $x\in D$

    \item[(2)] Define $\nu_x:=\Theta_D|_{T_xM}^{-1}:\mathbb{R}^d\rightarrow T_xM$, then $(x,u)\mapsto(\exp_x\circ\nu_x)(u)$ is smooth and Lipschitz on $D\times B_2(0)$ w.r.t the metric $d(x,x')+|u-u'|$

    \item[(3)] $x\mapsto\nu_x^{-1}\circ \exp_x^{-1}$ is a Lipschitz map from $D$ to $C^2(D,\mathbb{R}^d)=\{C^2\text{-maps from }D\text{ to }\mathbb{R}^d\}$.

    Let $\mathcal{D}$ be a finite cover of $M$ by such neighborhoods. Denote with $\varpi(\mathcal{D})$ the Lebesgue number of that cover: If $d(x,y)<\varpi(\mathcal{D})$ then $x$ and $y$ belong to the same $D$ for some $D$.
\end{enumerate}
\item[(b)] We say that two Pesin charts $\psi_{x_1}^{\eta_1},\psi_{x_2}^{\eta_2}$ {\em $I$-overlap} if $\eta_1=I^{\pm1}(\eta_2)$, and for some $D\in\mathcal{D}$, $x_1,x_2\in D$ and $d(x_1,x_2)+\|\Theta_D\circ C_0(x_1)-\Theta_D\circ C_0(x_2)\|<\eta_1^4\eta_2^4$.
\end{enumerate}

\noindent\textbf{Remark}: The overlap condition is symmetric and monotone: if $\psi_{x_i}^{\eta_i},i=1,2$ $I$-overlap, then $\psi_{x_i}^{\xi_i},i=1,2$ $I$-overlap for all $\eta_i\leq\xi_i\leq Q_\epsilon(x_i)$ s.t $\xi_1=I^{\pm1}(\xi_2)$.
    
\end{definition}

\begin{prop}\label{chartsofboxes} The following holds for all $\epsilon$ small enough: If $\psi_x:R_\eta(0)\rightarrow M$ and $\psi_y:R_\zeta(0)\rightarrow M$ $I$-overlap, then: \begin{enumerate}
\item $\psi_x[R_{I^{-2}(\eta)}(0)]\subset\psi_y[R_\zeta(0)]$ and $\psi_y[R_{I^{-2}(\zeta)}(0)]\subset\psi_x[R_\eta(0)]$,
\item $\mathrm{dist}_{C^{1+\beta/2}}(\psi_{x/y}^{-1}\circ\psi_{y/x},Id)<\epsilon\eta^3\zeta^3$ (recall \textsection \ref{notations}) where the $C^{1+\beta/2}$ distance is calculated on $R_{I^{-1} (r(M))}(0)$ and $r(M)$ is defined in \textsection \ref{forrho}.
\end{enumerate}
\end{prop}
\begin{proof}


Suppose $\psi_x^\eta$ and $\psi_y^\zeta$ $I$-overlap, and fix some $D\in\mathcal{D}$ which contains $x,y$ such that $d(x,y)+\|\Theta_D\circ C_0(x)-\Theta_D\circ C_0(y)\|<\eta^4\zeta^4$. For $z=x,y$, write $C_{z}=\Theta_D\circ C_0(z)$, and $\psi_{z}=\exp_{z}\circ \nu_{z}\circ C_{z}$. By the definition of Pesin charts, $\eta\leq Q_\epsilon(x)< C_{\beta,\epsilon}\|C_0^{-1}(x)\|^{-2\gamma}$, $\zeta\leq Q_\epsilon(y)<C_{\beta,\epsilon}\|C_0^{-1}(y)\|^{-2\gamma}$. In particular, $\eta,\zeta\leq C_{\beta,\epsilon}$. 
Our first constraint on $\epsilon$ is that it would be less than $1$, and so small that $$C_{\beta,\epsilon}<\epsilon<\frac{\min\{1,r(M),\rho(M)\}}{5(L_1+L_2+L_3+L_4)^3},$$ 
where $r(M),\rho(M)$ are defined on \textsection \ref{forrho}, and: \begin{enumerate}
\item $L_1$ is a uniform Lipschitz constant for the maps $(x,v)\mapsto(\exp_x\circ\nu_x)(v)$ on $D\times B_{r(M)}(0)$ ($D\in\mathcal{D}$).
\item  $L_2$ is a uniform Lipschitz constant for the maps $x\mapsto\nu_x^{-1}\exp_x^{-1}$ from $D$ into $C^2(D,\mathbb{R}^2)$ ($D\in\mathcal{D}$).
\item $L_3$ is a uniform Lipschitz constant for $\exp_x^{-1}:B_{\rho(M)}(x)\rightarrow T_xM$ ($x\in M$).
\item $L_4$ is a uniform Lipschitz constant for $\exp_x:B_{r(M)}(0)\rightarrow M$ ($x\in M$).
\end{enumerate}
We assume w.l.o.g that these constants are all larger than one. 

\textit{Part 1:} $\psi_x[R_{I^{-2}(\eta)}(0)]\subset\psi_y[R_\zeta(0)]$

Proof: Suppose $v\in R_{I^{-2}(\eta)}(0)$. Since $C_0(x)$ is a contraction: $|C_xv|=|C_0(x)v|\leq|v|$, and $(x,C_xv),(y,C_xv)\in D\times B_{r(M)}(0)$. Since $d(x,y)<\eta^4\zeta^4$:
$$d(\exp_y\circ\nu_y[C_xv],\exp_x\circ\nu_x[C_xv])<L_1\eta^4\zeta^4.$$
It follows that $\psi_x(v)\in B_{L_1\eta^4\zeta^4}(\exp_y\circ\nu_y(C_xv))$. Call this ball $B$. The radius of $B$ is less than $\rho(M)$ because of our assumptions on $\epsilon$. Therefore $\exp_y^{-1}$ is well defined and Lipschitz on $B$, and its Lipschitz constant is at most $L_3$. Writing $B=\exp_y[\exp_y^{-1}[B]]$ we deduce that
$$\psi_x(v)\in B\subset\exp_y[B^y_{L_3L_1\eta^4\zeta^4}(\nu_y(C_xv))]=:\psi_y[E],$$
where $E:=C_0^{-1}(y)[B_{L_3L_1\eta^4\zeta^4}^{y}(\nu_y(C_xv))]$.

Notice that $I^{-2}(\eta)-I^{-1}(\eta)<-\frac{\Gamma}{2}\cdot \eta^{1+\frac{1}{\gamma}} $ for small enough $\epsilon>0$.\footnote{Set $t:=I^{-2}(\eta)$, then $I^{-2}(\eta)-I^{-1}(\eta)=t-I(t)=y(1-e^{\Gamma t^\frac{1}{\gamma}})\leq -t\Gamma t^\frac{1}{\gamma}=-\Gamma\cdot (I^{-2}(\eta))^{1+\frac{1}{\gamma}}\leq -\frac{\Gamma}{2}\cdot \eta^{1+\frac{1}{\gamma}} $, for $\eta>0$ small enough w.r.t $\gamma$.} We claim that $E\subset R_\zeta(0)$. First note that $E\subset B_{\|C_0^{-1}(y)\|L_3L_1\eta^4\zeta^4}(C_y^{-1}C_xv)$, therefore if $w\in E$ then:
\begin{align*}
|w|_\infty\leq&|C_y^{-1}C_xv|_\infty+\|C_0^{-1}(y)\|L_3L_1\eta^4\zeta^4\\
\leq&|(C_y^{-1}C_x-Id)v|_\infty+|v|_\infty+\|C_0^{-1}(y)\|L_3L_1\eta^4\zeta^4\\
\leq&|v|_\infty+\sqrt{d}\|C_y^{-1}\|\cdot\|C_x-C_y\|\cdot|v|_\infty+\|C_0^{-1}(y)\|L_3L_1\eta^4\zeta^4\\
\leq& I^{-2}(\eta)+\|C_0^{-1}(y)\|(\eta^4\zeta^4\sqrt{d}\cdot I^{-2}(\eta)+L_3L_1\eta^4\zeta^4)\text{  }(\because\|C_x-C_y\|\leq\eta^4\zeta^4)\\
\leq &I^{-1}(\eta)+I^{-2}(\eta)-I^{-1}(\eta)+\eta^4 \leq I^{-1}(\eta)-\frac{\Gamma}{2}\eta^{1+\frac{1}{\gamma}}+\eta^4\\ \leq & I^{-1}(\eta)-\frac{\Gamma}{2}\eta^{2}+\eta^4 < I^{-1}(\eta) \leq \zeta \text{ }(\because \gamma\geq\frac{2}{\beta}\geq1,\eta<\epsilon\text{ small w.r.t }\Gamma).
\end{align*}
It follows that $E\subset R_\zeta(0)$. Thus $\psi_x(v)\in\psi_y[R_\zeta(0)]$. Part 1 follows.

\textit{Part 2:} $\mathrm{dist}_{C^{1+\beta/2}}(\psi_x^{-1}\circ\psi_y,Id)<\epsilon\eta^3\zeta^3$ on $R_{I^{-1}(r(M))}(0)$

Proof: One can show exactly as in the proof of part 1 that $\psi_x[R_{I^{-1}(r(M))}(0)]\subset\psi_y[R_{r(M)}(0)]$, therefore $\psi_x^{-1}\circ\psi_y$ is well defined on $R_{I^{-1}(r(M))}(0)$. We calculate the distance of this map from the identity:
\begin{align}\label{C2reg}\psi_x^{-1}\circ\psi_y=&C_x^{-1}\circ\nu_x^{-1}\circ\exp_x^{-1}\circ\exp_y\circ\nu_y\circ C_y\\\nonumber
=&C_x^{-1}\circ[\nu_x^{-1}\circ\exp_x^{-1}+\nu_y^{-1}\circ\exp_y^{-1}-\nu_y^{-1}\circ\exp_y^{-1}]\circ\exp_y\circ\nu_y\circ C_y\\\nonumber
=&C_x^{-1}C_y+C_x^{-1}\circ[\nu_x^{-1}\circ\exp_x^{-1}-\nu_y^{-1}\circ\exp_y^{-1}]\circ\psi_y\\\nonumber
=&Id+C_x^{-1}(C_y-C_x)+C_x^{-1}\circ[\nu_x^{-1}\circ\exp_x^{-1}-\nu_y^{-1}\circ\exp_y^{-1}]\circ\psi_y.\nonumber
\end{align}
The $C^{1+\beta/2}$ norm of the second summand is less than $\|C_x^{-1}\|\eta^4\zeta^4$. The $C^{1+\beta/2}$ norm of the third summand is less than $\|C_x^{-1}\|L_2d(x,y)L_4^{1+\beta/2}$. This is less than $\|C_x^{-1}\|L_2L_4^2\eta^4\zeta^4$. It follows that $\mathrm{dist}_{C^{1+\beta/2}}(\psi_x^{-1}\circ\psi_y,Id)<\|C_x^{-1}\|(1+L_2L_4^2)\eta^4\zeta^4$. This is (much) smaller than $\epsilon\eta^3\zeta^3$.

\end{proof}

\noindent\textbf{Remarks}:
\begin{enumerate}
	\item Notice that \eqref{C2reg} in fact gives a bound for the $C^2$-distance when $\epsilon$ is small enough. 
	\item By item (2) in the proposition above, the greater the distortion of $\psi_x$ or $\psi_y$ the closer they are one to another. This ``distortion compensating bound" will be used in what follows to argue that $\psi_{f(x)}^{-1}\circ f\circ\psi_x$ remains close to a linear hyperbolic map if we replace $\psi_{f(x)}$ by an overlapping map $\psi_y$  (Proposition \ref{3.4inomris} below).

\end{enumerate}

\begin{lemma}\label{overlap} Suppose $\psi_{x_1}^{\eta_1},\psi_{x_2}^{\eta_2}$ $I$-overlap, and write $C_1=\Theta_D\circ C_0(x_1), C_2=\Theta_D\circ C_0(x_2)$ for the $D$ s.t $D\ni x_1,x_2$. Then \begin{enumerate}
\item 
$\|C_1^{-1}-C_2^{-1}\|<2\epsilon \eta_1 \eta_2$,
\item $\frac{\|C_1^{-1}\|}{\|C_2^{-1}\|}=e^{\pm \eta_1\eta_2}.$
\end{enumerate}
\end{lemma}
The proof is similar to \cite[Proposition~3.2]{Sarig}.
\subsubsection{The form of $f$ in overlapping charts}

We introduce a notation: for two vectors $u\in\mathbb{R}^{d_1}$ and $v\in\mathbb{R}^{d_2}$, $(-u-,-v-)\in\mathbb{R}^{d_1+d_2}$ means the new vector whose coordinates are the coordinates of these two, put in the same order as written.
\begin{prop}\label{3.4inomris} The following holds for all $\epsilon$ small enough. Suppose $x,y\in 0$-$\mathrm{summ}$, $s(x)=s(y)$ and $\psi_{f(x)}^\eta$ $I$-overlaps $\psi_y^{\eta'}$, then $f_{xy}:=\psi_y^{-1}\circ f\circ\psi_x$ is a well defined injective map from $R_{Q_\epsilon(x)}(0)$ to $\mathbb{R}^d$, and there are matrices $D_s:\mathbb{R}^{s(x)}\rightarrow\mathbb{R}^{s(x)}, D_u:\mathbb{R}^{d-s(x)}\rightarrow\mathbb{R}^{d-s(x)}$ and differentiable maps $h_s:R_{Q_\epsilon(x)}(0)\rightarrow \mathbb{R}^{s(x)},h_u:R_{Q_\epsilon(x)}(0)\rightarrow \mathbb{R}^{u(x)}$, s.t $f_{xy}$ can be put in the form
$$f_{xy}(-v_s-,-v_u-)=(D_s v_s+h_s(v_s,v_u),D_u v_u+h_u(v_s,v_u)),$$ where $v_{s/u}$ are the ``stable"/"unstable" components of the input vector $u$ for $\psi_x$; and $\kappa^{-1}\leq\|D_s^{-1}\|^{-1},\|D_s\|\leq e^{-\frac{1}{S^2(x)}}$, $e^\frac{1}{U^2(x)}\leq\|D_u^{-1}\|^{-1},\|D_u\|\leq\kappa$, $|h_{s/u}(0)|<\epsilon\eta^3$, $\|\frac{\partial(h_s,h_u)}{\partial(v_s,v_u)}\|<\sqrt{\epsilon}\eta^{\beta/2}$ on $R_\eta(0)$ (in particular $\|\frac{\partial(h_s,h_u)}{\partial(v_s,v_u)}\rvert_0\|<\sqrt{\epsilon}\eta^{\beta/2}$), and $\|\frac{\partial(h_s,h_u)}{\partial(v_s,v_u)}\rvert_{v_1}-\frac{\partial(h_s,h_u)}{\partial(v_s,v_u)}\rvert_{v_2}\|\leq\sqrt{\epsilon}|v_1-v_2|^{\beta/2}$ on $R_{Q_\epsilon(x)}(0)$. A similar statement holds for $f_{xy}^{-1}$, assuming that $\psi_{f^{-1}(y)}^{\eta'}$ $I$-overlaps $\psi_x^\eta$.
\end{prop}
\begin{proof} We write $f_{xy}=(\psi_y^{-1}\circ\psi_{f(x)})\circ f_x$, and treat $f_{xy}$ as a perturbation of $f_x$. Recall Theorem \ref{linearization}. For small enough $\epsilon$, 
 $f_x$ has the following properties:
\begin{itemize}
\item It is well-defined, differentiable, and injective on $R_{Q_\epsilon(x)}(0)$.
\item $f_x(0)=0, d_0f_x=D_0(x)$.
\item For all $v_1,v_2\in R_{Q_\epsilon(x)}(0)$: $\|d_{v_1}f_x-d_{v_2}f_x\|\leq2\epsilon|v_1-v_2|^{\beta/2}$.

(Because the $C^{1+\beta/2}$ distance between $f_x$ and $d_0f_x$ on $R_{Q_\epsilon(x)}(0)$ is less than $Q_\epsilon(x)$)
\item For every $0<\eta<Q_\epsilon(x)$ and $v\in R_\eta(0):\|d_vf_x\|<3\kappa$, provided $\epsilon$ is small enough (because $\|d_vf_x\|\leq\|d_0f_x\|+\epsilon\eta^{\beta/2}<2\kappa+\epsilon$).
\end{itemize}
The second and fourth points imply that $f_x[R_{Q_\epsilon(x)}(0)]\subset B_{3\kappa Q_\epsilon(x)}(0)$. Now since $Q_\epsilon(x)<\epsilon^{3/\beta}$: $f_x[R_{Q_\epsilon(x)}(0)]\subset B_{3\kappa\epsilon^{3/\beta}}(0)$. If $\epsilon$ is so small that $3\kappa\epsilon^{3/\beta}<I^{-1}(r)$, then $f_x[R_{Q_\epsilon(x)}(0)]\subset B_{I^{-1}(r)}(0)\subset R_{I^{-1}(r)}(0)$, which is the domain of $\psi_y^{-1}\circ\psi_{f(x)}$ according to Proposition \ref{chartsofboxes}. Therefore $f_{xy}$ is well-defined, differentiable and injective on $R_{Q_\epsilon(x)}(0)$.

$h_s$ is defined by taking the first $s(x)$ coordinates of $f_{xy}(v_s,v_u)$, and subtracting $D_sv_s$: basically the equation in the statement defines $h_s,h_u$. It is left to check the properties:

$$(-h_s(0)-,-h_u(0)-)=f_{xy}(0)=\psi_y^{-1}(f(x))=(\psi_y^{-1}\circ\psi_{f(x)})(0)\text{ therefore }$$ $$|(-h_s(0)-,-h_u(0)-)|\leq d_{C^0}(\psi_y^{-1}\circ\psi_{f(x)},Id)<(\because\text{Proposition \ref{chartsofboxes}})<\epsilon\eta^3(\eta')^3<\epsilon\eta^3.$$ 
Now, by rearranging the derivative of $(\psi_y^{-1}\circ\psi_{f(x)})\circ f_x$:
$$d_vf_{xy}=[d_{f_x(v)}(\psi_y^{-1}\circ\psi_{f(x)})-Id]d_vf_x+[d_vf_x-d_0f_x]+d_0f_x.$$
The norm of the first summand is bounded by $3\kappa\epsilon\eta^3$. The norm of the second summand is less than $\epsilon|v|^{\beta/2}<\sqrt{d}\epsilon\eta^{\beta/2}$. The third term is $D_0(x)=d_0f_x$. Thus
\begin{align}\label{alsoForC2reg}\|\frac{\partial(h_s,h_u)}{\partial(v_s,v_u)}\|=\|d_vf_{xy}-D_0(x)\|<\epsilon\eta^{\beta/2}[3\kappa+\sqrt{d}]
\leq \sqrt{\epsilon}\eta^\frac{\beta}{2}\text{ for small enough }\epsilon>0.\end{align}
 In particular: $\|\frac{\partial(h_s,h_u)}{\partial(v_s,v_u)}\rvert_0\|<\sqrt{\epsilon}\eta^{\beta/2}$. From the expression for $d_vf_{xy}$ we get that for every $v_1,v_2\in R_{Q_\epsilon(x)}(0)$:
$$\|d_{v_1}f_{xy}-d_{v_2}f_{xy}\|\leq\|d_{f_x(v_1)}(\psi_y^{-1}\circ\psi_{f(x)})-d_{f_x(v_2)}(\psi_y^{-1}\circ\psi_{f(x)})\|\cdot\|d_{v_1}f_x\|+\|d_{v_1}f_x-d_{v_2}f_x\|\cdot(\|d_{f_x(v_1)}(\psi_y^{-1}\circ\psi_{f(x)})\|+1)$$
Using that $d_{C^{1+\beta/2}}(\psi_y^{-1}\circ\psi_{f(x)},Id)<\epsilon\eta^3$ (Proposition \ref{chartsofboxes}), we see that
\begin{align*}
    \|d_{v_1}f_{xy}-d_{v_2}f_{xy}\|\leq & \epsilon\eta^3|f_x(v_1)-f_x(v_2)|^{\beta/2}\cdot3\kappa+2\epsilon|u-v|^{\beta/2}(\epsilon\eta+2)\\ \leq & \epsilon\eta^3\sup\limits_{w\in R_{Q_\epsilon(x)}(0)}\|d_wf_x\|^{\beta/2}\cdot\|v_1-v_2\|^{\beta/2}\cdot3\kappa+5\epsilon|v_1-v_2|^{\beta/2}\\ \leq & \epsilon((3\kappa)^{1+\beta/2}\eta+5)|v_1-v_2|^{\beta/2}\leq6\epsilon|v_1-v_2|^{\beta/2}\leq \sqrt{\epsilon}|v_1-v_2|^{\beta/2}\text{, provided }\epsilon\text{ is small enough.}
\end{align*}
\end{proof}

\subsubsection{Coarse graining}
Recall the definition of $s(x)$ for a point $x\in 0$-$\mathrm{summ}$ in Definition  \ref{0-summ}
.
\begin{prop}\label{discreteness} The following holds for all $\epsilon>0$ small enough: There exists a countable collection $\mathcal{A}$ of Pesin charts with the following properties:
\begin{enumerate}
\item {\em Discreteness}:  $\{\psi\in\mathcal{A}:\psi=\psi_x^\eta,\eta>t\}$ is finite for every $t>0$.
\item {\em Sufficiency}: For every $x\in \RST$ and for every sequence of positive numbers $0<\eta_n\leq I^\frac{-1}{4}(Q_\epsilon(f^n(x)))$ in $\mathcal{I}$ s.t $\eta_n=I^{\pm1}(\eta_{n+1})$ there exists a sequence $\{\psi_{x_n}^{\eta_n}\}_{n\in\mathbb{Z}}$ of elements of $\mathcal{A}$ s.t for every $n$:
\begin{enumerate}
\item $\psi_{x_n}^{\eta_n}$ $I$-overlaps $\psi_{f^n(x)}^{\eta_n}$, $Q_\epsilon(f^n(x))=I^\frac{\pm1}{4}(Q_\epsilon(x_n))$ and $s(x_n)=s(f^n(x))$ \normalfont($=s(x)$, since Lyapunov exponents and dimensions are $f$-invariant);
\item $\psi_{f(x_n)}^{\eta_{n+1}}$ $I$-overlaps $\psi_{x_{n+1}}^{\eta_{n+1}}$;
\item $\psi_{f^{-1}(x_n)}^{\eta_{n-1}}$ $I$-overlaps $\psi_{x_{n-1}}^{\eta_{n-1}}$;
\item $\psi_{x_n}^{\eta_n'}\in\mathcal{A}$ for all $\eta_n'\in \mathcal{I}$ s.t $\eta_n\leq\eta_n'\leq\min\{Q_\epsilon(x_n),I(\eta_n)\}$.
\end{enumerate}
\end{enumerate}
\end{prop}
\begin{proof} The proof is the same as the proof of \cite[Proposition~3.5]{Sarig} except for two difference: One, in the higher-dimensional case $s(x)\in\{1,...,d-1\}$ has $d-1$ possible values instead of just one. To deal with this we apply the discretization of \cite{Sarig} to $\RST\cap [s(x)=s]$ for each $s=1,..., d-1$. Two, we use the finer approximation of $I$-overlapping charts rather than Sarig's $\epsilon$-overlapping charts, which is still possible by exactly the same arguments of pre-compactness by Sarig. 

\end{proof}
\subsection{$I$-chains and infinite-to-one Markov extension of $f$}\label{epsilonchains}
\subsubsection{Double charts and $I$-chains}

Recall that $\psi_x^\eta$ ($0<\eta\leq Q_\epsilon(x)$) stands for the Pesin chart $\psi_x:R_\eta(0)\rightarrow M$. An {\em $I$-double Pesin chart} (or just \enquote{double chart}) is a pair $\psi_x^{p^s,p^u}:=(\psi_x^{p^s},\psi_x^{p^u})$ where $0<p^u,p^s\leq Q_\epsilon(x)$.
\begin{definition}\label{edges} $\psi_{x}^{p^s,p^u}\rightarrow\psi_{y}^{q^s,q^u}$ means :
\begin{itemize}
\item $\psi_{x}^{p^s\wedge p^u}$ and $\psi_{f^{-1}(y)}^{p^s\wedge p^u}$ $I$-overlap;
\item $\psi_{f(x)}^{q^s\wedge q^u}$ and $\psi_{y}^{q^s\wedge q^u}$ $I$-overlap;
\item $q^u=\min\{I (p^u),Q_\epsilon(y)\}$ and $p^s=\min\{I(q^s),Q_\epsilon(x)\}$;
\item $s(x)=s(y)$.
\end{itemize}
\end{definition}
\begin{definition}\label{defepsilonchains} (see \cite{Sarig}, Definition 4.2)
$\{\psi_{x_i}^{p^s_i,p^u_i}\}_{i\in\mathbb{Z}}$ (resp. 
$\{\psi_{x_i}^{p^s_i,p^u_i}\}_{i\geq0}$ and 
$\{\psi_{x_i}^{p^s_i,p^u_i}\}_{i\leq0}$) is called an {\em $I$-chain} (resp. positive , negative $I$-chain) if $\psi_{x_i}^{p^s_i,p^u_i}\rightarrow\psi_{x_{i+1}}^{p^s_{i+1},p^u_{i+1}}$ for all $i$. We abuse terminology and drop the $I$ in ``$I$-chains".
\end{definition}

Let $\mathcal{A}$ denote the countable set of Pesin charts we have constructed in \textsection1.2.3 and recall that $\mathcal{I}=\{I^\frac{-\ell}{4}(1):\ell\geq0\}$.
\begin{definition}\label{graphosaurus} $\mathcal{G}$ is the directed graph with vertices $\mathcal{V}$ and $\mathcal{E}$ where:
\begin{itemize}
\item $\mathcal{V}:=\{\psi_{x}^{p^s,p^u}:\psi_{x}^{p^s\wedge p^u}\in \mathcal{A},p^s,p^u\in \mathcal{I}, p^s,p^u\leq Q_\epsilon(x)\}$.
\item $\mathcal{E}:=\{(\psi_{x}^{p^s,p^u},\psi_{y}^{q^s,q^u})\in \mathcal{V}\times\mathcal{V}:\psi_{x}^{p^s,p^u}\rightarrow\psi_{y}^{q^s,q^u}\}$.
\end{itemize}
\end{definition}
This is a countable directed graph. Every vertex has a finite degree, because of the following lemma and the discretization of $\mathcal{A}$ (Proposition \ref{discreteness}). Let 
$$\Sigma(\mathcal{G}):=\{x\in\mathcal{V}^\mathbb{Z}:(x_i,x_{i+1})\in\mathcal{E} ,\forall i\in\mathbb{Z}\}.$$
$\Sigma(\mathcal{G})$ is the Markov shift associated with the directed graph $\mathcal{G}$. It comes equipped with the left-shift $\sigma$, and the standard metric.

\begin{lemma}\label{lemma131} If $\psi_x^{p^s,p^u}\rightarrow\psi_y^{q^s,q^u}$ then $q^u\wedge q^s=I^{\pm1}(p^u\wedge p^s)$. Therefore for every $\psi_x^{p^s,p^u}\in\mathcal{V}$ there are only finitely many $\psi_y^{q^s,q^u}\in\mathcal{V}$ s.t $\psi_x^{p^s,p^u}\rightarrow\psi_y^{q^s,q^u}$ or $\psi_y^{q^s,q^u}\rightarrow\psi_x^{p^s,p^u}$.
\end{lemma}
\begin{proof}
Since $\psi_x^{p^s,p^u}\rightarrow\psi_y^{q^s,q^u}$: $q^u=\min\{I(p^u),Q_\epsilon(y)\}$ and $p^s=\min\{I(q^s),Q_\epsilon(x)\}$. Also recall $q^s\leq Q_\epsilon(y),p^u\leq Q_\epsilon(x)$; and that $I$ is a strictly increasing continuous function. It follows that
\begin{align*}q^u\wedge q^s=&\min\{I(p^u),Q_\epsilon(y),q^s\}=\min\{I(p^u),q^s\}\leq \min\{I(p^u),I^2(q^s)\}=I(\min \{p^u,I(q^s)\})\\
=& I(\min \{p^u,Q_\epsilon(x),I(q^s)\})= I(p^u\wedge\min\{Q_\epsilon(x),I(q^s)\})= I(p^u\wedge p^s).
\end{align*}
Similarly, $p^s\wedge p^u\leq I(q^u\wedge q^s)$.

\end{proof}

\begin{definition} Let $(Q_k)_{k\in\mathbb{Z}}$ be a sequence in $\mathcal{I}=\{I^\frac{-\ell}{4}(1)\}_{l\in\mathbb{N}}$. A sequence of pairs $\{(p^s_k,p^u_k)\}_{k\in\mathbb{Z}}$ is called {\em $I$-strongly subordinated} to $(Q_k)_{k\in\mathbb{Z}}$ if for every $k\in\mathbb{Z}$: \begin{itemize}
\item $0<p^s_k,p^u_k\leq Q_k$
\item $p^s_k,p^u_k\in \mathcal{I}$
\item $p_{k+1}^u=\min\{I( p_k^u),Q_{k+1}\}$ and $p_{k-1}^s=\min\{I (p_k^s),Q_{k-1}\}$
\end{itemize}
\end{definition}
For example, if $\{\psi_{x_k}^{p^s_k,p^u_k}\}_{k\in\mathbb{Z}}$ is a chain, then $\{(p^s_k,p^u_k)\}_{k\in\mathbb{Z}}$ is $I$-strongly subordinated to $\{Q_\epsilon(x_k)\}_{k\in\mathbb{Z}}$.

\begin{lemma}\label{subordinatedchain} (Compare with \cite[Lemma~4.6]{Sarig})
Let $(Q_k)_{k\in\mathbb{Z}}$ be a sequence in $\mathcal{I}$, and suppose $q_k\in \mathcal{I}$ satisfy $0<q_k\leq Q_k$ and $q_k=I^{\pm1}(q_{k+1})$ for all $k\in\mathbb{Z}$. Then there exists a sequence $\{(p^s_k,p^u_k)\}_{k\in\mathbb{Z}}$ which is $I$-strongly subordinated to $(Q_k)_{k\in\mathbb{Z}}$ and $p^s_k\wedge p^u_k\geq q_k$ for all $k\in\mathbb{Z}$.
\end{lemma}
\begin{proof}
By the assumptions on $q_k$: $Q_{k-n},Q_{k+n}\geq I^{-n}(q_k)$ for all $n\geq0$, therefore the following are well defined:
$$p_k^u:=\max\{t\in \mathcal{I}: I^{-n}(t)\leq Q_{k-n}, \forall n\geq0\},$$
$$p_k^s:=\max\{t\in \mathcal{I}: I^{-n}(t)\leq Q_{k+n}, \forall n\geq0\}.$$
The sequence $\{(p^s_k,p^u_k)\}_{k\in\mathbb{Z}}$ is $I$-strongly subordinated to $(Q_k)_{k\in\mathbb{Z}}$.
\end{proof}

\begin{lemma}\label{forrecurrenceinnextone}
 Suppose $\{(p^s_k,p^u_k)\}_{k\in\mathbb{Z}}$ is $I$-strongly subordinated to a sequence $(Q_k)_{k\in\mathbb{Z}}\subset \mathcal{I}$. If $\limsup_{n\rightarrow\infty}(p^s_n\wedge p^u_n)>0$ and $\limsup_{n\rightarrow-\infty}(p^s_n\wedge p^u_n)>0$, then $p_n^u$ (resp. $p_n^s$) is equal to $Q_n$ for infinitely many $n>0$, and for infinitely many $n<0$.
\end{lemma}
\begin{proof}
We prove the statement for $p_n^u$ (the proof for $p_n^s$ is similar). 
Let $p_n:=p_n^u\wedge p_n^s$ and define $m:=\frac{1}{2}\min\{\limsup_{n\rightarrow\infty}p_n,\limsup_{n\rightarrow\infty}p_{-n}\}$ and $N:=\min\{N'\in\mathbb{N}:I^{-N'}(1)< m\}$ (recall $I^{-n}(1)\downarrow0$ as $n\uparrow\infty$, see Lemma \ref{forI}). There exist infinitely many positive (resp. negative) $n$ s.t $p_n>m$. We claim that for every such $n$, there must exist some $k\in[n,n+N]$ s.t $p_k^u=Q_n$. Otherwise, by $I$-strong subordination:
$$p^u_{n+N}=\min\{Q_{n+N},I(p^u_{n+N-1})\}=I( p^u_{n+N-1})=...=I^{N}(p^u_n)\geq I^{N}(m)> 1.$$
A contradiction.
\end{proof}

\begin{prop}\label{prop131}
For every $x\in \RST$ there is a chain $\{\psi_{x_k}^{p_k^s,p_k^u}\}_{k\in\mathbb{Z}}\subset\Sigma(\mathcal{G})$ s.t $\psi_{x_k}^{p_k^s\wedge p_k^u}$ $I$-overlaps $\psi_{f^k(x)}^{p_k^s\wedge p_k^u}$ for every $k\in\mathbb{Z}$.
\end{prop}
\begin{proof} This follows from Proposition \ref{discreteness} as in \cite[Proposition~4.5]{Sarig}. We give the details below to account for the differences in the setups.

Suppose $x\in  \RST$ and let $\{q(f^n(x))\}_{n\in\mathbb{Z}}$ be given by its strong temperability. Choose $q_n\in \mathcal{I}\cap [I^{-\frac{1}{4}}(q(f^n(x))),I^{\frac{1}{4}}(q(f^n(x)))]$. The sequence $\{q_n\}_{n\in\mathbb{Z}}$ satisfies the assumptions of Lemma \ref{subordinatedchain}, therefore there exists a sequence $\{(q_n^s,q_n^u)\}_{n\in\mathbb{Z}}$ that is $I$-strongly subordinated to $\{I^\frac{-1}{4}(Q_\epsilon(f^n(x)))\}_{n\in\mathbb{Z}}$ and that satisfies $q_k^u\wedge q_k^s\geq q_k$. Let $\eta_n:=q_n^u\wedge q_n^s$. As the proof of Lemma \ref{lemma131} shows: $\eta_{n+1}=I^{\pm1}(\eta_n)$, so we may use Proposition \ref{discreteness} to construct an infinite sequence $\psi_{x_n}^{\eta_n}\in\mathcal{A}$ such that \begin{enumerate}
\item $\psi_{x_n}^{\eta_n}$ $I$-overlaps $\psi_{f^n(x)}^{\eta_n}$, $Q_\epsilon(f^n(x))=I^{\pm\frac{1}{4}}(Q_\epsilon(x_n))$ and $s(x)=s(f^n(x))=s(x_n)$
\item $\psi_{f(x_n)}^{\eta_{n+1}}$ $I$-overlaps $\psi_{x_{n+1}}^{\eta_{n+1}}$
\item $\psi_{f^{-1}(x_n)}^{\eta_{n-1}}$ $I$-overlaps $\psi_{x_{n-1}}^{\eta_{n-1}}$
\item $\psi_{x_n}^{\eta_n'}\in\mathcal{A}$ for all $\eta_n'\in \mathcal{I}$ s.t $\eta_n\leq\eta_n'\leq\min\{Q_\epsilon(x_n),I(\eta_n)\}$
\end{enumerate}
Using Lemma \ref{subordinatedchain}, we construct a sequence $\{(p_n^s,p_n^u)\}_{n\in\mathbb{Z}}$ which is $I$-strongly subordinated to $\{Q_\epsilon(x_n)\}_{n\in\mathbb{Z}}$ and which satisfies $p_n^s\wedge p_n^u\geq \eta_n$.

Claim 1: $\psi_{x_n}^{p_n^s,p_n^u}\in\mathcal{V}$ for all $n$.

Proof: It's sufficient to show that $q_n^u\wedge q_n^s\leq p_n^u\wedge p_n^s\leq I(q_n^u\wedge q_n^s)\text{ }(n\in\mathbb{Z})$, because property (4) with $\eta_n':=p_n^s\wedge p_n^u$ says that in this case $\psi_{x_n}^{p_n^u,p_n^s}\in\mathcal{A}$, whence $\psi_{x_n}^{p_n^u,p_n^s}\in\mathcal{V}$. We start by showing that there are infinitely many $n<0$ such that $p_n^u\leq I( q_n^u)$. Since $x\in \RST$, $\limsup_{n\rightarrow\infty}q_n,\limsup_{n\rightarrow-\infty}q_n>0$. Therefore by Lemma \ref{forrecurrenceinnextone} there are infinitely many $n<0$ for which $q_n^u=I^\frac{-1}{4}(Q_\epsilon(f^n(x)))$. Property (1) guarantees that for such $n$, $q_n^u\geq I^\frac{-2}{4}(Q_\epsilon(x_n))\geq I^\frac{-2}{4}(p_n^u)$, whence $p_n^u<I(q_n^u)$.

If $p_n^u\leq I( q_n^u)$, then $p_{n+1}^u\leq I( q_{n+1}^u)$ because
\begin{align*}
p_{n+1}^u=&\min\{I(p^u_n),Q_\epsilon(x_{n+1})\}=I(\min\{p_n^u,I^{-1}(Q_\epsilon(x_{n+1}))\})\\
\leq &I(\min\{I( q_n^u),I^\frac{-3}{4}(Q_\epsilon(f^{n+1}(x)))\})\leq I(\min\{I( q_n^u),I^\frac{-1}{4}(Q_\epsilon(f^{n+1}(x)))\})\equiv I( q_{n+1}^u).	
\end{align*}

It follows that $p_n^u\leq I(q_n^u)$ for all $n\in\mathbb{Z}$. Working with positive $n$ one can show in the same manner that $p_n^s\leq I(q_n^s)$ for all $n\in\mathbb{Z}$. Combining the two results we see that $p_n^u\wedge p_n^s\leq I(q_n^s)\wedge I( q_n^u)=I(q_n^u\wedge q_n^s)$ for all $n\in\mathbb{Z}$. Since by construction $p_n^s\wedge p_n^u\geq\eta_n=q_n^s\wedge q_n^u$ we obtain $q_n^u\wedge q_n^s\leq p_n^u\wedge p_n^s\leq I(q_n^u\wedge q_n^s)$ as needed.

Claim 2: For every $k\in\mathbb{Z}$: $\psi_{x_k}^{p_k^s,p_k^u}\rightarrow\psi_{x_{k+1}}^{p_{k+1}^s,p_{k+1}^u}$ and $\psi_{x_k}^{p_k^s\wedge p_k^u}$ $I$-overlaps $\psi_{f^k(x)}^{p_k^s\wedge p_k^u}$ 

Proof:  This follows from properties (1),(2) and (3) above, the inequality $p_n^s\wedge p_n^u\geq\eta_n$ and the monotonicity property of the ovrelap condition.
\end{proof}

\section{``Shadowing lemma"- admissible manifolds and the Graph Transform}\label{GTrans}
In this section we prove the refined shadowing lemma: for every $I$-chain $\{\psi_{x_i}^{p_i^s,p_i^u}\}_{i\in\mathbb{Z}}$ there is some $x\in M$ s.t $f^i(x)\in \psi_{x_i}[R_{p^s_i\wedge p^u_i}(0)]$ for all $i$. We say that the chain shadows the orbit of $x$.

The construction of $x$, and the properties of the shadowing operation, are established using a Graph Transform argument.

We present a suitable space of admissible manifolds, which is preserved under the action of the linear differential in charts, and prove the convergence of the Graph Transform in the absence of uniform estimates in charts.

The following definition is motivated by \cite{KM,Sarig}.
\begin{definition}\label{def135} Recall $s(x)=\dim H^s(x)$.
Let $x\in 0$-$\mathrm{summ}$, a {\em $u$-manifold} in $\psi_x$ is a manifold $V^u\subset M$ of the form
$$V^u=\psi_x[\{(F_1^u(t_{s(x)+1},...,t_d),...,F_{s(x)}^u(t_{s(x)+1},...,t_d),t_{s(x)+1},...t_d) : |t_i|\leq q\}],$$
where $0<q\leq Q_\epsilon(x)$, and $\overrightarrow{F}^u$ is a $C^{1+\beta/2}$ function s.t $\max\limits_{\overline{R_q(0)}}|\overrightarrow{F}^u|_\infty\leq Q_\epsilon(x)$.

Similarly we define an {\em $s$-manifold} in $\psi_x$:
$$V^s=\psi_x[\{(t_1,...,t_{s(x)},F_{s(x)+1}^s(t_1,...,t_{s(x)}),...,F_d^s(t_1,...,t_{s(x)})): |t_i|\leq q\}],$$
with the same requirements for $\overrightarrow{F}^s$ and $q$. We will use the superscript ``$u/s$" in statements which apply to both the $s$ case and the $u$ case. The function $\overrightarrow{F}=\overrightarrow{F}^{u/s}$ is called the {\em representing function} of $V^{u/s}$ at $\psi_x$. The parameters of a $u/s$ manifold in $\psi_x$ are: 
\begin{itemize}
\item $\sigma$-parameter: $\sigma(V^{u/s}):=\|d_{\cdot}\overrightarrow{F}\|_{\beta/2}:=\max\limits_{\overline{R_q(0)}}\|d_{\cdot}\overrightarrow{F}\|+\text{H\"ol}_{\beta/2}(d_{\cdot}\overrightarrow{F})$,

where $\text{H\"ol}_{\beta/2}(d_{\cdot}\overrightarrow{F}):=\max\limits_{\vec{t_1},\vec{t_2}\in\overline{R_q(0)}}\{\frac{\|d_{\overrightarrow{t_1}}\overrightarrow{F}-d_{\overrightarrow{t_2}}\overrightarrow{F}\|}{|\overrightarrow{t_1}-\overrightarrow{t_2}|^{\beta/2}}\}$ and $\|A\|:=\sup\limits_{v\neq0}\frac{|Av|_\infty}{|v|_\infty}$.
\item $\iota$-parameter: $\iota(V^{u/s}):=\|d_0\overrightarrow{F}\|$.
\item $\varphi$-parameter: $\varphi(V^{u/s}):=|\overrightarrow{F}(0)|_\infty$.
\item $q$-parameter: $q(V^{u/s}):=q$.
\end{itemize}

A {\em $(u/s,\sigma,\iota,\varphi,q
)$-manifold} in $\psi_x$ is a $u/s$-manifold $V^{u/s}$ in $\psi_x$ whose parameters satisfy $\sigma(V^{u/s})\leq\sigma,\iota(V^{u/s})\leq\iota,\varphi(V^{u/s})\leq\varphi,q(V^{u/s})\leq q
$.
\end{definition}
\begin{definition}\label{admissible} Suppose $\psi_x^{p^u,p^s}$ is a double chart. A $u/s$-admissible manifold in $\psi_x^{p^u,p^s}$ is a {\em $(u/s,\sigma,\iota,\varphi,q
)$-manifold} in $\psi_x$ s.t
$$\sigma\leq\frac{1}{2},\iota\leq\frac{1}{2}(p^u\wedge p^s)^{\beta/2},\varphi\leq10^{-3}(p^u\wedge p^s)^2,  \text{ and } q = \begin{cases} p^u & u\text{-manifolds} \\ 
p^s & s\text{-manifolds}. \end{cases}
$$
\end{definition}
\noindent\textbf{Remarks}: 
\begin{enumerate}
	\item 
This definition is 
slightly different than its parallels in \cite{Sarig,SBO}, since we require the $\varphi$-parameter to be bounded by $10^{-3}(p^s\wedge p^u)^2$ rather than $10^{-3}(p^s\wedge p^u)$. This stronger bound is needed for some of our estimates in the Graph Transform (see the bound for $\iota$ in Theorem \ref{graphtransform}(1)). The tradeoff is that we need to show that the stronger bound also carries respectively to the transformed manifold, as we do there in claim 2.
	\item If $\epsilon<1$ (as we always assume), then these conditions together with $p^u,p^s<Q_\epsilon(x)$ force $$\mathrm{Lip}(\overrightarrow{F})=\max\limits_{\vec{t_1},\vec{t_2}\in\overline{R_q(0)}}\frac{|\overrightarrow{F}(\overrightarrow{t_1})-\overrightarrow{F}(\overrightarrow{t_2})|_\infty}{|\overrightarrow{t}_1-\overrightarrow{t}_2|_\infty} \leq (p^{u/s})^\frac{\beta}{2} <\epsilon,$$ because for every $\overrightarrow{t}_{u/s}$ in the domain of $\overrightarrow{F}$, $|\overrightarrow{t}_{u/s}|_\infty\leq p^{u/s}\leq Q_\epsilon(x)<\epsilon^{3/\beta}$ and $\|d_{\overrightarrow{t}_{u/s}}F\|\leq\|d_0F\|+\text{H\"ol}_{\beta/2}(d_{\cdot}\overrightarrow{F})\cdot|\overrightarrow{t}_{u/s}|_\infty^{\beta/2}\leq\frac{1}{2}(p^u\wedge p^s)^{\beta/2}+\frac{1}{2}(p^{u/s})^{\beta/2}\leq(p^{u/s})^{\beta/2}<\epsilon$ and by Lagrange's mean-value theorem applied to the restriction of $\vec{F}$ to the interval connecting each $\vec{t}_1$  and $\vec{t}_2$, we are done.
\item If $\epsilon$ is small enough then $\max\limits_{\overline{R_q(0)}}|\overrightarrow{F}|_\infty<10^{-2}Q_\epsilon(x)$, since $\max\limits_{\overline{R_q(0)}}|\overrightarrow{F}|_\infty\leq|\overrightarrow{F}(0)|_\infty+\mathrm{Lip}(\overrightarrow{F})\cdot p^{u/s}<\varphi+\epsilon p^{u/s}<(10^{-3}+\epsilon)p^{u/s}<10^{-2}p^{u/s}$.
\end{enumerate}
\begin{definition}\label{def137}$\text{ }$\\\begin{enumerate}
    \item Let $V_1,V_2$ be two $u$-manifolds (resp. $s$-manifolds) in $\psi_x$ s.t $q(V_1)=q(V_2)$, then:
$$\mathrm{dist}(V_1,V_2):=\max|F_1-F_2|_\infty,\mathrm{dist}_{C^1}(V_1,V_2):=\max|F_1-F_2|_\infty+\max\|d_{\cdot}F_1-d_{\cdot}F_2\|$$
\item For a map $E:\mathrm{Dom}\rightarrow M_{r_1}(\mathbb{R})$, where $M_{r_1}(\mathbb{R})$ is the space of real matrices of dimensions $r_1\times r_1$, and $\mathrm{Dom}$ is the closure of some open and bounded subset of $\mathbb{R}^{r_2}$ ($r_1,r_2\in\mathbb{N}$), the {\em $\alpha$-norm} of $E(\cdot)$ is
$$\|E(\cdot)\|_{\alpha}:=\|E(\cdot)\|_\infty+\text{H\"ol}_{\alpha}(E(\cdot))\text{, where, }$$

$\|E(\cdot)\|_\infty=\sup\limits_{s\in\mathrm{Dom}}\|E(s)\|=\sup\limits_{s\in\mathrm{Dom}}\sup\limits_{v\neq 0}\frac{|E(s)v|_\infty}{|v|_\infty}$, and $\text{H\"ol}_{\alpha}(E(\cdot))=\sup\limits_{s_1,s_2 \in\mathrm{Dom}}\{\frac{\|E(s_1)-E(s_2)\|}{|s_1-s_2|_\infty^\alpha}\}$.
\end{enumerate}
\end{definition}
\noindent\textbf{Remarks}: \begin{enumerate}
    \item The second part of this definition is a generalization of the specific $\frac{\beta}{2}$-norm of admissible manifolds, as represented by the $\sigma$ parameter in Definition \ref{def135}.
    \item  Notice that the difference between $\|\cdot\|$ and $\|\cdot\|_\infty$ in our notation, is that $\|\cdot\|$ is the operator norm, while $\|\cdot\|_\infty$ is the supremum norm for operator valued maps.
    \item $\|\cdot\|_\alpha$ is submultiplicative: for any two  functions $\varphi$ and $\psi$ as in Definition \ref{def137}(2),
$$\|\psi\cdot\varphi\|_\alpha\leq\|\varphi\|_\alpha\cdot\|\psi\|_\alpha.
\footnote{To see that, it is enough to show that $\text{H\"ol}_\alpha(\varphi\cdot\psi)\leq \text{H\"ol}_\alpha(\psi)\cdot\|\varphi\|_\infty+\text{H\"ol}_\alpha(\varphi)\cdot\|\psi\|_\infty$. This holds since,
$$\sup\limits_{x,y\in\mathrm{Dom}}\frac{|\varphi(x)\psi(x)-\varphi(y)\psi(y)|}{|x-y|^\alpha}=\sup\limits_{x,y\in\mathrm{Dom}}\frac{|\varphi(x)\psi(x)-\varphi(x)\psi(y)+\varphi(x)\psi(y)-\varphi(y)\psi(y)|}{|x-y|^\alpha}\leq \mathrm{\mathrm{RHS}}.$$} $$
\end{enumerate}
\begin{lemma}\label{banachalgebras}
Let $E:\mathrm{dom}(E)\rightarrow M_r(\mathbb{R})$ be a map as in Definition \ref{def137}. Then for any $0<\alpha<1$, if $\|E(\cdot)\|_\alpha<1$, then  $$\|(Id+E(\cdot))^{-1}\|_\alpha\leq\frac{1}{1-\|E(\cdot)\|_\alpha}$$
\end{lemma}
\begin{proof}

In addition to the third item of the remark above, it is also easy to check that $BH:=\{\varphi:dom(E)\rightarrow M_r(\mathbb{R}): \|\varphi\|_\alpha<\infty\}$ is a complete Banach space, and hence a Banach algebra. By a well known lemma in Banach algebras, if $\|E\|_\alpha<1$ then $\|(Id+E(\cdot))^{-1}\|_\alpha\leq\frac{1}{1-\|E(\cdot)\|_\alpha}$.

\end{proof}

\begin{prop}\label{firstbefore} The following holds for all $\epsilon$ small enough. Let $V^u$ be a $u$-admissible manifold in $\psi_x^{p^s,p^u}$, and $V^s$ an $s$-admissible manifold in $\psi_x^{p^s,p^u}$, then:
\begin{enumerate}
\item $V^u$ and $V^s$ intersect at a unique point $P$,
\item $P=\psi_x(\overrightarrow{t})$ with $|\overrightarrow{t}|_\infty\leq10^{-2}(p^s\wedge p^u)^2$,
\item $P$ is a Lipschitz function of $(V^s,V^u)$, with Lipschitz constant less than 3 (w.r.t the distance $\mathrm{dist}(\cdot,\cdot)$ in Definition \ref{def137}).
\end{enumerate}
\end{prop}
\begin{proof} Assume $\epsilon\in(0,\frac{1}{2})$. Remark: We will omit the $\overrightarrow{\cdot}$ notation when it is clear that the object under inspection is a vector in $\mathbb{R}^t$ ($t=s(x),u(x)$). Write $V^u=\psi_x[\{(-F(t_u)-,-t_u-): |t_u|_\infty\leq p^u\}]$ and $V^s=\psi_x[\{(-t_s-,-G(t_s)-): |t_s|_\infty\leq p^s\}]$.

Let $\eta:=p^s\wedge p^u$. Note that $\eta<\epsilon$ and that $|F(0)|_\infty,|G(0)|_\infty\leq10^{-3}\eta^2$ and $\mathrm{Lip}(F),\mathrm{Lip}(G)<\epsilon$ by the remark following Definition \ref{admissible}. Hence the maps $F$ and $G$ are contractions, and they map the closed cubes $\overline{R_{10^{-2}\eta^2}(0)}$ in the respective dimensions into themselves: for every $(H,t)\in\{(F,t_u),(G,t_s)\}$:
$$|H(t)|_\infty\leq|H(0)|_\infty+\mathrm{Lip}(H)|t|_\infty<10^{-3}\eta^2+\epsilon10^{-2}\eta^2=(10^{-1}+\epsilon)10^{-2}\eta^2<10^{-2}\eta^2.$$
It follows that $G\circ F$ is a $\epsilon^2$ contraction of $\overline{R_{10^{-2}\eta^2}(0)}$ into itself. By the Banach fixed point theorem it has a unique fixed point: $G\circ F(w)=w$. Denote $v=F(w)$. We claim $V^s,V^u$ intersect at $$P=\psi_x(-F(w)-,-w-)=\psi_x(-v-,-w-)=\psi_x(-v-,-G(v)-),$$
\begin{itemize}
\item $P\in V^u$ since $v=F(w)$ and $|w|_\infty\leq10^{-2}\eta^2\ll p^u$.
\item $P\in V^s$ since $w=G\circ F(w)=G(v)$ and $|v|_\infty<|F(0)|_\infty+\mathrm{Lip}(F)|w|_\infty\leq10^{-3}\eta^2+\epsilon10^{-2}\eta^2<10^{-2}\eta^2\ll p^s$.
\end{itemize}
We also see that $|v|_\infty,|w|_\infty\leq10^{-2}\eta^2$.

We claim that $P$ is the unique intersection point of $V^s,V^u$. Denote $\xi=p^u \vee p^s$ and extend $F,G$ arbitrarily to $\epsilon$-Lipschitz functions $\widetilde{F},\widetilde{G}:\overline{R_\xi(0)}\rightarrow\overline{R_{Q_\epsilon(x)}(0)}$ 
using McShane's extension formula \cite{McShane}. \footnote{$\widetilde{F}_i(s)=\sup\limits_{u\in dom(F_i)}\{F_i(u)-\mathrm{Lip}(F_i)\cdot |s-u|_\infty\},\widetilde{G}_i(s)=\sup\limits_{u\in dom(G_i)}\{G_i(u)-\mathrm{Lip}(G_i)\cdot |s-u|_\infty\}$ for each coordinate of $F,G$, which in turn induces $\epsilon$-Lipschitz extensions to $F,G$ since the relevant norm is $|\cdot|_\infty$.\label{lipextensions}}
Let $\widetilde{V}^s,\widetilde{V}^u$
denote the $u/s$-sets represented by $\widetilde{F},\widetilde{G}$. Any intersection point of $V^s,V^u$ is an intersection point of $\widetilde{F},\widetilde{G}$, which takes the form $\widetilde{P}=\psi_x(-\widetilde{v}-,-\widetilde{w}-)$ where $\widetilde{v}=\widetilde{F}(\widetilde{w})$ and $\widetilde{w}=\widetilde{G}(\widetilde{v})$. Notice that $\widetilde{w}$ is a fixed point of $\widetilde{G}\circ\widetilde{F}$. The same calculations as before show that $\widetilde{G}\circ\widetilde{F}$ contracts $\overline{R_\xi(0)}$ into itself. Such a map has a unique fixed point--- hence $\widetilde{w}=w$, whence $\widetilde{P}=P$. This concludes parts 1,2.

Next we will show that $P$ is a Lipschitz function of $(V^s,V^u)$. Suppose $V_i^u,V_i^s$ ($i=1,2$) are represented by $F_i$ and $G_i$ respectively. Let $P_i$ denote the intersection points of $V_i^s\cap V_i^u$. We saw above that $P
_i=\psi_x(-v_i-,-w_i-)$, where $w_i$ is a fixed point of $f_i:=G_i\circ F_i:\overline{R_{10^{-2}\eta^2}(0)}\rightarrow\overline{R_{10^{-2}\eta^2}(0)}$. Then $f_i$ are $\epsilon^2$ contractions of $\overline{R_{10^{-2}\eta}(0)}$ into itself. Therefore:
\begin{align*}
|w_1-w_2|_\infty=&|f_1^n(w_1)-f_2^n(w_2)|_\infty\leq|f_1(f_1^{n-1}(w_1))-f_2(f_1^{n-1}(w_1))|_\infty+|f_2(f_1^{n-1}(w_1))-f_2(f_2^{n-1}(w_2))|_\infty\\
\leq&\max\limits_{\overline{R_{10^{-2}\eta^2}(0)}}|f_1-f_2|_\infty+\epsilon^2|f_1^{n-1}(w_1)-f_2^{n-1}(w_2)|_\infty\\
\leq&\cdots\leq\max\limits_{\overline{R_{10^{-2}\eta^2}(0)}}|f_1-f_2|_\infty(1+\epsilon^2+...+\epsilon^{2(n-1)})+\epsilon^{2n}|w_1-w_2|_\infty.
\end{align*}
Passing to the limit $n\rightarrow\infty$: 
\begin{align*}
|w_1-w_2|_\infty\leq(1-\epsilon^2)^{-1}\max\limits_{\overline{R_{10^{-2}\eta^2}(0)}}|f_1-f_2|_\infty.
\end{align*}
Similarly $v_i$ is a fixed point of $g_i=F_i\circ G_i$, a contraction from $\overline{R_{10^{-2}\eta^2}(0)}$ to itself. The same arguments give $|v_1-v_2|_\infty\leq(1-\epsilon^2)^{-1}\max\limits_{\overline{R_{10^{-2}\eta^2}(0)}}|g_1-g_2|_\infty$.
Since $\psi_x$ is 2-Lipschitz,
$$d(P_1,P_2)<\frac{2}{1-\epsilon^2}(\max\limits_{\overline{R_{10^{-2}\eta^2}(0)}}|g_1-g_2|_\infty+\max\limits_{\overline{R_{10^{-2}\eta^2}(0)}}|f_1-f_2|_\infty).$$
Now,
\begin{align*}
\max\limits_{\overline{R_{10^{-2}\eta^2}(0)}}|F_1\circ G_1-F_2\circ G_2|_\infty\leq&\max\limits_{\overline{R_{10^{-2}\eta^2}(0)}}|F_1\circ G_1-F_1\circ G_2|_\infty+\max\limits_{\overline{R_{10^{-2}\eta^2}(0)}}|F_1\circ G_2-F_2\circ G_2|_\infty\\
\leq &\mathrm{Lip}(F_1)\max\limits_{\overline{R_{10^{-2}\eta^2}(0)}}|G_1-G_2|_\infty+\max\limits_{\overline{R_{10^{-2}\eta^2}(0)}}|F_1-F_2|_\infty.
\end{align*}
Similarly,
$$\max\limits_{\overline{R_{10^{-2}\eta^2}(0)}}|G_1\circ F_1-G_2\circ F_2|_\infty\leq \mathrm{Lip}(G_1)\max\limits_{\overline{R_{10^{-2}\eta^2}(0)}}|F_1-F_2|_\infty+\max\limits_{\overline{R_{10^{-2}\eta^2}(0)}}|G_1-G_2|_\infty.$$
So since $\mathrm{Lip}(G),\mathrm{Lip}(F)\leq\epsilon$: $d(P_1,P_2)<\frac{2(1+\epsilon)}{1-\epsilon^2}[\mathrm{dist}(V_1^s,V_2^s)+\mathrm{dist}(V_1^u,V_2^u)]$.

It remains to choose $\epsilon$ small enough so that $\frac{2(1+\epsilon)}{1-\epsilon^2}<3$.

\end{proof}

The next theorem analyzes the action of $f$ on admissible manifolds. Similar Graph Transform Lemmas were used to show Pesin's stable manifold theorem (\cite{BP} chapter 7, \cite{Pesin}), see also \cite{perron1,Perron2}. Our analysis shows that the Graph Transform preserves admissibility. The general idea is similar to what Sarig does in \cite[\textsection~4.2]{Sarig} for the two dimensional case. Katok and Mendoza treat the general dimensional case in \cite{KM}. Our proof complements theirs by establishing additional analytic properties.

Recall the notation $\psi_x^{p^s,p^u}\rightarrow\psi_y^{q^s,q^u}$ from Definition \ref{edges}.
\begin{theorem}\label{graphtransform}
(Graph Transform) The following holds for all $\epsilon$ small enough: suppose $\psi_x^{p^s,p^u}\rightarrow\psi_y^{q^s,q^u}$, and $V^u$ is a $u$-admissible manifold in $\psi_x^{p^s,p^u}$, then:
\begin{enumerate}
\item $f[V^u]$ contains a $u$-manifold $\widehat{V}^u$ in $\psi_y^{q^s,q^u}$ with parameters:
\begin{itemize}
\item $\sigma(\widehat{V}^u)\leq e^{-\frac{1}{S^2(x)}-\frac{1}{U^2(x)}+8\sqrt{\epsilon}(p^u)^\frac{\beta}{2}}\cdot[\sigma+2\kappa\sqrt{\epsilon}(p^u)^\frac{\beta}{2}] $,
\item $\iota(\widehat{V}^u)\leq e^{-\frac{1}{U^2(x)}-\frac{1}{S^2(x)}+\sqrt{\epsilon}(p^u)^\frac{\beta}{2}}[\iota+\epsilon \eta^\beta]$,
\item $\varphi(\widehat{V}^u)\leq e^{-\frac{1}{S^2(x)}}[\varphi+\kappa\sqrt{\epsilon}\eta^{2+\frac{\beta}{2}}] $
.
\end{itemize}
\item $f[V^u]$ intersects any $s$-admissible manifold in $\psi_y^{q^s,q^u}$ at a unique point.
\item $\widehat{V}^u$ restricts to a $u$-manifold in $\psi_y^{q^s,q^u}$. This is the unique $u$-admissible manifold in $\psi_y^{q^s,q^u}$ inside $f[V^u]$. We call it $\mathcal{F}_u(V^u)$.
\item Suppose $V^u$ is represented by the function $F$. If $p:=(F(0),0)$ then $f(p)\in\mathcal{F}_u(V^u)$.
\end{enumerate}
Similar statements hold for the $f^{-1}$-image of an $s$-admissible manifold in $\psi_y^{q^s,q^u}$.
\end{theorem}
\begin{proof}
Let $V^u=\psi_x[\{(F_1^u(t_{s(x)+1},...,t_d),...,F_{s(x)}^u(t_{s(x)+1},...,t_d),t_{s(x)+1},...t_d) : |t_i|\leq p^u\}]$ be a $u$-admissible manifold in $\psi_x^{p^s,p^u}$. We omit the $^u$ super-script of $F$ to ease notations, and denote the parameters of $V^u$ by $\sigma,\iota,\varphi$ and $q$. Let $\eta:=p^s\wedge p^u$. By the admissibility of $V^u$, $\sigma\leq\frac{1}{2},\iota\leq\frac{1}{2}\eta^{\beta/2},\varphi\leq10^{-3}\eta, q=p^u$, and $\mathrm{Lip}(F)<\epsilon$. We analyze $\Gamma_y^u=\psi_y^{-1}[f[V^u]]\equiv f_{xy}[\mathrm{Graph} (F)]$, where $f_{xy}=\psi_y^{-1}\circ f\circ\psi_x$ and $graph(F):=\{(-F(t)-,-t-): |t|_\infty\leq q\}$. Since $V^u$ is admissible, $graph(F)\subset R_{Q_\epsilon(x)}(0)$. On this domain 
(Proposition \ref{3.4inomris}),
$$f_{xy}(v_s,v_u)=(D_sv_s+h_s(v_s,v_u),D_uv_u+h_u(v_s,v_u)),$$
where $\kappa^{-1}\leq\|D_s^{-1}\|^{-1},\|D_s\|\leq e^{-\frac{1}{S^2(x)}}$ and $e^\frac{1}{U^2(x)}\leq\|D_u^{-1}\|^{-1},\|D_u\|\leq\kappa$, 
and $\|\frac{\partial(h_s,h_u)}{\partial(v_s,v_u)}\rvert_{v_1}-\frac{\partial(h_s,h_u)}{\partial(v_s,v_u)}\rvert_{v_2}\|\leq\sqrt{\epsilon}|v_1-v_2|^{\beta/2}$.
By Proposition \ref{3.4inomris}, 
\begin{equation}\label{remarkproof} \|\frac{\partial(h_s,h_u)}{\partial(v_s,v_u)}\|<\sqrt{\epsilon} (p^u)^{\beta/2}\text{ on }R_{p^u}(0)\text{ and }|h_{s/u}(0)|<\epsilon \eta^3
.
\end{equation} Using the equation for $f_{xy}$ we can put $\Gamma_y^u$ in the following form:
\begin{equation}\label{astrix}
    \Gamma_y^u=\{(D_sF(t)+h_s(F(t),t),D_ut+ h_u(F(t),t): |t|_\infty\leq q\}, \text{ } \Big( t=(t_{s(x)+1},...,t_d)\Big).
\end{equation}
The idea (as in \cite{KM} and \cite{Sarig}) is to call the ``$u$" part of the coordinates $\tau$, solve for $t=t(\tau)$, and then substitute the result in the ``$s$" coordinates.\\

\textit{Claim 1}: The following holds for all $\epsilon$ small enough: $D_ut+h_u(F(t),t)=\tau$ has a unique solution $t=t(\tau)$ for all $\tau\in R_{e^{\frac{1}{U^2(x)}-\frac{1}{2}(p^u)^\frac{\beta}{2}}q}(0)$, and
\begin{enumerate}[label=(\alph*)]
\item $\mathrm{Lip}(t)<e^{-\frac{1}{U^2(x)}+(p^u)^\frac{\beta}{2}}$,
\item $|t(0)|_{\infty}<2\sqrt{\epsilon}\eta^{2+\frac{\beta}{2}}$,
\item The $\frac{\beta}{2}$-norm of $d_{\cdot}t$ is smaller than $e^{-\frac{1}{U^2(x)}}e^{8\sqrt{\epsilon}(p^u)^\frac{\beta}{2}}$.
\end{enumerate}
\textit{Proof}: Let $\tau(t):=D_ut+h_u(F(t),t)$. For every $|t|_\infty\leq q$ and a unit vector $v$:
\begin{align*}
|d_t\tau v|\geq & |D_u v|-\max\|\frac{\partial(h_s,h_u)}{\partial(v_s,v_u)}\|\cdot(\|d_tF\|+u(x))\geq|D_u v|-\max\|\frac{\partial(h_s,h_u)}{\partial(v_s,v_u)}\|\cdot(\|d_tF\|+d)\\
\geq & |D_u v|-\sqrt{\epsilon} (p^u)^\frac{\beta}{2}(\epsilon+d)\text{ }(\because\text{ }\eqref{remarkproof}\text{ })\\
\geq & |D_u v|(1-\sqrt{\epsilon} (p^u)^\frac{\beta}{2}(\epsilon+d))(\because |D_u v|\geq\|D_u^{-1}\|^{-1} \geq e^\frac{1}{U^2(x)}>1)\\
>&e^{-(p^u)^\frac{\beta}{2}}\|D_u^{-1} \|^{-1}\text{(provided }\epsilon\text{ is small enough)}.
\end{align*}
Since $v$ was arbitrary, it follows that $\tau$ is expanding by a factor of at least $e^{-(p^u)^\frac{\beta}{2}}\|D_u^{-1}\|^{-1}\geq e^{\frac{1}{U^2(x)}-Q_\epsilon^\frac{\beta}{2}(x)}>1$ ($\because \gamma\geq \frac{2}{\beta}$), whence one-to-one. Since $\tau$ is one-to-one, $\tau^{-1}$ is well-defined on $\tau[\overline{R_{q}(0)}]$. We estimate this set: Since $\tau$ is continuous and $e^{-(p^u)^\frac{\beta}{2}}\|D_u^{-1}\|^{-1}$ expanding (recall, these calculations are done w.r.t the supremum norm, as defined in Definition \ref{def135}): $\tau[\overline{R_{q}(0)}]\supset \overline{R_{e^{-(p^u)^\frac{\beta}{2}}\|D_u^{-1}\|^{-1}q}(\tau(0))}$. The center of the box can be estimated as follows:
\begin{align}\label{againAlsoForC2}|\tau(0)|_\infty=&|h_u(F(0),0)|_\infty\leq |h_u(0,0)|_\infty+\max\|\frac{\partial(h_s,h_u)}{\partial(v_s,v_u)}\||F(0)|_\infty\nonumber\\
\leq&\epsilon\eta^3+\sqrt{\epsilon}\eta^\frac{\beta}{2}\cdot 10^{-3} \eta^2\leq2\sqrt{\epsilon}\eta^{2+\frac{\beta}{2}}.\end{align}
Recall that $\eta\leq q$, and therefore $|\tau(0)|_\infty\leq2\sqrt{\epsilon} q^{2+\frac{\beta}{2}}\ll 2\sqrt{\epsilon} q^{\frac{\beta}{2}} $, and hence 

$$\tau[\overline{R_{q}(0)}]\supset \overline{R_{(e^{-(p^u)^\frac{\beta}{2}}\|D_u^{-1}\|^{-1}+2\sqrt{\epsilon} (p^u) ^{\frac{\beta}{2}})q}(0)}\supset \overline{R_{\|D_u^{-1}\|^{-1}(e^{-(p^u)^\frac{\beta}{2}}+2\epsilon (p^u)^\frac{\beta}{2})q}(0)}\supset \overline{R_{e^\frac{1}{U^2(x)} e^{-\frac{1}{2} (p^u)^\frac{\beta}{2}}q}(0)}$$ for $\epsilon$ small enough. So $\tau^{-1}$ is well-defined on this domain.

Since $t(\cdot)$ is the inverse of a $\|D_u^{-1}\|^{-1}e^{-(p^u)^\frac{\beta}{2}}$-expanding map, $\mathrm{Lip}(t)\leq e^{(p^u) ^\frac{\beta}{2}}\|D_u^{-1}\|<e^{-\frac{1}{U^2(x)}+ (p^u) ^\frac{\beta}{2}}$.

\medskip
We saw above that $|\tau(0)|_\infty<2\sqrt{\epsilon} \eta^{2+\frac{\beta}{2}}$. For all $\epsilon$ small enough, this is significantly smaller than $e^{\frac{1}{U^2(x)}-\frac{1}{2}(p^u)^\frac{\beta}{2}}q$, therefore $\tau(0)$ belongs to the domain of $t$. It follows that:
\begin{equation}\label{secondastrix}
    |t(0)|_\infty=|t(0)-t(\tau(0))|_\infty\leq \mathrm{Lip}(t)|\tau(0)|_\infty<
    2\sqrt{\epsilon}\eta^{2+\frac{\beta}{2}} e^{-\frac{1}{U^2(x)}+ (p^u)^\frac{\beta}{2}}.
\end{equation}
For all $\epsilon$ small enough this is less than $2\sqrt{\epsilon}\eta^{2+\frac{\beta}{2}}$, proving (b).

\medskip
Next we will calculate the $\frac{\beta}{2}$-norm of $d_\vartheta t$. We use the following notation: Let

$G(\vartheta):=(F(t(\vartheta)),t(\vartheta))$, then $d_{G(\vartheta)}h_u$ is a $u(x)\times d$ matrix, where $u(x):=d-s(x)$. The matrix $d_{t(\vartheta)}F$ is a $u(x)\times s(x)$ matrix, then let $A'(\vartheta)$:=\begin{tabular}{| l |}
\hline $d_{t(\vartheta)}F $  \\ \hline
$I_{u(x)\times u(x)}$   \\ \hline
\end{tabular}, and $A(\vartheta):=d_{G(\vartheta)}h_u \cdot A'(\vartheta)$; where the notation in the definition of $A'$ means that it is the matrix created by stacking the two matrices $d_{t(\vartheta)}F$ (represented in the standard bases, as implied by the notation) and $I_{u(x)\times u(x)}$.
Using this and the identity $$\vartheta=\tau(t(\vartheta))=D_ut(\vartheta)+h_u(F(t(\vartheta)),t(\vartheta))$$  we get: \begin{equation}\label{banana}
    d_\vartheta t=(D_u+A(\vartheta))^{-1}=(Id+D_u^{-1}A(\vartheta))^{-1}D_u^{-1}.
    \end{equation}
    From Lemma \ref{banachalgebras} we see that in order to bound the $\frac{\beta}{2}$-norm of $d_\cdot t$, it is enough to show that the $\frac{\beta}{2}$-norm of $D_u^{-1}A(\cdot)$ is less than 1. The lemma is applicable since $d_\vartheta t$ is a square matrix, as both $t$ and $\vartheta$ in this setup are vectors in $\mathbb{R}^{u(x)}$ (recall, $u(x)=d-s(x)$).
    
Here is the bound of the $\frac{\beta}{2}$-norm of $D_u^{-1}A$:

\begin{align*}
\|A(\vartheta_1)-A(\vartheta_2)\|=&\|d_{G(\vartheta_1)}h_uA'(\vartheta_1)-d_{G(\vartheta_2)}h_uA'(\vartheta_2)\|\\
=&\|d_{G(\vartheta_1)}h_uA'(\vartheta_1)-d_{G(\vartheta_1)}h_uA'(\vartheta_2)+d_{G(\vartheta_1)}h_uA'(\vartheta_2)-d_{G(\vartheta_2)}h_uA'(\vartheta_2)\|\\ \leq&\|d_{G(\vartheta_1)}h_u[A'(\vartheta_1)-A'(\vartheta_2)]\|+\|[d_{G(\vartheta_1)}h_u-d_{G(\vartheta_2)}h_u]A'(\vartheta_2)\|\\
\leq&\Big(\mathrm{Lip}(t)^\frac{\beta}{2}\cdot\|d_{G(\vartheta_1)}h_u\|\cdot \text{H\"ol}_{\beta/2}(d_\cdot F)+\text{H\"ol}_{\beta/2}(d_\cdot h_u)\cdot \mathrm{Lip}(G)^\frac{\beta}{2}\cdot\|d_{t(\vartheta_2)}F\|\Big)|\vartheta_1-\vartheta_2|_\infty^{\beta/2}.
\end{align*}
Recall:
\begin{itemize}
\item $\mathrm{Lip}(F)<(p^u)^\frac{\beta}{2}$
\item $\mathrm{Lip}(t)\leq e^{-\frac{1}{U^2(x)}+(p^u)^\frac{\beta}{2}}$
\item $\text{H\"ol}_{\beta/2}(d_\cdot h_u)\leq \sqrt{\epsilon}$
\item $d_0h_u\leq\sqrt{\epsilon}\eta^{\beta/2}$
\end{itemize}
From the identity $G(\vartheta)=(F(t(\vartheta)),t(\vartheta))$ we get that Lip($G)\leq\mathrm{Lip}(F)\cdot\mathrm{Lip}(t)+\mathrm{Lip}(t)$; and from that and the first two items above we deduce that $\mathrm{Lip}(G)\leq 1$, whence $\mathrm{Lip}(G)^\frac{\beta}{2}\leq 1$. We also know $\|d_{t(r)}F\|\leq \mathrm{Lip}(F)\leq (p^u)^\frac{\beta}{2}
$  (by the remark after Definition \ref{admissible}). From the last two items we deduce $\|d_\cdot h_u\|_\infty\leq(\sqrt{\epsilon}p^u)^{\beta/2}+\sqrt{\epsilon}(p^u)^{\beta/2}$. From the admissibility of the manifold represented by $F$, we get H\"ol$_{\frac{\beta}{2}}(d_\cdot F)\leq \frac{1}{2}$.

So it is clear that for a small enough $\epsilon$, $\text{H\"ol}_{\beta/2}(A(\cdot))<2\epsilon^2(p^u)^\frac{\beta}{2}$, and hence $\text{H\"ol}_{\beta/2}(D_u^{-1}A(\cdot))<2\epsilon^2(p^u)^\frac{\beta}{2} $. Also, $\|D_u^{-1}A(0)\|\leq e^{-\frac{1}{U^2(x)}}\|A(0)\|=e^{-\frac{1}{U^2(x)}}\|d_{(F(0),t(0))}h_uA'(0)\|$, and by Proposition \ref{3.4inomris} $\|d_{G(\vartheta)}h_u\|\leq\sqrt{\epsilon}(p^u)^\frac{\beta}{2}$. Since $\|d_\cdot F\|\leq(p^u)^\frac{\beta}{2}\leq \epsilon$,
for all small enough $\epsilon$, $\|A
(0)\|\leq \sqrt{\epsilon}(p^u)^\frac{\beta}{2}(1+\epsilon)<2 \sqrt{\epsilon}(p^u)^\frac{\beta}{2} $, and $\|D_u^{-1}A(0)\|\leq e^{-\frac{1}{U^2(x)}}2 \sqrt{\epsilon}(p^u)^\frac{\beta}{2} $. Hence $\|A(\vartheta)\|\leq\|A(0)\|+\text{H\"ol}_\frac{\beta}{2}(A(\cdot))\cdot(e^\frac{1}{U^2(x)} q)^\frac{\beta}{2}\leq3\sqrt{\epsilon} (p^u)^\frac{\beta}{2}$.

In total, $\|A(\cdot)\|_\frac{\beta}{2}\leq 3\sqrt{\epsilon} (p^u)^\frac{\beta}{2} +2\epsilon^2(p^u)^\frac{\beta}{2}<4\sqrt{\epsilon}(p^u)^\frac{\beta}{2}$, and so $\|D_u^{-1}A(\cdot)\|_{\beta/2}<e^{-\frac{1}{U^2(x)}} 4\sqrt{\epsilon}(p^u)^\frac{\beta}{2} <1$. Thus we can use Lemma \ref{banachalgebras} and \eqref{banana} to get
$$\| d_\cdot t\|_{\beta/2}\leq \frac{1}{1-4\sqrt{\epsilon}(p^u)^\frac{\beta}{2}}\cdot\|D_u^{-1}\|\leq e^{8\sqrt{\epsilon}(p^u)^\frac{\beta}{2}}\cdot e^{-\frac{1}{U^2(x)}}
,$$
for small enough $\epsilon$. Here we used that $D_u$ is at least $e^\frac{1}{U^2(x)}$-expanding. This completes claim 1. 

\medskip We now substitute $t=t(\vartheta)$ in \eqref{astrix}, and find that $$\Gamma_y^u\supset\{(H(\vartheta),\vartheta):|\vartheta|_\infty<e^{\frac{1}{U^2(x)}-\frac{1}{2}(p^u)^\frac{\beta}{2}}q\},$$ where $\vartheta\in \mathbb{R}^{u(x)}$ (as it is the right set of coordinates of $(H(\vartheta),\vartheta)$), and $H(\vartheta):=D_sF(t(\vartheta))+h_s(F(t(\vartheta)),t(\vartheta))$. Claim 1 guarantees that $H$ is well-defined and $C^{1+\beta/2}$ on $R_{e^{\frac{1}{U^2(x)}-\frac{1}{2}(p^u)^\frac{\beta}{2}}q}(0)$. We find the parameters of $H$:

\textit{Claim 2}: For all $\epsilon$ small enough, $|H(0)|_\infty
\leq e^{-\frac{1}{S^2(x)}}[\varphi+\kappa\sqrt{\epsilon}\eta^{2+\frac{\beta}{2}}] $, and $|H(0)|_\infty<10^{-3}(q^u\wedge q^s)^2$.

\textit{Proof}: 
By \eqref{secondastrix}, $|t(0)|_\infty\leq 2\sqrt{\epsilon}\eta^{2+\frac{\beta}{2}} e^{-\frac{1}{U^2(x)}+(p^u)^\frac{\beta}{2}} $. Since $\mathrm{Lip}(F)<(p^u)^\frac{\beta}{2},|F(0)|_\infty<\varphi$ and $\varphi\leq10^{-3}\eta^2$: $|F(t(0))|_\infty<\varphi+\epsilon\eta^{2+\frac{\beta}{2}}
$ provided $\epsilon$ is small enough. Thus
\begin{align*}
|H(0)|_\infty\leq&\|D_s\||F(t(0))|_\infty+|h_s(F(t(0)),t(0))|_\infty\\ \leq&\|D_s\| (\varphi+\epsilon\eta^{2+\frac{\beta}{2}})
+[|h_s(0)|_\infty+\max\|d_\cdot h_s\||(F(t(0)),t(0))|_\infty]\\
\leq&\|D_s\|(\varphi+\epsilon\eta^{2+\frac{\beta}{2}})
+[\epsilon\eta^3+\sqrt{\epsilon}\eta^\frac{\beta}{2}
\max\{|F(t(0))|_\infty,|t(0)|_\infty\}]\\
\leq&\|D_s\|(\varphi+\epsilon\eta^{2+\frac{\beta}{2}})+
3\epsilon\eta^{2+\frac{\beta}{2}}(\because \mathrm{Lip}(F)<1,|t(0)|_\infty<2\sqrt{\epsilon}\eta^{2+\frac{
\beta}{2}})\\
\leq&\|D_s\|[\varphi+\kappa\sqrt{\epsilon}\eta^{2+\frac{\beta}{2}}]= 10^{-3}(q^s\wedge q^u)^2\cdot \|D_s\|\cdot[1+ 10^3\kappa\sqrt{\epsilon}\eta^{\frac{\beta}{2}}]\cdot (\frac{p^s\wedge p^u}{q^s\wedge q^u})^2.
\end{align*}
Recalling that $\|D_s\|\leq e^{-\frac{1}{S^2(x)}}$,  $\eta\equiv p^u\wedge p^s\leq I(q^u\wedge q^s)$ and $\varphi\leq10^{-3}(p^u\wedge p^s)^2$:
\begin{align}
|H(0)|_\infty\leq & 10^{-3}(q^s\wedge q^u)^2 \cdot e^{-\frac{1}{S^2(x)}+\kappa 10^3 \sqrt{\epsilon}\eta^\frac{\beta}{2}+2\Gamma (q^s\wedge q^u)^\frac{1}{\gamma}}\\
\leq &10^{-3}(q^s\wedge q^u)^2 \cdot e^{-\frac{1}{S^2(x)}+\kappa 10^3 \sqrt{\epsilon}\eta^\frac{\beta}{2}+2\Gamma e^{\epsilon}\eta^\frac{1}{\gamma}}
\leq 10^{-3}(q^s\wedge q^u)^2 \cdot e^{-\frac{1}{S^2(x)}+3\Gamma \eta^\frac{1}{\gamma}},\nonumber
\end{align} and for all $\epsilon$ sufficiently small this is less than $10^{-3}(q^u\wedge q^s)^2$. The claim follows.

\medskip
\textit{Claim 3}: For all $\epsilon$ small enough, $\|d_0H\|<e^{-\frac{1}{U^2(x)}-\frac{1}{S^2(x)}+\sqrt{\epsilon}(p^u)^\frac{\beta}{2}}[\iota+\epsilon \eta^{\beta}]
$, and $\|d_0H\|<\frac{1}{2}(q^u\wedge q^s)^{\beta/2}$.

\textit{Proof}: $\|d_0H\|\leq\|d_0t\|[\|D_s\|\cdot\|d_{t(0)}F\|+\|d_{(F(t(0)),t(0))}h_s\|\cdot\max\|A'(0)\|]$ and
\begin{itemize}
\item $\|d_0t\|\leq \mathrm{Lip}(t)<e^{-\frac{1}{U^2(x)}+(p^u)^\frac{\beta}{2}}$ (claim 1).

\item By 
\eqref{alsoForC2reg} we can get the following fine bound:
\begin{align*}
\|d_{(F(t(0)),t(0))}h_s\|\leq & 2\epsilon | (F(t(0)),t(0)) |_\infty^\frac{\beta}{2}=2\epsilon |t(0)|^\frac{\beta}{2}\leq \epsilon^2 \eta^\beta
.
\end{align*}

\item $\|d_{t(0)}F\|
\leq\|d_0F\|+\text{H\"ol}_{\beta/2}(d_\cdot F)|t(0)|_\infty^{\beta/2}\leq\iota+\sigma\cdot(2\sqrt{\epsilon}\eta^{2+\frac{\beta}{2}})^{\beta/2}\leq\iota+ \epsilon^2\eta^{\beta}
$ for small enough $\epsilon>0$.

\item $\|A'(0)\|\leq 1+\epsilon<2$.
\end{itemize}

These estimates together give us that:
\begin{align*}
\|d_0H\|<&e^{-\frac{1}{U^2(x)}+\sqrt{\epsilon}(p^u)^\frac{\beta}{2}}\|D_s\|[\iota+\epsilon^2\eta^{\beta}+\|D_s\|^{-1} \epsilon^2\eta^\beta\cdot2]\leq e^{-\frac{1}{U^2(x)}-\frac{1}{S^2(x)}+\sqrt{\epsilon}(p^u)^\frac{\beta}{2}}[\iota+ \eta^{\beta}(\epsilon^2+ 2\kappa\epsilon^2)]\\
\leq &e^{-\frac{1}{U^2(x)}-\frac{1}{S^2(x)}+\sqrt{\epsilon} (p^u)^\frac{\beta}{2}}[\iota+(q^u\wedge q^s)^{\beta/2}\cdot \eta^\frac{\beta}{2}\epsilon^\frac{3}{2}(\frac{I(q^s\wedge q^u)}{q^s\wedge q^u})^\frac{\beta}{2}]\\
\leq &e^{-\frac{1}{U^2(x)}-\frac{1}{S^2(x)}+\sqrt{\epsilon} (p^u)^\frac{\beta}{2}}[\iota+(q^u\wedge q^s)^{\beta/2}\cdot \eta^\frac{\beta}{2}\epsilon^\frac{5}{4}]\text{ for all sufficiently small }\epsilon>0,\\
=& \frac{1}{2} (q^u\wedge q^s)^{\beta/2} \cdot e^{-\frac{1}{U^2(x)}-\frac{1}{S^2(x)}+\sqrt{\epsilon} (p^u)^\frac{\beta}{2}}[\frac{\iota}{\frac{1}{2}(q^u\wedge q^s)^{\beta/2}}+\epsilon\eta^\frac{\beta}{2}].
\end{align*}
The first line implies that for all $\epsilon$ small enough $\|d_0H\|<e^{-\frac{1}{U^2(x)}-\frac{1}{S^2(x)}+\sqrt{\epsilon}(p^u)^\frac{\beta}{2}}[\iota+\epsilon^\frac{3}{2}\eta^{\beta}]$, which is stronger than the estimate in the claim.

Since $\iota\leq\frac{1}{2}(p^u\wedge p^s)^{\beta/2}$ and $(p^u\wedge p^s)\leq I(q^u\wedge q^s)$  (by Lemma \ref{lemma131}), we also get that for all $\epsilon$ small enough: $\|d_0H\|\leq e^{-\frac{1}{U^2(x)}-\frac{1}{S^2(x)}+\sqrt{\epsilon}(p^u)^\frac{\beta}{2}}e^{\beta\Gamma\eta^\frac{1}{\gamma}}\cdot \frac{1}{2}(q^s\wedge q^u)^\frac{\beta}{2}\leq\frac{1}{2}(q^u\wedge q^s)^{\beta/2}$, as required. 

\medskip
\textit{Claim 4}: For all $\epsilon$ small enough $\|d_\cdot H\|_{\beta/2}\leq e^{-\frac{1}{S^2(x)}-\frac{1}{U^2(x)}+8\sqrt{\epsilon}(p^u)^\frac{\beta}{2}}\cdot[\sigma+2\kappa\sqrt{\epsilon}(p^u)^\frac{\beta}{2}] $, and $\|d_\cdot H\|_{\beta/2}\leq\frac{1}{2}$.

\textit{Proof}: By claim 1 and its proof: \begin{itemize}
\item $\|d_\cdot t\|_{\beta/2}\leq \|D_u^{-1}\|e^{8\sqrt{\epsilon}(p^u)^\frac{\beta}{2}} $.
\item $t$ is a contraction and $\|d_{t(\cdot)} F\|_{\beta/2}\leq \sigma$.
\item $\|d_{(F(t),t)} h_s\|_{\beta/2}<\frac{3}{2}\sqrt{\epsilon}(p^u)^\frac{\beta}{2}$.
\end{itemize}
We need the following fact: Suppose
$\psi,\varphi:\mathrm{Dom}\rightarrow M_{r_1}(\mathbb{R})$, where Dom is an open and bounded subset of $\mathbb{R}^{r_2}$ ($r_1,r_2\in\mathbb{N}$)- then  $\|\varphi\cdot\psi\|_\alpha\leq\|\varphi\|_\alpha\cdot\|\psi\|_\alpha$ (see item (3) in the remark after Definition \ref{def137}).

Recall that $H(\vartheta)=D_sF(t(\vartheta))+h_s(F(t(\vartheta)),t(\vartheta))$, hence $$d_\vartheta H=D_sd_{(t(\vartheta))}Fd_\vartheta t+ d_{G(\vartheta)}h_s \cdot A'(\vartheta)\cdot d_{\vartheta}t.$$
 $\|A'(\cdot)\|_\frac{\beta}{2}\leq1+ \|d_{t(\cdot)}F\|_\frac{\beta}{2} $ (recall $A'(\vartheta)$:=\begin{tabular}{| l |}
\hline $d_{t(\vartheta)}F $  \\ \hline
$I_{u(x)\times u(x)}$   \\ \hline
\end{tabular}
), we get $\|d_\cdot H\|_{\beta/2}\leq \|D_u^{-1}\|e^{8\sqrt{\epsilon}(p^u)^\frac{\beta}{2}}[\|D_s\|\cdot\sigma+\frac{3}{2}\sqrt{\epsilon}(p^u)^\frac{\beta}{2}\cdot(1+\epsilon)]$. This and $\|D_s\|\leq e^{-\frac{1}{S^2(x)}}$ gives: \begin{align*}\|d_\cdot H\|_{\beta/2}\leq &e^{-\frac{1}{S^2(x)}-\frac{1}{U^2(x)}+8\sqrt{\epsilon}(p^u)^\frac{\beta}{2}}\cdot[\sigma+2\kappa\sqrt{\epsilon}(p^u)^\frac{\beta}{2}]\\
=& \sigma+ (e^{-\frac{1}{S^2(x)}-\frac{1}{U^2(x)}+8\sqrt{\epsilon}(p^u)^\frac{\beta}{2}}-1)\sigma+2\kappa\sqrt{\epsilon}(p^u)^\frac{\beta}{2} e^{-\frac{1}{S^2(x)}-\frac{1}{U^2(x)}+8\sqrt{\epsilon}(p^u)^\frac{\beta}{2}}\\
\leq&\sigma -\frac{1}{4}(\frac{1}{S^2(x)}+\frac{1}{U^2(x)}-8\sqrt{\epsilon}(p^u)^\frac{\beta}{2})+\frac{1}{2}\epsilon^\frac{1}{3}(p^u)^\frac{\beta}{2}\leq \sigma -\frac{1}{4}(\frac{1}{S^2(x)}+\frac{1}{U^2(x)})+\epsilon^\frac{1}{3}(p^u)^\frac{\beta}{2}.\end{align*}

If $\epsilon$ is small enough, then since $\sigma\leq\frac{1}{2}$, $\|d_\cdot H\|_{\beta/2}\leq 
\frac{1}{2}$.

\medskip
\textit{Claim 5}: For all $\epsilon$ small enough $\widehat{V}^u:=\{\psi_y((H(\vartheta),\vartheta)):|\vartheta|_\infty\leq\min\{e^{\frac{1}{U^2(x)}-\frac{1}{2}(p^u)^\frac{\beta}{2}}q,Q_\epsilon(y)\}\}$ is a $u$-manifold in $\psi_y$. The parameters of $\widehat{V}^u$ are as in part 1 of the statement of the theorem, and $\widehat{V}^u$ contains a $u$-admissible manifold in $\psi_y^{q^u,q^s}$.

\textit{Proof}: To see that $\widehat{V}^u$ is a $u$-manifold in $\psi_y$ we have to check that $H$ is $C^{1+\beta/2}$ and $\|H\|_\infty\leq Q_\epsilon(y)
$. Claim 1 shows that $H$ is $C^{1+\beta/2}$
. To see $\|H\|_\infty\leq Q_\epsilon(y)$, we first observe that for all $\epsilon$ small enough $\mathrm{Lip}(H)<\epsilon$ because $\|d_\vartheta H\|\leq\|d_0 H\|+\text{H\"ol}_{\beta/2}(d_\cdot H)Q_\epsilon(y)^{\beta/2}\leq\frac{1}{2}\epsilon+\frac{1}{2}\epsilon=\epsilon$
.

It follows that $\|H\|_\infty\leq|H(0)|_\infty+\epsilon Q_\epsilon(y)<(10^{-3}+\epsilon)Q_\epsilon(y)<Q_\epsilon(y)$.

Next we claim that $\widehat{V}^u$ contains a $u$-admissible manifold in $\psi_y^{q^s,q^u}$. Since $\psi_x^{p^s,p^u}\rightarrow\psi_y^{q^s,q^u},q^u=\min\{I(p^u),Q_\epsilon(y)\}$. Consequently, for every $\epsilon$ small enough $$e^{\frac{1}{U^2(x)}-\frac{1}{2}(p^u)^\frac{\beta}{2}}q\equiv e^{\frac{1}{U^2(x)}-\frac{1}{2} (p^u)^\frac{\beta}{2}}p^u>e^{\Gamma (p^u)^\frac{1}{\gamma}} p^u =I(p^u)\geq q^u.$$ So $\widehat{V}^u$ restricts to a $u$-manifold with a $q$-parameter equal to $q^u$. Claims 2-4 guarantee that the parameter bounds for $\widehat{V}^u$ satisfy $u$-admissibility in $\psi_y^{q^s,q^u}$ (see Definition \ref{admissible}), and that part 1 of Theorem \ref{graphtransform} holds.

\medskip
\textit{Claim 6}: $f[V^u]$ contains exactly one $u$-admissible manifold in $\psi_y^{q^s,q^u}$. This manifold contains $f(p)$ where $p=(F(0),0)$.

\textit{Proof}: The previous claim shows existence. We prove uniqueness. Using the identity $\Gamma_y^u=\{(D_sF(t)+h_s(F(t),t),D_ut+h_u(F(t),t)):|t|_\infty\leq q\}$, any $u$-admissible manifold in $\psi_y^{q^s,q^u}$ which is contained in $f[V^u]$ is a subset of $$\psi_y[\{(D_sF(t)+h_s(F(t),t),D_ut+h_u(F(t),t)):|t|_\infty\leq q, |D_ut+h_u(F(t),t)|_\infty\leq q^u\}]$$ We just showed at the end of claim 5 that for all $\epsilon$ small enough, $q^u<e^{\frac{1}{U^2(x)}-\frac{1}{2}(p^u)^\frac{\beta}{2}}q$. By claim 1 the equation $\tau=D_ut+h_u(F(t),t)$ has a unique solution $t=t(\tau)\in R_q(0)$ for all $\tau\in R_{q^u}(0)$. Our manifold must therefore be equal to $$\psi_y[\{(D_sF(t(\tau))+h_s(F(t(\tau)),t(\tau)),\tau):|\tau|_\infty\leq q^u\}]$$ This is exactly the $u$-admissible manifold that we have constructed above.

Let $\mathcal{F}[V^u]$ denote the unique $u$-admissible manifold in $\psi_y^{q^s,q^u}$ contained in $f[V^u]$. We claim that $f(p)\in\mathcal{F}[V^u]$. By the previous paragraph it is enough to show that the second set  of coordinates (the $u$-part) of $\psi_y^{-1}[\{f(p)\}]$ has infinity norm of less than $q^u$. Call the $u$-part $\tau$, then:\begin{align*}
    |\tau|_\infty=&\text{second set of coordinates of }f_{xy}(F(0),0)=|h_u(F(0),0)|_\infty\\
\leq &|h_u(0,0)|_\infty+\max\|d_\cdot h_u\|\cdot|F(0)|_\infty<\epsilon\eta^3+\sqrt{\epsilon}\eta^\frac{\beta}{2}\cdot 10^{-3}\eta<\sqrt{\epsilon}\eta^{1+\frac{\beta}{2}}\ll(q^s\wedge q^u)\leq q^u.
\end{align*}
This concludes the proof of the claim.

\medskip
\textit{Claim 7}: $f[V^u]$ intersects every $s$-admissible manifold in $\psi_y^{q^s,q^u}$ at a unique point.

\textit{Proof}: 
 In this part we use the fact that $\psi_x^{p^s,p^u}\rightarrow\psi_y^{q^s,q^u}$ implies $s(x)=s(y)$. Let $W^s$ be an $s$-admissible manifold in $\psi_y^{q^s,q^u}$. We saw in the previous claim that $f[V^u]$ and $W^s$ intersect at least at one point. Now we wish to show that the intersection point is unique: Recall that we can put $f[V^u]$ in the form $$f[V^u]=\psi_y[\{(D_sF(t)+h_s(F(t),t),D_ut+h_u(F(t),t)):|t|_\infty\leq q\}].$$ We saw in the proof of claim 1 that the second coordinate $\tau(t):=D_ut+h_u(F(t),t)$ is a one-to-one continuous map whose image contains $R_{q^u}(0)$. We also saw that $\|d_\cdot t^{-1}\|^{-1}>e^{-(p^u)^\frac{\beta}{2}}\|D_u^{-1}\|^{-1}\geq e^{\frac{1}{U^2(x)}-(p^u)^\frac{\beta}{2}}>1$. Consequently the inverse function $t:Im(\tau)\rightarrow \overline{R_q(0)}$ satisfies $\|d_\cdot t\|<1$, and so $$f[V^u]=\psi_y[\{(H(\vartheta),\vartheta):\vartheta\in Im(\tau)\}],\text{ where }\mathrm{Lip}(H)\leq\epsilon.$$ Let $I:\overline{R_{q^u}(0)}\rightarrow\mathbb{R}^{u(x)}$ denote the function which represents $W^s$ in $\psi_y$, then $\mathrm{Lip}(I)\leq\epsilon$. Extend it to an $\epsilon$-Lipschitz function on $Im(\tau)$ (by McShane's  extension formula \cite{McShane}, similarly to footnote \footref{lipextensions}). The extension represents a Lipschitz manifold $\widetilde{W}^s\supset W^s$. We wish to use the same arguments we used in Proposition \ref{firstbefore} to show the uniqueness of the intersection point (this time of $f[V^u]$ and $\widetilde{W}^s$). We need the following observations:
\begin{enumerate}
 \item $\overline{R_{q^u}(0)}\subset \overline{R_{e^{\frac{1}{U^2(x)}-\frac{1}{2}(p^u)^\frac{\beta}{2}}q}(0)}\subset Im(\tau)$, as seen in the estimation of $Im(\tau)$ in the beginning of claim 1.
 \item $Im(\tau)$ is compact, as a continuous image of a compact set.
 \item $\tau(t)=D_ut+h_u(F(t),t)$, hence $\|d_\cdot \tau\|\leq \kappa+\epsilon$. Hence $\mathrm{diam}(Im(\tau))\leq q(1+\kappa)$. So for $\epsilon<\frac{1}{\kappa+1}\wedge\frac{1}{3}$, we get that for $I\circ H$ (an $\epsilon^2$-Lipschitz function from $Im(\tau)$ to itself), its image is contained in $\overline{R_{\epsilon^2\|d_\cdot \tau\|q+|I\circ H(\tau(0))|}(0)}\subset \overline{R_q(0)}\subset Im(\tau)$ (the last inclusion is due to the bound we have seen $|\tau(0)|\leq2\sqrt{\epsilon}\eta^{2+\frac{\beta}{2}}$)
\end{enumerate}
As in the proof of Proposition \ref{firstbefore}, we can use the Banach fixed point theorem
. 
This gives us that $f[V^u]$ and $W^s$ intersect in at most one point.

This concludes the proof of claim 7, and thus the proof of the theorem for $u$-manifolds. 

The case of $s$-manifolds follows from the symmetry between $s$ and $u$-manifolds:

\begin{enumerate}
\item $V$ is a $u$-admissible manifold w.r.t $f$ iff $V$ is an $s$-admissible manifold w.r.t to $f^{-1}$, and the parameters are the same.
\item $\psi_x^{p^s,p^u}\rightarrow\psi_y^{q^s,q^u}$ w.r.t $f$ iff $\psi_y^{q^s,q^u}\rightarrow\psi_x^{p^s,p^u}$ w.r.t $f^{-1}$. 
\end{enumerate}
\end{proof}

\begin{definition}
Suppose $\psi_x^{p^s,p^u}\rightarrow\psi_y^{q^s,q^u}$.
\begin{itemize}
    \item  The {\em Graph Transform} (of an $s$-admissible manifold) $\mathcal{F}_{s}$ maps an $s$-admissible manifold $V^{s}$ in $\psi_y^{q^s,q^u}$ to the unique $s$-admissible manifold in $\psi_x^{p^s,p^u}$ contained in $f^{-1}[V^{s}]$.
    \item The {\em Graph Transform} (of a $u$-admissible manifold) $\mathcal{F}_{u}$ maps a $u$-admissible manifold $V^{u}$ in $\psi_x^{p^s,p^u}$ to the unique $u$-admissible manifold in $\psi_y^{q^s,q^u}$ contained in $f[V^{u}]$.
\end{itemize}
\end{definition}

The operators $\mathcal{F}_{s},\mathcal{F}_{u}$ depend on the edge $\psi_x^{p^s,p^u}\rightarrow\psi_y^{q^s,q^u}$, even though the notation does not specify it.

\begin{prop}\label{prop133}
If $\epsilon$ is small enough then the following holds: For any $s/u$-admissible manifold $V_1^{s/u},V_2^{s/u}$ in $\psi_x^{p^s,p^u}$
$$\mathrm{dist}(\mathcal{F}_{s/u}(V_1^{s/u}),\mathcal{F}_{s/u}(V_2^{s/u}))\leq e^{-\frac{1-\epsilon}{U^2(x)/S^2(x)}}\mathrm{dist}(V_1^{s/u},V_2^{s/u}),$$
$$\mathrm{dist}_{C^1}(\mathcal{F}_{s/u}(V_1^{s/u}),\mathcal{F}_{s/u}(V_2^{s/u}))\leq e^{-(\frac{1-\epsilon}{U^2(x)}+\frac{1-\epsilon}{S^2(x)})}(\|d_\cdot F_1-d_\cdot F_2\|_\infty+ (\sqrt{\epsilon}+\epsilon^\beta)\|F_1-F_2\|_\infty^{\beta/2}).$$
\end{prop}
\begin{proof} (For the $u$ case) Suppose $\psi_x^{p^s,p^u}\rightarrow\psi_y^{q^s,q^u}$, set $\eta:=p^s\wedge p^u$, and let $V_i^u$ be two $u$-admissible manifolds in $\psi_x^{p^s,p^u}$. We take $\epsilon$ to be small enough for the arguments in the proof of Theorem \ref{graphtransform} to work. These arguments give us  that if $V_i=\psi_x[\{(F_i(t),t):|t|_\infty\leq p^u\}]$, then $\mathcal{F}_u[V_i]=\psi_y[\{(H_i(\tau),\tau):|\tau|_\infty\leq q^u\}]$, where $t_i$ and $\tau$ are of length $u(x)$ (as a right set of components) and
\begin{itemize}
\item $H_i(\tau)=D_sF_i(t_i(\tau))+h_s(F_i(t_i(\tau)),t_i(\tau))$,
\item $t_i(\tau)$ is defined implicitly by $D_ut_i(\tau)+h_u(F_i(t_i(\tau)),t_i(\tau))=\tau$ and $|d_\cdot t_i|_\infty<1$,
\item $\|D_s^{-1}\|\leq\kappa,\|D_s\|\leq e^{-\frac{1}{S^2(x)}},\|D_u^{-1}\|\leq e^{-\frac{1}{U^2(x)}},\|D_u\|\leq \kappa$,
\item $|h_{s/u}(0)|_\infty<\epsilon\eta^3, \text{H\"ol}_{\beta/2}(d_\cdot h_{s/u})\leq\sqrt{\epsilon},\max\|d_\cdot h_{s/u}\|<\frac{3}{2}\sqrt{\epsilon} (p^u)^\frac{\beta}{2}$.
\end{itemize}
In order to prove the proposition we need to estimate $\|H_1-H_2\|_\infty,\|d_\cdot H_1-d_\cdot H_2\|_\infty$ in terms of $\|F_1-F_2\|_\infty,\|d_\cdot F_1-d_\cdot F_2\|_\infty$.

\medskip
\textit{Part 1:} For all $\epsilon>0$ small enough $\|t_1-t_2\|_\infty\leq\epsilon^2\|F_1-F_2\|_\infty$

\textit{Proof}: By definition, $D_ut_i(\tau)+h_u(F_i(t_i(\tau)),t_i(\tau))=\tau$. By substracting both equations we get
\begin{align*}
|D_u(t_1-t_2)|_\infty\leq&|h_u(F_1(t_1),t_1)-h_u(F_2(t_2),t_2)|_\infty\leq\max\|d_\cdot h_u\|\cdot\max\{|F_1(t_1)-F_2(t_2)|_\infty,|t_1-t_2|_\infty\}\\
\leq&\max\|d_\cdot h_u\|\cdot\max\{|F_1(t_1)-F_1(t_2)|_\infty+|F_1(t_2)-F_2(t_2)|_\infty,|t_1-t_2|_\infty\}\\
\leq&\frac{3}{2}\sqrt{\epsilon}(p^u)^\frac{\beta}{2}\max\{\|F_1-F_2\|_\infty+\mathrm{Lip}(F_1)|t_1-t_2|_\infty,|t_1-t_2|_\infty\}\\
\leq &\frac{3}{2}\sqrt{\epsilon}(p^u) ^\frac{\beta}{2}\|F_1-F_2\|_\infty+ \frac{3}{2}\sqrt{\epsilon}(p^u)^\frac{\beta}{2}(\epsilon+1)|t_1-t_2|_\infty.
\end{align*}

Rearranging terms and recalling $\|D_u^{-1}\|\leq e^{-\frac{1}{U^2(x)}}$, for all $\epsilon>0$ small enough: $$\|t_1-t_2\|_\infty\leq\frac{\frac{3}{2}\sqrt{\epsilon}(p^u)^\frac{\beta}{2}\|F_1-F_2\|_\infty}{e^{\frac{1}{U^2(x)}}-2\sqrt{\epsilon}(p^u) ^\frac{\beta}{2}}\leq\frac{\frac{3}{2}\sqrt{\epsilon}(p^u)^\frac{\beta}{2}\|F_1-F_2\|_\infty}{\frac{1}{U^2(x)}-2\sqrt{\epsilon}(p^u)^\frac{\beta}{2}}\leq\frac{\frac{3}{2}\sqrt{\epsilon}(p^u) ^\frac{\beta}{2}\|F_1-F_2\|_\infty}{(1-\epsilon)\frac{1}{U^2(x)}}\leq \epsilon ^2\|F_1-F_2\|_\infty.$$ The claim follows.

\medskip
\textit{Part 2}: For all $\epsilon$ small enough $\|H_1-H_2\|_\infty<e^{-\frac{1-\epsilon}{S^2(x)}}\|F_1-F_2\|_\infty$, whence proving the first claim of the proposition.

\textit{Proof}: 
\begin{align*}
|H_1-H_2|_\infty\leq&\|D_s\|\cdot|F_1(t_1)-F_2(t_2)|_\infty+|h_s(F_1(t_1),t_1)-h_s(F_2(t_2),t_2)|_\infty\\
\leq&\|D_s\|\cdot|F_1(t_1)-F_2(t_2)|_\infty+\|d_\cdot h_s\|_\infty\max\{|F_1(t_1)-F_2(t_2)|_\infty,|t_1-t_2|_\infty\}\\
\leq&\|D_s\|\cdot(|F_1(t_1)-F_2(t_1)|_\infty+|F_2(t_1)-F_2(t_2)|_\infty)+\\
+&\|d_\cdot h_s\|_\infty\cdot\max\{|F_1(t_1)-F_2(t_1)|_\infty+|F_2(t_1)-F_2(t_2)|_\infty,|t_1-t_2|_\infty\}\\
\leq&\|D_s\|(\|F_1-F_2\|_\infty+\mathrm{Lip}(F_2)|t_1-t_2|)_\infty+\|d_\cdot h_s\|_\infty|t_1-t_2|_\infty\\
\leq&\Big(\|D_s\|(1+\mathrm{Lip}(F_2)\frac{|t_1-t_2|}{\|F_1-F_2\|})+\|d_\cdot h_s\|_\infty\frac{|t_1-t_2|_\infty}{\|F_1-F_2\|}\Big)\|F_1-F_2\|\\
\leq& \Big(e^{-\frac{1}{S^2(x)}}(1+\epsilon^2\cdot(p^u)^\frac{\beta}{2})+\epsilon^2\cdot\frac{3}{2}\sqrt{\epsilon}(p^u)^\frac{\beta}{2}\Big)\|F_1-F_2\|(\because\text{ part 1})\\
\leq& e^{-\frac{1}{S^2(x)}+2\epsilon^2\cdot(p^u)^\frac{\beta}{2}}\|F_1-F_2\|.
\end{align*}

Whenever $\epsilon
>0$ is sufficiently small the factor on the RHS is smaller than $e^{-\frac{1-\epsilon}{S^2(x)}}$.

\medskip
\textit{Part 3}: For all $\epsilon$ small enough: $$\|d_\cdot t_1-d_\cdot t_2\|_\infty<6\sqrt{\epsilon}(p^u)^\frac{\beta}{2}(\|d_\cdot F_1-d_\cdot F_2\|_\infty+\|F_1-F_2\|_\infty^{\beta/2}).$$

\textit{Proof}: Define as in claim 1 of Theorem \ref{graphtransform}, $A'_i(\vartheta)$:=\begin{tabular}{| l |}
\hline $d_{t_i(\vartheta)}F_i $  \\ \hline
$I_{u(x)\times u(x)}$   \\ \hline
\end{tabular} and $A_i(\vartheta):=d_{G_i(\vartheta)}h_u \cdot A_i'(\vartheta)$, where  $G_i(\vartheta)=(F_i(t_i(\vartheta)),t_i(\vartheta))$.

We have seen at the beginning of the Graph Transform that

$Id=(D_u+A_i(\vartheta))d_\vartheta t_i$. By taking differences we obtain:

$(D_u+A_1(\vartheta))(d_\vartheta t_1-d_\vartheta t_2)=(D_u+A_1(\vartheta))d_\vartheta t_1-(D_u+A_2)d_\vartheta t_2+(A_2-A_1)d_\vartheta t_2=Id-Id+(A_2-A_1)d_\vartheta t_2=$ 

$=(A_2-A_1)d_\vartheta t_2=(d_{G_1}h_u\cdot A'_1-d_{G_2}h_uA'_2)\cdot d_\vartheta t_2$

Define $J_i(t):=$\begin{tabular}{| r |}
\hline $d_{t}F_i $  \\ \hline
$I_{u(x)\times u(x)}$   \\ \hline
\end{tabular}, which is the same as $A'_i(\vartheta)$, with a parameter change ($t=t_i(\vartheta)$). So now we see that:
\begin{align*}
(D_u+A_1)(d_\vartheta t_1-d_\vartheta t_2)=&(d_{G_1}h_u\cdot J_1(t_1)-d_{G_2}h_u\cdot J_2(t_2))d_\vartheta t_2\\
=&(d_{G_1}h_u\cdot J_1(t_1)-d_{G_2}h_u\cdot J_1(t_1)+d_{G_2}h_u\cdot J_1(t_1)-d_{G_2}h_u\cdot J_2(t_2))d_\vartheta t_2
\\=&\Big([d_{G_1}h_u\cdot J_1(t_1)-d_{G_2}h_u\cdot J_1(t_1)]+[d_{G_2}h_u\cdot J_1(t_1)-d_{G_2}h_u\cdot J_2(t_1)]\\
+&[d_{G_2}h_u\cdot J_2(t_1)-d_{G_2}h_u\cdot J_2(t_2)]\Big)d_\vartheta t_2:=I+II+III.
\end{align*}
Since $\|D_u^{-1}\|^{-1}\geq e^\frac{1}{U^2(x)},\|d_\cdot F_1\|_\infty<1$ and $\|d_\cdot h_u\|_\infty<\frac{3}{2}\sqrt{\epsilon}(p^u)^\frac{\beta}{2}$, we get $\|(D_u+A_1)^{-1}\|^{-1}\geq e^\frac{1}{U^2(x)}-\frac{3}{2}\sqrt{\epsilon}(p^u)^\frac{\beta}{2} $, and hence (for $\epsilon>0$ small enough),
$$\|d_\cdot t_1-d_\cdot t_2\|_\infty\leq\frac{1}{e^\frac{1}{U^2(x)}-\frac{3}{2}\sqrt{\epsilon}(p^u)^\frac{\beta}{2}}\|I+II+III\|_\infty\leq e^{-\frac{1-\epsilon}{U^2(x)}}\left(\|I\|_\infty+\|II\|_\infty+\|III\|_\infty\right).$$

We begin analyzing: recall $\text{H\"ol}_{\beta/2}(d_\cdot h_{s/u})\leq\sqrt{\epsilon}$:
\begin{itemize}
\item $\|d_{(F_1(t_1),t_1)}h_u-d_{(F_2(t_1),t_1)}h_u\|\leq\sqrt{\epsilon}\|F_1-F_2\|_\infty^{\beta/2}$
\item $\|d_{(F_2(t_1),t_1)}h_u-d_{(F_2(t_2),t_1)}h_u\|\leq\sqrt{\epsilon}\|t_1-t_2\|_\infty^{\beta/2}(\because \mathrm{Lip}(F_2)<1)$ 
\item $\|d_{(F_2(t_2),t_1)} h_u-d_{(F_2(t_2),t_2)} h_u\|_\infty<\sqrt{\epsilon}|t_1-t_2|_\infty^{\beta/2}$
\end{itemize}
By part 1 $\|t_1-t_2\|_\infty\leq\epsilon^2\|F_1-F_2\|_\infty$, it follows that:
$$\|d_{G_1}h_u-d_{G_2}h_u\|<3\sqrt{\epsilon}\|F_1-F_2\|_\infty^{\beta/2},$$ and hence $$\|I\|_\infty\leq3\sqrt{\epsilon}\|F_1-F_2\|_\infty^{\beta/2}.$$
Using the facts that $\|d_\cdot t_{1/2}\|_\infty<1,\|d_\cdot h_u\|_\infty<\frac{3}{2}\sqrt{\epsilon}(p^u)^\frac{\beta}{2}
$ and $\text{H\"ol}_{\beta/2}(d_\cdot F_2)<1$ (from the definition of admissible manifolds and the proof of the Graph Transform) we get that:

\hspace{16pt}$\|II\|_\infty\leq\frac{3}{2}\sqrt{\epsilon}(p^u)^\frac{\beta}{2}\|d_\cdot F_1-d_\cdot F_2\|_\infty$,

\hspace{16pt}$\|III\|_\infty\leq\frac{3}{2}\sqrt{\epsilon}(p^u)^\frac{\beta}{2}\|t_1-t_2\|_\infty^\frac{\beta}{2}\leq \sqrt{\epsilon}(p^u)^\frac{\beta}{2}\|F_1-F_2\|_\infty^\frac{\beta}{2}
$.

So for all $\epsilon$ sufficiently small $\|d_\cdot t_1-d_\cdot t_2\|_\infty<6\sqrt{\epsilon}(p^u)^\frac{\beta}{2}(\|d_\cdot F_1-d_\cdot F_2\|_\infty+\|F_1-F_2\|_\infty^{\beta/2})$.

\medskip
\textit{Part 4}: $\|d_\cdot H_1-d_\cdot H_2\|_\infty<e^{-(1-\epsilon)(\frac{1}{U^2(x)}+\frac{1}{S^2(x)})}(\|d_\cdot F_1-d_\cdot F_2\|_\infty+ (\sqrt{\epsilon}+\epsilon^\beta)\|F_1-F_2\|_\infty^{\beta/2})$.

\textit{Proof}: By the definition of $H_i$: $d_\vartheta H_i=D_s\cdot d_{t_i} F_i\cdot d_\vartheta t_i+d_{G_i}h_u\cdot J_i(t_i)\cdot d_\vartheta t_i$

Taking differences we see that: $$\|d_\vartheta H_1-d_\vartheta H_2\|\leq\|d_\vartheta t_1-d_\vartheta t_2\|\cdot\|D_s d_{t_1}F_1+d_{G_1}h_u\|+\|d_\vartheta t_2\|\cdot\|D_s\|\cdot(\|d_{t_1}F_1-d_{t_1}F_2\|+\|d_{t_1}F_2-d_{t_2}F_2\|)$$
$$+\|d_\vartheta t_2\|\cdot\|d_{G_1}h_u-d_{G_2}h_u\|\cdot\|d_{t_1}F_1\|+\|d_\vartheta t_2\|\cdot\|d_{G_2}h_u\|\cdot\|d_{t_1}F_1-d_{t_2}F_2\|:=I'+II'+III'+IV'$$
Using the same arguments as in part 3 we can show that:
\begin{align*}
I'\leq&\|d_\cdot t_1-d_\cdot t_2\|(e^{-\frac{1}{S^2(x)}}(p^u)^\frac{\beta}{2}+\frac{3}{2}\sqrt{\epsilon}(p^u)^\frac{\beta}{2})\leq 12\epsilon(p^u)^{\beta}(\|d_\cdot F_1-d_\cdot F_2\|_\infty+\|F_1-F_2\|_\infty^{\beta/2}),\\
II'\leq& e^{-\frac{1}{S^2(x)}-\frac{1}{U^2(x)}+(p^u)^\frac{\beta}{2}}(\|d_\cdot F_1-d_\cdot F_2\|_\infty+\|t_1-t_2\|_\infty^{\beta/2})\\
\leq& e^{-\frac{1}{S^2(x)}-\frac{1}{U^2(x)}+(p^u)^\frac{\beta}{2}}(\|d_\cdot F_1-d_\cdot F_2\|_\infty+\epsilon^\beta\|F_1-F_2\|_\infty^{\beta/2})\text{ by part 1},\\
III'\leq& e^{-\frac{1}{U^2(x)}+(p^u)^\frac{\beta}{2}}(p^u)^\frac{\beta}{2}\cdot 3\sqrt{\epsilon}\|F_1-F_2\|_\infty^{\beta/2}\text{ (by the estimate of I in part 3)},\\
IV'\leq&\frac{3}{2}\sqrt{\epsilon}(p^u)^\frac{\beta}{2}e^{-\frac{1}{U^2(x)}+(p^u)^\frac{\beta}{2}}(\|d_\cdot F_1-d_\cdot F_2\|_\infty+\|F_1-F_2\|_\infty^{\beta/2}).
\end{align*}
It follows that $$\|d_\cdot H_1-d_\cdot H_2\|_\infty<(e^{-\frac{1}{U^2(x)}-\frac{1}{S^2(x)}+(p^u)^\frac{\beta}{2}}+O((p^u)^\frac{\beta}{2}))(\|d_\cdot F_1-d_\cdot F_2\|_\infty+(\sqrt{\epsilon}+\epsilon^\beta)\|F_1-F_2\|_\infty^{\beta/2}).$$
If $\epsilon$ is small enough then $$\|d_\cdot H_1-d_\cdot H_2\|_\infty<e^{-(1-\epsilon)(\frac{1}{U^2(x)}+\frac{1}{S^2(x)})}(\|d_\cdot F_1-d_\cdot F_2\|_\infty+ (\sqrt{\epsilon}+\epsilon^\beta)\|F_1-F_2\|_\infty^{\beta/2}).$$
\end{proof}

\subsubsection{A Markov extension}
Recall the definitions from the begining of \textsection \ref{defepsilonchains}:

$$\mathcal{V}=\{\psi_x^{p^s,p^u}:\psi_x^{p^s\wedge p^u}\in\mathcal{A},p^s,p^u\in \mathcal{I},p^s,p^u\leq Q_\epsilon(x)\}$$
$$\mathcal{E}=\{(\psi_x^{p^s,p^u},\psi_y^{q^s,q^u})\in\mathcal{V}\times\mathcal{V}:\psi_x^{p^s,p^u}\rightarrow\psi_y^{q^s,q^u}\}$$
$$\mathcal{G}\text{ is the directed graph with vertices }\mathcal{V}\text{ and edges }\mathcal{E}$$
\begin{definition}
    $\Sigma(\mathcal{G}):=\{v=(v_i)_{i\in\mathbb{Z}}:v_i\in\mathcal{V},v_i\rightarrow v_{i+1}\forall i\}$, equipped with the left-shift $\sigma$, and the metric $d(v,w)=e^{-\min\{|i|:v_i\neq w_i\}}$.
\end{definition}
\begin{definition} Suppose $(v_i)_{i\in\mathbb{Z}},v_i=\psi_{x_i}^{p_i^s,p_i^u}$ is a chain. Let $V_{-n}^u$ be a $u$-admissible manifold in $v_{-n}$. 
$\mathcal{F}_u^n(v_{-n})$ is the $u$-admissible manifold in $v_0$ which is a result of the application of the graph transform $n$ times along $v_{-n},...,v_{-1}$ (each application is the transform described in section 3 of Theorem \ref{graphtransform}). Similarly any $s$-admissible manifold in $v_n$ is mapped by $n$ applications of the graph transform to an $s$-admissible manifold in $v_0$: $\mathcal{F}_s^n(v_n)$. These two manifolds depend on $v_{-n},...,v_n$.
\end{definition}

\begin{definition}[Leaf cores]\label{tildeLeaves}
	Let $\psi_{x}^{p^s,p^u}$ be a double chart, and let $V^u$ and $V^s$ be a $u$-admissible manifold and an $s$-admissible manifold in $\psi_x^{p^s,p^u}$, respectively. The associated {\em leaf cores} are $\widecheck{V}(\underline{u}):=\psi_{x}\circ F^u [R_{\eta^2}(0)], \widecheck{V}^s(\underline{u}):= \psi_{x}\circ F^s [R_{\eta^2}(0)] $ where $\eta=p^s\wedge p^u$, and $F^u$ and $F^s$ are the representing functions of the $u$-admissible and $s$-admissble manifolds, respectively.
\end{definition}
\noindent\textbf{Remark}: $\widecheck{V}^u(\underline{u})\subseteq V^u(\underline{u}), \widecheck{V}^s(\underline{u})\subseteq V^s(\underline{u})$ and by Proposition \ref{firstbefore} the unique intersection point of $V^u(\underline{u})$ and $V^s(\underline{u})$ is also the unique intersection point of $\widecheck{V}^u(\underline{u})$ and $\widecheck{V}^s(\underline{u})$.

\begin{cor}\label{WideCheckprop133}
If $\epsilon$ is small enough then the following holds: For any $s/u$-admissible manifold $V_1^{s/u},V_2^{s/u}$,
$$\mathrm{dist}(\widecheck{\mathcal{F}}_{s/u}(V_1^{s/u}),\widecheck{\mathcal{F}}_{s/u}(V_2^{s/u}))\leq e^{-\frac{1-\epsilon}{U^2(x)/S^2(x)}}\mathrm{dist}(\widecheck{V}_1^{s/u},\widecheck{V}_2^{s/u}),$$
$$\mathrm{dist}_{C^1}(\widecheck{\mathcal{F}}_{s/u}(V_1^{s/u}),\widecheck{\mathcal{F}}_{s/u}(V_2^{s/u}))\leq e^{-(\frac{1-\epsilon}{U^2(x)}+\frac{1-\epsilon}{S^2(x)})}(\|d_\cdot F_1-d_\cdot F_2\|_{\infty,R_{(p^s\wedge p^u)^2}(0)}+ (\sqrt{\epsilon}+\epsilon^\beta)\|F_1-F_2\|_{\infty, R_{(p^s\wedge p^u)^2}(0) }^{\beta/2}).$$
\end{cor}

The proof is similar to the proof of Proposition \ref{prop133}, since if $\psi_y^{p^s,p^u}\rightarrow \psi_z^{q^s,q^u}$, then $(p^s\wedge p^u)^2=(I^{\pm1}(q^s\wedge q^u))^2$ (and so $f^{1/-1}[\widecheck{V}^{u/s}]\supseteq \widecheck{\mathcal{F}}(V^{u/s})$).

\begin{definition} Let $\psi_x^{p^s,p^u}$ be some double chart. Assume $V_n$ is a sequence of $s/u$-manifolds in $\psi_x$. We say that $V_n$ converges uniformly to $V$, an $s/u$-manifold in $\psi_x$, if the representing functions of $V_n$ converge uniformly to the representing function of $V$. We say that $\widecheck{V}_n$ converges uniformly to $\widecheck{V}$ if the representing functions of $V_n$ converge uniformly to the representing function of $V$ when restricted to $R_{(p^s\wedge p^u)^2}(0)$.
\end{definition}
Notice, $V_n\xrightarrow[uniformly]{} V$ implies $\widecheck{V}_n\xrightarrow [uniformly]{} \widecheck{V}$.

\medskip
The following proposition is a very important step in our proof, especially the fifth item, for the following reason. A symbolic coding is often a tool which is used in order to study measures of interests, and in particular equilibrium states. One of the most sought-after equilibrium states are SRB measures- equilibrium states of the geometric potential. In order to take advantage of the symbolic coding for this purpose, one needs to be able to lift the geometric potential to the symbolic space with the (morally minimal) property of {\em summable variations} (i.e the total variation of chains with the same $n$ first symbols is a summable quantity in $n$). When coding a set of points with a uniform bound on the Lyapunov exponents, one usually gets that the coding map is H\"older continuous (in $C^1$ distance for stable/unstable leaves as well), and so the geometric potential lifts to a H\"older continuous potential as well (some arguments are needed, see \cite{Sarig,SBO}). The H\"older continuity of the geometric potential implies immediately the (exponentially fast) summable variations. In our setup, there is no uniform Lyapunov exponent for the orbits we code, and in fact these orbits may exhibit $0$ Lyapunov exponents. This makes the continuity property of the factor map much weaker, and not H\"older continuous. Thus proving the summable variations property becomes a much more subtle question. This goal motivated the definition of leaf cores, which is used in the proof below, where we are able to show that indeed the factor map has summable variation, w.r.t $C^1$ distance for leaf cores.

\begin{prop}\label{firstofchapter}
Suppose $(v_i=\psi_{x_i}^{p_i^s,p_i^u})_{i\in\mathbb{Z}}$ is a chain of double charts, and choose arbitrary $u$-admissible manifolds $V^u_{-n}$ in $v_{-n}$, and arbitrary $s$-admissible manifolds $V_n^s$ in $v_n$. Then:
\begin{enumerate}
\item The limits $V^u((v_i)_{i\leq0}):=\lim\limits_{n\rightarrow\infty}\mathcal{F}_u^n(V^u_{-n}),V^s((v_i)_{i\geq0}):=\lim\limits_{n\rightarrow\infty}\mathcal{F}_s^n(V^s_n)$ exist, and are independent of the choice of $V_{-n}^u$ and $V_n^s$.
\item $V^{u}((v_i)_{i\leq0})/V^{s}((v_i)_{i\geq0})$ is an $s/u$-admissible manifold in $v_0$.
\item $f[V^s(v_i)_{i\geq0}]\subset V^s(v_{i+1})_{i\geq0}$ and $f^{-1}[V^u(v_i)_{i\leq0}]\subset V^u(v_{i-1})_{i\leq0}$
\item \begin{align*}
V&^s((v_i)_{i\geq0})=\{p\in\psi_{x_0}[R_{p_0^s}(0)]:\forall k\geq0, f^k(p)\in\psi_{x_k}[R_{10Q_\epsilon(x_k)}(0)]\}\\
V&^u((v_i)_{i\leq0})=\{p\in\psi_{x_0}[R_{p_0^u}(0)]:\forall k\geq0, f^{-k}(p)\in\psi_{x_{-k}}[R_{10Q_\epsilon(x_{-k})}(0)]\}
\end{align*}
\item The map $(v_i)_{i\in\mathbb{Z}}\mapsto \widecheck{V} ^u((v_i)_{i\leq0})$ has summable variations: $\exists C_1,\tau>0$ depending on the calibration parameters (but not on $\epsilon$) s.t $$\mathrm{dist}_{C^1}(\widecheck{V}^u((v_i)_{i\leq0}),\widecheck{V}^u((w_i)_{i\leq0}))\leq C_1[I^{-N}(p^u_0\wedge p^s_0)]^{\frac{1}{\gamma}+\tau}.$$

A similar statement holds for $\widecheck{V}^s(\cdot)$.
\end{enumerate}
\end{prop}
\begin{proof} Parts (1)--(4) are inspired by Pesin's Stable Manifold Theorem \cite{Pesin}
. 


We treat the case of $u$ manifolds, the stable case is similar.

\textit{Part 1}: By the Graph Transform theorem $\mathcal{F}_u^n(V_{-n}^u)$ is a $u$-admissible manifold in $v_0$. By 
Proposition \ref{prop133} for any other choice of $u$-admissible manifolds $W_{-n}^u$ in $v_{-n}$:
\begin{align}\label{emoji}\mathrm{dist}(\mathcal{F}_u^n(V_{-n}^u),\mathcal{F}_u^n(W_{-n}^u))\leq& e^{-\sum_{k=0}^{n-1}\frac{1-\epsilon}{S^2(x_{-k})}}\mathrm{dist}(V_{-n}^u,W_{-n}^u)\leq 2e^{-(1-\epsilon)C_{\beta,\epsilon}^{-\frac{1}{\gamma}}\sum_{k=0}^{n-1}(p^u_{-k}
)^\frac{1}{\gamma}}p_{-n}^u\nonumber\\
=&	2e^{-(1-\epsilon)C_ {\beta,\epsilon} ^{-\frac{1}{\gamma}}\sum_{k=0}^{n-1}(p^u_{-k}
)^\frac{1}{\gamma}}(p_{-n}^u\wedge p_{-n}^s)^2\frac{p_{-n}^u}{(p_{-n}^u\wedge p_{-n}^s)^2}\\
\leq& 2e^{-\left((1-\epsilon)C_ {\beta,\epsilon} ^{-\frac{1}{\gamma}}-\Gamma\right)\sum_{k=0}^{n-1}(p^u_{-k}\wedge p^s_{-k})^\frac{1}{\gamma}} (p_{0}^u\wedge p_{0}^s)^2 \frac{p_{-n}^u}{(p_{-n}^u\wedge p_{-n}^s)^2}\nonumber\\
\leq& 2e^{-2\cdot\frac{1}{2\Gamma}\left((1-\epsilon)C_ {\beta,\epsilon} ^{-\frac{1}{\gamma}}-\Gamma\right)\cdot \Gamma\sum_{k=0}^{n-1}(I^{-k}(p_0^u\wedge p_0^s))^\frac{1}{\gamma}}(p_{0}^u\wedge p_0^s)^2 \frac{p_{-n}^u}{(p_{-n}^u\wedge p_{-n}^s)^2}.\nonumber
\end{align}

In the first line we needed the fact that for any admissible manifold with representing function $F$, $\|F\|_{
\infty}\leq 
p^u
$ by definition. Notice that by choosing $\epsilon>0$ arbitrarily small (and thus making $\frac{1}{C_{\beta,\epsilon}}$ arbitrarily large), we can make the exponent on the RHS very large. Let $\epsilon>0$ small enough so ${\frac{1}{2\Gamma}\left((1-\epsilon)C_ {\beta,\epsilon} ^{-\frac{1}{\gamma}}-\Gamma\right) 
} \geq 1$, then by \eqref{emoji}, with the following estimates:

\begin{align}\label{foreq13self}
\mathrm{dist}(\mathcal{F}_u^n(V_{-n}^u),\mathcal{F}_u^n(W_{-n}^u))\leq& 2
e^{-\Gamma\sum_{k=0}^{n-1}(I^{-k}(p_0^u
))^\frac{1}{\gamma}}p_{-n}^u (\because p^u_{-k+1}\leq I(p^u_{-k}))\nonumber\\ 
\leq &2
e^{-\Gamma\sum_{k=0}^{n-1}(I^{-k}(p_0^u
))^\frac{1}{\gamma}}p^u_0\frac{p_{-n}^u}{p^u_0}\leq \frac{1}{p^u_0} I^{-n}(p_0^u).
\end{align}
When restricting to the leaf cores, $\mathrm{dist}(\widecheck{V}^u_{-n},\widecheck{V}^u_{-n})\leq 2\cdot 10^{-2}(p^u_{-n}\wedge p^s_{-n})^2$. In this case, with Corollary \ref{WideCheckprop133} replacing Proposition \ref{prop133}, the estimates which follow \eqref{emoji} give in addition

\begin{align}\label{link}
\mathrm{dist}(\widecheck{\mathcal{F}}_u^n(V_{-n}^u),\widecheck{\mathcal{F}}_u^n(W_{-n}^u))\leq& \left[
e^{-\Gamma\sum_{k=0}^{n-1}(I^{-k}(p_0^u\wedge p_0^s))^\frac{1}{\gamma}}(p_{0}^u\wedge p^s_{0})\right]^2
 \leq \left(
I^{-n}(
p_0^u\wedge p_0^s
)\right)^2
.
\end{align}

Equation \eqref{foreq13self} implies that if the limit exists then it is independent of $V_{-n}^u$. By Theorem \ref{graphtransform}, for every $n\geq m$, $\mathcal{F}_u^{n-m}(V^u_{-n})$ is a $u$-admissible manifold in $v_{-m}$, and so
\begin{align*}
 	\mathrm{dist}(\mathcal{F}_u^n(V^u_{-n}),\mathcal{F}_u^m(V^u_{-m}))= \mathrm{dist}(\mathcal{F}_u^m(\mathcal{F}_u^{n-m}(V^u_{-n})),\mathcal{F}_u^m(V^u_{-m}))\leq  \frac{1}{p^u_0}I^{-m}(p_0^u).
\end{align*}
Hence, $\mathcal{F}_u^n(V_{-n}^u)$ is a Cauchy sequence in a complete space, and therefore converges.

\medskip
\textit{Part 2}: Admissibility of the limit: Write $v_0=\psi_x^{p^s,p^u}$ and let $F_n$ denote the functions which represent $\mathcal{F}_u^n(V^u_{-n})$ in $v_0$. Since $\mathcal{F}_u^n(V^u_{-n})$ are $u$-admissible in $v_0$, for every $n$:

\begin{itemize}
\item $\|d_\cdot F_n\|_{\beta/2}\leq\frac{1}{2}$,
\item $\|d_0F_n\|\leq\frac{1}{2}(p^s\wedge p^u)^{\beta/2}$,
\item $|F_n(0)|_\infty\leq10^{-3}(p^s\wedge p^u)^2$,
\item $\sup|F_n|\leq p^u$.
\end{itemize}
Since $\mathcal{F}_u^n(V_{-n}^u)\xrightarrow[n\rightarrow\infty]{} V^u((v_i)_{i\leq0})$, $F_n\xrightarrow[n\rightarrow\infty]{}F$ uniformly, where $F$ represents $V^u((v_i)_{i\leq0})$. By the Arzela-Ascoli theorem, $\exists n_k\uparrow\infty$ s.t $d_\cdot F_{n_k}\xrightarrow[k\rightarrow\infty]{}L$ uniformly where $\|L\|_{\beta/2}\leq\frac{1}{2}$. For each term in the sequences, the following identity holds:$$\forall t\in R_{p^u}(0):\text{ }\vec{F}_{n_k}(t)=\vec{F}_{n_k}(0)+\int\limits_{0}^{1}d_{\lambda t}\vec{F}_{n_k}\cdot t\text{ }d\lambda.$$

Since $d_\cdot F_{n_k}$ converge uniformly, we get:
$$\forall t\in R_{p^u}(0):\text{ }\vec{F}(t)=\vec{F}(0)+\int\limits_{0}^{1}L(\lambda t)\cdot t\text{ }d\lambda.$$
The same calculations give us that $\forall t_0\in R_{p^u}(0):$
$$\forall t\in R_{p^u}(0):\text{ }\vec{F}(t)=\vec{F}(t_0)+\int\limits_{0}^{1}L(\lambda (t-t_0)+t_0)\cdot (t-t_0)\text{ }d\lambda=$$
$$=\vec{F}(t_0)+\int\limits_{0}^{1}[L(\lambda (t-t_0)+t_0)-L(t_0)]\cdot (t-t_0)\text{ }d\lambda+L(t_0)\cdot(t-t_0).$$
Since $L$ is $\frac{\beta}{2}$-H\"older on a compact set, and in particular uniformly continuous, the second and third summands are $o(|t-t_0|)$, and hence $F$ is differentiable and $ d_tF=L(t)$. We also see that $\{d_\cdot F_n\}$ can only have one limit point. Consequently $d_\cdot F_n\xrightarrow[n\rightarrow\infty]{uniformly}d_\cdot F$ . So $\|d_\cdot F\|_{\beta/2}\leq\frac{1}{2}$, $\|d_0F\|\leq\frac{1}{2}(p^s\wedge p^u)^{\beta/2}$, $|F(0)|_\infty\leq10^{-3}(p^s\wedge p^u)^2$, and $\sup|F|\leq p^u$, whence the $u$-admissibility of $V^u((v_i)_{i\leq0})$.

\medskip
\textit{Part 3:} Let $V^u:=V^u((v_i)_{i\leq0})=\lim\mathcal{F}_u^n(V^u_{-n})$, and $W^u:=V^u((v_{i-1})_{i\leq0})=\lim\mathcal{F}_u^n(V^u_{-n-1})$. $\forall n\geq0$,
\begin{align*}\mathrm{dist}(V^u,\mathcal{F}_u(W^u))\leq& \mathrm{dist}(V^u,\mathcal{F}_u^n(V_{-n}^u))+\mathrm{dist}(\mathcal{F}_u^n(V_{-n}^u),\mathcal{F}_u^{n+1}(V_{-n-1}^u))+\mathrm{dist}(\mathcal{F}_u^{n+1}(V_{-n-1}^u),\mathcal{F}_u(W^u))\\
\leq &\mathrm{dist}(V^u,\mathcal{F}_u^n(V_{-n}^u))+2I^{-n}(1)+e^{-\frac{1-\epsilon}{S^2(x_{-1})}}\mathrm{dist}(\mathcal{F}_u^{n}(V_{-n-1}^u),W^u).\end{align*}
The first and third summands tend to zero by definition, and the second goes to zero since by Lemma \ref{ladderFunc}. So $V^u=\mathcal{F}_u(W^u)\subset f[W^u]$.

\medskip
\textit{Part 4}: The inclusion $\subset$ is simple: Every $u$-admissible manifold $W_i^u$ in $\psi_{x_i}^{p_i^s,p_i^u}$ is contained in $\psi_{x_i}[R_{p_i^u}(0)]$, because if $W_i^u$ is represented by the function $F$ then any $p=\psi_{x_i}(v,w)$ in $W_i^u$ satisfies $|w|_\infty\leq p_i^u$ and $$|v|_\infty=|F(w)|_\infty\leq|F(0)|_\infty+\max\|d_\cdot F\|\cdot|w|_\infty\leq \varphi+\epsilon\cdot|w|_\infty\leq(10^{-3}+\epsilon)p_i^u<p_i^u$$ Applying this to $V^u:=V^u[(v_i)_{i\leq0}]$, we see that for every $p\in V^u$, $p\in \psi_{x_0}[R_{p_0^u}(0)]$. Hence by part 3 for every $k\geq0$:$$f^{-k}(p)\in f^{-k}[V^u]\subset V^u[(v_{i-k})_{i\leq 0}]\subset\psi_{x_{-k}}[R_{p_{-k}^u}(0)]\subset\psi_{x_{-k}}[R_{10Q_\epsilon(x_{-k})}(0)].$$

So now to prove $\supset$: Suppose $z\in\psi_{x_0}[R_{p_0^u}(0)]$ and $f^{-k}(z)\in\psi_{x_{-k}}[R_{10Q_\epsilon(x_{-k})}(0)]$ for all $k\geq0$. Write $z=\psi_{x_0}(v_0,w_0)$. We show that $z\in V^u$ by proving that $v_0=F(w_0)$ where $F$ is the representing function for $V^u$. Set $z':=\psi_{x_0}(F(w_0),w_0)=:\psi_{x_0}(v',w')$. For $k\geq0$, $f^{-k}(z),f^{-k}(z')\in\psi_{x_{-k}}[R_{10Q_\epsilon(x_{-k})}(0)]$ (the first point by assumption, and the second since $f^{-k}(z')\in f^{-k}[V^u]\subset V^u((v_{i-k})_{i\leq0}$). It is therefore possible to write: $$f^{-k}(z)=\psi_{x_{-k}}(v_{-k},w_{-k})\text{ and }f^{-k}(z')=\psi_{x_{-k}}(v'_{-k},w'_{-k})\text{  }(k\geq0)$$ where $|v_{-k}|_\infty,|w_{-k}|_\infty,|v'_{-k}|_\infty,|w'_{-k}|_\infty\leq10Q_\epsilon(x_{-k})$ for all $k\geq0$.

By Proposition \ref{3.4inomris} for $f^{-1}$, $\forall k\geq0$, $f^{-1}_{x_{-k-1}x_{-k}}=\psi_{x_{-k-1}}^{-1}\circ f^{-1}\circ\psi_{x_{-k}}$ can be written as $$f_{x_{-k-1}x_{-k}}^{-1}(v,w)=(D_{s,k}^{-1}v+h_{s,k}(v,w),D_{u,k}^{-1}w+h_{u,k}(v,w)),$$ where $\|D_{s,k}\|^{-1}\geq e^{\frac{1}{S^2(x_{-k})}},\|D_{u,k}^{-1}\|\leq e^{-\frac{1}{U^2(x_{-k})}}$, and $\max\limits_{R_{10Q_\epsilon(x_{-k})}(0)}\|d_\cdot h_{s/u,k}\|\leq \sqrt{\epsilon}(10 Q_\epsilon(x_{-k}))^\frac{\beta}{2}$ (provided $\epsilon$ is small enough). Define $\Delta v_{-k}:=v_{-k}-v'_{-k}$, and $\Delta w_{-k}:=w_{-k}-w'_{-k}$. Since for every $k\leq 0$, $(v_{-k-1},w_{-k-1})=f_{x_{-k-1}x_{-k}}^{-1}(v_{-k},w_{-k})$ and $(v'_{-k-1},w'_{-k-1})=f_{x_{-k-1}x_{-k}}^{-1}(v'_{-k},w'_{-k})$ we get the two following bounds:
$$|\Delta v_{-k-1}|_\infty\geq\|D_{s,k}\|^{-1}\cdot|\Delta v_{-k}|_\infty-\max\|d_\cdot h_{s,k}\|\cdot(|\Delta v_{-k}|_\infty+|\Delta w_{-k}|_\infty)\geq e^{\frac{1-\epsilon}{S^2(x_{-k})}}|v_{-k}|_\infty-\epsilon^\frac{1}{3}Q_\epsilon^\frac{\beta}{2}(x_{-k})|\Delta w_{-k}|_\infty,$$
$$|\Delta w_{-k-1}|_\infty\leq\|D_{u,k}^{-1}\|\cdot|\Delta w_{-k}|_\infty+\max\|d_\cdot h_{u,k}\|\cdot(|\Delta w_{-k}|_\infty+|\Delta v_{-k}|_\infty)\leq e^{-\frac{1-\epsilon}{U^2(x_{-k})}} |w_{-k}|_\infty+ \epsilon^\frac{1}{3}Q_\epsilon^\frac{\beta}{2}(x_{-k}) |\Delta v_{-k}|_\infty.$$
Denote $a_k:=|\Delta v_{-k}|_\infty,b_k:=|\Delta w_{-k}|_\infty$
. Then we get $$a_{k+1}\geq e^{\frac{1-\epsilon}{S^2(x_{-k})}}a_k-\epsilon^\frac{1}{3}Q_\epsilon^\frac{\beta}{2}(x_{-k}) b_k,b_{k+1}\leq e^{-\frac{1-\epsilon}{U^2(x_{-k})}}b_k+ \epsilon^\frac{1}{3}Q_\epsilon^\frac{\beta}{2}(x_{-k}) a_k.$$

We claim that $a_k\leq a_{k+1},b_k\leq a_k \forall k\geq0$: 
Proof by induction: This is true for $k=0$ since $b_0=0$. Assume by induction that $a_k\leq a_{k+1},b_k\leq a_k$
, then:
$$b_{k+1}\leq e^{-\frac{1-\epsilon}{U^2(x_{-k})}} b_k+ \epsilon^\frac{1}{3}Q_\epsilon^\frac{\beta}{2}(x_{-k}) a_k\leq e^{-\frac{1-2\epsilon}{U^2(x_{-k})}}a_k\leq a_k\leq a_{k+1},$$
$$\Rightarrow a_{k+2}\geq e^{\frac{1-\epsilon}{S^2(x_{-k-1})}}a_{k+1}-\epsilon^\frac{1}{3}Q_\epsilon^\frac{\beta}{2}(x_{-k-1}) b_{k+1}\geq e^{\frac{1-2\epsilon}{S^2(x_{-k-1})}}a_{k+1}\geq a_{k+1}.$$

We also see that $a_{k+1}\geq e^{\frac{1-2\epsilon}{S^2(x_{-k})}} a_k$ for all $k\geq0$, hence $a_k\geq e^{\sum_{j=0}^{k-1}\frac{1-2\epsilon}{S^2(x_{-j})}} a_0$. Notice, for all $\epsilon>0$ small enough $$\sum_{j=0}^{k-1}\frac{1-2\epsilon}{S^2(x_{-j})} \gg \sum_{j=0}^{k-1}(p^u_{-j}\wedge p^s_{-j})^\frac{1}{\gamma}\geq \sum_{j=0}^{k-1}(I^{-j}(p^u_{0}\wedge p^s_{0}))^\frac{1}{\gamma}.$$ Thus, by Lemma \ref{forI}, either $a_k\rightarrow\infty$ or $a_0=0$. But $a_k=|v_{-k}-v'_{-k}|_\infty\leq20Q_\epsilon(x_{-k})<20\epsilon$, so $a_0=0$, and therefore $v_0=v'_0$, whence $F(w'_0)=F(w_0)$. So

$$z=\psi_x(F(w_0),w_0)\in V^u,$$
which concludes the proof of part 4.

\medskip
\textit{Part 5}: If two chains $v=(v_i)_{i\in\mathbb{Z}},w=(w_i)_{i\in\mathbb{Z}}$ satisfy $v_i=w_i$ for $|i|\leq N$ then $\exists C_1,\tau>0$ depending on the calibration parameters s.t $\mathrm{dist}_{C^1}(\widecheck{V}^u((v_i)_{i\leq0}),\widecheck{V}^u((w_i)_{i\leq0}))\leq C_1[I^{-N}(p^u_0\wedge p^s_0)]^{\frac{1}{\gamma}+\tau}$: Let $V^u_k=V^u((v_{i})_{i \leq -k})$ be a $u$-admissible manifold in $v_{-N+k}$, and let $W^u_k=V^u((w_{i})_{i \leq -k}) $ be a $u$-admissible manifold in $w_{-k} (=v_{-k})$ for all $0\leq k\leq N$. Then $V^u((v_i)_{i \leq 0})=\mathcal{F}_u^{k}(V^u_k)$ and $V^u((w_i)_{i  \leq 0})=\mathcal{F}_u^{k}(W^u_k)$ are $u$-admissible manifolds in $w_{0}(=v_{0})$. Then by \eqref{link}, $\mathrm{dist}(\widecheck{V}^u_k,\widecheck{W}^u_k)\leq \left(I^{-N}(p^u_0\wedge p^s_0)\right)^2$. 
Represent $V_{-k}^u$ and $W_{-k}^u$ by $F_k$ and $G_k$ respectively. Let $\eta:=p^u_0\wedge p^s_0$. By Corollary \ref{WideCheckprop133},
\begin{align}\label{groudon2}
\|F_{k-1}-G_{k-1}\|_{R_\eta(0),\infty}\leq& e^{-\frac{1-\epsilon}{S^2(x_{-k})}}\|F_k-G_k\|_ {R_\eta(0),\infty},\\
\|d_\cdot F_{k-1}-d_\cdot G_{k-1}\|_ {R_\eta(0),\infty}\leq& e^{-\frac{1-\epsilon}{S^2(x_{-k})} -\frac{1-\epsilon}{U^2(x_{-k})}}(\|d_\cdot F_k-d_\cdot G_k\|_ {R_\eta(0),\infty} +(\sqrt\epsilon+\epsilon^\beta)\|F_k-G_k\|_ {R_\eta(0),\infty} ^{\beta/2}).\nonumber
\end{align}
Set $d_k:= \|d_\cdot F_{k}-d_\cdot G_{k}\|_ {R_\eta(0),\infty} $, $c_k:=\|F_k-G_k\|_ {R_\eta(0),\infty} $, and $r_k:=\frac{1-\epsilon}{S^2(x_{-k})}+ \frac{1-\epsilon}{U^2(x_{-k})} $, then iterating \eqref{groudon2} $N$ times we get
\begin{align*}
d_0\leq& e^{-r_{1}}d_{1}+(\sqrt\epsilon+\epsilon^\beta)e^{-r_{1}}c_{1}^\frac{\beta}{2}\leq \cdots\leq d_{N}e^{-\sum_{k=1}^Nr_{j}}+ (\sqrt\epsilon+\epsilon^\beta)\sum_{k=1}^{N}c_k^\frac{\beta}{2}e^{-\sum_{j=1}^kr_j}\\
\leq& 2(p_0^s\wedge p_0^u) e^{\Gamma\sum_{k=0}^{N-1}(p_{-k}^s\wedge p_{-k}^u)^\frac{1}{\gamma}}\cdot e^{-\sum_{k=0}^{N-1}r_{j}}+ (\sqrt\epsilon+\epsilon^\beta)\sum_{k=1}^{N}c_k^\frac{\beta}{2}e^{-\sum_{j=0}^{k-1}r_j}(\because d_N\leq 2(p_{-N}^s\wedge p_{-N}^u))\\
\leq& 2I^{-N}(p_0^u\wedge p_0^s) +(\sqrt\epsilon+\epsilon^\beta)\sum_{k=1}^N [e^{-\Gamma\sum_{j=0}^{N-k-1}I^{-j}(p_{-k}^u\wedge p^s_{-k})^\frac{1}{\gamma}}\cdot(p_{-k}^u\wedge p_{-k}^s)]^\beta
e^{-\sum_{j=0}^{k-1}r_j} (\because \eqref{link})\\
\leq& 2I^{-N}(p_0^u\wedge p_0^s)+(\sqrt\epsilon+\epsilon^\beta)\sum_{k=1}^N [e^{-\Gamma\sum_{j=0}^{N-k-1}I^{-j}(p_{-k}^u\wedge p^s_{-k})^\frac{1}{\gamma}}\cdot(p_{-k}^u\wedge p_{-k}^s)]^\beta
\cdot e^{-\sum_{j=0}^{k-1} r_j}\\
\leq & 2I^{-N}(p_0^u\wedge p_0^s)+(\sqrt\epsilon+\epsilon^\beta)\sum_{k=1}^N [e^{-\Gamma\sum_{j=0}^{N-k-1}I^{-j}(p_{-k}^u\wedge p^s_{-k})^\frac{1}{\gamma}}\cdot(p_{0}^u\wedge p_{0}^s)e^{\Gamma\sum_{j=0}^{k-1}(p_{-j}^u\wedge p^s_{-j})^\frac{1}{\gamma}}\cdot e^{-\frac{
1}{\beta}\sum_{j=0}^{k-1} r_j}]^\beta
\\
\leq& 2I^{-N}(p_0^u\wedge p_0^s)+(\sqrt\epsilon+\epsilon^\beta)\sum_{k=1}^N [e^{-\Gamma\sum_{j=0}^{N-1}I^{-j}(p_{-k}^u\wedge p^s_{-k})^\frac{1}{\gamma}}\cdot(p_{0}^u\wedge p_{0}^s)]^\beta
\\
=& 2I^{-N}(p_0^u\wedge p_0^s)+(\sqrt\epsilon+\epsilon^\beta)\cdot N\cdot [I^{-N}(p_{0}^u\wedge p_{0}^s)]^\beta
\leq 3\cdot N\cdot [I^{-N}(p_0^u\wedge p^s_0)]^\beta.
\end{align*}

By Definition \ref{calibrationParams} $\gamma>\frac{2}{\beta}$, and so for 
$\tau:=\frac{1}{2}(\beta- \frac{2}{\gamma})>0$, $\beta-2(\frac{1}{\gamma}+\tau) =0$. Then 
\begin{align}\label{formyPesinAC}
d_0\leq 3\left(N\cdot \left(I^{-N}(1)\right)^{\frac{1}{\gamma}+\tau}\right)\cdot\left(I^{-N}(p^u_0\wedge p^s_0)\right)^{\frac{1}{\gamma}+\tau}
.
\end{align}
By \eqref{forSummVar} in Lemma \ref{FaveForNow}, $\left(I^{-N}(1)\right)^{\frac{1}{\gamma}+\tau}=O(\frac{1}{N^{1+\delta}})$ for some $\delta>0$, and so $N\cdot\left(I^{-N}(1)\right)^{\frac{1}{\gamma}+\tau}$ is bounded. Set $C_1:=3\cdot\sup_{N\geq1}\{ N\cdot\left(I^{-N}(1)\right)^{\frac{1}{\gamma}+\tau}\}$, then $d_0\leq C_1\cdot [I^{-N}(p^u_0\wedge p^s_0)]^{\frac{1}{\gamma}+\tau}$, which is summable by Lemma \ref{FaveForNow}.

\end{proof}
\begin{theorem}\label{DefOfPi}
Given a chain of double charts $(v_i)_{i\in\mathbb{Z}}$, let $\pi((\underline v)):=$unique intersection point of $V^u((v_i)_{i\leq0})$ and $V^s((v_i)_{i\geq0})$. Then \begin{enumerate}
\item $\pi$ is well-defined and $\pi\circ\sigma=f\circ\pi$.
\item $\pi:\Sigma\rightarrow M$ is uniformly continuous.
\item $\pi[\Sigma]\supset\pi[\Sigma^\#]\supset \RST$
.
\end{enumerate}
\end{theorem}
\begin{proof}
\textit{Part 1}: $\pi$ is well-defined thanks to Proposition \ref{firstbefore}. Now, write $v_i=\psi_{x_i}^{p_i^s,p_i^u}$ and $z=\pi((\underline v))$. We claim that $f^k(z)\in\psi_{x_k}[R_{Q_\epsilon(x_k)}(0)]\text{  }(k\in\mathbb{Z})$. For $k=0$ this is true because $z\in V^s((v_i)_{i\geq0})$, and this is an $s$-admissible manifold in $\psi_{x_0}^{p_0^s,p_0^u}$. For $k>0$ we use Proposition \ref{firstofchapter} (part 3) to see that $$f^k(z)\in f^k[V^s((v_i)_{i\geq0})]\subset V^s((v_{i+k})_{i\geq0}).$$
Since $V^s((v_{i+k})_{i\geq0})$ is an $s$-admissible manifold in $\psi_{x_k}^{p_k^s,p_k^u}$,  $f^k(z)\in\psi_{x_k}[R_{Q_\epsilon(x_k)}(0)]$. The case $k<0$ can be handled the same way, using $V^u((v_i)_{i\leq0})$. Thus $z$ satisfies: $$f^k(z)\in\psi_{x_k}[R_{Q_\epsilon(x_k)}(0)]\text{  }\text{ for all }k\in\mathbb{Z}.$$
Any point which satisfies this must be $z$, since by Proposition \ref{firstofchapter} (part 4) it must lie in $V^u((v_i)_{i\leq0})\cap V^s((v_i)_{i\geq0})$ (and in fact even in $\widecheck V^u((v_i)_{i\leq0})\cap \widecheck V^s((v_i)_{i\geq0})$), and this equation characterizes $\pi(\underline v)=z$. Hence: $$f^k(f(\pi(\underline v)))=f^{k+1}(\pi(\underline v))\in\psi_{x_{k+1}}^{p_{k+1}^s,p_{k+1}^u}\text{  }\forall k\in\mathbb{Z}$$ and this is the condition that characterizes $\pi(\sigma\underline v)$.

\medskip
\textit{Part 2}: We saw that $\underline v\mapsto \widecheck{V}^u((v_i)_{i\leq0})$ and $\underline v\mapsto \widecheck{V}^s((v_i)_{i\geq0})$ are uniformly continuous (Proposition \ref{firstofchapter}). Since the intersection of an $s$-admissible manifold and a $u$-admissible manifold is a Lipschitz function of those manifolds (Proposition \ref{firstbefore}), $\pi$ is also uniformly continuous.

\medskip
\textit{Part 3}: We prove $\pi[\Sigma]\supset \RST$: Suppose $x\in \RST$, then by Proposition \ref{prop131} there exist $\psi_{x_k}^{p_k^s,p_k^u}\in\mathcal{V}$ s.t $\psi_{x_k}^{p_k^s,p_k^u}\rightarrow\psi_{x_{k+1}}^{p_{k+1}^s,p_{k+1}^u}$ for all $k$, and s.t $\psi_{x_k}^{p_k^s\wedge p_k^u}$ $I$-overlaps $\psi_{f^k(x)}^{p_k^s\wedge p_k^u}$ for all $k\in\mathbb{Z}$. This implies $$f^k(x)=\psi_{f^k(x)}(0)\in\psi_{f^k(x)}[R_{p_k^s\wedge p_k^u}(0)]\subset\psi_{f^k(x)}[R_{Q_\epsilon(x_k)}(0)].$$
Then $x$ satisfies $f^k(x)\in\psi_{x_k}[R_{Q_\epsilon(x_k)}(0)]\text{  }(k\in\mathbb{Z})$. It follows that $x=\pi(\underline v)$ with $\underline v=(\psi_{x_i}^{p_i^s,p_i^u})_{i\in\mathbb{Z}}$. Let $\{q(f^n(x))\}_{n\in\mathbb{Z}}$ be given by the strong temperability of $x$, then by Lemma \ref{subordinatedchain} applied to $\{Q_\epsilon(f^n(x))\}_{n\in\mathbb{Z}}$ and $\{q(f^n(x))\}_{n\in\mathbb{Z}}$, the sequence $\{(p^s_i,p^u_i)\}_{i\in\mathbb{Z}}$ is $I$-strongly subordinated to $\{Q_\epsilon(f^i(x))\}_{i\in\mathbb{Z}}$ and satisfies $p^s_i\wedge p^u_i\geq q(f^i(x))$ for all $i\in\mathbb{Z}$
.
By the definition of $\RST$ there exist subsequences $i_k,j_k\uparrow\infty$ for which $p_{i_k}^s\wedge p_{i_k}^u\geq q(f^{i_k}(x))$ and $p_{-j_k}^s\wedge p_{-j_k}^u\geq q(f^{-j_k}(x))$ are bounded away from zero (see Proposition \ref{subordinatedchain}). By the discreteness property of $\mathcal{A}$ (Proposition \ref{discreteness}), $\psi_{x_i}^{p_i^s,p_i^u}$ must repeat some symbol infinitely often in the past, and some symbol (possibly a different one) infinitely often in the future. Thus the above actually proves that $$\pi[\Sigma^\#]\supset \RST,$$
where $\Sigma^\#=\{\underline v\in\Sigma: \exists v,w\in\mathcal{V},n_k,m_k\uparrow\infty\text{ s.t }v_{n_k}=v,v_{-m_k}=w\}$.
\end{proof}
\noindent\textbf{Remark:} After we prove the Inverse Problem properties in \textsection \ref{chapter333}, we in fact identify the image of 
 $\Sigma^\#$, and not just a subset of it. 

\noindent\textbf{Summary up to here}
\begin{itemize}
    \item We defined $I$-chains (Definition \ref{epsilonchains}).
    \item We proved the shadowing lemma: for every chain $v=(\psi^{p_i^s,p_i^u}_{x_i})_{i\in\mathbb{Z}}$ there exists $x\in M$ s.t $f^i(x)\in \psi_{x_i}[R_{p^s_i\wedge p^u_i}(0)]$ $(i\in\mathbb{Z})$ [proof of Theorem \ref{DefOfPi}].
    \item We found a countable set of charts $\mathcal{A}$ s.t the countable Markov shift $$\Sigma=\{\underline v\in \mathcal{A}^\mathbb{Z}: \underline v\text{ is an }I\text{-chain}\}$$ is a Markov extension of a set containing $\RST$ (Theorem \ref{DefOfPi});
    But the coding $\pi:\Sigma\rightarrow M$ is not finite-to-one.
\end{itemize}
Along the way we saw that if $x$ is shadowed by $v=(v_i)_{i\in\mathbb{Z}}$ then $V^s((v_i)_{i\geq0}),V^u((v_i)_{i\leq0})$ are local (perhaps weakly) stable/unstable manifolds of $x$. (Proposition \ref{firstofchapter}).

\section{
Chains which shadow the same orbit are close}\label{chapter333}
\subsection{The inverse problem
}
The aim of this part is to show that the map $\pi:\Sigma
^\#\rightarrow \RST$ from Theorem \ref{DefOfPi}, is ``almost invertible": if $\pi((\psi_{x_i}^{p_i^s,p_i^u})_{i\in\mathbb{Z}})=\pi((\psi_{y_i}^{q_i^s,q_i^u})_{i\in\mathbb{Z}})$, then $\forall i$: $x_i\approx y_i$, $p^{s/u}_i\approx q^{s/u}_i$ and $C_0(x_i)\approx C_0(y_i)$, where $\approx$ means that the respective parameters belong to the same compact sets. The compact sets can be as small as we wish by choosing small enough $\epsilon$. The discretization we constructed in \textsection 1.2.3 in fact allows these compact sets to contain only a finite number of such possible parameters for our charts. This does not mean that $\pi$ is finite-to-one, but it will allow us to construct a finite-to-one coding later. 
\subsubsection{Comparing orbits of tangent vectors}
\begin{lemma}\label{firstchapter2}
Let $(\psi_{x_i}^{p_i^s,p_i^u})_{i\in\mathbb{Z}},(\psi_{y_i}^{q_i^s,q_i^u})_{i\in\mathbb{Z}}$ be two chains s.t $\pi((\psi_{x_i}^{p_i^s,p_i^u})_{i\in\mathbb{Z}})=\pi((\psi_{y_i}^{q_i^s,q_i^u})_{i\in\mathbb{Z}})=p$, then $d(x_i,y_i)<\frac{\sqrt{d}}{50}\max\{p_i^s\wedge p_i^u,q_i^s\wedge q_i^u\}^2$.
\end{lemma}
\begin{proof}$f^i(p)$ is the intersection of a $u$-admissible manifold and an $s$-admissible manifold in $\psi_{x_i}^{p_i^s,p_i^u}$, therefore (by Proposition \ref{firstbefore}) $f^i(p)=\psi_{x_i}(v_i,w_i)$ where $|v_i|_\infty,|w_i|_\infty<10^{-2}(p_i^s\wedge p_i^u)^2$. $d(f^i(p),x_i)=d(\exp_{x_i}(C_0(x_i)(v_i,w_i)),x_i)=|C_0(x_i)(v_i,w_i)|_2\leq \sqrt{d}|(v_i,w_i)|_\infty\leq\frac{\sqrt{d}}{100}(p_i^s\wedge p_i^u)^2$. Similarly $d(f^i(x),y_i)<\frac{\sqrt{d}}{100}(q_i^s\wedge q_i^u)^2$. It follows that $d(x_i,y_i)<\frac{\sqrt{d}}{50}\max\{q_i^s\wedge q_i^u,p_i^s\wedge p_i^u\}^2$.
\end{proof}
The following definitions are taken from \cite{Sarig}.
\begin{definition}
Let $V^u$ be a $u$-admissible manifold in the double chart $\psi_x^{p^s,p^u}$. We say that $V^u$ {\em stays in windows} if there exists a negative chain $(\psi_{x_i}^{p_i^s,p_i^u})_{i\leq0}$ with $\psi_{x_0}^{p_0^s,p_0^u}=\psi_x^{p^s,p^u}$ and $u$-admissible manifolds $W_i^u$ in $\psi_{x_i}^{p_i^s,p_i^u}$ s.t $f^{-|i|}[V^u]\subset W_i^u$ for all $i\leq0$.
\end{definition}
\begin{definition}
Let $V^s$ be an $s$-admissible manifold in the double chart $\psi_x^{p^s,p^u}$. We say that $V^s$ {\em stays in windows} if there exists a positive chain $(\psi_{x_i}^{p_i^s,p_i^u})_{i\geq0}$ with $\psi_{x_0}^{p_0^s,p_0^u}=\psi_x^{p^s,p^u}$ and $s$-admissible manifolds $W_i^s$ in $\psi_{x_i}^{p_i^s,p_i^u}$ s.t $f^i[V_i^u]\subset W_i^s$ for all $i\geq0$.
\end{definition}

\noindent\textbf{Remark:} If $\underline v$ is a chain, then $V_i^u:=V^u((v_k)_{k\leq i})$ and $V_i^s:=V^s((v_k)_{k\geq i})$ stay in windows, since $f^{-k}[V_i^u]\subset V_{i-k}^u$ and $f^{k}[V_i^s]\subset V_{i+k}^s$ for all $k\geq0$ (by Proposition \ref{firstofchapter}).
\begin{prop}\label{Lambda} The following holds for all $\epsilon$ small enough: Let $V^s$ be an admissible $s$-manifold in $\psi_x^{p^s,p^u}$, and suppose $V^s$ stays in windows.
\begin{enumerate}
\item For every $y,z\in V^s$: $d(f^k(y),f^k(z))<4 I^{-k}(p^s)$ for all $k\geq0$.
\item For every $y\in V^s$ and $u\in T_yV^s(1)$: $|d_yf^ku|\leq4\cdot I^{-k}(p_0^s)\cdot\frac{\|C_0(x)^{-1}\|}{p^s}$ for all $k\geq0$.

\noindent Recall: $T_yV^s(1)$ is the unit ball in $T_yV^s$, w.r.t the Riemannian norm.
\item For $y\in V^s$, let $\omega_s(y)$ be a normalized volume form of $T_yV^s$. For all $y,z\in V^s$
,  $\forall k\geq0$:
$$|\log|d_yf^k\omega_s(y)|-\log|d_zf^k \omega_s(z)||<\epsilon^2 (\frac{d(y,z)}{(p^s)^3})^{\beta/4},$$
\end{enumerate}
 
\noindent where $|d_yf^k\omega_s(y)|=\bigl|\det d_yf^k|_{T_yV^s}\bigr|,|d_zf^k\omega_s(z)|=\bigl|\det d_zf^k|_{T_zV^s}\bigr|$. The symmetric statement holds for $u$-admissible manifolds which stay in windows, for ``$s$" replaced by ``$u$" and $f$ by $f^{-1}$.
\end{prop}
\begin{proof}

Suppose $V^s$ is an $s$-admissible manifold in $\psi_{x}^{p^s,p^u}$, which stays in windows, then there is a positive chain $(\psi_{x_i}^{p_i^s,p_i^u})_{i\geq0}$ s.t $\psi_{x_0}^{p_0^s,p_0^u}=\psi_x^{p^s,p^u}$, and there are $s$-admissible manifolds $W_i^s$ in $\psi_{x_i}^{p_i^s,p_i^u}$ s.t $f^i[V^s]\subset W_i^s$ for all $i\geq0$. We write: \begin{itemize}
\item $V^s=\{(t,F_0(t)): |t|_\infty\leq p^s\}$
\item $W_i^s=\psi_{x_i}[\{(t,F_i(t)):|t|_\infty\leq p_i^s\}]$
\item $\eta_i:=p_i^s\wedge p_i^s$
\end{itemize} Admissibility means that $\|d_\cdot F_i\|_{\beta/2}\leq\frac{1}{2},\|d_0F_i\|\leq\frac{1}{2}(p_i^s\wedge p_i^u)^{\beta/2}$, $|F_i(0)|_\infty\leq10^{-3}(p_i^s\wedge p_i^u)^2$ and $\mathrm{Lip}(F_i)<\epsilon$ ($\because$ Definition \ref{admissible}). 

\textit{Part 1}: 
For every $y,z\in V^s$: $d(f^k(y),f^k(z))\leq4 I^{-k}(p_0^s)$ for all $k\geq0$.

Since $V^s$ stays in windows: $f^k[V^s]\subset\psi_{x_k}[R_{Q_\epsilon(x_k)}(0)]$ for all $k\geq0$. Therefore $\forall y,z\in V^s$ one can write $f^k(y)=\psi_{x_k}(y_k')$ and $f^k(z)=\psi_{x_k}(z_k')$, where $y_k':=(y_k,F_k(y_k)),z_k':=(z_k,F_k(z_k))$ belong to $R_{Q_\epsilon(x_k)}(0)$. For every $k$: $y_{k+1}'=f_{x_kx_{k+1}}(y_k')$ and $z_{k+1}'=f_{x_kx_{k+1}}(z_k')$. (Reminder: $f_{x_kx_{k+1}}=\psi_{x_{k+1}}^{-1}\circ f\circ\psi_{x_k}$).
By Proposition \ref{3.4inomris},

\begin{align}\label{forSummability}
	|y_{k+1}-z_{k+1}|_\infty\leq &\|D_{s,k}\|\cdot|y_k-z_k|_\infty+\sqrt\epsilon\eta_k^\frac{\beta}{2}(|y_k-z_k|_\infty+\mathrm{Lip}(F_k)|y_k-z_k|_\infty)\nonumber\\
\leq&(e^{-\frac{1}{S^2(x_k)}}+2\sqrt\epsilon\eta_k^\frac{\beta}{2})|y_k-z_k|_\infty<e^{-\frac{1-\epsilon^2}{S^2(x_k)}}|y_k-z_k|_\infty\leq\ldots\leq e^{-\sum_{j=0}^{k}\frac{1-\epsilon^2}{S^2(x_j)}}|y_0-z_0|_\infty\nonumber\\
\leq& e^{-(1-\epsilon^2)C_{\beta,\epsilon}^{-\frac{1}{\gamma}}\sum_{j=0}^{k}(p^s_{j}
)^\frac{1}{\gamma}}|y_0-z_0|_\infty=e^{-\frac{(1-\epsilon^2)C_{\beta,\epsilon}^{-\frac{1}{\gamma}}}{\Gamma}\cdot\Gamma\sum_{j=0}^{k}(I^{-j}(p^s_{0}
))^\frac{1}{\gamma}}|y_0-z_0|_\infty\nonumber\\
=& e^{-\frac{(1-\epsilon^2)C_{\beta,\epsilon}^{-\frac{1}{\gamma}}}{\Gamma}\cdot\Gamma\sum_{j=0}^{k}(I^{-j}(p^s_{0}
))^\frac{1}{\gamma}}p_0^s\frac{|y_0-z_0|_\infty}{p_0^s}\nonumber\\
\leq &e^{-\Gamma\sum_{j=0}^{k}(I^{-j}(p^s_{0}
))^\frac{1}{\gamma}}p_0^s\frac{|y_0-z_0|_\infty}{p_0^s}= I^{-k}(p_0^s)\cdot \frac{|y_0-z_0|_\infty}{p_0^s},
\end{align}

where the last inequality is true since $\epsilon$ is so small that $\frac{(1-\epsilon^2)C_{\beta,\epsilon}^{-\frac{1}{\gamma}}}{\Gamma}>2\geq1$ (see \eqref{foreq13self}). Since $y_0',z_0'$ are on the graph of an $s$-admissible manifold in $\psi_{x_0}^{p_0^u,p_0^s}$, their stable components are in $R_{p_0^s}(0)$, so $|y_0-z_0|\leq2p_0^s$. So we get, 
\begin{equation}\label{fornewpart3inverse}
|y_k-z_k|_\infty\leq2 I^{-k}(p_0^s).
\end{equation}
 In fact, we even got
\begin{equation}\label{forpart3here}
	|y_k-z_k|_\infty\leq 2 (I^{-k}(p_0^s))^2\frac{1}{p_0^s}.
\end{equation}

 Now since $y_k'=(y_k,F_k(y_k)),z_k'=(z_k,F_k(z_k))$ and $\mathrm{Lip}(F_k)<\epsilon$: $|y_k'-z_k'|_\infty\leq2I^{-k}(p_0^s)$. Pesin charts have Lipschitz constant less than $2$
, so $$d(f^k(y),f^k(z))<4I^{-k}(p_0^s).$$ 

For every $k\geq0$, $y_{k+1}'=f_{x_kx_{k+1}}(y_k')$ and $z_{k+1}'=f_{x_kx_{k+1}}(z_k')$. (Reminder: $f_{x_kx_{k+1}}=\psi_{x_{k+1}}^{-1}\circ f\circ\psi_{x_k}$). Then,

\begin{align}\label{calcin6.3.1}  
|y_k'-z_k'|_\infty\leq& I^{-k}(p_0^s)\frac{|y_0'-z_0'|_\infty}{p_0^s}=I^{-k}(p_0^s)\frac{|\psi_x^{-1}(y)-\psi_x^{-1}(z)|_\infty}{p_0^s}\\
\leq&2\|C_0^{-1}(x)\|\cdot I^{-k}(p_0^s)\cdot \frac{d(y,z)}{p_0^s}
<2\|C_0^{-1}(x)\|\cdot (I^{-k}(p_0^s))^{1-\frac{1}{2\gamma}}(p_0^s)^\frac{1}{2\gamma}\cdot \frac{d(y,z)}{p_0^s}\nonumber\\
\leq& C_{\beta,\epsilon}^\frac{1}{\gamma}\cdot 2(I^{-k}(p_0^s))^{1-\frac{1}{2\gamma}}\cdot \frac{d(y,z)}{p_0^s}, \text{ }\forall k\geq0. \nonumber
\end{align}

Similarly, by \eqref{forpart3here}, we get 
\begin{equation}\label{extendforpart3here}
|y_k'-z_k'|_\infty\leq 2\|C_0^{-1}(x)\|\cdot (I^{-k}(p_0^s))^2\cdot \frac{d(y,z)}{(p_0^s)^2}.
\end{equation}

In addition $f^k(y)=\psi_{x_k}(y_k'),f^k(z)=\psi_{x_k}(z_k')$, whence \begin{equation}\label{chromecastsonic}d(f^k(y),f^k(z))\leq \mathrm{Lip}(\psi_{x_k})|y_k'-z_k'|\leq C_{\beta,\epsilon}^\frac{1}{\gamma}\cdot 4 (I^{-k}(p_0^s))^{1-\frac{1}{2\gamma}}\cdot \frac{d(y,z)}{p_0^s}.\end{equation}

\textit{Part 2:} 
Let $y\in V^s$ and let $u\in T_yV^s(1)$ some unit vector tangent to $V^s$. $f^k(y)\in W_k^s\subset\psi_{x_k}[R_{Q_\epsilon(x_k)}(0)]$, so $d_yf^ku=d_{y_k'}\psi_{x_k}\Big(\begin{array}{c}
\xi_k\\
\zeta_k\\
\end{array}\Big)$ where $\Big(\begin{array}{c}
\xi_k\\
\zeta_k\\
\end{array}\Big)$ is tangent to the graph of $F_k$. Since $\mathrm{Lip}(F_k)<\epsilon$, we get that $|\zeta_k|_\infty\leq\epsilon|\xi_k|_\infty$ for all $k$. The defining relation in the sequence is $$\Big(\begin{array}{c}
\xi_{k+1}\\
\zeta_{k+1}\\
\end{array}\Big)=\Big[\begin{pmatrix}D_{s,k}  &   \\  & D_{u,k}
\end{pmatrix}+\begin{pmatrix}d_{y_k'}h_s^{(k)} \\ d_{y_k'}h_u^{(k)}
\end{pmatrix}\Big]\Big(\begin{array}{c}
\xi_k\\
\zeta_k\\
\end{array}\Big)$$
where $D_{s,k}, D_{u,k}$ are block matrices. So we get
$$|\xi_{k+1}|_\infty\leq(e^{\frac{-1}{S^2(x_k)}}+\sqrt\epsilon\eta_k^\frac{\beta}{2})|\xi_k|_\infty\leq e^{-\frac{1-\epsilon}{S^2(x_k)}}|\xi_k|_\infty\leq e^{-\sum_{j=0}^k\frac{1-\epsilon}{S^2(x_j)}}|\xi_0|_\infty.$$
Recalling the definition $d_yf^ku=d_{y_k'}\psi_{x_k}\Big(\begin{array}{c}
\xi_k\\
\zeta_k\\
\end{array}\Big)$, and recalling that $\|d_\cdot\psi_{x_k}\|\leq2$  we see that, 
%
$|d_yf^ku|_\infty\leq4\cdot I^{-k}(p_0^s)\cdot\frac{\|C_0(x_0)^{-1}\|}{p_0^s}$ as wished.

\medskip

\textit{Part 3:} Fix $n\in\mathbb{N}$, and denote $A:=\Big|\log|d_yf^n\omega_s(y)|-\log|d_zf^n\omega_s(z)|\Big|$. For every $p\in V^s
$: $$d_pf^n\omega_s(p)=d_{f(p)}f^{n-1}d_pf(\omega_s(p))=|d_pf\omega_s(p)|\cdot d_{f(p)}f^{n-1}\omega_s(f(p))=\cdots=\prod_{k=1}^{n}|d_{f^k(p)}f\omega_s(f^k(p))|\cdot \omega_s(f^{n}(p)),$$ where $\omega_s(f^k(p)):=\frac{d_pf^k\omega_s(p)}{|d_pf^k\omega_s(p)|}$ is a normalized volume form of $T_{f^k(p)}f^k[V^s]$. Thus: $$A:=\Big|\log\frac{|d_yf^n\omega_s(y)|}{|d_zf^n\omega_s(z)|}\Big|\leq\sum_{k=1}^{n}\Big|\log|d_{f^k(y)}f\omega_s(f^k(y))|-\log|d_{f^k(z)}f\omega_s(f^k(z))|\Big|.$$

In section 1.2 we have covered $M$ by a finite collection $\mathcal{D}$ of open sets $D$, equipped with a smooth map $\theta_D:TD\rightarrow\mathbb{R}^d$ s.t $\theta_D|_{T_xM}:T_xM\rightarrow\mathbb{R}^d$ is an isometry and $\nu_x:=\theta_D^{-1}:\mathbb{R}^d\rightarrow TM$ has the property that $(x,v)\mapsto\nu_x(v)$ is Lipschitz on $D\times B_1(0)$. Since $f$ is $C^{1+\beta}$ and $M$ is compact, $d_pfv$ depends in a $\beta$-H\"older way on $p$, and in a Lipschitz way on $v$. It follows that there exists a constant $H_0>1$ s.t for every $D,D'\in\mathcal{D}; y',z'\in D,\text{ }f(y'),f(z')\in D'; u,v\in\mathbb{R}^d(1)$: $\Big|\Theta_{D'}d_{y'}f\nu_{y'}(u)-\Theta_{D'}d_{z'}f\nu_{z'}(v)\Big|<H_0(d(y',z')^\beta+|u-v|_2)$. Choose $D_k\in\mathcal{D}$ s.t $D_k\ni f^k(y),f^k(z)$. Such sets exist provided $\epsilon$ is much smaller than the Lebesgue number of $\mathcal{D}$, since by part 1 $d(f^k(y),f^k(z))<4\epsilon$. By writing $Id=\Theta_{D_k}\circ\nu_{f^k(y)}$ and $Id=\Theta_{D_k}\circ\nu_{f^k(z)}$, and recalling that $\Theta_{k}$ are isometries, we see that:
\begin{align}\label{NewA}
A\leq&\sum_{k=1}^n\Big|\log|\Theta_{D_{k+1}}d_{f^k(y)}f\nu_{f^k(y)}\Theta_{D_k}\omega_s(f^k(y))|-\log|\Theta_{D_{k+1}}d_{f^k(z)}f\nu_{f^k(z)}\Theta_{D_k}\omega_s(f^k(z))|\Big|\\
\leq& \sum_{k=1}^nM_f^d\Big|\Theta_{D_{k+1}}d_{f^k(y)}f\nu_{f^k(y)}\Theta_{D_k}\omega_s(f^k(y))-\Theta_{D_{k+1}}d_{f^k(z)}f\nu_{f^k(z)}\Theta_{D_k}\omega_s(f^k(z))\Big|,\nonumber
\end{align}
by the inequalities $|\log a-\log b|\leq\frac{|a-b|}{\min\{a,b\}}$ and the following estimate:
\begin{align*}\inf_{p\in\{y,z\},k\geq0}|\Theta_{D_{k+1}}d_{f^k(p)}f\nu_{f^k(p)}\Theta_{D_k}\omega_s(f^k(p))|=&\inf_{p\in\{y,z\}}|\mathrm{Jac}(d_pf|_{T_pV^s})|\\
=&\inf_{p\in\{y,z\}}\mathrm{Jac}(\sqrt{(d_pf|_{T_pV^s})^td_pf|_{T_pV^s}})\geq M_f^{-d}.
\end{align*}
Let $\omega_1$ be the standard volume form for $\mathbb{R}^{s(x)}$, and consider the map $t\overset{G_k}{\mapsto}(t,F_k(t))$, $G_k:R_{p^s_k}(0)(\subset \mathbb{R}^{s(x)})\rightarrow \mathbb{R}^d$. Then $$\omega_s(f^k(y))=\pm\frac{d_{y_k}(\psi_{x_k}\circ G_k)\omega_1}{|d_{y_k}(\psi_{x_k}\circ G_k)\omega_1|},\text{ }\omega_s(f^k(z))=\pm\frac{d_{z_k}(\psi_{x_k}\circ G_k)\omega_1}{|d_{z_k}(\psi_{x_k}\circ G_k)\omega_1|},$$
where the $\pm$ notation means the equality holds for one of the signs. Since we are only interested in the absolute value of the volume forms, we can assume w.l.o.g that the sign is identical for both of the volume forms. $\forall u\in \mathbb{R}^d(1)$, $|d_{y_k/z_k}(\psi_{x_k}\circ G_k)u|\geq\frac{1}{2\cdot\|C_0^{-1}(x_k)\|}$, then $|d_{y_k/z_k}(\psi_{x_k}\circ G_k)\omega_1|\geq\frac{1}{(2\|C_0^{-1}(x_k)\|)^{s(x)}}$. 
Recalling that $\Theta_{D_k}$ are isometries, we get
\begin{align}\label{shezifhadash}
    |\Theta_{D_k}\omega_s(f^k(y))-\Theta_{D_k}\omega_s(f^k(z))|&=\Big|\frac{\Theta_{D_k}d_{y_k}(\psi_{x_k}\circ G_k)\omega_1}{|\Theta_{D_k}d_{y_k}(\psi_{x_k}\circ G_k)\omega_1|}-\frac{\Theta_{D_k}d_{z_k}(\psi_{x_k}\circ G_k)\omega_1}{|\Theta_{D_k}d_{z_k}(\psi_{x_k}\circ G_k)\omega_1|}\Big|\\
    &\leq(2\|C_0^{-1}(x_k)\|)^{s(x)}\Big|\Theta_{D_k}d_{y_k}(\psi_{x_k}\circ G_k)\omega_1-\Theta_{D_k}d_{z_k}(\psi_{x_k}\circ G_k)\omega_1\Big|.\nonumber
\end{align}
Substituting this back in \eqref{NewA} yields $\Big($recall Hadamard's inequality, which asserts that if a volume form $\widetilde{\omega}$ is the wedge product $\widetilde{\omega}=u_1\wedge...\wedge u_l$, then $|\widetilde{\omega}|\leq\prod_{i=1}^l|u_i|_2 \Big)$,
\begin{align}\label{wubbalubba}
    A\leq & M_f^d\sum_{k=1}^n\Big|\Theta_{D_{k+1}}d_{f^k(y)}f\nu_{f^k(y)}\Theta_{D_k}\omega_s(f^k(y))-\Theta_{D_{k+1}}d_{f^k(z)}f\nu_{f^k(z)}\Theta_{D_k}\omega_s(f^k(y))\nonumber\\
    + & \Theta_{D_{k+1}}d_{f^k(z)}f\nu_{f^k(z)}\Theta_{D_k}\omega_s(f^k(y))-\Theta_{D_{k+1}}d_{f^k(z)}f\nu_{f^k(z)}\Theta_{D_k}\omega_s(f^k(z))\Big|\nonumber\\
    \leq & M_f^d\sum_{k=1}^n\|d_{f^k(z)}f\|^{s(x)}\cdot|\Theta_{D_k}\omega_s(f^k(y))-\Theta_{D_k}\omega_s(f^k(z))|+\|\Theta_{D_{k+1}}d_{f^k(y)}f\nu_{f^k(y)}-\Theta_{D_{k+1}}d_{f^k(z)}f\nu_{f^k(z)}\|\nonumber\\
    \leq & M_f^d\cdot \sum_{k=1}^n\Big( M_f^d\cdot (2\|C_0^{-1}(x_k)\|)^{s(x)}\cdot\Big|\Theta_{D_k}d_{y_k}(\psi_{x_k}\circ G_k)\omega_1-\Theta_{D_k}d_{z_k}(\psi_{x_k}\circ G_k)\omega_1\Big|+H_0d(f^k(y),f^k(z))^\beta\Big)\nonumber\\
    \leq& M_f^d\cdot \sum_{k=1}^n\Big( M_f^d\cdot (2\|C_0^{-1}(x_k)\|)^{s(x)}\cdot[\text{H\"ol}_{\frac{\beta}{2}}(d_\cdot (\psi_{x_k}\circ G_k))|y_k-z_k|^\frac{\beta}{2}]^{s(x)}+H_0d(f^k(y),f^k(z))^\beta\Big)\nonumber\\
    \leq& M_f^d\cdot \sum_{k=1}^n\Big( M_f^d\cdot(2\cdot\frac{1}{2})^{s(x)}(2\|C_0^{-1}(x_k)\||y_k-z_k|^\frac{\beta}{2})^{s(x)}+H_0d(f^k(y),f^k(z))^\beta\Big)\nonumber\\
    \leq& M_f^d\cdot \sum_{k=1}^n\Big( M_f^d(2C_{\beta,\epsilon}^\frac{\beta}{4}|y_k-z_k|^\frac{\beta}{4})^{s(x)}+H_0d(f^k(y),f^k(z))^\beta\Big)\nonumber\\
    \leq& M_f^d\cdot \sum_{k=1}^n\Big( M_f^d\cdot( 2C_{\beta,\epsilon}^\frac{\beta}{4}(2\|C_0^{-1}(x)\|\cdot (I^{-k}(p_0^s))^2\cdot \frac{d(y,z)}{(p_0^s)^2})^\frac{\beta}{4})+H_0(C_{\beta,\epsilon}^\frac{1}{\gamma}\cdot2(I^{-k}(p_0^s))^{1-\frac{1}{2\gamma}}\cdot \frac{d(y,z)}{p_0^s})^\beta\Big)\nonumber\\
    \leq& M_f^d\cdot \sum_{k=1}^n\Big( M_f^d\cdot( 2C_{\beta,\epsilon}^\frac{\beta}{4}((I^{-k}(p_0^s))^2\cdot \frac{d(y,z)}{(p_0^s)^3})^\frac{\beta}{4})+H_0(C_{\beta,\epsilon}^\frac{1}{\gamma}\cdot2(I^{-k}(p_0^s))^{1-\frac{1}{2\gamma}}\cdot \frac{d(y,z)}{p_0^s})^\beta\Big)\nonumber\\
   \leq& 3M_f^{2d}H_0\cdot C_{\beta,\epsilon}^{\min\{\frac{\beta}{4},\frac{1}{\gamma}\}}\cdot\sum_{k=1}^n (I^{-k}(p_0^s))^\frac{\beta}{2}\cdot (\frac{d(y,z)}{(p_0^s)^3})^\frac{\beta}{4}\leq\epsilon ^2 \cdot (\frac{d(y,z)}{(p_0^s)^3})^\frac{\beta}{4},
\end{align}
where H\"ol$_\frac{\beta}{2}(d_\cdot G_k)\leq\frac{1}{2}$, and Lip($d_\cdot\psi_{x_k})\leq2$ since $\|C_0(x_k)\|\leq1$. The seventh inequality in \eqref{wubbalubba} uses the fact that $M_f>1$, and the penultimate inequality uses the fact that $|y_k-z_k|<1$ to replace $s(x)$ with $1$, and uses \eqref{extendforpart3here}. 

By the choice of $Q_\epsilon(\cdot)$ (Definition \ref{balls}) and the inequality $p^s_k\leq Q_\epsilon(x_k)$, and since $\gamma>\frac{2}{\beta}$, we have $\|C_0^{-1}(x_k)\|\cdot (p_k^s)^\frac{\beta}{4}<C_{\beta,\epsilon}^\frac{\beta}{4}$ for all $k\geq0$. In addition, Lemma \ref{FaveForNow} tells us that $\sum_{k\geq0}(I^{-k}(p_0^s))^\frac{\beta}{2}\leq \sum_{k\geq0}(I^{-k}(1))^\frac{\beta}{2}<\infty$. So for small enough $\epsilon$, \eqref{wubbalubba} holds. 
This estimate is uniform in $n$, and Part (3) is proven.

\end{proof}

\begin{cor}\label{RegularOnVs}
Let $V^s$ be an admissible $s$-manifold in $\psi_x^{p^s,p^u}$ which stays in windows. Then $\forall y\in V^s$,
$$S^2(y)\leq 2\cdot4\frac{\|C_0^{-1}(x)\|}{p^s}\sum_{m\geq0}(I^{-m}(1))^2<\infty.$$	
A similar bound holds for $U^2(\cdot)$ on $u$-manifolds which stay in windows.
\end{cor}

\noindent\underline{Notice}: In a setup where we are given two double charts $\psi_x^{p^s,p^u}\rightarrow \psi_y^{q^s,q^u}$, if $V^s$ is an $s$-admissible manifold in $\psi_y^{q^s,q^u}$ which stays in windows, then also $\mathcal{F}(V^s)$ is an $s$-admissible manifold in $\psi_x^{p^s,p^u}$ which stays in windows.
For the next lemma, recall Definition \ref{scalingfuncs} for $S(\cdot,\cdot)$.

\begin{lemma}\label{boundimprove} Let $\psi_x^{p^s,p^u}\rightarrow \psi_y^{q^s,q^u}$. Let $V^s$ be an $s$-admissible manifold in $\psi_y^{q^s,q^u}$ which stays in windows. Let $\mathcal{F}(V^s)$ be the Graph Transform of $V^s$ w.r.t to the edge $\psi_x^{p^s,p^u}\rightarrow \psi_y^{q^s,q^u}$. Let $q\in f[\mathcal{F}(V^s)]\subset V^s$.
Then
$\exists\pi_x:T_{f^{-1}(q)}\mathcal{F}(V^s)\rightarrow H^s(x),\pi_y:T_qV^s\rightarrow H^s(y)$ invertible linear transformations s.t $\forall \rho\geq e^{\Gamma Q_\epsilon^{\frac{\beta}{4}-\frac{1}{4\gamma}}(x)}$: if $\frac{S(q,\lambda)}{S(y,\pi_y \lambda)}\in [\rho^{-1},\rho]$ then
$$\frac{S(f^{-1}(q),\xi)}{S(x,\pi_x\xi)}\in [e^{\Gamma Q_\epsilon(x)^{\frac{\beta}{4}+\frac{3}{4\gamma}}}\rho^{-1},e^{-\Gamma Q_\epsilon(x)^{\frac{\beta}{4}+\frac{3}{4\gamma}}}\rho],$$
where $\lambda=d_{f^{-1}(q)}f\xi$ and $\xi\in T_{f^{-1}(q)}\mathcal{F}(V^s)$ is arbitrary.
A similar statement holds for the unstable case. We will also see that $\frac{|\pi_x\xi|}{|\xi|}=e^{\pm Q_\epsilon(x)^{\frac{\beta}{2}-\frac{1}{2\gamma}}}$ for all $\xi\in T_{f^{-1}(q)}\mathcal{F}(V^s)$ (similarly for $\pi_y$).
\end{lemma}
\begin{proof}  Recall the disks $\mathcal{D}=\{D\}$ and the isometries $\Theta_D:TD\rightarrow\mathbb{R}^d,\nu_a=\Theta_D^{-1}|_{T_aM}, a\in D\subset M$ from Definition \ref{isometries}. If $\epsilon$ is smaller than the Lebesgue number of $\mathcal{D}$, then $\exists D$ which contains $x, f^{-1}(q), f^{-1}(y)$. Let $\Theta=\Theta_D$, and set $C_1:=\Theta\circ C_0(x),C_2:=\Theta\circ C_0(f^{-1}(y))$.
Let $F,G$ denote the representing functions of $V^s,\mathcal{F}(V^s)$, respectively. 

Let $\xi\in T_{f^{-1}(q)}\mathcal{F}(V^s)$. Write $t=\psi_x^{-1}(f^{-1}(q))=(t_s,G(t_s))$, and
$\xi=d_{t}\psi_x\Big(\begin{array}{c}
v\\
d_{t_s}Gv\\
\end{array}\Big)$ for a vector $v\in \mathbb{R}^{s(x)}$. We define
$\pi_x\xi:=d_{(0,0)}\psi_x\Big(\begin{array}{c}
v\\
0\\
\end{array}\Big)$ (
$\pi_y$ is defined analogously). Denote by $E_0$ the $|\cdot|_\infty$-Lipschitz constant of $(x,\underline{u},\underline{v})\mapsto[\Theta_D\circ d_{\underline{v}}\exp][\nu_x \underline{u}]$.

We begin with bounding the distortion of $\pi_{x}$ (assume w.l.o.g\footnote{Otherwise cancel its size from both the denominator and numerator of the LHS of eq. \eqref{eq9}.} $|\xi|=1$):

\begin{align}\label{step1new}
|\Theta\xi-\Theta\pi_x\xi|=&|\Theta d_{(t_s,G(t_s))}\psi_x\Big(\begin{array}{c}
v\\
d_{t_s}Gv\\
\end{array}\Big)-\Theta d_{(0,0)}\psi_x\Big(\begin{array}{c}
v\\
0\\
\end{array}\Big)|\\
=&|\Theta \Big(d_{C_0(x)(t_s,G(t_s))}\exp_x\Big)\nu_x\circ C_1\Big(\begin{array}{c}
v\\
d_{t_s}Gv\\
\end{array}\Big)-\Theta \Big(d_{C_0(x)(0,0)}\exp_x\Big)\nu_x\circ C_1\Big(\begin{array}{c}
v\\
0\\
\end{array}\Big)|\nonumber\\
\leq& E_0\Big(\Big| C_1\Big( \Big(\begin{array}{c}
v\\
d_{t_s}Gv\\
\end{array}\Big) -\Big(\begin{array}{c}
v\\
0\\
\end{array}\Big)\Big)\Big|+\Big|C_0(x)\Big((t_s,G(t_s))-(0,0)\Big)\Big|\Big)\nonumber\\ \leq&E_0\Big(|v|\cdot\|d_{t_s}G\|+ |\exp_x^{-1}f^{-1}(q)|\Big)= E_0\Big(|v|\cdot\|d_{t_s}G\|+ d(f^{-1}(q),x)\Big)(\because C_1\text{ is a contraction})\nonumber\\
\leq&3E_0 Q_\epsilon(x)^{\beta/2}\|C_0^{-1}(x)\|\leq 3E_0C_{\beta,\epsilon}^\frac{1}{2\gamma}Q_\epsilon(x)^{\frac{\beta}{2}-\frac{1}{2\gamma}} \nonumber\\
\leq &\frac{1}{2}Q_\epsilon(x)^{\frac{\beta}{2}-\frac{1}{2\gamma}}\text{ }(\because |v|\leq 2\|C_0^{-1}(x)\|,\|d_{t_s}G\|\leq Q_\epsilon(x)^{\beta/2}).\nonumber
\end{align} 
Therefore $\frac{|\pi_x\xi|}{|\xi|}=e^{\pm Q_\epsilon(x)^{\frac{\beta}{2}-\frac{1}{2\gamma}}}$  (similarly, $\frac{|\pi_y\xi'|}{|\xi'|}=e^{\pm Q_\epsilon(y) ^{\frac{\beta}{2}-\frac{1}{2\gamma}}}, \xi'\in T_q V^s\setminus\{0\}$). We continue to decompose:
\begin{equation}\label{multipliers}\frac{S(x,\pi_x\xi)}{S(f^{-1}(q),\xi)}=\frac{S(f^{-1}(y),\eta)}{S(f^{-1}(q),\xi)}\cdot\frac{S(x,\pi_x \xi)}{S(f^{-1}(y),\eta)},\end{equation}
where $\eta:= d_y (f^{-1})\pi_y d_{f^{-1}(q)}f\xi\in H^s(f^{-1}(y))$  ($\|\pi_y\|\leq e^{Q_\epsilon^{\frac{\beta}{2}-\frac{1}{2\gamma}}(y)}$, hence $|\eta|\leq2M_f^2$).
We study the second factor. Denote
$$u_1:=C_0^{-1}(x)\xi,u_2:=C_0^{-1}(f^{-1}(y))\eta.$$
By the definitions of $S(\cdot,\cdot)$ and $C_0(\cdot)$ (Definition \ref{scalingfuncs}, Theorem \ref{pesinreduction}), the inverse of the second factor satisfies:
\begin{align}
\Big|\frac{S(f^{-1}(y),\eta)}{S(x,\pi_x\xi)}-1\Big|=&\Big|\frac{|C_0^{-1}(f^{-1}(y))\eta|}{|C_0^{-1}(x)\pi_x\xi|}-1\Big|=\Big|\frac{|C_2^{-1}\Theta\eta|}{|C_1^{-1}\Theta\pi_x\xi|}-1\Big|\label{eq9}\\ \leq&\Big|\frac{|C_2^{-1}\Theta\eta-C_1^{-1}\Theta\eta|+|C_1^{-1}\Theta\eta-C_1^{-1}\Theta\pi_x\xi|+|C_1^{-1}\Theta\pi_x\xi|}{|C_1^{-1}\Theta\pi_x\xi|}-1\Big|\nonumber\\
\leq&\frac{1}{|C_0^{-1}(x)\pi_x\xi|}\Big(\|C_2^{-1}-C_1^{-1}\|\cdot|\eta|+\|C_1^{-1}\|\cdot|\Theta\pi_x\xi-\Theta \eta|\Big)\nonumber\\
\leq&2\Big( 2M_f^2Q_\epsilon^4(x)Q_\epsilon^4(f^{-1}(y))+\|C_0^{-1}(x)\|\cdot|\Theta\pi_x\xi-\Theta\eta|\nonumber\Big)\\
(\because&\|C_0(x)\|\leq1\text{, Lemma \ref{overlap}},|\eta|\leq 2M_f^2,|\pi_x\xi|\geq e^{-2Q_\epsilon^{\frac{\beta}{2}-\frac{1}{2\gamma}}(x)}\geq\frac{1}{2}).\nonumber
\end{align}
By the definition of $\omega_0$ (Lemma \ref{omega0}), $\frac{Q_\epsilon(f^{-1}(y))}{Q_\epsilon(y)}\leq \omega_0$. Substituting this gives:
\begin{align}\label{thirdterm}
\Big|\frac{S(f^{-1}(y),\eta)}{S(x,\pi_x\xi)}-1\Big|\leq&4M_f^2\omega_0Q_\epsilon^4(x)Q_\epsilon^4(y)+2\cdot\|C_0^{-1}(x)\|\cdot|\Theta\pi_x\xi-\Theta\eta|\nonumber\\
\leq&\frac{1}{4}Q_\epsilon^3(x)Q_\epsilon^3(y)+2\|C_0^{-1}(x)\|\cdot|\Theta\pi_x\xi-\Theta\eta|,\end{align}
for small enough $\epsilon$. Next we bound $|\Theta\pi_x\xi-\Theta\eta|$:

\medskip
\textit{Step 1}: $|\Theta\xi-\Theta\pi_x\xi|\leq 3E_0C_{\beta,\epsilon}^\frac{1}{2\gamma}Q_\epsilon(x)^{\frac{\beta}{2}-\frac{1}{2\gamma}}\leq\frac{1}{2}Q_\epsilon(x)^{\frac{\beta}{2}-\frac{1}{2\gamma}}$ by \eqref{step1new}.


\medskip
\textit{Step 2}: Estimation of $|\Theta\xi-\Theta \eta|$. Let $\lambda:=d_{f^{-1}(q)}f\xi$ and $\hat{\lambda}=\frac{\lambda}{|\lambda|}$ \text{ }\footnote{$\hat{\lambda}$ is well defined since $f$ is a diffeomorphism and $|\xi|\neq0$.}.
Then
\begin{align*}
|\Theta\eta-\Theta\xi|\equiv&|\Theta d_y(f^{-1})\pi_yd_{f^{-1}(q)}f\xi
-\Theta\xi|=|\Theta d_y(f^{-1})\pi_yd_{f^{-1}(q)}f\xi-\Theta d_q(f^{-1})d_{f^{-1}(q)}f\xi|\\
=&|\Theta d_y(f^{-1})\pi_y\lambda-\Theta d_q(f^{-1})\lambda|=|\Theta d_y(f^{-1})\nu_y\Theta'\pi_y\lambda-\Theta d_q(f^{-1})\nu_q\Theta'\lambda|\\
\leq&|\Theta d_y(f^{-1})\nu_y\Theta'\pi_y\lambda-\Theta d_q(f^{-1})\nu_q\Theta'\pi_y\lambda|+|\Theta d_q(f^{-1})\nu_q\Theta'\pi_y\lambda-\Theta d_q(f^{-1})\nu_q\Theta'\lambda|\\
\leq &\|\Theta d_y(f^{-1})\nu_y-\Theta d_q(f^{-1})\nu_q\|\cdot|\pi_y\lambda|+M_f|\Theta'\pi_y\lambda-\Theta'\lambda|\\
\leq &H_0\cdot d(q,y)^\beta|\pi_y\lambda|+M_f|\Theta'\pi_y\hat{\lambda}-\Theta'\hat{\lambda}|\cdot|\lambda|,    
\end{align*}
where $\Theta'$ is the linear isometry for the neighborhood containing $y$ and $q$, as in Definition \ref{isometries}, and $H_0$ is the bound for H\"ol$(d_\cdot f^{-1})$ as in the proof of Proposition \ref{Lambda}. 
$|\lambda|\leq M_f$. As in \eqref{step1new}, $|\Theta'\pi_y\hat{\lambda}-\Theta'\hat{\lambda}|\leq 3E_0C_{\beta,\epsilon}^\frac{1}{2\gamma} Q_\epsilon(y)^{\frac{\beta}{2}-\frac{1}{2\gamma}}$, and hence also $|\pi_y\lambda|=|\pi_y\hat{\lambda}|\cdot|\lambda|\leq  (1+Q_\epsilon(y)^{\frac{\beta}{2}-\frac{1}{2\gamma}})\cdot M_f<2M_f$. Therefore,
$$|\Theta \eta
-\Theta\xi|\leq M_f\cdot(2H_0\cdot Q_\epsilon(y)^{\beta}+ M_f\cdot 3E_0C_{\beta,\epsilon}^\frac{1}{2\gamma} Q_\epsilon(y)^{\frac{\beta}{2}-\frac{1}{2\gamma}})<2M_f^2(H_0+3E_0)C_{\beta,\epsilon}^\frac{1}{2\gamma}Q_\epsilon(y)^{\frac{\beta}{2}-\frac{1}{2\gamma}}.$$
Since $\psi_x^{p^s,p^u}\rightarrow \psi_y^{q^s,q^u}$,  $\psi_{f(x)}^{p^s\wedge p^u}$ and $\psi_y^{q^s\wedge q^u}$ $I$-overlap. Hence Lemma \ref{overlap} gives us the bound \begin{align*}\frac{Q_\epsilon(y)}{Q_\epsilon(x)}=&\frac{Q_\epsilon(f(x))}{Q_\epsilon(x)}\cdot\frac{Q_\epsilon(y)}{Q_\epsilon(f(x))}\leq\omega_0 \frac{Q_\epsilon(y)}{Q_\epsilon(f(x))}\leq\omega_0 \Big(\frac{\|C_0^{-1}(y)\|}{\|C_0^{-1}(f(x))\|}\Big)^{-2\gamma}\\
\leq&\omega_0 (e^{Q_\epsilon(f(x))Q_\epsilon(y)})^{2\gamma}\leq\omega_0 e^\epsilon\text{ (for small enough }\epsilon\text{)}.
\end{align*}
Thus
$$|\Theta\eta
-\Theta\xi|\leq2M_f^2(H_0+3E_0)(e^\epsilon\omega_0)^{\frac{\beta}{2}-\frac{1}{2\gamma}}C_{\beta,\epsilon}^\frac{1}{2\gamma}Q_\epsilon(x)^{\frac{\beta}{2}-\frac{1}{2\gamma}}\leq \frac{1}{2}C_{\beta,\epsilon}^\frac{1}{4\gamma}Q_\epsilon(x)^{\frac{\beta}{2}-\frac{1}{2\gamma}}\text{ for all }\epsilon\text{ enough}.$$

\medskip
\textit{Step 3}: Adding the estimates in steps 1 and 2, we obtain
\begin{align*}|\Theta\pi_x\xi-\Theta\eta|\leq&|\Theta\pi_x\xi-\Theta\xi|+|\Theta\xi-\Theta\eta|\\
\leq &Q_\epsilon(x)^{\frac{\beta}{2}-\frac{1}{2\gamma}}(\frac{1}{2}C_{\beta,\epsilon}^\frac{1}{4\gamma}+3E_0C_{\beta,\epsilon}^\frac{1}{2\gamma})\leq C_{\beta,\epsilon}^\frac{1}{4\gamma}Q_\epsilon(x)^{\frac{\beta}{2}-\frac{1}{2\gamma}}.\end{align*}

Substituting this in \eqref{thirdterm} gives (for small enough $\epsilon$):
\begin{equation}\label{ferligator}
\Big|\frac{S(f^{-1}(y),\eta)}{S(x,\pi_x\xi)}-1\Big|\leq\frac{1}{4}Q^3_\epsilon(x)Q^3_\epsilon(y)+2\|C_0^{-1}(x)\|\cdot C_{\beta,\epsilon}^\frac{1}{4\gamma}Q_{\epsilon}^{\frac{\beta}{2}-\frac{1}{2\gamma}}(x)\leq \frac{1}{2}C_{\beta,\epsilon}^\frac{1}{2\gamma}Q_\epsilon(x)^{\frac{\beta}{2}-\frac{1}{\gamma}},
\end{equation}

and hence, $$\frac{S(f^{-1}(y),\eta)}{S(x,\pi_x\xi)}= e^{\pm C_{\beta,\epsilon}^\frac{1}{2\gamma}Q_\epsilon(x)^{\frac{\beta}{2}-\frac{1}{\gamma}}}.$$ This concludes the study of the second factor of the RHS of \eqref{multipliers}.

\medskip
It remains to study the first factor of the RHS of \eqref{multipliers}. We begin with the identity

\begin{align*}S^2(f^{-1}(y),\eta)=&2\sum_{m=0}^{\infty}|d_{f^{-1}(y)}f^m\eta|^2
=2\Big(|\eta|^2+\sum_{m=1}^\infty|d_yf^{m-1}d_{f^{-1}(y)}f\eta|^2\Big)
=2|\eta|^2+S^2(y,d_{f^{-1}(y)}f\eta).\end{align*}
Similarly, and by the assumption of the lemma 
:
\begin{align*}S^2(f^{-1}(q),\xi)=&2|\xi|^2+S^2(q,d_{f^{-1}(q)}f\xi)\leq2|\xi|^2+\rho^2 S^2(y,\pi_yd_{f^{-1}(q)}f\xi).\end{align*}
By the choice of $\eta$, $d_{f^{-1}(y)}f\eta=\pi_yd_{f^{-1}(q)}f\xi$, so
$$S^2(f^{-1}(q),\xi)\leq2|\xi|^2+\rho^2S^2(y,d_{f^{-1}(y)}f\eta)$$
and so,
\begin{align}\frac{S^2(f^{-1}(q),\xi)}{S^2(f^{-1}(y),\eta)}\leq&\frac{2|\xi|^2+\rho^2S^2(y,d_{f^{-1}(y)}f\eta)}{2|\eta|^2+S^2(y,d_{f^{-1}(y)}f\eta)}=
\rho^2-\frac{2(\rho^2-1)|\eta|^2+2(|\eta|^2-|\xi|^2)}{2|\eta|^2+S^2(y,d_{f^{-1}(y)}f\eta)}\nonumber\\=&\rho^2-\frac{2(\rho^2-1)|\eta|^2+2(|\eta|^2-|\xi|^2)}{S^2(f^{-1}(y),\eta)}\label{kofiko}.\end{align}

By the bound achieved in step 2, and since $|\xi|=1$, $|\eta|=e^{\pm C_{\beta,\epsilon}^\frac{1}{4\gamma}Q_\epsilon(x)^{\frac{\beta}{2}-\frac{1}{2\gamma}}}$ (for small enough $\epsilon$). Hence $\Big||\xi|^2-|\eta|^2\Big|\leq2(e^{ C_{\beta,\epsilon}^\frac{1}{4\gamma}Q_\epsilon(x)^{\frac{\beta}{2}-\frac{1}{2\gamma}}}-1)\leq4 C_{\beta,\epsilon}^\frac{1}{4\gamma}Q_\epsilon(x)^{\frac{\beta}{2}-\frac{1}{2\gamma}}$. We get that for $\epsilon$ small enough, since $Q_\epsilon(x)\ll \epsilon$,
 \begin{align*}\rho^2-1&\geq e^{2\Gamma Q_\epsilon(x)^{\frac{\beta}{4}-\frac{1}{4\gamma}}}-1\geq 2\Gamma Q_\epsilon(x)^{\frac{\beta}{4}-\frac{1}{4\gamma}}\\\Rightarrow&
 (\rho^2-1)|\eta|^2+(|\xi|^2-|\eta|^2)
 \geq(\rho^2-1)|\eta|^2-4C_{\beta,\epsilon}^\frac{1}{4\gamma}Q_\epsilon(x)^{\frac{\beta}{2}-\frac{1}{2\gamma}}=(\rho^2-1)[|\eta|^2-\frac{4C_{\beta,\epsilon}^\frac{1}{4\gamma}Q_\epsilon(x)^{\frac{\beta}{2}-\frac{1}{2\gamma}}}{\rho^2-1}]\\
 &\geq(\rho^2-1)[e^{-2C_{\beta,\epsilon}^\frac{1}{4\gamma}Q_\epsilon(x)^{\frac{\beta}{2}-\frac{1}{2\gamma}}}-\frac{4C_{\beta,\epsilon}^\frac{1}{4\gamma}Q_\epsilon(x)^{\frac{\beta}{2}-\frac{1}{2\gamma}}}{2\Gamma Q_\epsilon(x)^{\frac{\beta}{4}-\frac{1}{4\gamma}}}]
\geq (\rho^2-1)e^{-C_{\beta,\epsilon}^\frac{1}{5\gamma}Q_\epsilon(x)^{\frac{\beta}{4}-\frac{1}{4\gamma}}}.
 \end{align*}
Plugging this in \eqref{kofiko} yields
\begin{align*}\frac{S^2(f^{-1}(q),\xi)}{S^2(f^{-1}(y),\eta)}\leq\rho^2-\frac{2(\rho^2-1)e^{-C_{\beta,\epsilon}^\frac{1}{5\gamma}Q_\epsilon(x)^{\frac{\beta}{4}-\frac{1}{4\gamma}}}}{S^2(f^{-1}(y),\eta)}\leq\rho^2-\frac{2(\rho^2-1)e^{-2C_{\beta,\epsilon}^\frac{1}{5\gamma}Q_\epsilon(x)^{\frac{\beta}{4}-\frac{1}{4\gamma}}}}{S^2(x,\pi_x\xi)},
\end{align*}
where the last inequality is because, as we saw above, $\frac{S(f^{-1}(y),\eta)}{S(x,\pi_x\xi)}=e^{\pm C_{\beta,\epsilon}^\frac{1}{2\gamma}} Q_\epsilon(x)^{\frac{\beta}{2}-\frac{1}{\gamma}}$ 
. So
$$\frac{S^2(f^{-1}(q),\xi)}{S^2(f^{-1}(y),\eta)}\leq\rho^2-\frac{2(\rho^2-1)e^{-2C_{\beta,\epsilon}^\frac{1}{5\gamma}Q_\epsilon(x)^{\frac{\beta}{4}-\frac{1}{4\gamma}}}}{\|C_0^{-1}(x)\|^2\cdot|\pi_x\xi|^2}.$$
Substituting $|\xi|=1$, $\frac{|\pi_x\xi|}{|\xi|}= e^{\pm Q_\epsilon(x)^{\frac{\beta}{2}-\frac{1}{2\gamma}}}$ gives:
$$\frac{S^2(f^{-1}(q),\xi)}{S^2(f^{-1}(y),\eta)}\leq\rho^2-\frac{2(\rho^2-1)e^{-2C_{\beta,\epsilon}^\frac{1}{5\gamma}Q_\epsilon(x)^{\frac{\beta}{4}-\frac{1}{4\gamma}}}}{\|C_0^{-1}(x)\|^2\cdot e^{2 Q_\epsilon(x)^{\frac{\beta}{2}-\frac{1}{2\gamma}}}}\leq\rho^2\Big(1-\frac{2(1-\frac{1}{\rho^2})e^{-3C_{\beta,\epsilon}^\frac{1}{5\gamma}Q_\epsilon(x)^{\frac{\beta}{4}-\frac{1}{4\gamma}}}}{\|C_0^{-1}(x)\|^2}\Big)$$
For $\rho\geq e^{\Gamma Q_\epsilon^{\frac{\beta}{4}-\frac{1}{4\gamma}}(x)}$, $1-\frac{1}{\rho^2}\geq1-e^{-2\Gamma Q_\epsilon^{\frac{\beta}{4}-\frac{1}{4\gamma}}(x)}\geq\Gamma Q_\epsilon^{\frac{\beta}{4}-\frac{1}{4\gamma}}(x)$, so:
$$\frac{S^2(f^{-1}(q),\xi)}{S^2(f^{-1}(y),\eta)}\leq\rho^2\Big(1-\frac{2\Gamma Q_\epsilon^{\frac{\beta}{4}-\frac{1}{4\gamma}}(x)e^{-3C_{\beta,\epsilon}^\frac{1}{5\gamma}Q_\epsilon(x)^{\frac{\beta}{4}-\frac{1}{4\gamma}}}}{\|C_0^{-1}(x)\|^2}\Big)
.$$
By definition, $Q_\epsilon(x)\leq\frac{C_{\beta,\epsilon}}{\|C_0^{-1}(x)\|^{2\gamma}}$. Hence,
$$\frac{S^2(f^{-1}(q),\xi)}{S^2(f^{-1}(y),\eta)}\leq\rho^2(1-3\Gamma Q_\epsilon(x)^{\frac{\beta}{4}-\frac{1}{4\gamma}+\frac{1}{\gamma}})\leq\rho^2 e^{-3\Gamma Q_\epsilon(x)^{\frac{\beta}{4}+\frac{3}{4\gamma}}}.$$
Similarly, $\frac{S^2(f^{-1}(q),\xi)}{S^2(f^{-1}(y),\eta)}\geq\rho^{-2}e^{3\Gamma Q_\epsilon(x)^{\frac{\beta}{4}+\frac{3}{4\gamma}}}$.

\medskip
In summary, $\frac{S^2(x,\pi_x\xi)}{S^2(f^{-1}(q),\xi)}$ is the product of two factors (see \eqref{multipliers}). The second factor is bounded by $e^{\pm C_{\beta,\epsilon}^\frac{1}{2\gamma}Q_\epsilon(x)^{\frac{\beta}{2}-\frac{1}{\gamma}}}$, 
and the first factor takes values in $[\rho^{-2}e^{3\Gamma Q_\epsilon(x)^{\frac{\beta}{4}+\frac{3}{4\gamma}}},\rho^2e^{-3\Gamma Q_\epsilon(x)^{\frac{\beta}{4}+\frac{3}{4\gamma}}}]$.
Therefore, 
\begin{align*}
\frac{S^2(f^{-1}(q),\xi)}{S^2(f^{-1}(y),\eta)}\leq&\rho^2 e^{-3\Gamma Q_\epsilon(x)^{\frac{\beta}{4}+\frac{3}{4\gamma}}+C_{\beta,\epsilon}^\frac{1}{5\gamma}Q_\epsilon(x)^{\frac{\beta}{2}-\frac{1}{\gamma}}}\\
\leq&\rho^2 e^{-2\Gamma Q_\epsilon^{\frac{\beta}{4}+\frac{3}{4\gamma}}(x)\left(\frac{3}{2}-\frac{C_{\beta,\epsilon}^\frac{1}{5\gamma}}{2\Gamma}Q_\epsilon^{\frac{\beta}{2}-\frac{1}{\gamma}-\frac{\beta}{4}-\frac{3}{4\gamma}}(x)\right)}\\
= & \rho^2 e^{-2\Gamma Q_\epsilon^{\frac{\beta}{4}+\frac{3}{4\gamma}}(x)\left(\frac{3}{2}-\frac{C_{\beta,\epsilon}^\frac{1}{5\gamma}}{2\Gamma}Q_\epsilon^{\frac{\beta}{4}+\frac{1}{4\gamma}}(x)\right)}
\leq\rho^2\Big(e^{-\Gamma Q_\epsilon(x)^{\frac{\beta}{4}+\frac{3}{4\gamma}}}\Big)^2.
\end{align*}
and $\frac{S^2(f^{-1}(q),\xi)}{S^2(f^{-1}(y),\eta)}\geq\rho^{-2}\Big(e^{\Gamma Q_\epsilon(x)^{\frac{\beta}{4}+\frac{3}{4\gamma}}}\Big)^2$ for small enough $\epsilon$. This concludes the proof.


\end{proof}

\begin{lemma}\label{chainbounds}
The following holds for any $\epsilon$ small enough. For any two 
 chains $(\psi_{x_i}^{p^s_i,p^u_i})_{i\in\mathbb{Z}},(\psi_{y_i}^{q^s_i,q^u_i})_{i\in\mathbb{Z}}$ with each having some symbol repeating infinitely often in the future, if $\pi[(\psi_{x_i}^{p^s_i,p^u_i})_{i\in\mathbb{Z}}]=\pi[(\psi_{y_i}^{q^s_i,q^u_i})_{i\in\mathbb{Z}}]=p$ then $\forall k\in\mathbb{Z}$, and for any $\xi\in T_{f^k(p)}V^s_k$,
$$\frac{S(x_k,\pi_{x,k}\xi)}{S(y_k,\pi_{y,k}\xi)}= e^{\pm 2\Gamma Q_\epsilon^{\min\{\frac{\beta}{2}-\frac{1}{\gamma},
\frac{\beta}{4}-\frac{1}{4\gamma}\}
}(x_k)},$$
 where $\pi_{x,k},\pi_{y,k}$ denote respectively the maps $\pi_{x_k}:T_{f^k(p)}V^s_k\rightarrow H^s(x_k),\pi_{y_k}:T_{f^k(p)}U^s_k\rightarrow H^s(y_k)$ as in Lemma \ref{boundimprove},
and $V^s_k=V^s((\psi_{x_i}^{p_i^s,p_i^u})_{i\geq k}),U^s_k=V^s((\psi_{y_i}^{q_i^s,q_i^u})_{i\geq k})$.
A similar statement holds for the unstable manifold.
\end{lemma}
\noindent\textbf{Remark}: Notice that $T_{f^k(p)}V^s_k=T_{f^k(p)}U^s_k$, which makes the choice and use of $\xi$ well defined. 
This can be seen by the facts that the orbit $(f^k(p))_{k\in\mathbb{Z}}$ stays on the intersection of stable and unstable manifolds, and Proposition \ref{Lambda}(2) gives lower bounds for the contraction/expansion by the differential on the respective manifolds' tangent spaces. Thus, the 
stable space of $f^k(p)$ (which exists by Proposition \ref{Lambda}(2))
coincides with $T_{f^k(p)}V^s_k$, and $T_{f^k(p)}V^s_k=\text{stable space of }f^k(p)=T_{f^k(p)}U^s_k$.

\begin{proof} Denote $\vec{v}=(\psi_{x_i}^{p^s_i,p^u_i})_{i\in\mathbb{Z}}$ and $\vec{u}=(\psi_{y_i}^{q^s_i,q^u_i})_{i\in\mathbb{Z}}$. $V^s_k$ and $U^s_k$ are $s$-admissible manifolds which stay in windows. We will prove that:
\begin{equation}\label{thatexpression}\forall\text{ } \xi\in T_{f^k(p)}V^s_k=T_{f^k(p)}U^s_k:\text{ }\frac{S(f^k(p),\xi)}{S(x_k,\pi_{x,k}\xi)},\frac{S(f^k(p),\xi)}{S(y_k,\pi_{y,k}\xi)}= e^{\pm \Gamma Q_\epsilon^{\min\{\frac{\beta}{2}-\frac{1}{\gamma},
\frac{\beta}{4}-\frac{1}{4\gamma}\}
}(x_k)}.
\end{equation}
This is sufficient, since 
$$\frac{S(x_k,\pi_{x,k}\xi)}{S(y_k,\pi_{y,k}\xi)}=\frac{S(x_k,\pi_{x,k}\xi)}{S(f^k(p),\xi)}\cdot\frac{S(f^k(p),\xi)}{S(y_k,\pi_{y,k}\xi)}
.$$
We show $\frac{S(f^k(p),\xi)}{S(x_k,\pi_{x,k}\xi)}= e^{\pm \Gamma Q_\epsilon^{\min\{\frac{\beta}{2}-\frac{1}{\gamma},
\frac{\beta}{4}-\frac{1}{4\gamma}\}
}(x_k)} $; the case of $\frac{S(f^k(p),\xi)}{S(y_k,\pi_{y,k}\xi)}$ is identical. By assumption, there exists a double chart $v$ and a sequence $n_k\uparrow \infty$ s.t $v_{n_k}=v$ for all $k$. Write $v=\psi_x^{p^s,p^u}$.

\medskip
For all $k\geq0$, $f^{n_k}(p)\in V^s(\sigma^{n_k}\underline{v})$, and by Corollary \ref{RegularOnVs}, $\exists C_{v}>0$ s.t $S(f^{n_k}(p))\leq C_{v}$. Set $$k_v:=\lceil\frac{\log C_v}{\Gamma Q_\epsilon^{\frac{\beta}{4}+\frac{3}{4\gamma}}(x)}\rceil.$$

Let $\xi\in T_{f^{n_{k_v}}(p)}V^s(\sigma^{n_{k_v}}\underline{v})(1)$. Pulling back $n_{k_v}$ times, Lemma \ref{boundimprove} tells us that for every $j\in\{0,\ldots, n_{k_v}\}$ either $\frac{S(f^j(p),d_{f^{n_{k_v}}}f^{-(n_{k_v}-j)}\xi)}{S(x_j,\pi_{x,j} d_{f^{n_{k_v}}}f^{-(n_{k_v}-j)}\xi)}=e^{\pm\Gamma Q_{\epsilon}^{\frac{\beta}{4}-\frac{1}{4\gamma}}(x_j)}$, or the ratio improves by $e^{-\Gamma Q_\epsilon^{\frac{\beta}{4}+\frac{3}{4\gamma}}(x)}$. However, the ratio starts from a quantity bounded by $C_v$. It follow that there must be $j\in\{0,\ldots,n_{k_v}\}$ s.t 
$$\frac{S(f^j(p),d_{f^{n_{k_v}}}f^{-(n_{k_v}-j)}\xi)}{S(x_j,\pi_{x,j} d_{f^{n_{k_v}}}f^{-(n_{k_v}-j)}\xi)}=e^{\pm\Gamma Q_{\epsilon}^{\frac{\beta}{4}-\frac{1}{4\gamma}}(x_j)}.$$

With every further iteration, the ratio can only improve, or remain bounded by $e^{\pm C_{\beta,\epsilon}^\frac{1}{2\gamma}Q_\epsilon^{\frac{\beta}{2}-\frac{1}{\gamma}}(x_j)}$ (recall \eqref{ferligator}). This holds for every $\xi\in T_{f^{n_{k_v}}(p)}V^s(\sigma^{n_{k_v}}\underline{v})(1)$, and so $\forall \xi'\in T_pV^s(\underline v)(1)$
$$\frac{S(p,\xi')}{S(x_0,\pi_{x,0}\xi')}= e^{\pm\max\{C_{\beta,\epsilon}^\frac{1}{2\gamma}Q_\epsilon^{\frac{\beta}{2}-\frac{1}{\gamma}}(x_0),\Gamma Q_{\epsilon}^{\frac{\beta}{4}-\frac{1}{4\gamma}}(x_0)\}}= e^{\pm\Gamma Q_\epsilon^{\min\{\frac{\beta}{2}-\frac{1}{\gamma},\frac{\beta}{4}-\frac{1}{4\gamma}\}}(x_0)}.$$

This proves the lemma for $k=0$. For other $k$, apply this to $(\psi_{x_{i+k}}^{p^s_{i+k},p^u_{i+k}})_{i\in\mathbb{Z}}$,$(\psi_{y_{i+k}}^{q^s_{i+k},q^u_{i+k}})_{i\in\mathbb{Z}}$.

\end{proof}

\noindent
\textbf{Remark:} The scheme of the proof only required that we have a uniform bound on the $S(\cdot)$ parameter of a stable leaf infinitely often in the future of the chain. The scheme works similarly for chains which may not have a symbol recurring infinitely often, but lie within an irreducible component of $\Sigma$. The principle which allows for it is replacing the role of the recurring symbol with a concatenation arbitrarily far away with a fixed periodic word (which is possible by irreducibility).

\medskip
Our next task is to show that if $\pi((\psi_{x_{i}}^{p^s_{i},p^u_{i}})_{i\in\mathbb{Z}})=\pi((\psi_{y_{i}}^{q^s_{i},q^u_{i}})_{i\in\mathbb{Z}})$ then $C_0(x_i)^{-1}:T_{x_i}M\rightarrow\mathbb{R}^d,C_0(y_i)^{-1}:T_{y_i}M\rightarrow\mathbb{R}^d$ are ``approximately the same." There are two issues in expressing this formally:
\begin{enumerate}[label=(\alph*)]
\item $C_0(x_i)^{-1},C_0(y_i)^{-1}$ are linear maps with different domains.
\item $C_0(\cdot)$ is only determined up to orthogonal self maps of $H^s(\cdot),H^u(\cdot)$. In the two-dimensional case this means choosing a number from $\{\pm1\}$, which commutes with composition of linear maps, but we wish to deal with the higher dimensional case, which is more subtle because of non-commutativity.
\end{enumerate}
We address these issues in the following proposition.
\begin{prop}\label{Xi} Under the assumptions of the previous lemma:
$$\exists\Xi_i:T_{x_i}M\rightarrow T_{y_i}M\text{ s.t }\forall\xi\in T_{x_i}M(1):\frac{|C_0^{-1}(x_i)\xi|}{|C_0^{-1}(y_i)\Xi_i\xi|}=e^{\pm2\Gamma Q_\epsilon^{\min\{\frac{\beta}{2}-\frac{1}{\gamma},
\frac{\beta}{4}-\frac{1}{4\gamma}\}
}(x_i)},$$
where $\Xi_i$ is an invertible linear transformation, and $\|\Xi_i\|,\|\Xi_i^{-1}\|\leq e^{Q_\epsilon^{\frac{\beta}{2}-\frac{1}{\gamma}}(x_i)}$.
\end{prop}
\begin{proof} 

Denote the following operators as defined in Lemma \ref{boundimprove}
: $\pi_{x,k}^s:T_{f^k(p)}V_k^s\rightarrow H^s(x_k)$, $\pi_{x,k}^u:T_{f^k(p)}V^u_k\rightarrow H^u(x_k)$, $\pi_{y,k}^s:T_{f^k(p)}U_k^s\rightarrow H^s(y_k)$, $\pi_{y,k}^u:T_{f^k(p)}U_k^u\rightarrow H^u(x_k)$.
Using the notations from the previous lemma, define $\pi_{x,k}^s:=\pi_{x,k}:T_{f^k(p)}V^s_k\rightarrow H^s(x_k),\pi_{y,k}^s:=\pi_{y,k}:T_{f^k(p)}U^s_k\rightarrow H^s(y_k)$. Let $V^u_k:=V^u((\psi_{x_i}^{p_i^s,p_i^u})_{i\leq k}),U^u_k:=V^u((\psi_{y_i}^{q_i^s,q_i^u})_{i\leq k})$, and define $\pi_{x,k}^u:=\pi_{x,k}:T_{f^k(p)}V^u_k\rightarrow H^u(x_k),\pi_{y,k}^u:=\pi_{y,k}:T_{f^k(p)}U^u_k\rightarrow H^u(y_k)$ as in the unstable case of Lemma \ref{boundimprove}. Define $\Xi_k^s:H^s(x_k)\rightarrow H^s(y_k)$, $\Xi_k^u:H^u(x_k)\rightarrow H^u(y_k)$ the following way:
$$\Xi_k^s:=\pi_{y,k}^s\circ
(\pi_{x,k}^s)^{-1}\text{ , }\Xi_k^u:=\pi_{y,k}^u\circ
(\pi_{x,k}^u)^{-1}.$$
The composition of $\pi_{y,k}^{s/u}\circ(\pi_{x,k}^{s/u})^{-1}$ is well defined by the remark after Lemma \ref{chainbounds}. Define $\Xi_k:T_{x_k}M\rightarrow T_{y_k}M$:
\begin{equation}\label{defxi}\forall \xi\in T_{x_k}M\text{:  }\Xi_k\xi:=\Xi_k^s\xi_s+\Xi_k^u\xi_u\text{, where }\xi=\xi_s+\xi_u, \xi_{s/u}\in H^{s/u}(x_k).\end{equation}
This is well defined since $T_{x_k}M=H^s(x_k)\oplus H^u(x_k)$. Notice that $\Xi_k[H^{s/u}(x_k)]=H^{s/u}(y_k)$. 
Hence, by Lemma \ref{chainbounds}:
\begin{equation}\label{shlukit}
\frac{|C_0^{-1}(x_k)\xi|}{|C_0^{-1}(y_k)\Xi_k\xi|}=\sqrt{\frac{S^2(x_k,\xi_s)+U^2(x_k,\xi_u)}{S^2(y_k,\Xi_k\xi_s)+U^2(y_k,\Xi_k\xi_u)}}=e^{\pm2\Gamma Q_\epsilon^{\min\{\frac{\beta}{2}-\frac{1}{\gamma},
\frac{\beta}{4}-\frac{1}{4\gamma}\}
}(x_k)}  .  
\end{equation}
It remains to bound $\|\Xi_k\|$ and $\|\Xi_k^{-1}\|$. We begin by showing bounds on the norms of $\Xi_k^{s/u}$. By the bounds for  $\|\pi_{x,k}^{s/u}\|$,$\|(\pi_{x,k}^{s/u})^{-1}\|$, $\|\pi_{y,k}^{s/u}\|$,$\|(\pi_{y,k}^{s/u})^{-1}\|$
from Lemma \ref{boundimprove}
,
\begin{equation}\label{shluk}
\|\Xi_k^{s/u}\|\leq \exp\Big( Q_\epsilon(x_k)^{\frac{\beta}{2}-\frac{1}{2\gamma}}+
Q_\epsilon(y_k)^{\frac{\beta}{2}-\frac{1}{2\gamma}}
\Big)
.
\end{equation}

We continue to bound $\|\Xi_k\|$. To ease notations we will omit the '$k$' subscripts. Recall $\rho(M)$ from \textsection\ref{forrho}. By Lemma \ref{firstchapter2}, $d(x,f^k(p))<\frac{\sqrt{d}}{100}Q_\epsilon(x)<\rho(M)$ for $\epsilon$ small enough.  Similarly, $d(y,f^k(p))<\rho(M)$. Therefore the following is well defined: $\psi_x^{-1}(f^k(p))=:z_x'$, $\psi_y^{-1}(f^k(p))=:z_y'$. The vectors of the first $s(x)$ coordinates of $z_{x}'/z_y'$ will be called $z_{x}/z_y$ respectively.

\medskip
$\textit{Part 1}$: Let $\xi\in T_xM$, $\xi=\xi_s+\xi_u$, $\xi_{s/u}\in H^{s/u}(x)$, then $|\xi_s|,|\xi_{u}|\leq \|C_0^{-1}(x)\|\cdot|\xi|$.

\textit{Proof}: w.l.o.g $|\xi|=1$. Notice that the size of $\xi_{s/u}$ can be very big, even when $|\xi_s+\xi_u|=|\xi|=1$. This can happen when 
$\sphericalangle(H^s(x),H^u(x)):=\inf\limits_{\eta_s\in H^s(x),\eta_u\in H^u(x)}|\sphericalangle(\eta_s,\eta_u)|$ is very small, and $\xi_s$ and $\xi_u$ are almost parallel, of the same size, and are pointing to almost opposite directions. Consider the triangle created by the tips of $\xi_s$ and $-\xi_u$, and the origin. Denote its angle at the origin by $\alpha$ (the angle between $\xi_s$ and $-\xi_u$). The size of the edge in front of $\alpha$ has length 1 (since $|\xi_s-(-\xi_u)|=|\xi_s+\xi_u|=1$). Angles between two non-parallel vectors are in $(0,\pi)$, and their sine is always positive. By the sine theorem,
\begin{equation}\label{darksalt}
|\xi_s|,|\xi_u|\leq \frac{1}{|\sin\alpha|}\leq\frac{1}{\inf\limits_{\eta_s\in H^s(x),\eta_u\in H^u(x)}|\sin\sphericalangle(\eta_s,\eta_u)|}=\frac{1}{\sin\sphericalangle(H^s(x),H^u(x))}.
\end{equation}
Notice that $\frac{1}{|\sin\sphericalangle(\eta_s,\eta_u)|}\leq\|C_0^{-1}(x)\|$, $\forall\eta_{s/u}\in H^{s/u}(x)$. To see this, notice first the identity
$$\|C_0^{-1}(x)\|^2=\sup_{\eta_s\in H^s(x),\eta_u\in H^u(x),|\eta_s+\eta_u|=1}\{S^2(x,\eta_s)+U^2(x,\eta_u)\}.$$

Choose $\eta_s\in H^s(x)(1),\eta_u\in H^u(x)(1)$ which minimize $\sin\sphericalangle(\eta_s,\eta_u)$. w.l.o.g $\cos\sphericalangle(\eta_s,\eta_u)>0$. Let $\zeta:=\frac{\eta_s-\eta_u}{|\eta_s-\eta_u|}$. Since $C_0^{-1}(x)\eta_s\perp C_0^{-1}(x)\eta_u$ and $\|C_0(x)\|\leq1$,
$$|C_0^{-1}(x)\zeta|^2\geq \frac{|\eta_s|^2}{|\eta_s-\eta_u|^2}+\frac{|\eta_u|^2}{|\eta_s-\eta_u|^2}=\frac{2}{2-2\cos\sphericalangle(\eta_s,\eta_u))}\geq\frac{1}{1-\cos^2\sphericalangle(\eta_s,\eta_u)}=\frac{1}{\sin^2\sphericalangle(\eta_s,\eta_u)}.$$
So $\|C_0^{-1}(x)\|\geq\frac{1}{|\sin\sphericalangle(\eta_s,\eta_u)|}$. By \eqref{darksalt}, $|\xi_s|,|\xi_u|\leq\|C_0^{-1}(x)\|$.

\medskip
\noindent QED

$\medskip$
$\textit{Part 2}$: $\Xi=d_{C_0(x)z_x'}(\exp_y^{-1}\exp_x)+E_1$, where $E_1:T_xM\rightarrow T_yM$ is linear and $\|E_1\|\leq \frac{1}{2}Q_\epsilon(x)^{\frac{\beta}{2}-\frac{1}{\gamma}}$.

\textit{Proof}: We will begin by showing that the expression above is well defined for $\epsilon<\frac{\rho(M)}{2}$. This is true because $\exp_x[B_\epsilon(0)]\subset B_{\epsilon+d(x,y)}(y)\subset B_{2\epsilon}(y)$; $|C_0(x)z_x'|=d(x,f^k(p))<\epsilon$; so $\exp_y^{-1}\exp_x$ is defined on an open neighborhood of $C_0(x)z_x'$, and hence is differentiable on it.

By definition, $\Xi^s=\pi_{y,k}^s\circ(\pi_{x,k}^s)^{-1}$.
\begin{itemize}
\item $\xi_s\in H^s(x)$ is mapped by $(\pi_{x,k}^s)^{-1}$ to $\xi_s'$, a tangent vector at $f^k(p)$
given by $\xi_s'=d_{z_x'}\psi_x\Big(\begin{array}{c}
v\\
d_{z_x}Gv\\ \end{array}\Big)$, where $G$ represents $V^s_k$ in $\psi_{x_k}^{p_k^s,p^u_k}$ and $v=C_0^{-1}(x)\xi_s$. We write $\xi_s'$ as follows:
$$\xi_s'=d_{z_x'}\psi_x\Big(\begin{array}{c}
v\\
d_{z_x}Gv\\ \end{array}\Big)=(d_{z_y'}\psi_y)(d_{z_y'}\psi_y)^{-1}d_{z_x'}\psi_x\Big(\begin{array}{c}
v\\
d_{z_x}Gv\\ \end{array}\Big)=:d_{z_y'}\psi_y\Big(\begin{array}{c}
w\\
d_{z_y}Fw\\ \end{array}\Big),$$
for some $w$, where $F$ is the representing function of $U^s_k$ in $\psi_{y_k}^{q^s_k,q^u_k}$. This representation is possible because $\xi_s'$ is tangent to the stable manifold $U^s_k$ in $\psi_{y_k}^{q^s_k,q^u_k}$ at $f^k(p)=\psi_y(z_y')$.
\item $\xi_s'=(\psi_{x,k}^s)^{-1}\xi_s$ is mapped by $\pi_{y,k}^s$ to
$$\Xi^s\xi_s=\pi_{y,k}^s\xi_s'=d_0\psi_y\Big(\begin{array}{c}
w\\
0\\ \end{array}\Big).$$
\end{itemize}

By construction,
$$\Big(\begin{array}{c}
w\\
d_{z_y}Fw\\ \end{array}\Big)=(d_{z_y'}\psi_y)^{-1}(d_{z_x'}\psi_x)\Big(\begin{array}{c}
v\\
d_{z_x}Gv\\ \end{array}\Big).$$

\noindent We abuse notation and use $w,v$ to represent both vectors in $\mathbb{R}^{s(x)}\times \{0\}
$ and vectors in $\mathbb{R}^{s(x)}
$. Then, $$w=(d_{z_y'}\psi_y)^{-1}d_{z_x'}\psi_xv+\Big[ (d_{z_y'}\psi_y)^{-1}d_{z_x'}\psi_x\Big(\begin{array}{c}
0\\
d_{z_x}Gv\\ \end{array}\Big)-\Big(\begin{array}{c}
0\\
d_{z_y}Fw\\ \end{array}\Big)\Big].$$
Call the term in the brackets $E_s\xi_s$, then
$\Xi_s=C_\chi(y)(d_{z_y'}\psi_y)^{-1}d_{z_x'}\psi_xC_0^{-1}(x)+C_0(y)E_s$.
\begin{itemize}
  \item  $\Xi^s\xi_s=d_0\psi_y\Big(\begin{array}{c}
w\\
0\\ \end{array}\Big)\Rightarrow |w|=|(d_0\psi_y)^{-1}\Xi^s\xi_s|=|C_0^{-1}(y)\Xi^s \xi_s|\leq\|C_0^{-1}(y)\|\cdot\|\Xi^s\|\cdot|\xi_s|\leq$

$\leq\|C_0^{-1}(y)\|\cdot e^{Q_\epsilon(x_k)^{\frac{\beta}{2}-\frac{1}{2\gamma}}+
Q_\epsilon(y_k)^{\frac{\beta}{2}-\frac{1}{2\gamma}}}|\xi_s|\leq 2\|C_0^{-1}(y)\|\cdot |\xi_s|(\because \text{ \eqref{shluk} })$.
\item $\mathrm{Lip}(G)\leq Q_\epsilon(x)^\frac{\beta}{2},\mathrm{Lip}(F)\leq Q_\epsilon(y)^\frac{\beta}{2}$ (by the remark after Definition \ref{admissible}).
\item $\xi_s=C_\chi(x)v\Rightarrow|v|\leq \|C_\chi^{-1}(x)\|\cdot |\xi_s|$.
\end{itemize} Hence:

\begin{align}\label{Es}
|C_0(y)E_s\xi_s|\leq&\|d_{z_y'}\exp_y^{-1}\|\cdot\|d_{z_x'}\psi_x\|\cdot\|d_{z_x}G\|\cdot |v|+\|d_{z_y}F\|\cdot|w|\nonumber\\
\leq&2\cdot2Q_\epsilon(x)^\frac{\beta}{2}\cdot\|C_0^{-1}(x)\|\cdot|\xi_s|+Q_\epsilon(y)^\frac{\beta}{2}\cdot2\|C_0^{-1}(y)\|\cdot|\xi_s|\nonumber\\
\leq&(Q_\epsilon(x)^{\frac{\beta}{2}-\frac{1}{2\gamma}}+Q_\epsilon(y)^{\frac{\beta}{2}-\frac{1}{2\gamma}})|\xi_s|.
\end{align}
Similarly, $\Xi_u=C_0(y)(d_{z_y'}\psi_y)^{-1}d_{z_x'}\psi_xC_0^{-1}(x)+C_0(y)E_u$, where $E_u:H^u(x)\rightarrow T_yM$ and 
\begin{equation}\label{Eu}
\|C_0(y)E_u\|
\leq (Q_\epsilon(x)^{\frac{\beta}{2}-\frac{1}{2\gamma}}+Q_\epsilon(y)^{\frac{\beta}{2}-\frac{1}{2\gamma}})
.
\end{equation}
For a general tangent $\xi\in T_xM$, write $\xi=\xi_s+\xi_u,\xi_{s/u}\in H^{s/u}(x)$, then

\begin{align*}\Xi\xi=&\Xi(\xi_s+\xi_u)=\Xi_s\xi_s+\Xi_u\xi_u \\
=&\Big(C_0(y)(d_{z_y'}\psi_y)^{-1}d_{z_x'}\psi_xC_0^{-1}(x)+C_0(y)E_s\Big)\xi_s+\Big(C_0(y)(d_{z_y'}\psi_y)^{-1}d_{z_x'}\psi_xC_0^{-1}(x)+C_0(y)E_u\Big)\xi_u\\
=&C_0(y)(d_{z_y'}\psi_y)^{-1}d_{z_x'}\psi_xC_0^{-1}(x)(\xi_s+\xi_u)+C_0(y)(E_s\xi_s+E_u\xi_u)\\
=&C_0(y)(d_{z_y'}\psi_y)^{-1}d_{z_x'}\psi_xC_0^{-1}(x)\xi+C_0(y)(E_s\xi_s+E_u\xi_u).
\end{align*}

Hence, since $
d_{z'_{x}}\psi_{x}=d_{C_0(x)z_{x}'}\exp_{x}\circ C_0(x),d_{z'_{y}}\psi_{y}=d_{C_0(y)z_{y}'}\exp_{y}\circ C_0(y)$, we get
$$\Xi\xi=d_{f^k(p)}\exp_y^{-1}d_{C_0(x)z_x'}\exp_x\xi+C_0(y)(E_s\xi_s+E_u\xi_u)
\equiv d_{C_0(x)z_x'}(\exp_y^{-1}\exp_x)\xi+E_1\xi,$$
where
\begin{align*}|E_1\xi|\leq& |C_0(y)E_s\xi_s|+|C_0(y)E_u\xi_u|\leq((Q_\epsilon(x)^{\frac{\beta}{2}-\frac{1}{2\gamma}}+Q_\epsilon(y)^{\frac{\beta}{2}-\frac{1}{2\gamma}})|)(|\xi_s|+|\xi_u|)\text{ } (\because\text{ \eqref{Es},\eqref{Eu}})\\
\leq& ((Q_\epsilon(x)^{\frac{\beta}{2}-\frac{1}{2\gamma}}+Q_\epsilon(y)^{\frac{\beta}{2}-\frac{1}{2\gamma}})|)(\|C_0^{-1}(x)\|\cdot|\xi|+\|C_0^{-1}(x)\|\cdot|\xi|)\text{ }(\because\text{Part 1})\\
\leq&2\|C_0^{-1}(x)\|(Q_\epsilon(x)^{\frac{\beta}{2}-\frac{1}{2\gamma}}+Q_\epsilon(y)^{\frac{\beta}{2}-\frac{1}{2\gamma}})|\xi|.\end{align*}
Therefore,
$$\Xi=d_{C_0(x)z_x'}(\exp_y^{-1}\exp_x)+E_1\Rightarrow\|\Xi\|\leq \|d_\cdot\exp_y^{-1}\|\cdot\|d_\cdot\exp_x\|+\|E_1\|\leq4+2\|C_0^{-1}(x)\|(Q_\epsilon(x)^{\frac{\beta}{2}-\frac{1}{2\gamma}}+Q_\epsilon(y)^{\frac{\beta}{2}-\frac{1}{2\gamma}}).$$
Now using this again in \eqref{shlukit}:
\begin{align*}
\|C_0^{-1}(x)\|\leq& e^{2\Gamma Q_\epsilon^{\min\{\frac{\beta}{2}-\frac{1}{\gamma},
\frac{\beta}{4}-\frac{1}{4\gamma}\}
}(x_k)}\|C_0^{-1}(y)\Xi\|\leq e^{2\Gamma Q_\epsilon^{\min\{\frac{\beta}{2}-\frac{1}{\gamma},
\frac{\beta}{4}-\frac{1}{4\gamma}\}
}(x_k)}\|C_0^{-1}(y)\|\cdot\|\Xi\|\\
\leq& e^{2\Gamma Q_\epsilon^{\min\{\frac{\beta}{2}-\frac{1}{\gamma},
\frac{\beta}{4}-\frac{1}{4\gamma}\}
}(x_k)}\|C_0^{-1}(y)\|\Big(4+2\|C_0^{-1}(x)\|(Q_\epsilon(x)^{\frac{\beta}{2}-\frac{1}{2\gamma}}+Q_\epsilon(y)^{\frac{\beta}{2}-\frac{1}{2\gamma}})\Big)\\
=& \|C_0^{-1}(y)\|\cdot e^{2\Gamma Q_\epsilon^{\min\{\frac{\beta}{2}-\frac{1}{\gamma},
\frac{\beta}{4}-\frac{1}{4\gamma}\}
}(x_k)}\Big(4+2\|C_0^{-1}(x)\|Q_\epsilon(x)^{\frac{\beta}{2}-\frac{1}{2\gamma}}\Big)\\
+&\|C_0^{-1}(x)\|\cdot2e^{2\Gamma Q_\epsilon^{\min\{\frac{\beta}{2}-\frac{1}{\gamma},
\frac{\beta}{4}-\frac{1}{4\gamma}\}
}(x_k)}\Big(\|C_0^{-1}(y)\|Q_\epsilon(y)^{\frac{\beta}{2}-\frac{1}{2\gamma}}\Big)\\
\leq&5(1-\epsilon)\|C_0^{-1}(y)\|+\epsilon\|C_0^{-1}(x)\|
\Rightarrow (1-\epsilon)\|C_0^{-1}(x)\|\leq (1-\epsilon)5\|C_0^{-1}(y)\|\Rightarrow\|C_0^{-1}(x)\|\leq 5\|C_0^{-1}(y)\|.
\end{align*}
By symmetry we will also get $\|C_0^{-1}(y)\|\leq 5\|C_0^{-1}(x)\|$. In particular, $\frac{Q_\epsilon(y)}{Q_\epsilon(x)}=5^{2\gamma}$. Substituting these in the bounds for $\|E_1\|$, for all $\epsilon$ small enough:
\begin{align*}
\|E_1\|\leq&2\|C_0^{-1}(x)\|Q_\epsilon(x)^{\frac{\beta}{2}-\frac{1}{2\gamma}}+2\|C_0^{-1}(x)\|Q_\epsilon(y)^{\frac{\beta}{2}-\frac{1}{2\gamma}}\\
\leq&2\|C_0^{-1}(x)\|Q_\epsilon(x)^{\frac{\beta}{2}-\frac{1}{2\gamma}}+2\cdot(5^{2\gamma})^{\frac{\beta}{2}-\frac{1}{2\gamma}}\|C_0^{-1}(x)\|Q_\epsilon(x)^{\frac{\beta}{2}-\frac{1}{2\gamma}}\leq \frac{1}{2}Q_\epsilon(x)^{\frac{\beta}{2}-\frac{1}{\gamma}}.	
\end{align*}
QED

\medskip
$\textit{Part 3}$: $\forall u\in R_\epsilon(0)$: $\exists D\in \mathcal{D}$ s.t $\|\Theta_D d_{C_0(x)u}(\exp_y^{-1}\exp_x)-\Theta_D\|\leq Q_\epsilon(x)$ (see Definition \ref{isometries}).

\textit{Proof}: Begin by choosing $\epsilon$ smaller than $\frac{\varpi(\mathcal{D})}{2\sqrt{d}}$ (as in Definition \ref{isometries}). Then $\exists D\in\mathcal{D}$ such that $\exp_x[R_\epsilon(0)]\subset B_{\frac{\varpi(\mathcal{D})}{2}}(x)$, because $d(x,\exp_xC_0(x)v)=|C_0(x)v|_2\leq\sqrt{d}|v|_\infty<\epsilon\sqrt{d}$, so $\exp_x[R_\epsilon(0)]$ has diameter less than $2\sqrt{d}\epsilon<\varpi(\mathcal{D})$. Hence the choice of $D$ and of $\Theta_D$ is proper. Since

$$\Theta_D\exp_y^{-1}\exp_x=(\Theta_D\exp_y^{-1}-\Theta_D\exp_x^{-1})\exp_x+\Theta_D,$$
we get that for any $u\in R_\epsilon(0)$:
$$\Theta_Dd_{C_0(x)u}(\exp_y^{-1}\exp_x)=d_{C_0(x)u}(\Theta_D\exp_y^{-1}\exp_x)=d_{\exp_xC_0(x)u}(\Theta_D\exp_y^{-1}-\Theta_D\exp_x^{-1})d_{C_0(x)u}\exp_x+\Theta_D.$$
Then
$$\|\Theta_Dd_{C_0(x)u}(\exp_y^{-1}\exp_x)-\Theta_D\|\leq \|\Theta_D\exp_y^{-1}-\Theta_D\exp_x^{-1}\|_{C^2}\cdot \|d_{C_0(x)u}\exp_x\|\leq L_2d(x,y)\|d_{C_0(x)u}\exp_x\|,$$
where $L_2$ is the uniform Lipschitz const. of $x\mapsto\nu_x^{-1}\exp_x^{-1}$ (Proposition \ref{chartsofboxes}). This is due to the fact that $d(x,\exp_xC_0(x)u)<\sqrt{d}\epsilon\Rightarrow\exp_xC_0(x)u\in D$. In addition, $\|d_{C_0(x)u}\exp_x\|\leq\|Id_{T_xM}\|+E_0|C_0(x)u-0|\leq1+E_0\epsilon$ ($E_0$ is a constant introduced in Lemma \ref{boundimprove}). Therefore, in total,
\begin{equation}\label{explicitcalculation}\|\Theta_Dd_{C_0(x)u}(\exp_y^{-1}\exp_x)-\Theta_D\|\leq L_2(1+\epsilon E_0)\cdot d(x,y)<Q_\epsilon(x) \text{ }\because\text{Lemma \ref{firstchapter2}}.
\end{equation}
QED

Adding the estimates of part 2 and part 3, we get 
\begin{align*}
\|\Xi\|=&\|\Theta_Dd_{C_0(x)v}(\exp_y^{-1}\exp_x)-\Theta_D+\Theta_D+\Theta_D E_1\|\\ 
\leq&\|\Theta_D d_{C_0(x)v}(\exp_y^{-1}\exp_x)-\Theta_D\|+\|\Theta_D\|+\|E_1\|\leq Q_\epsilon(x)+1+\frac{1}{2}Q_\epsilon^{\frac{\beta}{2}-\frac{1}{\gamma}}(x)\leq e^{Q_\epsilon^{\frac{\beta}{2}-\frac{1}{\gamma}}(x)}.
\end{align*}

If we exchange the roles of $x_k,y_k$ then we get $\Xi_k^{-1}$, so $\|\Xi_k^{-1}\|$ is also less than $e^{Q_\epsilon^{\frac{\beta}{2}-\frac{1}{\gamma}}(x_k)}$.
\end{proof}

\begin{cor}\label{QforGamma} Under the assumptions and notations of the previous lemma:
$$\forall i\in\mathbb{Z}: \frac{\|C_0^{-1}(x_i)\|}{\|C_0^{-1}(y_i)\|}=e^{\pm(2\Gamma+1)Q_\epsilon(x_i)^{\min\{\frac{\beta}{2}-\frac{1}{\gamma},
\frac{\beta}{4}-\frac{1}{4\gamma}\}}},$$
and so
$$\frac{Q_\epsilon(x_i)}{Q_\epsilon(y_i)}=e^{\pm 2\gamma(2\Gamma+1)Q_\epsilon(x_i)^{\min\{\frac{\beta}{2}-\frac{1}{\gamma},
\frac{\beta}{4}-\frac{1}{4\gamma}\}}}.$$
\end{cor}
\begin{proof}
By Proposition \ref{Xi}:
\begin{align*}
\|C_0^{-1}(x_i)\|\leq& e^{2\Gamma Q_\epsilon(x_i)^{\min\{\frac{\beta}{2}-\frac{1}{\gamma},
\frac{\beta}{4}-\frac{1}{4\gamma}\}}}\|C_0^{-1}(y_i)\Xi_i\|\leq e^{(2\Gamma+1)Q_\epsilon(x_i)^{\min\{\frac{\beta}{2}-\frac{1}{\gamma},
\frac{\beta}{4}-\frac{1}{4\gamma}\}}}\|C_0^{-1}(y_i)\|,	
\end{align*}

and by symmetry we get the other inequality, and we are done.
\end{proof}

\subsubsection{Comparing frame parameters}\label{forGammaBounds}\text{ }

\noindent
\textbf{Remark:} Notice that if $\gamma>\frac{5}{\beta}$, and $\epsilon>0$ is sufficiently small, then Corollary \ref{QforGamma} in fact tell us
\begin{equation}\label{IofQ}Q_\epsilon(y_i)=I^{\pm1}(Q_\epsilon(x_i)),\text{ }i\in\mathbb{Z}.\end{equation}
It is important that the bound we get on the resemblance of $Q_\epsilon(x_i)$ and $Q_\epsilon(y_i)$ is of a form which commutes with the map $I$. This is the reason that it is important to require $\gamma>\frac{5}{\beta}$ to make sure the bounds of Corollary \ref{QforGamma} can be expressed in terms of $I$ ($\geq\frac{5}{\beta}$ is sufficient, if one carries the constants more carefully). To see where this commutativity relation is needed, see \eqref{forCommentNov21} below.

\begin{cor}\label{newLifeOfPi}
$\pi[\Sigma^\#]=\RST$.	
\end{cor}
\begin{proof}
In Theorem \ref{DefOfPi} we saw	$\pi[\Sigma^\#]\supseteq \RST$. By Corollary \ref{RegularOnVs}, for every $\underline{u}=(\psi_{x_i}^{p^s_i,p^u_i})_{\in\mathbb{Z}}\in\Sigma^\#$, $p:=\pi(\underline{u})\in0$-summ. $\{p^s_i\wedge p^u_i\}_{i\in\mathbb{Z}}$ is $I$-strongly subordinated to $\{Q_\epsilon(x_i)\}$ by Lemma \ref{lemma131}. By the proof of Proposition \ref{Xi}, the map $\pi_{x,i}:T_{f^i(p)}M\to T_{x_k}M$ satisfies $\|\pi_{x,i}\|,\|\pi_{x,i}^{-1}\|\leq e^{Q_\epsilon(x_i)^{\frac{\beta}{2}-\frac{1}{\gamma}}}$. By the proof of Lemma \ref{chainbounds}, $\frac{\|C_0^{-1}(f^i(p))\|}{\|C_0^{-1}(x_i)\|}=e^{\pm2\Gamma Q_\epsilon(x_i)^{\min\{\frac{\beta}{2}-\frac{1}{\gamma},\frac{\beta}{4}-\frac{1}{4\gamma}\}}}$. As in \eqref{IofQ}, we get $Q_\epsilon(f^i(p))=I^{\pm1}(Q_\epsilon(x_i))$. Then it follows that $$I^{-1}(p_i^s\wedge p_i^u)\leq I^{-1}(Q_\epsilon(x_i))\leq Q_\epsilon(f^i(x)).$$
Thus $p\in ST$ (with the strongly tempered kernel $I^{-1}(p_i^s\wedge p_i^u)$). Since $\limsup_{n\to\pm\infty}p_i^s\wedge p_i^u>0$, $p\in\RST$.
\end{proof}

\begin{definition}
(Compare with \cite[Definition~8.1]{Sarig})
\medskip
\begin{enumerate}
\item If $v$ is a double chart, then $p^{u/s}(v)$ means the $p^{u/s}$ in $\psi_x^{p^s,p^u}=v$.
\item A negative chain $(v_i)_{i\leq0}$ is called {\em $I$-maximal} if it has a symbol recurring infinitely often, and $$p^u(v_0)\geq I^{-1}(p^u(u_0))$$ for every regular chain $(u_i)_{i\in\mathbb{Z}}$ for which there is a positive regular chain $(v_i)_{i\geq0}$ s.t $\pi((u_i)_{i\in\mathbb{Z}})=\pi((v_i)_{i\in\mathbb{Z}})$.
\item A positive chain $(v_i)_{i\geq0}$ is called {\em $I$-maximal} if it has a symbol recurring infinitely often, and $$p^s(v_0)\geq I^{-1}(p^s(u_0))$$ for every chain $(u_i)_{i\in\mathbb{Z}}\in\Sigma^\#$ s.t 
$\pi((u_i)_{i\in\mathbb{Z}})=\pi((v_i)_{i\in\mathbb{Z}})$.
\end{enumerate}
\end{definition}

\begin{prop}\label{prop213}
For every chain $(v_i)_{i\in\mathbb{Z}}\in\Sigma^\#$, $(u_i)_{i\leq0}$ and $(u_i)_{i\geq0}$ are $I$-maximal.
\end{prop}
\begin{proof}\text{ } 

Step 1: Let $\underline{u},\underline{v}\in\Sigma^\#$ be two chains s.t $\pi(\underline u)=\pi(\underline v)$. If $u_0=\psi_x^{p^s,p^u}$ and $v_0=\psi_y^{q^s,q^u}$, then $Q_\epsilon(y)=I^{\pm1}(Q_\epsilon(x))$, by \eqref{IofQ}.


Step 2: Every negative chain $(v_i)_{i\leq0}$ with a symbol recurring infinitely often s.t $v_0=\psi_x^{p^s,p^u}$ where $p^u=Q_\epsilon(x)$ is $I$-maximal, and every positive chain with a symbol recurring infinitely often $(v_i)_{i\leq0}$ s.t $v_0=\psi_x^{p^s,p^u}$ where $p^s=Q_\epsilon(x)$ is $I$-maximal.

Proof: Let $(v_i)_{i\leq0}$ be a negative chain with a symbol recurring infinitely often, and assume $v_0=\psi_x^{p^s,p^u}$ where $p^u=Q_\epsilon(x)$. We show that it is $I$-maximal. Suppose $(v_i)_{i\in\mathbb{Z}}\in\Sigma^\#$ is an extension of $(v_i)_{i\leq0}$, and let $(u_i)_{i\in\mathbb{Z}}\in\Sigma^\#$ be a chain s.t $\pi((u_i)_{i\in\mathbb{Z}})=\pi((v_i)_{i\in\mathbb{Z}})$. Write $u_0=\psi_y^{q^s,q^u}$. Now, by step 1, $p^u=Q_\epsilon(x)\geq I^{-1}(Q_\epsilon(y))\geq I^{-1}(q^u)$. The second half of step 2 is shown similarly.

Step 3: Let $(v_i)_{i\leq0}$ be a negative chain with a symbol recurring infinitely often, and suppose $v_0\rightarrow v_1$. If $(v_i)_{i\leq0}$ is $I$-maximal, then also $(v_i)_{i\leq1}$ is $I$-maximal. Let $(v_i)_{i\geq0}$ be a positive chain with a symbol recurring infinitely often, and suppose $v_{-1}\rightarrow v_0$. If $(v_i)_{i\geq0}$ is $I$-maximal, then also $(v_i)_{i\geq-1}$ is $I$-maximal.

Proof: Let $(v_i)_{i\leq0}$ be an $I$-maximal negative chain with a symbol recurring infinitely often, and suppose $v_0\rightarrow v_1$. Suppose that $(u_i)_{i\in\mathbb{Z}}\in\Sigma^\#$ and that $(v_i)_{i\leq1}$ has a symbol repeating infinitely often, and that there is an extension $(v_i)_{i\leq1}$ to a chain $(v_i)_{i\in\mathbb{Z}}\in \Sigma^\#$ s.t $\pi((u_{i+1})_{i\in\mathbb{Z}})=\pi((v_{i+1})_{i\in\mathbb{Z}})$. We write $v_i=\psi_{x_i}^{p_i^s,p_i^u}$, and $u_i=\psi_{y_1}^{q_i^s,q_i^u}$; and we show that $p_1^u\geq I^{-1}(q_1^u)$. Since $\pi((u_{i+1})_{i\in\mathbb{Z}})=\pi((v_{i+1})_{i\in\mathbb{Z}})$ and $\pi\circ\sigma=f\circ\pi$: $$\pi((u_{i})_{i\in\mathbb{Z}})=\pi((v_{i})_{i\in\mathbb{Z}}).$$ Therefore, since $(v_i)_{i\leq0}$ is $I$-maximal, $p_0^u\geq I^{-1}(q_0^u)$. Also, by step 1: $Q_\epsilon(x_1)\geq I^{-1}(Q_\epsilon(y_1))$. It follows that 
\begin{align}\label{forCommentNov21}
p_1^u=&\min\{I(p_0^u),Q_\epsilon(x_1)\}\text{   }(\because v_0 \rightarrow v_1)\nonumber\\
\geq &\min\{ I(I^{-1}(q_0^u)),I^{-1}(Q_\epsilon(y_1))\}\nonumber\\
=& I^{-1}(\min\{I(q_0^u),Q_\epsilon(y_1)\})=I^{-1}(q_1^u)\text{   }(\because u_0\rightarrow u_1\text{ and }I^{-1}\circ I=I\circ I^{-1}).
\end{align}
This proves the case of step 3 dealing with negative chains. The proof for the case of positive chains is shown similarly.

Step 4: Proof of the proposition: Suppose $(v_i)_{i\in\mathbb{Z}}\in\Sigma^\#$, and write $v_i=\psi_{x_i}^{p_i^s,p_i^u}$. Since $(v_i)_{i_\mathbb{Z}}$ is a chain, we get that $\{(p_i^s,p_i^u)\}_{i\in\mathbb{Z}}$
is $I$-strongly subordinated to $\{Q_\epsilon(x_i)\}_{i\in\mathbb{Z}}$. Since $(v_i)_{i\in\mathbb{Z}}\in \Sigma^\#$, $\limsup_{i\rightarrow\infty}(p_i^s\wedge p_i^u)>0$. Therefore by Lemma \ref{forrecurrenceinnextone}, $p_n^u=Q_\epsilon(x_n)$ for some $n<0$, and $p_l^s=Q_\epsilon(x_l)$ for some $l>0$. By step 2, $(v_i)_{i\leq n}$ is an $I$-maximal negative chain, and $(v_i)_{i\geq l}$ is an $I$-maximal positive chain. By step 3: $(v_i)_{i\leq 0}$ is an $I$-maximal negative chain, and $(v_i)_{i\geq 0}$ is an $I$-maximal positive chain.
\end{proof}
\begin{lemma}\label{prop214}
Let $(\psi_{x_i}^{p_i^s,p_i^u})_{i\in\mathbb{Z}},(\psi_{y_i}^{q_i^s,q_i^u})_{i\in\mathbb{Z}}\in\Sigma^\#$ be two chains s.t $\pi((\psi_{x_i}^{p_i^s,p_i^u})_{i\in\mathbb{Z}})=\pi((\psi_{y_i}^{q_i^s,q_i^u})_{i\in\mathbb{Z}})=p$. Then for all $k\in\mathbb{Z}$, $p_k^u=I^{\pm1}(q_k^u),p_k^s=I^{\pm1}(q_k^s)$.
\end{lemma}
\begin{proof}
By Proposition \ref{prop213}, $(\psi_{x_i}^{p_i^s,p_i^u})_{i\leq0}$ is $I$-maximal, so $p_0^u\geq I^{-1}(q_0^u)$. $(\psi_{x_i}^{p_i^s,p_i^u})_{i\leq0}$ is also $I$-maximal, so $q_0^u\geq I^{-1}(p_0^u)$. It follows that $p_0^u=I^{\pm1}(q_0^u)$. Similarly $p_0^s=I^{\pm1}(q_0^s)$. Working with the shifted sequences $(\psi_{x_{i+k}}^{p_{i+k}^s,p_{i+k}^u})_{i\in\mathbb{Z}}$, and $(\psi_{y_{i+k}}^{q_{i+k}^s,q_{i+k}^u})_{i\in\mathbb{Z}}$, we obtain $p_k^u=I^{\pm1}(q_k^u),p_k^s=I^{\pm1}(q_k^s)$ for all $k\in\mathbb{Z}$.
\end{proof}

\subsubsection{Comparing Pesin charts}

\begin{theorem}\label{beforefinal}
The following holds for all $\epsilon$ small enough.  Suppose $(\psi_{x_i}^{p_i^s,p_i^u})_{i\in\mathbb{Z}},(\psi_{y_i}^{q_i^s,q_i^u})_{i\in\mathbb{Z}}\in\Sigma^\#$ are chains s.t $\pi((\psi_{x_i}^{p_i^s,p_i^u})_{i\in\mathbb{Z}})=p=\pi((\psi_{y_i}^{q_i^s,q_i^u})_{i\in\mathbb{Z}})$ then for all i:
\begin{enumerate}
\item $d(x_i,y_i)<\epsilon$,
\item $p_i^u=I^{\pm1}(q_i^u),p_i^s=I^{\pm1}(q_i^s)$,
\item $(\psi_{y_i}^{-1}\circ\psi_{x_i})(u)=O_iu+a_i+\Delta_i(u)$ for all $u\in R_\epsilon(0)$, where $O_i\in O(d)$ is a $d\times d$ orthogonal matrix, $a_i$ is a constant vector s.t $|a_i|_\infty\leq10^{-1}(q_i^u\wedge q_i^s)$, and $\Delta_i$ is a vector field s.t $\Delta_i(0)=0,\|d_v\Delta_i\|<\frac{1}{2}\epsilon^\frac{1}{3}$. $O_i$ preserves $\mathbb{R}^{s(x)}$ and $\mathbb{R}^{d-s(x)}$.\footnote{Recall, in our notations $\mathbb{R}^{s(x)}:=C_0^{-1}(x_i)[H^s(x_i)]=C_0^{-1}(y_i)[H^s(y_i)]\subset \mathbb{R}^d$ and $\mathbb{R}^{d-s(x)}:=(\mathbb{R}^{s(x)})^\perp=C_0^{-1}(x_i)[H^u(x_i)]=C_0^{-1}(y_i)[H^u(y_i)]$.}
\end{enumerate}
\end{theorem}
\begin{proof} Parts $(1)$ and $(2)$ are the content of Lemma \ref{firstchapter2} and Lemma \ref{prop214} respectively; Part $(3)$ is shown similarly to \cite[Theorem~4.13]{SBO}.
\end{proof}
\subsection{Similar charts have similar manifolds}
\subsubsection{Comparing manifolds}
\begin{prop}\label{prop221} The following holds for all $\epsilon$ small enough. Let $V^s$ (resp. $U^s$) be an $s$-admissible manifold in $\psi_x^{p^s,p^u}$ (resp. $\psi_y^{q^s,q^u}$). Suppose $V^s, U^s$ stay in windows. If $x=y$ then either $V^s,U^s$ are disjoint, or one contains the other. The same statement holds for $u$-admissible manifolds.
\end{prop}
Proof is similar to \cite[Proposition~4.15]{SBO}.

\section{Markov partitions and symbolic dynamics}\label{MarkoPolo}
In sections 1,2,3 of this paper we developed the shadowing theory needed to construct Markov partitions \`a la Bowen (\cite{B3,B4}), following the improvements of Sarig to the shadowing theory \cite{Sarig}, and its extension to higher dimensions \cite{SBO}, while carrying it in a setup where hyperbolicity is replaced by $0$-summability. What remains is to carry out this construction. This is the content of this part. As in \cite{SBO}, we follow \cite{B3} and \cite{Sarig} closely here.
\subsection{A locally finite countable Markov partition}
\subsubsection{The cover}
In \textsection \ref{epsilonchains} we constructed a countable Markov shift $\Sigma$ with countable alphabet $\mathcal{V}$ and a uniformly continuous map $\pi:\Sigma\rightarrow M$ which commutes with the left shift $\sigma:\Sigma\rightarrow\Sigma$, so that $\pi[\Sigma]\supseteq\RST$. Moreover, when
$$\Sigma^\#=\{u\in\Sigma: u \text{ is a regular chain}\}=\{u\in\Sigma: \exists v,w\in \mathcal{V}\exists n_k,m_k\uparrow\infty\text{ s.t   }u_{n_k}=v,u_{-m_k}=w\}$$
then $\pi[\Sigma^{\#}]\supset\RST$. However, $\pi$ is not finite-to-one. In this section we study the following countable cover of $\RST$:
\begin{definition}\label{ZV} $\mathcal{Z}:=\{Z(v):v\in\mathcal{V}\}$, where $Z(v):=\{\pi(u):u\in\Sigma^\#,u_0=v\}$.
\end{definition}
\begin{theorem}\label{Zlocallyfinite}
For every $Z\in\mathcal{Z}$, $|\{Z'\in\mathcal{Z}:Z'\cap Z\neq \varnothing\}|<\infty$.
\end{theorem}
\begin{proof} 
Fix some $Z=Z(\psi_x^{p^s,p^u})$. If $Z'=Z(\psi_y^{q^s,q^u})$ intersects $Z$ then there must exist two chains $v,w\in\Sigma^\#$ s.t $v_0=\psi_x^{p^s,p^u},w_0=\psi_y^{q^s,q^u}$ and $\pi(v)=\pi(w)$. Lemma \ref{prop214} says that in this case $q^u\geq I^{-1}(p^u),q^s\geq I^{-1}(p^s)$.
It follows that $Z'$ belongs to $\{Z(\psi_y^{q^s,q^u}):\psi_y^{q^s,q^u}\in\mathcal{V},q^u\wedge q^s\geq I^{-1}(p^u\wedge p^s)\}$. By definition, the cardinality of this set is less than or equal to:
$$|\{\psi_y^\eta\in\mathcal{A}:\eta\geq I^{-1}(p^s\wedge p^u)\}|\times|\{(q^s,q^u)\in \mathcal{I}\times \mathcal{I}: q^s\wedge q^u\geq I^{-1}(p^s\wedge p^u)\}|.$$
This is a finite number because of the discreteness of $\mathcal{A}$ (Proposition \ref{discreteness}).
\end{proof}
\subsubsection{Product structure}
Suppose $x\in Z(v)\in \mathcal{Z}$, then $\exists \underline{u}\in\Sigma^\#$ s.t $u_0=v$ and $\pi(\underline{u})=x$. Associated to $\underline{u}$ are two admissible manifolds in $v$: $V^u((u_i)_{i\leq 0})$ and $V^s((u_i)_{i\geq 0})$. These manifolds do not depend on the choice of $\underline{u}$: If $\underline{w}\in\Sigma^\#$ is another  chain s.t $w_0=v$ and $\pi(\underline{w})=x$ then $$V^u((u_i)_{i\leq 0})=V^u((w_i)_{i\leq 0})\text{ and }V^s((u_i)_{i\geq 0})=V^s((w_i)_{i\geq 0}),$$
because of Proposition \ref{prop221} and the equalities $p^{s/u}(w_0)=p^{s/u}(u_0)=p^{s/u}(v)$. We are therefore free to make the following definitions:
\begin{definition}
Suppose $Z=Z(v)\in\mathcal{Z}$. For any $x\in Z$:
\begin{enumerate}
    \item $V^u(x,Z):=V^u((u_i)_{i\leq 0})$ for some (any) $\underline{u}\in\Sigma^\#$ s.t $u_0=v$ and $\pi(\underline{u})=x$.
    
    We also set $W^u(x,Z):=V^u(x,Z)\cap Z$.
    \item $V^s(x,Z):=V^s((u_i)_{i\geq 0})$ for some (any) $\underline{u}\in\Sigma^\#$ s.t $u_0=v$ and $\pi(\underline{u})=x$.
    
    We also set $W^s(x,Z):=V^s(x,Z)\cap Z$.
\end{enumerate}
\end{definition}
\begin{prop}\label{forBuzzi1} Suppose $Z\in\mathcal{Z}$, then for any $x,y\in Z$: $V^u(x,Z),V^u(y,Z)$ are either equal or they are disjoint. Similarly for $V^s(x,Z),V^s(y,Z)$; and for $W^s(x,Z),W^s(y,Z)$ and for $W^u(x,Z),W^u(y,Z)$.
\end{prop}
\begin{proof}
The statement holds for $V^{s/u}(x,Z)$ because of Proposition \ref{prop221}. The statement for $W^{s/u}$ is an immediate corollary.
\end{proof}
\begin{prop}\label{propForBracketZ} Suppose $Z\in\mathcal{Z}$, then for any $x,y\in Z$: $V^u(x,Z)$ and $V^s(y,Z)$ intersect at a unique point $z$, and $z\in Z$. Thus $W^u(x,Z)\cap W^s(y,Z)=\{z\}$.
\end{prop}
\begin{proof} See \cite[Proposition~10.5]{Sarig}.
\end{proof}
\begin{definition}\label{bracketZ} The {\em Smale bracket} of two points $x,y\in Z\in\mathcal{Z}$ is the unique point $[x,y]_Z\in W^u(x,Z)\cap W^s(y,Z)$.
\end{definition}
\begin{lemma} Suppose $x,y\in Z(v_0)$, and $f(x),f(y)\in Z(v_1)$. If $v_0\rightarrow v_1$ then $f([x,y]_{Z(v_0)})=[f(x),f(y)]_{Z(v_1)}$.
\end{lemma}
\begin{proof} This is proved as in \cite[Lemma~10.7]{Sarig} using Theorem \ref{graphtransform}(2).
\end{proof}

\begin{lemma}
Suppose $Z,Z'\in\mathcal{Z}$. If $Z\cap Z'\neq\varnothing$ then for any $x\in Z,y\in Z'$: $V^s(x,Z)$ and $V^s(y,Z')$ intersect at a unique point.

We do not claim that this point is in $Z$ nor $Z'$.
\end{lemma}
Proof is similar to \cite[Lemma~5.8]{SBO}.
\subsubsection{The symbolic Markov property}
\begin{prop}\label{symbolicmarkovproperty}
If $x=\pi((v_i)_{i\in\mathbb{Z}})$, where $\underline{v}\in\Sigma^\#$, then $f[W^s(x,Z(v_0))]\subset W^s(f(x),Z(v_1))$ 

and $f^{-1}[W^u(f(x),Z(v_1))]\subset W^u(x,Z(v_0))$.
\end{prop}
\begin{proof} See \cite[Proposition~10.9]{Sarig}.
%
%
%
%
\end{proof}
\begin{lemma}\label{lastOne} Suppose $Z,Z'\in\mathcal{Z}$ and $Z\cap Z'\neq\varnothing$, then:
\begin{enumerate}
    \item If $Z=Z(\psi_{x_0}^{p_0^s,p_0^u})$ and $Z'=Z(\psi_{y_0}^{q_0^s,q_0^u})$ then $Z\subset \psi_{y_0}[R_{q_0^s\wedge q_0^u}(0)]$
    \item For any $x\in Z\cap Z'$: $W^u(x,Z)\subset V^u(x,Z')$ and $W^s(x,Z)\subset V^s(x,Z')$
\end{enumerate}
\end{lemma}
Proof is similar to \cite[Lemma~5.10]{SBO}.
\section{Proof of Theorem \ref{t4.1.1} and of Theorem \ref{t4.1.2}}\label{chapter5}
\subsection{Summary of the properties of the Markov extension}\label{summaryextension}

We present a summary of the objects we have constructed so far, and their properties.

We have constructed a family of sets $\mathcal{Z}$ s.t
\begin{enumerate}
    \item $\mathcal{Z}$ is \textbf{countable}  (Definition \ref{ZV}; see also Proposition \ref{discreteness} and Definition \ref{graphosaurus}).
    \item $\mathcal{Z}$ covers (exactly) $\RST$ 
(Corollary \ref{newLifeOfPi}).
    \item $\mathcal{Z}$ is \textbf{locally finite}: $\forall Z\in \mathcal{Z}, \#\{Z'\in\mathcal{Z}:Z\cap Z'\neq\varnothing\}<\infty$  (Theorem \ref{Zlocallyfinite}).
    \item Every $Z\in\mathcal{Z}$ has \textbf{product structure}: There are subsets $W^{s}(x,Z),W^u(x,Z)\subset Z$ $(x\in Z)$
s.t \begin{enumerate}
    \item $x\in W^{s}(x,Z)\cap W^u(x,Z)$.
    \item $\forall x,y\in Z \exists ! z\in Z$ s.t $W^u(x,Z)\cap W^s(y,Z)=\{z\}$, we write $z=[x,y]_Z.$
    \item $\forall x,y\in Z$ $W^u(x,Z),W^u(y,Z)$ are equal or disjoint (similarly for $W^s(x,Z),W^s(y,z)$).
\end{enumerate}
 (Proposition \ref{forBuzzi1}, Proposition \ref{propForBracketZ}).
\item The \textbf{symbolic Markov property}: if $x\in Z(u_0), f(x)\in Z(u_1)$,  and $u_0\rightarrow u_1$, then $f[W^s(x,Z_0)]\subset W^s(f(x),Z_1)$ and $f^{-1}[W^u(f(x),Z_1)]\subset W^u(x,Z_0)$ (Proposition \ref{symbolicmarkovproperty}).\footnote{To apply Proposition \ref{symbolicmarkovproperty}, represent $x=\pi(v)$ with $v\in\Sigma^\#$, $v_0=u_0$, and $f(x)=\pi(w)$ with $w\in\Sigma^\#$, $w_0=u_1$; and note that $x=\pi(u), u\in \Sigma^\#$ where $u_i=v_i$ $(i\leq0)$, $u_i=w_{i-1}$ $(i\geq1)$.}
\item 
\textbf{Uniform expansion/contraction on $W^{s}(x,Z),W^{u}(x,Z)$}: 
$\forall Z\in \mathcal{Z}, \forall z\in Z$, $\forall x\in W^s(z,Z),y\in W^u(z,Z)$, $\forall n\geq0$, $d(f^n(x),f^n(z)),d(f^{-n}(y), f^{-n}(z))\leq 4I^{-n}(1)$ (Proposition \ref{Lambda}(1)).

\item 
\textbf{Bowen's property with finite degree} (see \cite[Definition~1.11]{BoyleBuzzi}): \begin{enumerate}
    \item 
    $\forall u,v\in \Sigma:$ $\Big(\forall i, Z(u_i)\cap Z(v_i)\neq\varnothing\Big)\Rightarrow \pi(u)=\pi(v)$  (Proposition \ref{firstofchapter}(4), proof of Proposition \ref{prop213}).
    \item $\forall u,v\in \Sigma^\#$: $\pi(u)=\pi(v)\Leftrightarrow \Big(Z(u_i)\cap Z(v_i)\neq\varnothing$ $\forall i\in\mathbb Z\Big)$ ((a), Definition \ref{ZV}).
    \item $\forall Z\in \mathcal{Z}$:  $\#\{Z'\in\mathcal{Z}: Z'\cap Z\neq\emptyset\}<\infty$  (Theorem \ref{Zlocallyfinite}).
\end{enumerate}
\end{enumerate}
For the sake of completeness, we present below the refinement of $\mathcal{Z}$ into a partition, while the same arguments appear in verbatim in \cite[\textsection~6]{SBO} and beforehand in \cite{Sarig}. This refinment is used to prove Theorem \ref{t4.1.1} and Theorem \ref{t4.1.2}.

\subsection{Construction of a partition with a finite-to-one almost everywhere coding}\label{Forpartitione} We explain how the results which are listed in \textsection \ref{summaryextension} are sufficient to prove Theorem \ref{t4.1.1} and Theorem \ref{t4.1.2}.

\medskip
\underline{Step 1:} $\mathcal{Z}$ can be refined into a countable \textbf{partition} $\mathcal{R}$ with the following properties:
\begin{enumerate}
    \item[(a)] Product structure.
    \item[(b)] Symbolic Markov property.
    \item[(c)] Every $Z\in\mathcal{Z}$ contains only finitely many $R\in\mathcal{R}$.
\end{enumerate}

\underline{Proof:} This is done exactly as in \cite[\textsection~11]{Sarig}, using the Bowen-Sinai refinement procedure \cite{B1,B4,B3}, the local-finiteness property (3), and the product structure (4) of $Z_i\in\mathcal{Z}$. The partition elements are the equivalence classes of the following equivalence relation on $\bigcup\mathcal{Z}$ (see \cite[Proposition~11.2]{Sarig}):
$$x\sim y \iff \forall Z,Z'\in \mathcal{Z}, 
\begin{bmatrix}
    x\in Z\iff y\in Z  \\
    W^u(x,Z)\cap Z'\neq\varnothing \iff W^u(y,Z)\cap Z'\neq\varnothing   \\
     W^s(x,Z)\cap Z'\neq\varnothing \iff W^s(y,Z)\cap Z'\neq\varnothing
\end{bmatrix}.$$


\underline{Step 2:} For every finite chain $R_0\rightarrow\cdots\rightarrow R_m$,  $\bigcap_{j=0}^{m}f^{-j}[R_j]\neq\varnothing$.

\underline{Proof:} This follows from the fact that $\mathcal{R}$ is a Markov partition, using a classical argument due to Adler and Weiss, and Sinai. The proof is by induction on $m$. Assume the statement holds for $m$, and prove for $m+1$: Take $p\in f^m\Big[\bigcap_{j=0}^{m}f^{-j}[R_j]\Big]$, and take $p'\in R_{m}\cap f^{-1}R_{m+1}$. $p,p'\in R_{m}$, whence belong to some $Z\in\mathcal{Z}$; therefore $[p,p']_Z$ is well defined by Definition \ref{bracketZ} and Proposition \ref{propForBracketZ}. By the symbolic Markov property, $f^{-m}([p,p']_Z)\in \bigcap_{j=0}^{m+1}f^{-j}[R_{j}]$.

\underline{Step 3:} Let $\widehat{\mathcal{G}}$ be a directed graph with set of vertices $\mathcal{R}$ and set of edges $\{(R,S)\in \mathcal{R}\times\mathcal{R}: R\cap f^{-1}[S]\neq\varnothing\}$. Let $\widehat{\Sigma}:=\Sigma(\widehat{\mathcal{G}})$ and define $\widehat{\pi}:\widehat{\Sigma}\rightarrow M$ by $\{\widehat{\pi}(\underline{R})\}=\bigcap\limits_{n\geq0}\overline{\bigcap\limits_{i=-n}^n f^{-i}[R_i]}$. Then $\widehat{\pi}$ is well defined, and:
\begin{enumerate}
    \item[(a)] $\widehat{\pi}[\widehat{\Sigma}]\supset \RST$.
    \item[(b)] $\widehat{\pi}\circ\sigma=f\circ\widehat{\pi}$.
    \item[(c)] $\widehat{\pi}$ is uniformly continuous.
    \item[(d)] $\widehat{\Sigma}$ is locally compact.
\end{enumerate}

\underline{Proof:} This proof is a concise presentation of \cite[Theorem~12.3, theorem~12.5]{Sarig}. Step 2 guarantees that the intersection in the expression defining $\{\widehat{\pi}\}$ is not empty. 
The intersection is a singleton, because $\mathrm{diam}(\bigcap_{i=-n}^nf^{-i}[R_i])\leq 8I^{-n}(1)$. To see this, fix $x,y\in \bigcap_{i=-n}^nf^{-i}[R_i]$, and fix $\underline{u}\in\Sigma^\#$ s.t $\pi(\underline{u})=x$. It follows that $\bigcap_{i=-n}^nf^{-i}[R_i]\subset \bigcap_{i=-n}^nf^{-i}[Z(u_i)]$, whence $y\in\bigcap_{i=-n}^nf^{-i}[Z(u_i)]$. Let $z:=[x,y]_{Z(u_0)}$. By (4) and (5), $z\in \bigcap_{i=-n}^nf^{-i}[Z(u_i)]$, $f^n(z)\in W^u(f^n(x),Z(u_n))$, $f^{-n}(z)\in W^s(f^{-n}(y),Z(u_{-n}))$, whence by (6) $d(x,y)\leq d(x,z)+d(z,y)=d(f^{-n}(f^n(x)),f^{-n}(f^n(z)))+d(f^{n}(f^{-n}(z)),f^{n}(f^{-n}(y)))\leq 8I^{-n}(1)$.
Therefore $\widehat{\pi}$ is well defined. 
\begin{enumerate}
    \item[(a)] Let $x\in \RST$. By Theorem \ref{DefOfPi}, $f^i(x)\in\pi[\Sigma^\#]$ for all $i\in\mathbb{Z}$. By definition $\mathcal{Z}$ is a cover of $\pi[\Sigma^\#]$, and $\mathcal{R}$ refines $\mathcal{Z}$; whence $f^i(x)\in\bigcupdot\mathcal{R}$ $\forall i\in\mathbb{Z}$. Denote by $R_i$ the unique partition member of $\mathcal{R}$ which contains $f^i(x)$. $\underline{R}$ is an admissible chain by definition, and $\widehat{\pi}(\underline{R})=x$. Let $\underline{u}\in\Sigma^\#$ s.t $\pi(\underline{u})=x$, then $Z(u_i)\supset R_i$ $\forall i\in\mathbb{Z}$. $\mathcal{Z}$ is locally-finite, so by the pigeonhole principle, $\vec{R}$ must be in $\widehat{\Sigma}^\#$.
    \item[(b)] $$\{(f\circ\widehat{\pi})(\underline{R})\}=\bigcap_{n\geq0}\bigcap_{i=-n}^n\overline{f^{-(i-1)}[R_{i-1+1}]}=\{\widehat{\pi}(\sigma \underline{R})\}\text{ }(\because f\text{ is a diffeomorphism}).$$
    \item[(c)] Suppose $d(\underline{R},\underline{S})=e^{-n}$, then $R_i=S_i$ for $|i|\leq n$, whence $\widehat{\pi}(\underline{R}),\widehat{\pi}(\underline{S})\in\overline{\bigcap_{i=-n}^nf^{-i}[R_i]}$. We just saw in the beginning of the proof of step 3 that $\mathrm{diam}(\bigcap_{i=-n}^nf^{-i}[R_i])\leq 8I^{-n}(1)$. So $d(\widehat{\pi}(\underline{R}),\widehat{\pi}(\underline{S}))\leq 8\cdot I^{ \lceil\log d(\underline{R},\underline{S})\rceil}(1)$.
    \item[(d)] Fix a vertex $R\in \mathcal{R}$. Assume $R\rightarrow S$. Let $u\in \mathcal{V}$ s.t $R\subset Z(u)$. Take $x\in R\cap f^{-1}[S]\subset Z(u)$, then $\exists \underline{u}\in\Sigma^\#$ s.t $u_0=u$ and $\pi(\underline{u})=x$. Then $f(x)\in S\cap Z(u_1)$. Thus, by step 1(c), and Lemma \ref{lemma131}, $$\#\{S\in\mathcal{R}:R\rightarrow S\}\leq\sum_{u\in \mathcal{V}:R\subset Z(u)}\sum_{v\in\mathcal{V}:u\rightarrow v}\#\{S\in\mathcal{R}:S\subset Z(v)\}
    <\infty.$$
    A similar calculation shows $\#\{S':S'\rightarrow R\}<\infty$.
\end{enumerate}

\underline{Step 4:} $\widehat{\pi}[\widehat{\Sigma}^\#]=\RST$
, and $\widehat{\pi}|_{\widehat{\Sigma}^\#}$ is finite-to-one.

\underline{Proof:} In Corollary \ref{newLifeOfPi} we saw $\bigcup\mathcal{Z}=\pi[\Sigma^\#]=\RST$. $\mathcal{R}$ is a refinement of $\mathcal{Z}$. \cite[Proposition~4.11]{LifeOfPi} shows $\widehat{\pi}[\widehat{\Sigma}^\#]=\pi[\Sigma^\#]$ (with our $\epsilon>0$)\footnote{The proof of \cite[Proposition~4.11]{LifeOfPi} does not depend on $\epsilon$ at all; $\epsilon$ only appears in the statement to satisfy the equality $\pi[\Sigma^\#]=\RST$ which we have already seen independently in Corollary \ref{newLifeOfPi}.}, and completes the proof that $\widehat{\pi}[\widehat{\Sigma}^\#]=\RST$. The finite-to-one property follows from Theorem \ref{t4.1.2}, whose proof follows verbatim the proof of \cite[Theorem~1.3]{SBO} (which follows the proof of Sarig in \cite{Sarig}, which in turn adapts the earlier works of Bowen \cite{B1,B2,B3}).

\medskip
\begin{prop}
For every $x\in \widehat{\pi}[\widehat{\Sigma}]$, $T_xM=E^s(x)\oplus E^u(x)$, where
$$(1)\text{ }\sum_{n\geq0}\|d_xf^n|_{E^s(x)}\|^2<\infty\text{ , }(2)\text{ }\sum_{n\geq0}\|d_xf^{-n}|_{E^u(x)}\|^2<\infty.
$$
The maps $\underline{R}\mapsto E^{s/u}(\widehat{\pi}(\underline{R}))$ have summable variations as maps from $\widehat{\Sigma}$ to $TM$.
\end{prop}
\begin{proof}
The proof is similar to \cite[Proposition~12.6]{Sarig}, so as in \cite[Proposition~6.1]{SBO}, we limit ourselves to a brief sketch; One first writes $x$ as the image of a chain $\underline{u}$ in $\Sigma$ (it is possible since the partition $\mathcal{R}$ is a refinement of the cover $\mathcal{Z}$). Then the tangent space $T_xM$ splits by $T_x\widecheck{V}^s(\underline{u})\oplus T_x\widecheck{V}^u(\underline{u})$. The rest is clear by Lemma \ref{Lambda}(2) and Proposition \ref{firstofchapter}(5).
\end{proof}
\section{Absolute Continuity of Weak Foliations}\label{chapter7}

In \textsection \ref{GTrans} we constructed local stable and unstable leaves for points in $\RST$. In fact, by using a strengthened method of the Graph Transform, we constructed local stable and unstable manifolds for every chain, while Theorem \ref{DefOfPi}
says that every point in $\RST$ can be coded by a chain in $\Sigma^\#$.

\medskip
These local stable manifolds contract forwards in time, as shown in Proposition \ref{Lambda}. However, they may contract in a strictly slower than an exponential rate, which distinguishes their behavior from the behavior of classical local stable leaves of non-uniformly hyperbolic orbits. For that reason, the foliations by such leaves are called {\em weak foliations}. In an upcoming paper, we present a few examples of systems which exhibit orbits with \textbf{strictly weak foliations}, while these orbits are shown to carry invariant measures of interest (alas not probability measures, as every invariant probability measure carried by $\RST$ is hyperbolic), and we show furthermore that these orbits have full Riamannian volume. 

\medskip
In this section of the paper we wish to prove that weak foliations are absolutely continuous w.r.t the holonomy map and the Riemannian volume.\footnote{This terminology is common- however may be slightly misleading; since in the general case the stable leaves may not generate a foliation, and may not exist for every point.} This property was shown by Anosov and Sinai for the stable foliations of Anosov systems in \cite{Anosov_1967}, and later extended by Pesin to the non-uniformly hyperbolic setup (see \cite[\textsection~8.6]{BP} for proof). Pesin's proof deals with the difficulty that not every point may admit a local stable leaf, and with the fact that the size of local stable leaves may vary greatly in small neighborhoods. Pesin's approach is to restrict to a regular set (i.e Pesin block), and consider the leaves which pass through points in a fixed such set. 

\medskip
The proof we present below follows the idea of Pesin, while utilizing one more property: A careful examination of Pesin's proof shows that a minimal requirement for his scheme to work is that the leaves contract fast enough so that the sequence of their diameters over a forwards-orbit is summable. In our proof, restricting to a partition element takes the role of restricting to a regular set. Proposition \ref{Lambda} and Lemma \ref{FaveForNow} show the summability property on partition elements, which is needed for the scheme to work. Although, one still has to show that the $\beta$-H\"older continuity of $d_\cdot f$ is sufficient for summability when $\beta$ may be arbitrarily small, which was obvious for the exponentially contracting rate in Pesin's case.

\medskip
For the simplicity of presentation we prove the absolute continuity of the holonomy map between two unstable leaves of chains in a fixed chart; but the proof easily extends for two admissible $u$-manifolds in the chart similarly (recall Definition \ref{admissible}). The principle which is used is that the contraction on the stable leaves overpowers any possible contraction in the transversals in a summable manner. 

\subsection{Bounded leaf measures}
Let $\mathcal{V}$ be the collection of double charts from Definition \ref{graphosaurus}. Fix $u\in\mathcal{V}$, and the set $A(u):=\pi[\Sigma\cap[u]]$. Notice, $A(u)$ is compact and $A(u)\supseteq Z(u)$ which is covered by finitely many partition elements (recall \textsection \ref{Forpartitione}). Fix $(u_i^1)_{i\leq0}$ and $(u_i^2)_{i\leq0}$ two negative chains 
s.t $u^1_0=u^2_0= u=\psi_x^{p^s,p^u}$. Denote the respective representative functions of $V^u((u_i^1)_{i\leq0})$ and $V^u((u_i^2)_{i\leq0})$ by $F^1$ and $F^2$.

\medskip
\noindent Notice: every point in $V^u((u_i^j)_{i\leq0})\cap A(u)$ can be coded by a chain $\underline{u}\in\Sigma
$ s.t $u_i^j=u_i$ for all $i\leq 0$ and $j=1,2$. 

\begin{lemma}\label{niniloulou}
Let $x\in V^u((u_i^1)_{i\leq0})\cap Z(u)$, and let $\underline{u}\in\Sigma
$ s.t $u_i^j=u_i$ for all $i\leq 0$ and $\pi(\underline{u})=x$. Write $y:= \pi((u_i^2)_{i\leq0}\cdot (u_i)_{i\geq0})$, where $\cdot$ denotes an admissible concatenation. Write $u_i=\psi_{x_i}^{p^s_i,p^u_i}$, $i\in\mathbb{Z}$.

For every $n\geq0$, let $\ul{u}^{1,n}:=(u_i^1)_{i\leq0}\cdot(u_j)_{j=0}^n$ and $\ul{u}^{2,n}:=(u_i^2)_{i\leq0}\cdot(u_j)_{j=0}^n$ $F_n^1$, and let $F_n^2$ be the representative functions of $V^u(\ul{u}^{1,n})$ and $V^u(\ul{u}^{2,n})$, respectively. Write $\widetilde{F}^j_n(v):=(F^j_n(v),v)$, $n\geq0$.

Then for all $n\geq 0$, for all $\eta\in (0, (I^{-n}(p^s\wedge p^u))^2)$, 
$$\left|\frac{\mathrm{Vol}(\psi_{x_n}\circ \widetilde{F}^1_n[R_\eta(0)])}{\mathrm{Vol}(\psi_{x_n}\circ \widetilde{F}^2_n[R_\eta(0)])}-1\right|\leq \epsilon^2 (I^{-n}(p^s\wedge p^u))^{\frac{1}{\gamma}+\tau'},$$
where $\tau'=\frac{1}{4\gamma}>0$, $\Vol$ denotes the intrinsic volume of a submanifold with an induced Riemannian volume, and $R_\eta(0)$ is an $\|\cdot\|_\infty$-norm ball of radius $\eta$ around $0$ in $\mathbb{R}^{u(x)}$.
\end{lemma}

\begin{proof} We fix $n\geq0$, and omit the subscript in the proof where it is clear, as all analysis is in a fixed chart $u_n=\psi_{x_n}^{p^s_n,p^u_n}$. 

\medskip
\textit{Step 1:} Define $$I:\psi_{x_n}(F^{1}(t),t)\mapsto\psi_{x_n}(F^{2}(t),t).$$ Equivalently, $I=\psi_{x_n}\circ \tilde{F}^{2}\circ \pi_u\circ \psi_{x_n}^{-1}$, and  $\tilde{F}^{1/2}$ are diffeomorphisms onto their images, where $\pi_u$ is the projection onto the $u$-coordinates. So $I$ is a diffeomorphism from $\widecheck{V}^u(\ul{u}^{1,n})$ to $\widecheck{V}^u(\ul{u}^{2,n})$. $I$ satisfies
\begin{equation}\label{adidax}
d(I(x'),x')\leq \|\psi_{x_n}\|d_{C^0}(F^{1},F^{2})\leq 2d_{C^0}(F^{1},F^{2}).	
\end{equation}

The identity map of $\widecheck{V}^u(\ul{u}^{1,n})$ can be written as $\mathrm{Id}=\psi_{x_n}\circ \tilde{F}^{1}\circ\pi_u\circ\psi_{x_n}^{-1}:\widecheck{V}^u(\ul{u}^{1,n})\rightarrow \widecheck{V}^u(\ul{u}^{1,n})$. Let $z\in \widecheck{V}^u(\ul{u}^{1,n})$, and let $\Theta_{D}:TD\rightarrow \mathbb{R}^d$ be a local isometry as in Definition \ref{isometries} s.t. $D$ is a neighborhood which contains $z,I(z)$. Then,
\begin{align}\label{ams}
\|\Theta_D d_zI-\Theta_D d_z\mathrm{Id}\|&=\|\Theta_Dd_z(\psi_{x_n}\circ \tilde{F}^{2}\circ\pi_u\circ\psi_{x_n}^{-1})-\Theta_Dd_z(\psi_{x_n}\circ \tilde{F}^{1}\circ\pi_u\circ\psi_{x_n}^{-1}) \nonumber\|\\
&= \|[\Theta_Dd_t(\psi_{x_n}\circ\tilde{F}^{\ul{v}})-\Theta_Dd_t(\psi_{x_n}\tilde{F}^{\ul{u}})]\pi_u\circ d_z(\psi_{x_n}^{-1}))\|, \hspace{0.25cm}\text{where } t=\pi_u\psi_{x_n}^{-1}(z),
\nonumber\\
&\leq 2\|C_0^{-1}(x_n)\|\cdot\|\Theta_Dd_t(\psi_{x_n}\circ\tilde{F}^{2})-\Theta_Dd_t(\psi_{x_n}\circ\tilde{F}^{1})\|\nonumber\\
&=2\|C_0^{-1}(x_n)\|\cdot\|\Theta_Dd_{\tilde{F}^{2}(t)}\psi_{x_n}d_t\tilde{F}^{2}-\Theta_Dd_{\tilde{F}^{1}(t)}\psi_{x_n}d_t\tilde{F}^{1}\|\nonumber\\
&=2\|C_0^{-1}(x_n)\|\cdot\|\Theta_Dd_{\tilde{F}^{\ul{v}}(t)}\psi_{x_n}(d_t\tilde{F}^{2}-d_t\tilde{F}^{1})+(\Theta_Dd_{\tilde{F}^{2}(t)}\psi_{x_n}-\Theta_Dd_{\tilde{F}^{1}(t)}\psi_{x_n})d_t\tilde{F}^{1}\|\nonumber\\
&\leq2\|C_0^{-1}(x_n)\|\cdot\left[d_{C^1}(\tilde{F}^{1},\tilde{F}^{2})\|d_\cdot\psi_{x_n}\|+\mathrm{Lip}(d_\cdot\psi_{x_n})d_{C^0}(\tilde{F}^{1},\tilde{F}^{2})\cdot\|d_\cdot\tilde{F}^{1}\|\right]\nonumber\\
&\leq 8\|C_0^{-1}(x_n)\|\cdot d_{C^1}(\widecheck{V}^u(\ul{u}^{1,n}),\widecheck{V}^u(\ul{u}^{2,n}))\leq 8(C_{\beta,\epsilon}^{-1} \cdot p_n^s\wedge p_n^u)^\frac{-1}{2\gamma}\cdot d_{C^1}(\widecheck{V}^u(\ul{u}^{1,n}),\widecheck{V}^u(\ul{u}^{2,n}))\nonumber\\
\leq& 8(C_{\beta,\epsilon}^{-1} \cdot I^{-n}(p^s\wedge p^u))^\frac{-1}{2\gamma}\cdot d_{C^1}(\widecheck{V}^u(\ul{u}^{1,n}),\widecheck{V}^u(\ul{u}^{2,n})).
\end{align}
It follows that,
\begin{align}\label{adidax2}
	|\Jac(d_zI)-1|=&|\Jac(d_zI)-\Jac(d_z\mathrm{Id})|\leq C_1 \|\Theta_D d_zI-\Theta_D d_z\mathrm{Id}\|
	\nonumber\\
	\leq& 8C_1(C_{\beta,\epsilon}^{-1} \cdot I^{-n}(p^s\wedge p^u))^\frac{-1}{2\gamma}\cdot d_{C^1}(\widecheck{V}^u(\ul{u}^{1,n}),\widecheck{V}^u(\ul{u}^{2,n})),
\end{align}
where $C_1$ is the Lipschitz constant for the absolute value of the determinant on the ball of $(d\times d)$-matrices with an operator norm bounded by $2
$, where $\|d_zI\|\leq 2
$ by \eqref{ams}.\footnote{Write $u=dim V^u(\ul{R})$, and let $A, B$ be $u\times u$ matrices s.t. $|a_{ij}-b_{ij}|\leq \delta$ for all $i,j\leq u$. So $\det A=\sum_{\sigma\in S_{u}}\mathrm{sgn}(\sigma)\prod_{i=1}^{u}a_{i\sigma(i)}$. Then $|\det A-\det B|\leq \sum_{\sigma\in S_{u}}|\prod_{i=1}^{u}a_{i\sigma(i)}-\prod_{i=1}^{u}b_{i\sigma(i)}|\leq |S_{u}|\cdot u\|B\|_{Fr}^{u-1}\delta$. Take maximum over $u\leq d-1$.}

By Proposition \ref{firstofchapter}, $d_{C^1}(\widecheck{V}^u(\ul{u}^{1,n}),\widecheck{V}^u(\ul{u}^{2,n}))\leq C(I^{-n}(p^s\wedge p^u))^{\frac{1}{\gamma}+\tau}$, for $C,\tau>0$ constants depending on the calibration parameters (but not on $\epsilon$)
. Furthermore, by \eqref{formyPesinAC} in Proposition \ref{firstofchapter}, when $\gamma>\frac{5}{\beta}$, $\beta-\frac{2}{\gamma}>\frac{3}{\gamma}=(\frac{1}{\gamma}+\tau')+(\frac{1}{\gamma}+\frac{1}{2\gamma}+\tau')$, where $\tau':=\frac{1}{4\gamma}>0$. Then for $C':=3\sup_{N\geq0}\{N\cdot (I^{-N}(1))^{\frac{1}{\gamma}+\tau'}\}<\infty$, we get 

\begin{equation*}
d_{C^1}(\widecheck{V}^u(\ul{u}^{1,n}),\widecheck{V}^u(\ul{u}^{2,n}))\leq C'(I^{-n}(p^s\wedge p^u))^{\frac{1}{\gamma}+\frac{1}{2\gamma}+\tau'}.
\end{equation*}


\noindent Substituting this in \eqref{adidax2} gives, for all sufficiently small $\epsilon>0$,
\begin{equation}\label{DeltaAir}|\Jac(d_zI)-1|\leq8C_1C_{\beta,\epsilon}^\frac{1}{2\gamma}\cdot C'(I^{-n}(p^s\wedge p^u))^{\frac{1}{\gamma}+\tau'}\leq \epsilon^2 (I^{-n}(p^s\wedge p^u))^{\frac{1}{\gamma}+\tau'}.\end{equation}

\textit{Step 2:}

\begin{align*}
| \mathrm{Vol}(\psi_{x_n}\circ \widetilde{F}^1[R_\eta(0)])-\mathrm{Vol}(\psi_{x_n}\circ \widetilde{F}^2[R_\eta(0)]) | =&|\int\limits_{R_\eta(0)}\Jac(d_t(\psi_{x_n}\tilde{F}^{1}))d\mathrm{Leb}(t)-\int_{R_\eta(0)}\Jac(d_t(\psi_{x_n}\tilde{F}^{2}))d\mathrm{Leb}(t)|\\
\leq&\int_{R_\eta(0)}|\Jac(d_t(\psi_{x_n}\tilde{F}^{1}))-\Jac(d_t(\psi_{x_n}\tilde{F}^{2}))|d\mathrm{Leb}(t)\\
=& \int\limits_{R_\eta(0)}\Jac(d_t (\psi_{x_n}\widetilde{F}^1))\cdot |1-\Jac(d_t(\psi_{x_n}\tilde{F}^{2}))\\
&\cdot \Jac(d_{\psi_{x_n}\tilde{F}^{1}(t)}((\psi_{x_n}\tilde{F}^{1})^{-1}))|d\mathrm{Leb}(t) \\
=& \int\limits_{R_\eta(0)} \Jac(d_t(\psi_{x_n}\widetilde{F}^1))\cdot|1-\Jac(d_{\psi_{x_n}\tilde{F}^{1}(t)}I)|d\mathrm{Leb}(t).
\end{align*}
By  \eqref{DeltaAir},

\begin{align}\label{BestWest5}
\left|\frac{\Vol(\psi_{x_n}\circ \widetilde{F}^2[R_\eta(0)])}{\Vol(\psi_{x_n}\circ \widetilde{F}^1[R_\eta(0)])}-1\right|=&\frac{|\Vol(\psi_{x_n}\circ \widetilde{F}^1[R_\eta(0)])-\Vol(\psi_{x_n}\circ \widetilde{F}^2[R_\eta(0)])|}{\Vol(\psi_{x_n}\circ \widetilde{F}^1[R_\eta(0)])} \nonumber\\\leq&\frac{\epsilon^2 (I^{-n}(p^s\wedge p^u))^{\frac{1}{\gamma}+\tau'}\int_ {R_\eta(0)} \Jac(d_t(\psi_{x_n}\circ \widetilde{F}^1[R_\eta(0)])\mathrm{Leb}(t)}{\int_{R_\eta(0)} \Jac(d_t(\psi_{x_n}\circ \widetilde{F}^1[R_\eta(0)])\mathrm{Leb}(t)}\nonumber\\
= &\epsilon^2 (I^{-n}(p^s\wedge p^u))^{\frac{1}{\gamma}+\tau'}.
\end{align}
\end{proof}

\begin{lemma}\label{Bethisda}
In the notation of Lemma \ref{niniloulou}, $\forall n\geq0$

$$\frac{|d_{f^n(x)}f^{-n}\omega_u(f^n(x))|}{|d_{f^n(y)}f^{-n}\omega_u(f^n(y))|}= e^{\pm \epsilon^2(\frac{d(x,y)}{(p^s\wedge p^u)^3})^\frac{\beta}{4}},$$
where $\omega_u(\cdot)$ is the normalized volume form of $T_\cdot V^u(\cdot)$.	
\end{lemma}
\begin{proof}
Showing $\frac{|d_{x}f^{n}\omega_u(x)|}{|d_{y}f^{n}\omega_u(y)|}=e^{\pm \epsilon^2(\frac{d(x,y)}{(p^s\wedge p^u)^3})^\frac{\beta}{4}}$ is enough. Proposition \ref{Lambda} 
treats the case for the normalized volume forms of $T_\cdot V^s(\cdot)$. The case for the normalized volume forms of $T_\cdot V^u(\cdot)$ is similar. 
\end{proof}

\begin{cor}\label{afterafterniniloulou}
In the notation of Lemma \ref{niniloulou}, $\forall n\geq0$ sufficiently large (w.r.t $d(x,y)$ and $p^s\wedge p^u$), $\forall \eta\in(0,(I^{-n}(p^s\wedge p^u))^{3+\frac{4}{\beta}(\frac{1}{\gamma}+\tau')}]$,

$$\frac{\Vol(f^{-n}[\psi_{x_n}\circ \widetilde{F}^1_n[R_\eta(0)]])}{\Vol(f^{-n}[\psi_{x_n}\circ \widetilde{F}^2_n[R_\eta(0)]])}=e^{\pm2\epsilon^2(\frac{d(x,y)}{(p^s\wedge p^u)^3})^\frac{\beta}{4}},$$

where $\tau'=\frac{1}{4\gamma}>0$.
\end{cor}
\begin{proof}
For any $t_1,t_2\in R_\eta(0)$, $d(\psi_{x_n}(t_1),\psi_{x_n}(t_2))\leq \frac{3}{2}\eta$. We have, $$\Vol(f^{-n}[\psi_{x_n}\circ \widetilde{F}^1_n[R_\eta(0)]]) =\int_{\psi_{x_n}\circ \widetilde{F}^1_n[R_\eta(0)])}|d_{f^n(t)}f^{-n}\omega_u(f^n(t))|d\Vol_{V^u(f^n(x))}(t).$$
By Proposition \ref{Lambda}, part (3), since $\eta\leq I^{-n}(p^s\wedge p^u))^{3+\frac{4}{\beta}(\frac{1}{\gamma}+\tau')} $ and $(p^s_n)^3\geq (p^s_n\wedge p^u_n)^3\geq  (I^{-n}(p^s\wedge p^u))^3 $,
	$$\Vol(f^{-n}[\psi_{x_n}\circ \widetilde{F}^1_n[R_\eta(0)]])=\Vol(\psi_{x_n}\circ \widetilde{F}^1_n[R_\eta(0)])\cdot |d_{f^n(x)}f^{-n}\omega_u(f^n(x))| e^{\pm \epsilon^2(\frac{3}{2})^\frac{\beta}{4}(I^{-n}(p^s\wedge p^u))^{\frac{1}{\gamma}+\tau'}}.$$
		Similarly, 
		$$\Vol(f^{-n}[\psi_{x_n}\circ \widetilde{F}^2_n[R_\eta(0)]])=\Vol(\psi_{x_n}\circ \widetilde{F}^2_n[R_\eta(0)])\cdot |d_{f^n(y)}f^{-n}\omega_u(f^n(y))| e^{\pm \epsilon^2(\frac{3}{2})^\frac{\beta}{4}(I^{-n}(p^s\wedge p^u))^{\frac{1}{\gamma}+\tau'}}.$$
		
By Lemma \ref{Bethisda}, 

$$\frac{\Vol(f^{-n}[\psi_{x_n}\circ \widetilde{F}^2_n[R_\eta(0)]])}{\Vol(f^{-n}[\psi_{x_n}\circ \widetilde{F}^1_n[R_\eta(0)]])}=\frac{\Vol(\psi_{x_n}\circ \widetilde{F}^2_n[R_\eta(0)])}{\Vol(\psi_{x_n}\circ \widetilde{F}^1_n[R_\eta(0)])}e^{\pm 2\cdot (\frac{3}{2})^\frac{\beta}{4}\epsilon^2 (I^{-n}(p^s\wedge p^u))^{\frac{1}{\gamma}+\tau'}} e^{\pm \epsilon^2(\frac{d(x,y)}{(p^s\wedge p^u)^3})^\frac{\beta}{4}}.$$
		By Lemma \ref{niniloulou}, for all $\epsilon$ sufficiently small, $\frac{\Vol(\psi_{x_n}\circ \widetilde{F}^2[R_\eta(0)])}{\Vol(\psi_{x_n}\circ \widetilde{F}^1[R_\eta(0)])}=e^{\pm(1+(3-2(\frac{3}{2})^\frac{\beta}{4}))\epsilon^2 (I^{-n}(p^s\wedge p^u))^{\frac{1}{\gamma}+\tau'}}$. Thus in total, $\frac{\Vol(f^{-n}[\psi_{x_n}\circ \widetilde{F}^2_n[R_\eta(0)]])}{\Vol(f^{-n}[\psi_{x_n}\circ \widetilde{F}^1_n[R_\eta(0)]])}= e^{\pm2\epsilon^2(\frac{d(x,y)}{(p^s\wedge p^u)^3})^\frac{\beta}{4}}$ for all $n$ sufficiently large.
\end{proof}

\subsection{Covering lemmas and the proof of the absolute continuity of weak foliations}

\begin{lemma}[Besicovitch-Pesin]\label{BesiPesi}
	For all sufficiently large $n\geq0$, there exist $\{x^{(i)}\}_{i=1}^\ell\subseteq A(u)$, and sets $\{D_i^1\}_{i=1}^\ell$ s.t
	\begin{enumerate}
	\item $x^{(i)}=\pi(\underline{u}^{(i)})$ where $u^{(i)}_j=u^1_j$, $j\geq0$; and $u_j^{(i)}=\psi_{x_j^{(i)}}^{p^{s,i}_j,p^{u,i}_j}$.	
	\item $D_i^1\subseteq \psi_{x_n^{(i)}}\circ \widetilde{F}^{1,i}_n[R_\eta(0)]$ where $F^{1,i}$ is the representative function of $V^u((u^{(i)}_j)_{j\leq n})$ in $u_n^{(i)}$, and $\eta\in(0, (I^{-n}(p^s\wedge p^u))^{3+\frac{4}{\beta}(\frac{1}{\gamma}+\tau')}]$, $\tau':=\frac{1}{4\gamma}>0$.
	\item $\forall i\neq i'$, $\Vol(D_i^1\cap D_{i'}^1)=0$ ($\Vol$ is the intrinsic $f^n[V^u((u^1_i)_{i\leq0})]$-volume).
	\item $A(u)\cap V^u((u^1_i)_{i\leq0})\subseteq \bigcup_{i=1}^{\ell}f^{-n}[D_i^1]$.
	\item There exist $R_n$,  and a constant $c_m>0$ depending on the dimension of $M$ s.t for $\eta_n:= (I^{-n}(p^s\wedge p^u))^{3+\frac{4}{\beta}(\frac{1}{\gamma}+\tau')}$, for all $1\leq i\leq \ell$, $$B(x^{(i)},c_m R_n)\subseteq D_i^1\subseteq f^{-n}[\psi_{x_n^{(i)}}\circ \widetilde{F}^{1,i}_n[R_{\eta_n}(0)]]) \subseteq B(x^{(i)},R_n).$$
	\end{enumerate}

The collection depends on the leaf $V^u((u^1_i)_{i\leq0})$, but a collection $\{D_i^2\}_{i=1}^{\ell} $ with similar properties exists for $V^u((u^2_i)_{i\leq0})$, and constitutes a cover of $A(u)\cap V^u((u^2_i)_{i\leq0})$. In particular, the two covers share the same multiplicity $\ell<\infty$. In addition, the collections $\{D_i^{1/2}\}_{i=1}^\ell$ can be chosen in such a way that $\forall 1\leq i \leq \ell$,
\begin{equation}\label{afterniniloulou}
\frac{\Vol(D_i^1)}{\Vol(D_i^2)}=e^{\pm \epsilon^2 (I^{-n}(p^s\wedge p^u))^{\frac{1}{\gamma}+\tau'}}.
\end{equation}
\end{lemma}
For proof	see \cite[Lemma~8.6.3]{BP}. For the estimate of $\eqref{afterniniloulou}$ use the Jacobian bounds of the map $I$ from Lemma \ref{niniloulou}.

\medskip
\noindent\textbf{Remark:} The collections $\{D_i^{1/2}\}_{i=1}^\ell$ depends on $n$. To ease notation we omit notating it, when dependence is clear by context. In addition, similarly.

\begin{cor}\label{CorBesiPesi}
For all $n\geq0$ sufficiently large, for all $1\leq i\leq \ell$,	$$\frac{\Vol(f^{-n}[D_i^1])}{\Vol(f^{-n}[D_i^2])}=e^{\pm2\epsilon^2(\frac{d(x,y)}{(p^s\wedge p^u)^3})^\frac{\beta}{4}} .$$
\end{cor}
Shown similarly as in Corollary \ref{afterafterniniloulou}.

\medskip
Now that we have useful covers in our arsenal, and good estimates on the leaf measures of the cover elements, we wish to begin the proof of the absolute continuity property. If the leaf volume of $A(u)$ is $0$ in both leaves, we are done. So we may assume that the volume in one of them is positive- w.l.o.g $V^u((u_i^1)_{i\leq0})$- and get estimates for the measure in the other. We start by recalling a useful measure theoretical lemma:
\begin{lemma}\label{measuretheoretic}
	Given $\delta>0$, let $U_\delta^1$ be a $\delta$-neighborhood of $V^u((u_i^1)_{i\leq0})\cap A(u)$ in the leaf $V^u((u_i^1)_{i\leq0})$. Then $\exists \delta_0>0$ s.t $\forall \delta\in(0,\delta_0]$
	$$\frac{\Vol(A(u))}{\Vol(U_\delta^1)}\geq e^{-2\epsilon^2(\frac{2d_{C_0}(V^u((u_i^1)_{i\leq0}),V^u((u_i^2)_{i\leq0}))}{(p^s\wedge p^u)^3})^\frac{\beta}{4}},$$
	where $\Vol$ is the intrinsic $V^u((u_i^1)_{i\leq0})$-volume.
\end{lemma}

In particular, $\delta_0$ does not depend on $n$. Therefore, $\forall \delta\in(0,\delta_0)$ we may choose $n_\delta$ s.t $\forall n\geq n_\delta$, 
\begin{equation}\label{forfinalprodbounds}
	\bigcup_{i=1}^\ell f^{-n}[D_i^2]\subseteq U_\delta^2,
\end{equation}
where $U^2_{\delta}$ is the $\delta$-neighborhood of $V^u((u_i^2)_{i\leq0})\cap A(u)$ in the leaf $V^u((u_i^2)_{i\leq0})$.

\begin{cor}\label{lastoneplease} $\forall \delta\in (0,\delta_0)$
$$\frac{\Vol(U^2_{\delta})}{\Vol(U^1_{\delta})}\geq e^{-4\epsilon^2(\frac{2d_{C_0}(V^u((u_i^1)_{i\leq0}),V^u((u_i^2)_{i\leq0}))}{(p^s\wedge p^u)^3})^\frac{\beta}{4}}.$$
\end{cor}
\begin{proof} Let $n_\delta$ as in \eqref{forfinalprodbounds}. Then
\begin{align*}
	\frac{\Vol(U^2_{\delta})}{\Vol(U^1_{\delta})}= \frac{\Vol(U^2_{\delta})}{\Vol(\bigcup_{i=1}^\ell f^{-n}[D_i^2])}\cdot \frac{\Vol(\bigcup_{i=1}^\ell f^{-n}[D_i^2])}{\Vol(\bigcup_{i=1}^\ell f^{-n}[D_i^1])}\cdot \frac{\Vol(\bigcup_{i=1}^\ell f^{-n}[D_i^1])}{\Vol(U^1_{\delta})}.
\end{align*}
	By \eqref{forfinalprodbounds}, the first factor is greater or equal to $1$. By Lemma \ref{measuretheoretic}, the third factor is greater or equal to $e^{-2\epsilon^2(\frac{2d_{C_0}(V^u((u_i^1)_{i\leq0}),V^u((u_i^2)_{i\leq0}))}{(p^s\wedge p^u)^3})^\frac{\beta}{4}}$. By Lemma \ref{BesiPesi} and Corollary \ref{BesiPesi},
	$$\frac{\Vol(\bigcup_{i=1}^\ell f^{-n}[D_i^2])}{\Vol(\bigcup_{i=1}^\ell f^{-n}[D_i^1])}= \frac{\sum_{i=1}^\ell \Vol(f^{-n}[D_i^2])}{\sum_{i=1}^\ell \Vol(f^{-n}[D_i^1])}=e^{\pm 2\epsilon^2(\frac{2d_{C_0}(V^u((u_i^1)_{i\leq0}),V^u((u_i^2)_{i\leq0}))}{(p^s\wedge p^u)^3})^\frac{\beta}{4}}.$$
	This completes the proof.
\end{proof}
\begin{cor}
$\Vol(A(u)\cap V^u((u_i^2)_{i\leq0}))>0$.
\end{cor}
\begin{proof}
	$\Sigma$ is locally compact, $\pi$ is uniformly continuous, and so $A(u)$ is compact. Therefore, if $\Vol(A(u)\cap V^u((u_i^2)_{i\leq0}))=0$, then $\lim_{\delta\to0}\Vol(U^2_\delta)=0$ as well, but this is a contradiction to Corollary \ref{lastoneplease}.
\end{proof}

\medskip
 Now we may apply Lemma \ref{measuretheoretic} to the set $A(u)\cap V^u((u_i^2)_{\leq0})$, and get a $\delta_1\in (0,\delta_0]$ s.t $\forall \delta\in (0,\delta_1)$, 
 $$\frac{\Vol(A(u))}{\Vol(U_\delta^2)}\geq e^{-2\epsilon^2(\frac{2d_{C_0}(V^u((u_i^1)_{i\leq0}),V^u((u_i^2)_{i\leq0}))}{(p^s\wedge p^u)^3})^\frac{\beta}{4}},$$
where $\Vol$ is the intrinsic $V^u((u_i^2)_{i\leq0})$-volume.

\medskip
Therefore, we may use Lemma \ref{lastoneplease} with replacing the roles of $U^1_\delta$ and $U^2_\delta$, for all $\delta\in (0,\delta_1)$,
\begin{equation}\label{inofficern}
\frac{\Vol(U^1_\delta)}{\Vol(U^2_\delta)}=e^{\pm4\epsilon^2(\frac{2d_{C_0}(V^u((u_i^1)_{i\leq0}),V^u((u_i^2)_{i\leq0}))}{(p^s\wedge p^u)^3})^\frac{\beta}{4}}.	
\end{equation}

It then follows from Lemma \ref{measuretheoretic} and Corollary \ref{lastoneplease} that $\forall \delta\in(0,\delta_1)$,
\begin{align}\label{finallyAC}
	\frac{\Vol(V^u((u_i^1)_{i\leq0})\cap A(u))}{\Vol(V^u((u_i^2)_{i\leq0})\cap A(u))}=&\frac{\Vol(V^u((u_i^1)_{i\leq0})\cap A(u))}{\Vol(U^1_\delta)}\cdot \frac{\Vol(U^1_\delta)}{\Vol(U^2_\delta)}\cdot \frac{\Vol(U^2_\delta)}{\Vol(V^u((u_i^2)_{i\leq0})\cap A(u))}\nonumber\\
	=&e^{\pm6\epsilon^2(\frac{2d_{C_0}(V^u((u_i^1)_{i\leq0}),V^u((u_i^2)_{i\leq0}))}{(p^s\wedge p^u)^3})^\frac{\beta}{4}}.
\end{align}

\subsection{Construction of strictly weak leaves}
In \textsection \ref{GTrans} we construct weak manifolds- local stable and unstable leaves for orbits shadowed by $I$-chains. In particular, this construction does not make use of an asymptotic exponential contraction on the stable or unstable subspaces of the shadowed orbit; when exponentially-fast contracting manifolds are the standard stable and unstable manifolds in Pesin's stable manifold theorem. A natural question which rises is whether there are weak manifolds, which cannot be constructed using Pesin's stable manifolds theorem. That is, local stable and unstable leaves which exhibit strictly sub-exponential contraction. Such manifolds are {\em strictly weak} stable and unstable manifolds. 

In this chapter we give a general result regarding the existence of strictly weak leaves. It uses the observation that every hyperbolic periodic orbit, and every uniformly hyperbolic measure, can be coded by our coding. In \textsection \ref{codability} we discuss a general condition for the ``codability" of measures- i.e property of having a lift in our coding space.

\medskip
Given $x\in\RST$, let $W^s(x):=\bigcup_{n\geq0} f^{-n}[V^s(f^n(x))]$, $W^u(x):=\bigcup_{n\geq0} f^{n}[V^s(f^{-n}(x))]$, which are increasing unions. 

\begin{definition}[Ergodic homoclinic class]
	Let $p$ be a periodic hyperbolic point. Define
	$$H(p):=\{x\in\RST: W^u(x)\pitchfork W^s(p)\neq\varnothing, W^u(p)\pitchfork W^s(x)\neq\varnothing\}.$$
\end{definition}
This definition is due to Rodriguez-Hertz, Rodriguez-Hertz, Tahzibi, and Ures (\cite{RodriguezHertz}).

\begin{definition}[Maximal irreducible components]

Let $\mathcal{R}$ be partition constructed in \textsection \ref{Forpartitione}.
\begin{enumerate}
\item  We say that $R\overset{n}{\to} S$ if there exists a word of length $n$, $(w_0,w_1,\ldots,w_{n-2},w_{n-1})$, s.t $w_i\in\mathcal{R}$ for all $i\leq n-1$, $w_0=R$ and $w_{n_1}=S$, and for all $i\leq n-2$ $f^{-1}[w_{i+1}]\cap w_i\neq\varnothing$.
\item We write for two partition elements $R,S\in\mathcal{R}$, we write $R\sim S$ if $\exists n,m\geq0$ s.t $R \overset{n}{\to} S$ and	 $S\overset{m}{\to}R$. $\sim$ defines a symmetric and transitive relation on $\mathcal{R}$. When restricting to $R\in \mathcal{R}$ s.t $R\sim R$, $\sim$ defines an equivalence relation.
\item Let $R\in\mathcal{R}$ s.t $R\sim R$, then $\langle R \rangle$ is the equivalence class of $R$ w.r.t $\sim$.
\end{enumerate}

We call equivalence classes of the equivalence relation $\sim$, {\em maximal irreducible components}.
	
\end{definition}

\begin{lemma}[Transitive coding]\label{transitivecoding}
	Given an ergodic homoclinc class, there is a maximal irreducible component of a $\widehat{\Sigma}$, $\widetilde{\Sigma}$, s.t $\widetilde{\Sigma}$ can code every conservative measure on $H(p)$. In particular, it codes every periodic orbit of every hyperbolic periodic point $q$ which is homoclinically related to $p$. 
\end{lemma}
 This lemma is due to Buzzi, Crovisier, and Sarig \cite{BCS}. This formulation of the lemma appears in \cite[\textsection~5]{LifeOfPi} (with another set taking the role of $\RST$ in the definition of ergodic homoclinic classes, but both sets carry all uniformly hyperbolic measures and it is sufficient for the proof).

\begin{definition}
Given an ergodic invariant hyperbolic measure $\mu$, we write $\chi^u(\mu):=\sum_{i=1}^d \max(\chi_i(\mu),0)$ and $\chi^s(\mu):=\sum_{i=1}^d \max(-\chi_i(\mu),0)$, where $\chi_1(\mu),\ldots,\chi_d(\mu)$ are the Lyapunov exponents of $\mu$ (with multiplicity). If $p$ is a hyperbolic periodic point, $\chi^{s/u}(p)$ means $\chi^{s/u}(\mu_p)$, where $\mu_p$ is the invariant measure supported on the orbit of $p$.\footnote{The Lyapunov exponents of a periodic orbit are well-defined, as working in coordinates, the differential can be put in a Jordan-block form, with respective eigenvalues. See \cite[Lemma~5.4]{LifeOfPi}
 for more details.}
\end{definition}

\begin{prop}[Existence of strictly weak foliations]\label{sweak}
	Let $H(p)$ s.t $\inf\{\chi^s(q):q\in H(p)\text{ periodic}\}=0$. Then $H(p)$ admits strictly weak stable manifolds. A similar statement applies to strictly weak unstable manifolds.
\end{prop}

\begin{proof}
By Lemma \ref{transitivecoding} we have a maximal irreducible component $\widetilde{\Sigma}$ which codes all $q\in H(p)$ which are periodic. Therefore, there are periodic words in $\widetilde{\Sigma}$, $\{\underline{w}^{(i)}\}_{i\geq0}$, $\underline{w}^{(i)}=(w^{(i)}_0,\ldots,w^{(i)}_{l_i-1})$ s.t the Lyapunov exponents $\chi_i:=\chi^s(\underline{w}^{(i)})\to 0$.

Fix $w_0^{(0)}$ and denote it by $R$. Let $n_i:=\min\{n\geq0: R\overset{n}{\to}w_0^{(i)}\}$ (which is well-defined since $\widetilde{\Sigma}$ is an irreducible component); and let $\underline{c}^{(i)}$ be a connecting word of such length. Similarly, let $m_i:=\min\{m\geq0: w_0^{(i)} \overset{m}{\to}R\}$; and let $\underline{d}^{(i)}$ be a connecting word of such length.

Define for all $i\geq1$, 
\begin{equation}\label{tooverpower}p_i:=\min\{p\geq0: p\cdot l_i\geq \frac{1}{\chi_i}\cdot d\cdot\log M_f \cdot (m_{i}+p_{i-1}+n_{i+1}+l_{i+1})\}.
\end{equation}

Consider the chain $$\underline{R}:=(\underline{w}^{(0)})\cdot \underline{c}^{(1)}\cdot (\underline{w}^{(1)})^{p_1}\cdot \underline{d}^{(1)}\cdots (\underline{w}^{(n)})\cdot \underline{c}^{(n+1)}\cdot (\underline{w}^{(n+1)})^{p_{n+1}}\cdot \underline{d}^{(n+1)}\cdots,$$

where $\cdot$ means a concatenation, and $(\cdot)^\ell$ means a concatenation of a word to itself $\ell$ times. By construction, this is as an admissible chain, which returns to $R$ infinitely often. Take any recurrent extension of $\underline{R}$ to the negative coordinates, $\underline{R}^\pm\in \widehat{\Sigma}^\#$, and write $x:=\widehat{\pi}(\underline{R}^\pm)$.

\medskip
For simplicity of presentation, we first prove the proposition under the assumption that $\mathrm{dim} H^s(x)=1$, and then show how it extends to higher dimensions. When $\mathrm{dim} H^s=1$, there is only one normalized stable tangent vector at the point (up to a sign), denoted by $\omega^s(\cdot)$. Then the quantity $S^2(\cdot)$ is well-defined. Consider the chain $\underline{w}^{(i,\pm)}$ which is the periodic extension of $\underline{w}^{(i)}$ to $\widehat{\Sigma}$. Write $w_j^{(i)}=\psi_{x_j^{(i)}}^{p^{s,(i)}_j,p^{u,(i)}_j}$. Then \eqref{lowerboundstoonice} tells us that

\begin{align}\label{overpowered}
\log|d_{\widehat{\pi}(\underline{w}^{(i,\pm)})}f^{l_i}\omega^s(\widehat{\pi}(\underline{w}^{(i,\pm)}))|=&\sum_{j=0}^{l_i-1}\log\sqrt{1-\frac{2}{S^2(f^j(\widehat{\pi}(\underline{w}^{(i,\pm)})))}}+\log S(\widehat{\pi}(\underline{w}^{(i,\pm)}))-\log S(\widehat{\pi}(\underline{w}^{(i,\pm)})),\nonumber\\
\log|d_{f^{t_i}(x)}f^{l_i}\omega^s(f^{t_i(x)})|=&\sum_{j=0}^{l_i-1}\log\sqrt{1-\frac{2}{S^2(f^{j+t_i}(x))}}+\log S(f^{t_i}(x))-\log S(f^{t_i+l_i}(x)).
\end{align}

The proof of Lemma \ref{chainbounds} allows to compare multiplicatively the quantities of both equations of \eqref{overpowered}, through the centers of charts $\{x^{(i)}_j\}_{j=t_i}^{t_i+l_i}$ (notice $S^2(\cdot)\geq \frac{2}{1-M_f^{-2}}$). The choice of $p_i$ in \eqref{tooverpower} guarantees, together with the fact that $\chi_i\to0$, that $\chi^s(x)=0$.

\medskip
Now, in order to see the proof when $\mathrm{dim}H^s(x)>1$, the role of $\omega^s(\cdot)$ is replaced by the unique (up to sign) normalized volume form on $H^s(\cdot)$. One can define

$$\widehat{S}^2(x):=2\sum_{m\geq0}|d_xf^m\omega^s(x)|^2,$$
which is finite by the Hadamard inequality and the fact that $x\in0$-summ. \eqref{lowerboundstoonice} derives similarly $|d_xf\omega^s(x)|=\frac{\widehat{S}(x)}{\widehat{S}(f(x))}\sqrt{1-\frac{2}{\widehat{S}^2(x)}}$, and so \eqref{overpowered} follows. The comparison of the two quantities in \eqref{overpowered} through the center of charts is done similarly by an analogous version of Lemma \ref{chainbounds}, which follows from Lemma \ref{chainbounds}: the map $\pi_x^{i,j}:H^s(f^{t_i+j}(x))\to H^s(x_j^{(i)})$ satisfies $\frac{|\pi_x^{i,j}\xi|}{|\xi|}=e^{\pm Q_\epsilon(x_j)^{\frac{\beta}{2}-\frac{1}{2\gamma}}}$ (Lemma \ref{boundimprove}). Therefore $\Jac(\pi_x^{i,j})=e^{\pm d Q_\epsilon(x_j)^{\frac{\beta}{2}-\frac{1}{2\gamma}}}$; whence the vector estimates of Lemma \ref{chainbounds} pull to volume form estimates as well. Strictly sub-exponential contraction of the volume form on $V^s(\underline{R})$ implies no tangent vector at $x$ which contracts exponentially fast.  
\end{proof}

\section{Codability}\label{codability}


In \textsection \ref{Forpartitione} we saw that $\widehat{\pi}[\widehat{\Sigma}^\#]=\RST$, so in particular we can code any measure which is carried by $\RST$. It is clear that every hyperbolic measure is carried by $0$-summ, but it is not clear for which ergodic measures almost every orbit is strongly temperable (recall Definition \ref{NUHsharp}). 

First, we should notice that every periodic hyperbolic orbit, and every uniformly hyperbolic measure, are carried by $\RST$, since $\|C_0^{-1}(\cdot)\|$ is bounded on every orbit. Being able to code uniformly hyperbolic measures is important, since it allows to ``lift" the pressure of H\"older continuous potentials (as well as the geometric potential) over hyperbolic measures. That is,

$$\sup\{h_\mu(f)+\int \phi d\mu:\mu\text{ hyperbolic}\}=\sup\{h_\mu(f)+\int \phi d\mu:\mu\text{ uniformly hyperbolic}\},$$
for a H\"older continuous $\phi:M\to\mathbb{R}$. This is true since every hyperbolic measure can be coded by a TMS which lifts the potential \cite{SBO}, and \cite[Theorem~4.3]{SarigTDF} shows that it can be approximated by the pressure over compact subshifts (which project to uniformly hyperbolic horseshoes). 

In particular, since our coding lifts all hyperbolic periodic orbits, one can always construct invariant probability measures on the shift space. Such measures can even be Markov measures and fully supported on the shift space. They project to invariant probability measures on the manifold; and as long as the dynamics we code are not uniformly hyperbolic, these measures will be non-uniformly hyperbolic, and carried by $\RST$. 

One can nonetheless ask, {\em in what generality can measures lift to $\widehat{\Sigma}$?} Below we give a general condition which is sufficient (though may be hard to test), and in a closing discussion later in this chapter, we review in what sense this condition is optimal. 

\begin{prop}\label{forcodability}
Let $x\in0\mathrm{-summ}$. If $\|C_0^{-1}(f^n(x))\|^2=o(|n|)$, $n\to\pm\infty$, then $x\in \ST$.
\end{prop}
In particular, this condition does not depend on the calibrating parameters $\gamma,\Gamma$, nor $\beta$, nor $\epsilon$. So an invariant probability measure carried by such orbits will be carried by any Markov partition constructed per this paper. 
\begin{proof}
	Let $\delta>0$ to be determined later. Let $n_\delta\in\mathbb{N}$ s.t $\forall n\geq n_\delta$, $\|C_0^{-1}(f^n(x))\|^2\leq \delta n$. Set $c_\delta:=\min\{\frac{C_{\beta,\epsilon}}{\|C_0^{-1}(f^n(x))\|^{2\gamma}}:0\leq n\leq n_\delta\}$.
	
	For every $n\geq 0$, define $q^+(f^n(x)):=\min\{c_\delta,\frac{C_{\beta,\epsilon}}{(\delta n)^{\gamma}}\}\leq \frac{C_{\beta,\epsilon}}{\|C_0^{-1}(f^n(x))\|^{2\gamma}}$. In order to bound the variation of $q^+$, it is enough to bound the variation of $\frac{C_{\beta,\epsilon}}{(\delta n)^{\gamma}}$ for $n\geq n_\delta+1$. Then for any $n\geq n_\delta+1$, (w.l.o.g $n_\delta$ is big enough so $(1-\frac{1}{n})\geq e^{-2\frac{1}{n}}$)
	
	\begin{align*}
	\frac{\frac{C_{\beta,\epsilon}}{(\delta n)^{\gamma}}}{\frac{C_{\beta,\epsilon}}{(\delta (n-1))^{\gamma}}}\geq (\frac{n-1}{n})^\gamma\geq e^{-2\frac{\gamma}{n}}=  e^{-2\gamma\cdot \frac{\delta}{(C_{\beta,\epsilon})^\frac{1}{\gamma}}\cdot (\frac{C_{\beta,\epsilon}}{(\delta n)^\gamma})^\frac{1}{\gamma}}\geq 	e^{-\Gamma\cdot (\frac{C_{\beta,\epsilon}}{(\delta n)^\gamma})^\frac{1}{\gamma}},
	\end{align*}
where the last inequality is when we choose $\delta>0$ sufficiently small.  Then we get,
$$I(\frac{C_{\beta,\epsilon}}{(\delta n)^\gamma})=\frac{C_{\beta,\epsilon}}{(\delta n)^{\gamma}}e^{\Gamma\cdot (\frac{C_{\beta,\epsilon}}{(\delta n)^\gamma})^\frac{1}{\gamma}}\geq \frac{C_{\beta,\epsilon}}{(\delta (n-1))^{\gamma}},\text{ whence }\frac{C_{\beta,\epsilon}}{(\delta n)^\gamma}\geq I^{-1}(\frac{C_{\beta,\epsilon}}{(\delta (n-1))^{\gamma}}).$$

This construction can be done for $f^k(x)$, $k\in\mathbb{Z}$, the same way. Therefore the Ledrappier formula is well-defined for every $k\in\mathbb{Z}$,

$$p^s_k:=\max\{t\in\mathcal{I}:I^{-n}(t)\leq \frac{C_{\beta,\epsilon}}{\|C_0^{-1}(f^n(f^k(x)))\|^{2\gamma}}\text{ for all }n\geq0\}.$$
	
Similarly, the following is well-defined as well:

$$p^u_k:=\max\{t\in\mathcal{I}:I^{-n}(t)\leq \frac{C_{\beta,\epsilon}}{\|C_0^{-1}(f^{-n}(f^k(x)))\|^{2\gamma}}\text{ for all }n\geq0\}.$$

It follows from the Ledrappier formula that	for all $k\in\mathbb{Z}$,
\begin{enumerate}
\item $p^s_k,p^u_k\leq \frac{C_{\beta,\epsilon}}{\|C_0^{-1}(f^k(x))\|^{2\gamma}}$,
\item 
$p^u_{k+1}=\min\{I(p_k^u),\frac{C_{\beta,\epsilon}}{\|C_0^{-1}(f^{k+1}(x))\|^{2\gamma}}\}$, $p^s_{k-1}=\min\{I(p_k^s),\frac{C_{\beta,\epsilon}}{\|C_0^{-1}(f^{k-1}(x))\|^{2\gamma}}\}$.
\end{enumerate}
Therefore $q(f^k(x)):=\min\{p^s_k,p^u_k\}$ is a strongly tempered kernel for $x$, as in Lemma \ref{lemma131}.
\end{proof}

\subsection{Examples and discussion}
In this section we consider a family of systems which is called {\em almost Anosov}, introduced by Hu and Young in \cite{YoungCounterExample}, and we mention a few results regarding our coding of these systems. The full details appear in an upcoming work. We then have a short discussion regarding optimality of our construction (i.e the notions of $0$-summability and strong temperability, and the condition in Proposition \ref{forcodability} for codability).

	\begin{definition}[Almost-Anosov]
	Let $f\in \mathrm{Diff}^r(\mathbb{T}^2)$, $r\geq2$, be a topologically transitive diffeomorphism of the two dimensional torus. Assume that $f$ admits an invariant splitting of the tangent bundle everywhere $T_xM=E^s(x)\oplus E^u(x)$ and an indifferent fixed point $p$, s.t
	\begin{enumerate}
	\item $\exists K^s<1$, $\forall x$ $|d_xf|_{E^s}|\leq K^s$
	\item $\exists K^u(\cdot)\geq1$ continuous on $M$, s.t $\forall x$ $|d_xf|_{E^u}|\geq K^u(x)$,
	\item For any $x\neq p$, $K^u(x)>1$, and $|d_pf|_{E^u}|=K^u(p)=1$.
	\end{enumerate}
\end{definition}

Hu and Young introduce this class of systems to show that while these systems are very nicely behaving (``almost" uniform hyperbolicity), the admit no SRB measures. 

The following result appears with details in an upcoming work, and for this paper we assume it. 
\begin{theorem}\label{inmyexamples}
Given an almost Anosov diffeomorphism, every $x\neq p$ is $0$-summable.
\end{theorem}

\begin{cor}
Given an almost Anosov system, every invariant measure ($\neq \delta_p$) is carried by $\RST$ (for any choice of the calibrating parameters), and hence codable.
\end{cor}
\begin{proof}
Let $\mu$ be an invariant probability measure. W.l.o.g $\mu$ is ergodic. If $\mu(p)=0$, then $\mu$-a.e point is summable by Theorem \ref{inmyexamples}. In addition, $|d_xf|_{H^u(x)}|\geq 1$ everywhere, therefore
\begin{align}\label{fordiscussion}
U^2(f(x))=&2\sum_{m\geq0}|d_{f(x)}f^{-m}|_{H^u(f(x))}|^2=2+|d_{f(x)}f^{-1}|_{H^u(f(x))}|^2\cdot 2\sum_{m\geq0}|d_{x}f^{-m}|_{H^u(x)}|^2\nonumber\\
=&2+|d_{f(x)}f^{-1}|_{H^u(f(x))}|^2\cdot U^2(x)\leq 2+ U^2(x).
\end{align}

Thus, $U^2(f^n(x))\leq U^2(x)+2n$ for all $n\geq0$. Then $(U^2\circ f-U^2)^+\in L^1(\mu)$. One can then show that necessarily $U^2\circ f-U^2\in L^1(\mu)$ and $\int (U^2\circ f-U^2)d\mu=0$ (full details appear in the upcoming work where we show Theroem \ref{inmyexamples}). Then by the point-wise ergodic theorem, $\mu$-a.e,
$$\frac{U^2\circ f^n}{n}=\frac{1}{n}\sum_{k=0}^{n-1}(U^2\circ f-U^2)\circ f^k+\frac{U^2}{n}\to0.$$
Similarly $U^2\circ f^{-n}=o(n)$, $n\to\infty$, $\mu$-a.e.

Finally, notice that the continuous splitting on a compact manifold implies that $\sphericalangle(H^s(x),H^u(x))$ is bounded from below everywhere. In addition, $S^2\leq 2\sum_{m\geq0}(K^s)^{2m}=\frac{2}{1-(K^s)^2}<\infty$. Then $\|C_0^{-1}(x)\|^2\leq C_M\cdot U^2(x)$ for all $x\neq p$, where $C_M$ is some global constant. Thus $\|C_0^{-1}\circ f^n\|^2=o(|n|)$, $n\to\pm\infty$, $\mu$-a.e. By Proposition \ref{forcodability}, $\mu$ is carried by $\RST$ and is codable.
\end{proof}

\noindent\textbf{Remark:} We saw that $\widehat{\pi}[\widehat{\Sigma}^\#]=\RST$ carries every invariant probability measure in the almost Anosov setup. Furthermore, we can see that in some cases we must have measures with arbitrarily small Lyapunov exponents, while we can code all of them simultaneously in a finite-to-one almost everywhere fashion. This follows from the fact that almost Anosov systems can be achieved as a perturbation (or a ``slow-down") of a mixing hyperbolic toral automorphism (which admit a mixing finite Markov partition). One can make the element of the partition which contains $p$ arbitrarily small, call it $a$. Fix another element $b$, and construct a word from $b$ to $b$ with an arbitrarily long segment $a\ldots a$ (admissible since $p$ is a fixed point). Then This codes a periodic (hence hyperbolic) orbit, with arbitrarily small Lyapunov exponent. One then gets the existence of strictly weak manifolds by Proposition \ref{sweak}.

\medskip
After establishing the relationship between being codable and a sub-linear growth-rate of $\|C_0^{-1}(f^n(\cdot))\|^2$, one can ask the following three questions:
\begin{enumerate}
\item Can the definition of $0$-summable orbits be changed to get stronger contractions in charts? (Thus allowing us to require weaker properties than strong temperability).
\item Can the definition of strong temperability be weakened and still allow for a shadowing theory?
\item Is there a weaker condition than sub-linear growth-rate of $\|C_0^{-1}(f^n(\cdot))\|^2$ to guarantee codability?
\end{enumerate}
We will give non-formal heuristic answers to these questions, which justify the ``optimality" of the notions in this paper.

\begin{enumerate}
\item The answer to this question has three parts. First, one can ask wether we can consider orbits which converge in a power higher than $2$. First notice that for every hyperbolic orbit, the sum converges regardless of the power, and that every periodic orbit which converges in some power must be hyperbolic. Furthermore, every measure which is uniformly summable in some power must be hyperbolic. Therefore one does not lose much by choosing the power $2$, but gains a lot by the fact that the parallelogram law dictates that only for the power $2$, the scaling functions will constitute a norm on the tangent space which is induced by an inner-product. This bilinear form which induces the scaling functions (i.e the Lyapunov inner-product) is paramount to our construction of the Lyapunov change of coordinates.

One can also ask if perhaps adding an exponential term in the definition of summability, which depends on the Lyapunov exponent of the orbit, can yield a simpler construction while still coding all ``sufficiently nice" hyperbolic orbits. The heuristic answer is that this is conceptually wrong. A Lyapunov exponent is not a property of a point (resp. symbol), it is a property of the orbit (resp. chain). Imposing an edge condition which requires some correspondence between the Lyapunov exponents of neighboring charts, goes against the Markov structure which we sought to construct. Without imposing some correspondence on the exponents through an edge condition, one cannot use a composition of Pesin charts $d_xf^n=C_0(f^n(x))\circ(C_0^{-1}(f^n(x))d_{f^{n-1}(x)}f C_0(f^{n-1}x))\circ \cdots\circ(C_0^{-1}(f(x))d_xf C_0(x))\circ C_0^{-1}(x)$ without accumulating errors. 

Lastly, once can ask if some weights can be added in the definition of summability to aid/hinder convergence. A natural choice of such weights which has been studied before is exponential weights. Their short-comes are clear, as they restrict us to $\chi$-hyperbolic orbits, where $\chi$ depends on the exponent in the weights. If one chooses rates where the ratio between them does not go to a constant number, then shifting the sequence does not change it by a factor, and again, creates discrepancy. For these heuristic reason, the notion of $0$-summability is quite natural. 
\item \eqref{fordiscussion} shows us that even for a family of systems which are very nicely behaving, one gets naturally the relation $\frac{1}{U^{2\gamma}\circ f}\geq \frac{1}{U^{2\gamma}}e^{-2(\frac{1}{U^{2\gamma}})^\frac{1}{\gamma}}$. Therefore, heuristically, when we wish to give a general description to a phenomenon which has a relatively easily understood example, we cannot expect behavior nicer than that of the example generally. It turns out that strong temperability, which utilizes exactly the relation of the form  $\frac{C_{\beta,\epsilon}}{\|C_0^{-1}\|^{2\gamma}\circ f}\geq \frac{C_{\beta,\epsilon}}{\|C_0^{-1}\|^{2\gamma}}e^{-\Gamma(\frac{C_{\beta,\epsilon}}{\|C_0^{-1}\|^{2\gamma}\circ f})^\frac{1}{\gamma}}$, is sufficient for a shadowing theory in the general case.  
\item In Lemma \ref{FaveForNow} we saw that being strongly temperable implies that $\|C_0^{-1})\|^{2\gamma}\circ f^{n}=O(n^\gamma)$, where the big $O$ constant depends on $\gamma$ and $\Gamma$. Therefore an orbit which is coded by our coding or is strongly temperable for any choice of the calibrating parameters, will have to satisfy $\|C_0^{-1})\|^{2}\circ f^{n}=o(n)$. This implies that this condition captures a significant portion of the orbits we wish to code. 
\end{enumerate}

\section{Temperability and rate of contraction}\label{duality}

In this section we present a relationship between two properties which are studied in this paper: Temperability and the rate of contraction on stable and unstable manifolds.

We begin by showing that as long as an orbit visits a level set of $S^2$ with some positive frequency, it must contract tangent vectors on the stable direction exponentially fast asymptotically (similarly for $U^2$ and the unstable direction). This shows us that hyperbolicity emerges as a phenomenon of positive recurrence to level sets of the scaling function. In particular, every invariant probability measure carried by $0$-summ must be hyperbolic. However, the lack of positive recurrence does not imply the lack of recurrence. 

We therefore continue to study the rate of contraction of tangent vectors in the stable or unstable subspaces, without assuming positive recurrence. We show that the temperability properties  induce a rate of contraction. By temperability, we mean asymptotic growth-rate bounds for $\|C_0^{-1}\|^2$ (or $S^2/U^2$). We saw in Proposition \ref{forcodability} that a sub-linear growth-rate implies being strongly temperable. Proposition \ref{rateo} considers more rates.

\subsection{Hyperbolicity as positive recurrence of the scaling functions}

\begin{prop}
	Let $x\in0\mathrm{-summ}$, and let $\xi\in H^s(x)$ s.t $|\xi|=1$. Write $\xi_k:=\frac{d_xf^k\xi}{|d_xf^k\xi|}$. Then if $\lim\frac{1}{n}\log S(f^n(x))=0$,
	\begin{enumerate}
		\item $\limsup\limits_{n\to\infty}\frac{1}{n}\log |d_xf^n\xi|<0$ if and only if $\exists C>0$ s.t $\liminf\limits_{n\to\infty} \frac{1}{n}\#\{0\leq i\leq n: S^2(f^k(x),\xi_k)\leq C\}>0$.
		\item $\liminf\limits_{n\to\infty}\frac{1}{n}\log |d_xf^n\xi|<0$ if and only if $\exists C>0$ s.t $\limsup\limits_{n\to\infty} \frac{1}{n}\#\{0\leq i\leq n: S^2(f^k(x),\xi_k)\leq C\}>0$.
	\end{enumerate}
\end{prop}

A similar statement holds for the backwards orbit and $U^2$.
\begin{proof}

Write $\zeta^u_n:=\log\frac{1}{\sqrt{1-\frac{2}{S^2(f^n(x),\xi_n)}}}\in (0,\log M_f]$. By \eqref{lowerboundstoonice}, $\forall n\geq0$,

\begin{equation}\label{enddddalready}\frac{1}{n}\log |d_xf^n\xi|=-\frac{1}{n}\sum_{k=0}^{n-1}\zeta^u_k+\frac{1}{n}\log S(x,\xi)-\frac{1}{n}\log S(f^n(x),\xi_n)\leq -\frac{1}{n}\sum_{k=0}^{n-1}\zeta^u_k+\frac{1}{n}\log S(x,\xi).\end{equation}

Then the following claim implies the proposition.

Claim:
\begin{enumerate}
\item	 $\limsup\frac{1}{n}\sum_{i=0}^{n-1}\zeta^u_i>0$ if and only if $\exists \delta>0$ s.t $\limsup\frac{1}{n}\#\{0\leq i\leq n: S^2(f^i(x),\xi_i)\leq \frac{1}{\delta}\}>\delta$,
\item  $\liminf\frac{1}{n}\sum_{i=0}^{n-1}\zeta^u_i>0$ if and only if $\exists \delta>0$ s.t $\liminf\frac{1}{n}\#\{0\leq i\leq n: S^2(f^i(x),\xi_i)\leq \frac{1}{\delta}\}>\delta$.
\end{enumerate}

Proof: We start by showing item (1). Let $n_k$ be a subsequence.

 \underline{$\Rightarrow$:} Assume for contradiction that for any $\delta>0$  there exists $k_\delta\geq 0$ s.t $\forall k\geq k_\delta$,
$$\frac{1}{n_k}\#\{i\leq n_k: S^2(f^i(x),\xi_i)\leq \frac{1}{\delta}\}\leq \delta.$$
Notice, if $S^2(f^i(x),\xi_i)\geq \frac{1}{\delta}$, then $\zeta^u_i\leq \log\frac{1}{\sqrt{1-2\delta}}\leq 2\delta$ (for all sufficiently small $\delta>0$). Therefore,

\begin{equation}\label{arewetherealready}
	\frac{1}{n_k}\sum_{i=0}^{n_k-1}\zeta^u_i=\frac{1}{n_k}\sum_{i\leq n_k:S^2(f^i(x),\xi_i)< \frac{1}{\delta}}\zeta^u_i+\frac{1}{n_k}\sum_{i\leq n_k:S^2(f^i(x),\xi_i)\geq\frac{1}{\delta}}\zeta^u_i\leq \log M_f\cdot \delta+2\delta \leq (\log M_f+2)\cdot\delta.
\end{equation}

This is a contradiction when $0<\delta<\frac{1}{2}\frac{\limsup\frac{1}{n}\sum_{i=0}^{n-1}\zeta^u_i}{\log M_f+2}=\frac{1}{2}\frac{\lim\frac{1}{n_k}\sum_{i=0}^{n_k-1}\zeta^u_i}{\log M_f+2}$, for some subsequence $n_k$.

\underline{$\Leftarrow$:} Assume that $\exists \delta>0$ s.t $\limsup\frac{1}{n}\#\{0\leq i\leq n: S^2(f^i(x),\xi_i)\leq \frac{1}{\delta}\}>\delta$. Notice that if $S^2(f^i(x),\xi)\leq \frac{1}{\delta}$ then $\zeta^u_i\geq \log \frac{1}{\sqrt{1-2\delta}}\geq \delta$. Then by \eqref{arewetherealready}, for $n_k$ s.t $\frac{1}{n_k}\#\{0\leq i\leq n_k: S^2(f^i(x),\xi_i)\leq \frac{1}{\delta}\}>\frac{\delta}{2}$,

$$\frac{1}{n_k}\sum_{i=0}^{n_k-1}\zeta^u_i>\frac{1}{n_k}\sum_{i\leq n_k:S^2(f^i(x),\xi_i)< \frac{1}{\delta}}\zeta^u_i\geq \frac{\delta}{2} \cdot \delta>0.$$

This concludes item (1). The proof of item (2) is similar, and we leave it to the reader.	
\end{proof}

\noindent\textbf{Remark:} \eqref{enddddalready} shows that for the ``if" direction, the assumption $\lim \frac{1}{n}\log S\circ f^n=0$ is not needed. This is in fact the direction which implies how hyperbolicity emerges from positive recurrence.

\subsection{Temperability and rate of contraction}

Let $F(x):=e^{\alpha x}$, where $\alpha,x>0$. We get that $\frac{F(x)}{F'(x)}=\frac{1}{\alpha}<\infty$. That is, for an exponential $F$, this quantity is bounded away from infinity uniformly in $x$ (however big the bound may be). 

We wish to consider all rates which are strictly weaker than exponential, in the following sense: Let $F:\mathbb{R}^+\to \mathbb{R}^+$ be a differentiable function s.t $$\lim_{x\to\infty}\frac{F(x)}{F'(x)}=\infty,$$
and assume that $\frac{F(x)}{F'(x)}$ is non-decreasing. So, as long as $\frac{F(x)}{F'(x)}$ grows and is unbounded (no matter how slowly it grows), we consider it. The following proposition tells us that for any rate which is weaker than exponential in the sense above, there is a temperability rate which implies such contraction on the stable or unstable subspaces. 

\begin{prop}\label{rateo}
	Let $F:\mathbb{R}^+\to \mathbb{R}^+$ differentiable s.t $\frac{F(x)}{F'(x)}\uparrow\infty$. 
	Let $x\in 0\mathrm{-summ}$, $\xi\in H^s(x)$ s.t $|\xi|=1$. Write $\xi_n:=\frac{d_xf^n\xi}{|d_xf^n\xi|}$, and assume that $\exists C>0$ s.t $
	\forall n\geq 0$, $\frac{1}{2}S^2(f^n(x),\xi_n)\leq C \frac{F(n)}{F'(n)}$. Then $|d_xf^n\xi|= O(\frac{1}{(F(n))^\frac{1-\tau}{C}})$, $\forall \tau>0$. A similar statement holds for contraction on $H^u(x)$.
\end{prop}

\begin{proof}
We start by proving an inequality. Let $n_1> n_0\geq 1$. Write $G(t):=\frac{F'(t)}{F(t)}$, which is monotone, 
then,
\begin{equation}\label{finallyIwannaGoHome}
|\int_{n_0}^{n_1}G(t)dt-\sum_{k=n_0}^{n_1-1}G(k)	|=\sum_{k=n_0}^{n_1-1}G(k)-\int_{k}^{k+1}G(t)dt\leq \sum_{k=n_0}^{n_1-1}G(k)-G(k+1)=G(n_0)-G(n_1)\leq G(1).
\end{equation}

We now may continue with proving the proof. Write $\zeta_n^u:=\log\frac{1}{\sqrt{1-\frac{2}{S^2(f^n(x),\xi_n)}}}$. Let $\tau>0$, then for all $n\geq n_x$, where $n_x$ is large enough so $\frac{F'(n_x)}{C F(n_x)}$ is sufficiently small, $\zeta^u_n\geq \frac{(1-\tau)}{C}\frac{F'(n)}{F(n)}$.	 Then, for all $n\geq n_x$, by \eqref{enddddalready},
\begin{align*}
|d_xf^n\xi|\leq &S(x,\xi)\cdot\exp\left(-\sum_{k=n_x}^{n-1}\frac{(1-\tau)}{C}\frac{F'(k)}{F(k)}\right)\leq S(x) \cdot	\exp\left(-\frac{(1-\tau)}{C}\sum_{k=n_x}^{n-1}\frac{F'(k)}{F(k)}\right)\\
\leq&  S(x) \cdot	\exp\left(\frac{F'(1)}{F(1)C}-\frac{(1-\tau)}{C}\int_{n_x}^{n}\frac{F'(t)}{F(t)}dt\right)=C_x\cdot \exp\left(-\frac{(1-\tau)}{C}\int_{n_x}^{n}(\log F(t))'dt\right)\text{ }\because\eqref{finallyIwannaGoHome}\\
=& C_x\cdot \exp\left(-\frac{(1-\tau)}{C}(\log F(n)-\log F(n_x))dt\right)=C_x'\cdot \frac{1}{(F(n))^\frac{1-\tau}{C}}.
\end{align*}

In particular, when $\|C_0^{-1}(f^n(x))\|^2=o(n)$ then $\|d_xf^n|_{H^s(x)}\|$ decreases super-polynomially. Unfortunately, this relationship is not easily reversible. That is, super-polynomial contraction does not necessarily imply strong temperability, since we cannot commute a bound on the sum in the exponent, into a bound onto its terms (while the other way around works).

\end{proof}

\bibliographystyle{alpha}
\bibliography{Elphi}

\end{document}